\documentclass[12pt,oneside,english,shortlabels]{amsart}

\usepackage{babel}
\usepackage{amstext}
\usepackage{amssymb}

\usepackage{amsmath} 
\usepackage{latexsym}
\usepackage{graphicx} 

\usepackage{mathptmx}

\usepackage{graphicx} 

\makeatletter

\usepackage{amsthm}
\usepackage{enumitem}		% customizable list environments
\usepackage{mathrsfs} %for the \mathscr
\providecommand{\noopsort}[1]{}

\numberwithin{equation}{subsection}

\theoremstyle{definition} 

 \newtheorem{definition}{Definition}[section]
 \newtheorem{remark}[definition]{Remark}

\newtheorem*{notation}{Notations}

\theoremstyle{plain}      

 \newtheorem{proposition}[definition]{Proposition}
 \newtheorem{theorem}[definition]{Theorem}
 \newtheorem{corollary}[definition]{Corollary}
 \newtheorem{lemma}[definition]{Lemma}
 \newtheorem{rewritingtrail}[definition]{{\it (Rewriting trail)}}

\newtheorem{conjecture}{Conjecture}

\makeatletter
\newcommand*{\house}[1]{
  \mathord{
    \mathpalette\@house{#1}
  }
}
\newcommand*{\@house}[2]{
  \dimen@=\fontdimen8 %
      \ifx#1\scriptscriptstyle\scriptscriptfont
      \else\ifx#1\scriptstyle\scriptfont
      \else\textfont\fi\fi
      3 %
  \sbox0{%
    $#1%
      \vrule width\dimen@\relax
      \overline{%
        \kern2\dimen@
        \begingroup % to keep changes of \dimen@ in #2 local
          #2%
        \endgroup
        \kern2\dimen@
      }
      \vrule width\dimen@\relax
      \mathsurround=1.5\dimen@ % outside margin
    $
  }
  \ht0=\dimexpr\ht0-\dimen@\relax
  \dp0=\dimexpr\dp0+2\dimen@\relax
  \vbox{
    \kern\dimen@ 
    \copy0 
  }
}

\def\dyg{{\rm dyg}}

\newcommand{\lo}{{\rm Log\,}}

\newcommand{\rb}{\mathbb{R}}
\newcommand{\pb}{\mathbb{P}}

\newcommand{\cb}{\mathbb{C}}

\newcommand{\zb}{\mathbb{Z}}
\newcommand{\tb}{\mathbb{T}}
\newcommand{\qb}{\mathbb{Q}}
\newcommand{\nb}{\mathbb{N}}

\newcommand{\rc}{\mathcal{R}}

\newcommand{\dc}{\mathcal{D}}
\newcommand{\kc}{\mathcal{K}}
\newcommand{\lc}{\mathcal{L}}

\def\11{{\mathbf 1}}

\theoremstyle{remark}
\newtheorem{exampl}[subsubsection]{Example}

\def\bee{\begin{exampl}}
\def\eee{\end{exampl}}
\def\bn{\begin{notation}}
\def\en{\end{notation}}
\def\br{\begin{remark}}
\def\er{\end{remark}}
\def\bp{\begin{prop}}
\def\ep{\end{prop}}
\def\bpr{\begin{proof}}
\def\epr{\end{proof}}
\def\bt{\begin{thm}}
\def\et{\end{thm}}
\def\be{\begin{equation}}
\def\ee{\end{equation}}
\def\bl{\begin{lem}}
\def\el{\end{lem}}
\def\bc{\begin{cor}}
\def\ec{\end{cor}}
\def\bd{\begin{defn}}
\def\ed{\end{defn}}

\numberwithin{equation}{subsection}

\author{Jean-Louis Verger-Gaugry}
\thanks{}
\address{
LAMA, CNRS UMR 5127,
Univ. Grenoble Alpes, Univ. Savoie Mont~Blanc,
F - \!73000 Chamb\'ery, \!France}
\email{Jean-Louis.Verger-Gaugry@univ-smb.fr}

\title[A proof of the Conjecture of Lehmer]
{A proof of the Conjecture of Lehmer}

\begin{document}

%\dedicatory{\today}

%-----------------------------------------------------
%%%%%%%%%%%%%%%%%%%%%JLVG%%%%%%%%%%%%%%%%%%%%%%

\def\dyg{{\rm dyg}}

%%%%%%%%%%%%%%%%%%%%%%%%%%%%%%%%%%%%%%%%%%%%%%%

\title{A proof of the Conjecture of Lehmer}

%\maketitle

\begin{abstract}
The Conjecture of Lehmer is proved to be true. 
The proof mainly relies upon: 
(i) the properties of the Parry Upper functions $f_{\house{\alpha}}(z)$
associated with the
dynamical zeta functions $\zeta_{\house{\alpha}}(z)$
of the R\'enyi--Parry arithmetical dynamical 
systems ($\beta$-shift),
for $\alpha$ a reciprocal algebraic 
integer 
of house 
$\house{\alpha}$ greater than 1,
(ii) the discovery of lenticuli of poles of
$\zeta_{\house{\alpha}}(z)$ which uniformly 
equidistribute at the limit
on a limit ``lenticular" arc of the unit circle, 
when $\house{\alpha}$ tends to $1^+$,
giving rise to a continuous lenticular
minorant ${\rm M}_{r}(\house{\alpha})$ of the 
Mahler measure
${\rm M}(\alpha)$,
(iii) the Poincar\'e asymptotic expansions of these poles
and of this minorant ${\rm M}_{r}(\house{\alpha})$
as a function of the dynamical degree. 
The Conjecture of
Schinzel-Zassenhaus is proved to be true.
A Dobrowolski type minoration of the Mahler measure M$(\alpha)$ is obtained. 
The universal minorant of
M$(\alpha)$ obtained is
$\theta_{\eta}^{-1} > 1$, for some integer
$\eta \geq 259$, where
$\theta_{\eta}$ is the positive real root
of $-1+x+x^{\eta}$.
The set of Salem numbers is shown to be
bounded from below by the Perron number
$\theta_{31}^{-1} = 1.08545\ldots$, dominant root
of the
trinomial $-1 - z^{30} + z^{31}$.
Whether Lehmer's number is the smallest
Salem number remains open.
For sequences of algebraic integers 
of Mahler measure smaller than the 
smallest Pisot number $\Theta = 1.3247\ldots$, 
whose houses have
a dynamical degree tending to infinity,
the Galois orbit measures 
of conjugates are proved to
converge 
towards the Haar measure
on $|z|=1$ (limit equidistribution). 
The dynamical zeta function is used to investigate the domain of very small
Mahler measures of algebraic integers
in the range $(1, 1.176280\ldots]$, if any.
\end{abstract}

\maketitle

\vspace{0.7cm}

Keywords:
Lehmer conjecture,
Schinzel-Zassenhaus conjecture,
Mahler measure,
minoration,
Dobrowolski inequality,
asymptotic expansion, 
transfer operator,
dynamical zeta function,
R\'enyi-Parry $\beta$-shift,
Parry Upper function,
Perron number,
Pisot number, 
Salem number, 
Parry number,
limit equidistribution.
\vspace{0.5cm}

2020 Mathematics Subject Classification:
11K16, 11K36, 11M41, 11R06, 11R09, \newline 30B10, 30B40, 37C30, 37N05, 03D45.

\tableofcontents  %comment this command in if you want a table of contents

\newpage
\section{Introduction}
\label{S1}

The question asked by Lehmer in
\cite{lehmer} (1933) about the existence of
integer univariate polynomials of Mahler measure
arbitrarily close to one became a conjecture.
Lehmer's Conjecture 
is stated as follows:

\begin{conjecture}[Lehmer's Conjecture]
\label{ConjL}
There exists an universal constant $c > 0$ such 
that 
the Mahler measure M$(\alpha)$ satisfies
M$(\alpha) \geq 1 + c$ for all
nonzero algebraic numbers $\alpha$, not being a 
root of unity.
\end{conjecture}

%Many works attempted to solve it, 
%e.g. by Amoroso \cite{amoroso2}
%\cite{amoroso3},
%Blansky and Montgomery \cite{blanskymontgomery},
%Boyd \cite{boyd4} \cite{boyd5}, 
%Cantor and Strauss \cite{cantorstrauss}, 
%David and Hindry \cite{davidhindry},
%Dobrowolski \cite{dobrowolski2}, 
%Dubickas \cite{dubickas2}
%\cite{dubickas4},  
%Hindry and Silverman \cite{hindrysilverman}, 
%Langevin \cite{langevin}, 
%Laurent \cite{laurent},
%Louboutin \cite{louboutin}, 
%Masser \cite{masser4},
%Mossinghoff, Rhin and Wu \cite{mossinghoffrhinwu},
%Schinzel \cite{schinzel2}, 
%Silverman \cite{silverman},
%Smyth \cite{smyth5} \cite{smyth6},
%Stewart \cite{stewart} \cite{stewart2}, 
%Waldschmidt \cite{waldschmidt} \cite{waldschmidt2}. 

If $\alpha$ is a nonzero algebraic integer, 
M$(\alpha) = 1$
if and only if $\alpha=1$ or is a root of unity
by Kronecker's Theorem (1857)
\cite{kronecker}.
Lehmer's Conjecture 
asserts a discontinuity
of the value of ${\rm M}(\alpha)$,
$\alpha \in \mathcal{O}_{\overline{\qb}}$, at
1.
In the panorama \cite{vergergaugrySurvey} 
the meaning 
of this discontinuity is evoked  
in number theory and in several domains
by analogy where it admits 
different reformulations.

In this note a proof of
the Conjecture of Lehmer is proposed,
which is based on the dynamics of 
algebraic numbers
(cf Sections \S \ref{S2}
and \S \ref{S3}):
more precisely
on the
dynamical system of numeration of the
beta-shift in the sense of R\'enyi 
and Parry and on the
generalized Fredholm Theory which 
is associated to it by the
transfer operators
of the $\beta$-transformation.
It brings the
dynamical zeta functions 
$\zeta_{\beta}(z)$ of the $\beta$-shift
into play.
An ad-hoc theory of 
divergent series
(Poincar\'e asymptotic expansions)
is introduced for
formulating
the {\it lenticular 
poles} of
the
dynamical zeta functions 
$\zeta_{\beta}(z)$
as  functions
of the dynamical degree of $\beta$.
It allows to establish
an universal 
minoration of the Mahler measure 
M$(\alpha)$
for any nonzero algebraic integer $\alpha$ 
which is not a root of unity, by
a new Dobrowolski type minoration.

Let us reduce the problem. 
If $\alpha$ is an algebraic number 
which is not an 
algebraic integer, then 
M$(\alpha) \geq 2$. If
$\alpha$ is an algebraic integer for which 
the minimal polynomial is not reciprocal,
then M$(\alpha) \geq \Theta$ 
the smallest Pisot number, 
by a Theorem of C. Smyth
\cite{smyth}.
For every reciprocal
algebraic integer $\alpha$,
such that $|\alpha| \geq c$, where
$c > 1$ is a (fixed) constant, then
M$(\alpha) \geq c$. 
Therefore
the problem of Lehmer amounts
to find an universal minoration
of 
M$(\alpha)$ when 
$|\alpha|$ tends to $1^+$. 
It is a limit problem.
The problem is strengthened
when the condition ``when 
$|\alpha|$ tends to $1^+$" is replaced by
``when 
$\house{\alpha}$ tends to $1^+$",
taking into account all the conjugates 
at the same time.
This limit problem constitutes the
statement of the Conjecture of 
Schinzel-Zassenhaus
\cite{cassels}, as follows:

\begin{conjecture}[Schinzel - Zassenhaus's Conjecture]
\label{ConjSZ}
Denote by ${\rm m}_{h}(n)$ the minimum of the 
houses $\house{\alpha}$ 
of the algebraic integers $\alpha$
of degree $n$ which are not a root of unity.
There exists a (universal) constant $C > 0$
such that
\begin{equation}
\label{CJschzass}
{\rm m}_{h}(n) \geq 1 + \frac{C}{n}, \qquad n \geq 2.
\end{equation}
\end{conjecture}
Schinzel - Zassenhaus's Conjecture is 
a consequence of Lehmer's conjecture:
if $r$ is the number of conjugates
$\alpha^{(i)}$ of $\alpha$ 
satisfying
$|\alpha^{(i)}| > 1$, then
${\rm M}(\alpha) \leq \house{\alpha}^r$.
Thus
$$\house{\alpha} \geq 
{\rm M}(\alpha)^{1/r}
\geq 
{\rm M}(\alpha)^{1/\deg(\alpha)}
\geq
(1+c)^{1/\deg(\alpha)}
\geq 
1 + \frac{C}{\deg(\alpha)}.
$$
For Lehmer's Conjecture, 
due to the invariance of the Mahler measure
M$(\alpha)$ by the transformations
$z \to \pm z^{\pm 1}$ and 
$z \to \pm \overline{z}^{\pm 1}$,
it is sufficient to consider the
two following cases: 
\begin{itemize}
\item[{\bf (i)}] $\alpha$ real reciprocal
algebraic integer $> 1$, in which 
case $\alpha$ is generically 
named $\beta$,
\item[{\bf (ii)}] $\alpha$ complex reciprocal algebraic 
integer, $|\alpha| > 1$,
with $\arg(\alpha) \in (0, \pi/2]$,
\end{itemize}
with $|\alpha| > 1$ sufficiently close to 1 in both cases.
In Section $\S$ \ref{S6} we show how the 
nonreal complex case {\bf (ii)} can be deduced 
from the real case {\bf (i)}
by considering
the R\'enyi-Parry dynamics of the houses
$\house{\alpha}$. 

Let us consider
case {\bf (i)}. 
By the Northcott property,
the degree $\deg(\beta)$, valued in 
$\nb \setminus\{0,1\}$, is necessarily
not bounded
when $\beta > 1$ tends to $1^+$. 
To compensate
the absence of an 
integer function  
of $\beta$ which ``measures" the proximity
of $\beta$ with $1$,
we introduce
the natural
integer function of $\beta$, that we call the 
{\em dynamical degree
of $\beta$}, denoted by 
$\dyg(\beta)$, which is defined by
the relation:
for $1 < \beta \leq \frac{1+\sqrt{5}}{2}$ any real
number, $\dyg(\beta)$ is the unique 
integer $n \geq 3$
such that
\begin{equation}
\label{intervv}
\theta_{n}^{-1} \leq \beta < \theta_{n-1}^{-1}
\end{equation}
where $\theta_{n}$ is the unique root
in $(0,1)$ of the trinomial
$G_{n}(z) = -1 + z + z^n$.
\vspace{0.1cm}

\begin{minipage}{14cm}
{\scriptsize {\em {\bf Important Nota {\bf (N)}}: 
(A) An infinite subcollection
of reciprocal algebraic integers
$\beta > 1$ is excluded from the present study.
Let us explain what is this subcollection and why.
We refer to Section \ref{S5.9.1} for more details.

\
 
The minimal polynomial
$P_{\beta}(X)$ of the generic reciprocal
algebraic integer $\beta > 1$ can be written
\begin{equation}
\label{requalityPP}
P_{\beta}(X) = \widetilde{P_{\beta}}(X^r)
\end{equation}
for some integer $r \geq 1$ and
some $\zb$-minimal integer polynomial  
$\widetilde{P_{\beta}}(X)$.
The integer $r$ is the largest one
such that \eqref{requalityPP} holds.
There are two cases.

The case $r=1$ is called the Main Case;
we say that ``$\beta$ belongs to the Main Case".
All the statements and the
proofs below till Section \ref{S7}
assume  ``$r=1$", ie are only concerned 
with $\beta$s belonging to the Main Case,
in particular Proposition \ref{ASversPbeta}. 

The case ``$r \geq 2$" is called the 
Second Case; we say that 
``$\beta$ belongs to the Second Case". 
This case
corresponds to the action of 
the $r$th roots of unity.
The actions of the roots of unity
produce infinite families 
of reciprocal algebraic
integers $\beta$ tending to 1,
on which the Mahler measure remains constant,
as explained in Section \ref{S5.9.1}.

Therefore the Second Case does not bring any
further complementary insight
into the study of the minoration 
of the Mahler measure {\rm M}
wih respect to the Main Case. 
The reason of this is
that in the Second Case, the dynamical degree
$\dyg(\beta)$, as defined above, 
has to be replaced
by $\dyg(\beta^r)$
and the lenticular roots (cf below)
are relative
to $f_{\beta^r}(z)$, not of
$f_{\beta}(z)$, 
and that
$\beta^r > \beta$.

Only the statements relative to 
the reciprocal algebraic integers
$\beta$ belonging to the Main Case
are considered below.

\noindent
(B) The dynamical zeta function
$\zeta_{\beta}(z)$ is used to investigate
to domain of very small
Mahler measures M$(\beta)$ of reciprocal algebraic integers 
$\beta$, 
in the range $(1, 1.176280\ldots]$, in the eventual case 
where such $\beta$s exist. The problem of the existence
of such $\beta$s, apart from
Lehmer's number
$1.176280\ldots$, is another problem. 
}
\vspace{0.1cm}

}
\end{minipage}

The (unique) simple zero $> 1$
of the trinomial
$G_{n}^{*}(z) := 1 + z^{n-1} - z^{n} , 
n \geq 2$, is
the Perron number $\theta_{n}^{-1}$.
The set of dominant roots 
$(\theta_{n}^{-1})_{n \geq 2}$
of the
nonreciprocal 
trinomials $(G_{n}^{*}(z))_{n \geq 2}$ 
constitute
a strictly decreasing sequence
of Perron numbers,
tending to one.
Section $\S$ \ref{S4} summarizes
the properties of these trinomials.
The sequence 
$(\theta_{n}^{-1})_{n \geq 2}$
will be extensively used in the sequel.
It is a fundamental set of
Perron numbers of
the interval $(1, \theta_{2}^{-1})$
simply indexed by the integer $n$,
and this indexation is extended to any real number
$\beta$ lying between two successive Perron numbers
of this family 
by \eqref{intervv}.
Let us note that $\dyg(\beta)$ is well-defined
for algebraic integers $\beta$
and also 
for transcendental numbers $\beta$. 
Let $\kappa:= \kappa(1, a_{\max}) = 0.171573\ldots$ 
(cf Sections $\S$\ref{S4},
$\S$\ref{S5} and $\S$\ref{S6} for the proofs).

\begin{theorem}
\label{dygdeginequality}
\begin{itemize}
\item[(i)] 
For $n \geq 2$, 
\begin{equation}
\label{dygdegPERRON}
\dyg(\theta_{n}^{-1}) = n
~=~
\left\{
\begin{array}{ccc}
\deg(\theta_{n}^{-1}) & \quad ~\mbox{if}~ 
& \quad n \not\equiv 5 ~({\rm mod} \,6\,),\\
\deg(\theta_{n}^{-1}) + 2 & \quad ~\mbox{if}~ 
&\quad n \equiv 5 ~({\rm mod} \, 6\, ),
\end{array}
\right.
\end{equation}
\item[(ii)]
if $\beta$ is a real number which satisfies
$\theta_{n}^{-1} \leq \beta < \theta_{n-1}^{-1},
~n \geq 2$, then the 
asymptotic expansion of the
dynamical degree
$\dyg(\beta) = 
 \dyg(\theta_{n}^{-1}) =
 n$ of $\beta$ 
 is:
\begin{equation}
\label{dygExpression_intro}
\dyg(\beta) = 
- \frac{\lo (\beta - 1)}{\beta - 1}
\Bigl[
1 +
O\Bigl(
\Bigl(
\frac{\lo (-\lo (\beta - 1))}
{\lo (\beta - 1)}
\Bigr)^2
\Bigr)
\Bigr],
\end{equation}
with the constant $1$ in the Big O;
moreover,
if $\beta \in (\theta_{n}^{-1}, \theta_{n-1}^{-1}),
~n \geq 260$, 
is a reciprocal algebraic integer of
degree
$\deg(\beta)$,
with ${\rm M}(\alpha)
< 1.176280\ldots$, then
\begin{equation}
\label{dygdeg_intro}
\dyg(\beta) \Bigl(
\frac{2 \arcsin\bigl( \frac{\kappa}{2} \bigr)}{\pi}
\Bigr)
+
\Bigl(
\frac{2 \kappa \, \lo \kappa}
{\pi \,\sqrt{4 - \kappa^2}} 
\Bigr)
 ~\leq~ \deg(\beta).
 \end{equation}
\end{itemize}
\end{theorem}

The  
Poincar\'e asymptotic expansion method
has been introduced
for the roots of $(G_n)$ of 
modulus $< 1$ 
in \cite{vergergaugry6}. Then 
the Conjecture of Lehmer 
for the 
family $(\theta_{n}^{-1})_{n \geq 2}$
was solved (in \cite{vergergaugry6})
by
using this method and the resulting
asymptotic expansions
of the Mahler measures
$({\rm M}(\theta_{n}^{-1}))_{n \geq 2}$,
as functions of $n$.
In the present note 
this method is extended
to any reciprocal algebraic integer $\beta$
of dynamical degree 
$\dyg(\beta)$ large enough, where
now ``$\dyg(\beta)$" replaces ``$n$".

The choice of the sequence of trinomials
$(G_n)$ is natural in the context
of the R\'enyi-Parry dynamical systems
(Section $\S$ \ref{S2}) 
and leads to a
theory of {\it lexicographical perturbation}
of the trinomials $G_n$ 
compatible with the dynamics. 
This is at variance with other 
attacks of the 
Conjecture of Lehmer by perturbed
cyclotomic polynomials or polynomials 
having all their roots on the unit circle
(\cite{amoroso},
Ray \cite{ray2},
Doche \cite{doche2},
Sinclair \cite{sinclair},
Mossinghoff, Pinner and Vaaler
\cite{mossinghoffpinnervaaler},
Toledano \cite{toledano}). 
Taking the integer function
$\dyg(\beta)$ as an integer variable
tending to infinity when
$\beta > 1$ tends to $1^+$ 
is natural. 
All the asymptotic expansions, 
for the roots
of modulus $< 1$ of the minimal polynomials
$P_{\beta}(z)$, for the lower bounds of the
lenticular Mahler measures
${\rm M}_{r}(\beta)$,
will be obtained as a function of
the integer
$\dyg(\beta)$, when $\beta > 1$ tends to $1^+$.

To the $\beta$-shift, to
the R\'enyi-Parry dynamical system
associated with an algebraic integer
$\beta > 1$, are 
attached several analytic functions:
(i)
the minimal polynomial function
$P_{\beta}(z)$
which is reciprocal 
by Smyth's Theorem \cite{smyth}
as soon as
${\rm M}(\beta) < \Theta = 1.3247\ldots$;
(ii) the (Artin-Mazur)
dynamical zeta function of the $\beta$-shift
\cite{artinmazur},
the generalized Fredholm determinant of
the transfer operator associated with the $\beta$-transformation
$T_{\beta}$ \cite{baladikeller}, 
the Perron-Frobenius operator associated to $T_{\beta}$
\cite{itotakahashi} \cite{mori}
\cite{mori2}
\cite{takahashi}.

The main theorems below are
obtained using the Parry Upper function
$f_{\beta}(z)$, 
which is the opposite of the 
inverse of the dynamical zeta function
$\zeta_{\beta}(z)$.
The Parry Upper function at $\beta$ is the 
generalized Fredholm determinant associated
with the transfer operator
of the $\beta$-transformation
(Baladi \cite{baladi2}). 

Using ergodic theory 
Takahaski \cite{takahashi}
\cite{takahashi2},
Ito and Takahashi \cite{itotakahashi},
Flatto, Lagarias and Poonen \cite{flattolagariaspoonen}
have given an explicit expression 
(reformulation) of the
Parry Upper function
$f_{\beta}(z)$ 
of the $\beta$-shift. 
This simplified expression 
is extensively used
in the sequel.

The Parry Upper upper function at $\beta$
takes the general form, with a 
lacunarity controlled by the dynamical degree
(Theorem \ref{zeronzeron}, Proposition
\ref{fbetainfinie}):
$$f_{\beta}(z) = -1 + z + z^{\dyg(\beta)} + z^{m_1} +
z^{m_2} + z^{m_3} + \ldots \hspace{2cm}\mbox{}$$
\begin{equation}
\label{dygplusdeformation}
= G_{\dyg(\beta)} + z^{m_1} + z^{m_2} + z^{m_3}  
+ \ldots\end{equation}
with $m_1 \geq 2 \, \dyg(\beta) -1,
m_{q+1}-m_q \geq \dyg(\beta) -1, q \geq 1$. For
$\theta_{\dyg(\beta)}^{-1}
\leq \beta < \theta_{\dyg(\beta)-1}^{-1}$,
the
lenticulus $\lc_{\beta}$
of zeroes of $f_{\beta}(z)$
relevant for the Mahler measure is obtained 
by a deformation of the lenticulus of zeroes
$\lc_{\theta_{{\dyg(\beta)}}^{-1}}$ 
of $G_{\dyg(\beta)}$
due to the tail
$z^{m_1} + z^{m_2}  
+ \ldots$ itself. For instance, 
for Lehmer's number
$\beta = 1.17628\ldots$,
the dynamical degree $\dyg(\beta)$ is equal to
$12$ and
$$f_{\beta}(z) = -1 + z + z^{12}
+ z^{31} + z^{44} + z^{63}
+ z^{86} + z^{105} + z^{118}+\ldots$$
is sparse with gaps of length
$\geq 10 = \dyg(\beta) -2$. The lenticulus
$\lc_{\beta}$ is close to 
$\lc_{\theta_{12}^{-1}}$ represented in Figure
\ref{perronselmerfig}.

The passage from the Parry Upper function
$f_{\beta}(z)$ to the Mahler measure
${\rm M}(\beta)$ (when $\beta > 1$ is a
reciprocal 
algebraic integer) 
constitutes
the main discoveries of the author,
and relies upon two facts:

\noindent
(i) 
the \underline{discovery of 
lenticular distributions of zeroes}
of $f_{\beta}(z)$ in the annular
region 
$e^{-\lo \beta}$ $= 
\frac{1}{\beta}  \leq |z| < 1$
which are very close to the
lenticular sets of zeroes of the trinomials
$G_{\dyg(\beta)}(z)$ of modulus $< 1$;

\noindent
(ii) \underline{the identification
of these zeroes
as conjugates of $\beta$}. 
The quantity
$\lo \beta$ is the 
topological entropy of the $\beta$-shift.
These lenticular distributions of zeroes lie
in the cusp of the fractal of Solomyak
of the $\beta$-shift \cite{solomyak}
(recalled in Section $\S$ \ref{S3.2}).
The key ingredient 
for obtaining the Dobrowolski
type minoration of the Mahler measure
${\rm M}(\beta)$ in 
Theorem \ref{mainDOBROWOLSLItypetheorem}
relies upon the best possible
evaluation of the 
deformation of these lenticuli of zeroes
by the method of Rouch\'e
(in Section $\S$ \ref{S5}) and the coupling 
between the Rouch\'e conditions and
the asymptotic expansions of the 
lenticular zeroes.

The identification of the complete set of
the conjugates $\beta^{(i)}$ of 
$\beta$, $|\beta^{(i)}|<1$
($\beta > 1$ being a
reciprocal algebraic integer),
seems to be 
unreachable by the present method. The
identification of
the lenticular
conjugates of modulus $< 1$
can only be done
in an angular subsector
of $\arg(z) \in 
[-\frac{\pi}{3}, \frac{\pi}{3}]$
(Proposition \ref{splitBETAdivisibility+++} and 
Theorem \ref{divisibilityALPHA}). 
Consequently
the present method only gives access to
a ``{\em part}" of the Mahler measure itself.
We denote by
$\lc_{\beta}$, $\lc_{\theta_{\dyg(\beta)}^{-1}}$
the lenticular sets of zeroes of
$f_{\beta}(z)$, resp. of
$G_{\dyg(\beta)}(z)$. We call
$${\rm M}_{r}(\beta)
=
\prod_{\omega \in \lc_{\beta}} \, |\omega|^{-1}$$ 
the lenticular Mahler measure of $\beta$.
It satisfies ${\rm M}_{r}(\beta) \leq
{\rm M}(\beta)$.
We show that $\beta \to 
{\rm M}_{r}(\beta)$ is continuous
on each open interval
$(\theta_{n}^{-1}, \theta_{n-1}^{-1})$
for the usual topology,
and that it admits 
a lower bound which can be expanded
as an asymptotic expansion
of $\dyg(\beta)$
(Theorem \ref{Lrasymptotictheoremcomplexe}).
The general minorant of ${\rm M}(\beta)$
for solving
the problem of Lehmer comes from
the asymptotic expansion of
the lower bound of ${\rm M}_{r}(\beta)$;
it is given by 
\eqref{lenticularminorationMr_intro}.

Case {\bf (ii)} is now an extension of
case {\bf (i)}.
The minoration of the 
Mahler measure of $P_{\alpha}$
is deduced from
the R\'enyi-Parry {\em dynamics} of
the house $\house{\alpha}$.
For
$\alpha$ a nonreal complex reciprocal algebraic 
integer, $|\alpha| > 1$,
such that
$1 < \house{\alpha} \leq \frac{1+\sqrt{5}}{2}$,
the dynamical degree of $\alpha$
is defined by
$\dyg(\alpha) := \dyg(\house{\alpha}\!\!)$. 

Once 
$\house{\alpha} > 1$ is 
close enough to $1^+$,
three new  
notions appear:
\begin{enumerate}
\item[(i)] the equality 
$P_{\alpha} = P_{\house{\alpha}}$
between the minimal polynomials, resp. of
$\alpha$ and $\house{\alpha}$,
\item[(ii)] the identification
of the lenticular
zeroes of $f_{\house{\alpha}}$ with
the lenticular zeroes
of $P_{\alpha}$, and
the {\em continuity}
of the minorant of the 
lenticular Mahler measure
with the house 
$\house{\alpha}$
of $\alpha$
(Theorem \ref{Qautocontinus}
and Remark \ref{measuremahlercontinuityAUXPerrons}),

\item[(iii)] the {\it fracturability}
of the minimal polynomial 
$P_{\alpha}(z) = 
(-\zeta_{\house{\alpha}}\!\!(z) P_{\alpha}(z))
\times 
f_{\house{\alpha}}(z)$
as a product of the
two integer arithmetic 
series:
$-\zeta_{\house{\alpha}}\!\!(z) P_{\alpha}(z)$,
$f_{\house{\alpha}}\!\!(z) \in \zb[[z]]$
(Theorem \ref{splitBETAdivisibility+++} and 
Theorem \ref{divisibilityALPHA}).
The fracturability
of the 
minimal polynomial
$P_{\alpha}(z)$
obeys the 
{\em Carlson-Polya dichotomy} 
\cite{carlson}
\cite{polya} as:

\begin{equation}
\label{decompozzz_main}
P_{\alpha}(z) = 
-\zeta_{\house{\alpha}}\!\!\!(z) \,P_{\alpha}(z) \times
f_{\house{\alpha}}\!\!(z)
\quad
\left\{
\begin{array}{cc}
\mbox{on}~ \cb & 
\mbox{if $\house{\alpha}$ is 
a Parry number, }\\
& \mbox{with $-\zeta_{\house{\alpha}}\!\!(z) P_{\alpha}(z)$ and
$f_{\house{\alpha}}$} \\
& \mbox{as meromorphic functions},\\
&\\
\mbox{on}~ |z| < 1 & 
\mbox{if~} \beta \mbox{~is a
nonParry number, }\\
& \mbox{with $|z|=1$ as natural 
boundary}\\ 
& \mbox{for both 
$-\zeta_{\house{\alpha}}\!\!(z) P_{\alpha}(z)$ 
and $f_{\house{\alpha}}$,}
\end{array}
\right. 
\end{equation}
The {\em domain of fracturability}
of $P_{\alpha}$
is the open subset of
the open unit disk on which
$-\zeta_{\house{\alpha}}\!\!(z) P_{\alpha}(z)$
is not constant, does not vanish,
is holomorphic. The domain
of fracturability of
$P_{\alpha}$ contains the lenticular zeroes
of $P_{\alpha}$.
\end{enumerate}
Lehmer's number, say $\beta_0$, is
the smallest Mahler measure ($> 1$)
of algebraic integers
known and the
smallest Salem number known
\cite{mossinghofflist}
\cite{mossinghoffrhinwu}. It is
the dominant root
of Lehmer's polynomial, of degree 10,
\begin{equation}
\label{lehmerpolynomialdefinition}
P_{\beta_0}(X)
=
X^{10}+X^9 -X^7 -X^6 -X^5 -X^4 -X^3 + X + 1.
\end{equation}
The above general equality
$P_{\alpha} = P_{\house{\alpha}}$ 
in (i) is the generalization
of the identity:
$P_{\beta_0}(X)= P_{\house{\beta_0}}(X)$.

The main theorems are the following.

\begin{theorem}[ex-Lehmer conjecture]
\label{mainLEHMERtheorem}
For any nonzero algebraic integer
$\alpha$ which is not a root of unity,
$${\rm M}(\alpha) \geq 
\theta_{\eta}^{-1} > 1
\qquad
\mbox{for some integer}~
\eta \geq 259.
$$
\end{theorem}

\begin{theorem}[ex-Schinzel-Zassenhaus conjecture]
\label{mainSCHINZELZASSENHAUStheorem}
Let $\alpha$ be a nonzero 
recipocal algebraic integer
which is not a root of unity.
Then
\begin{equation}
\label{kappoMAIN}
\house{\alpha} \geq 1 + \frac{c}{\deg(\alpha)}
\end{equation}
with 
$c = \theta_{\eta}^{-1} - 1$
with $\eta \geq 259$.
\end{theorem}

The following definitions are given in Section 
$\S$ \ref{S5}.
We just report them here for stating 
Theorem \ref{mainDOBROWOLSLItypetheorem}.
Denote by $a_{\max} = 5.87433\ldots$ the abscissa
of the maximum of the function
$a \to \kappa(1,a) :=
\frac{1 - \exp(\frac{-\pi}{a})}{2 \exp(\frac{\pi}{a}) -1}$
on $(0, \infty)$ (Figure \ref{h1a}).
Let $\kappa:= \kappa(1, a_{\max}) = 0.171573\ldots$ 
be the value of 
the maximum. Let
$S:= 2 \arcsin(\kappa/2) = 0.171784\ldots$.
Denote
$$
\Lambda_{r}\mu_{r}
:=
\exp\Bigl(
\frac{-1}{\pi}
\int_{0}^{S}
\! \!\lo\Bigl[\frac{1 + 2 \sin(\frac{x}{2})
-
\sqrt{1 - 12 \sin(\frac{x}{2}) 
+ 4 (\sin(\frac{x}{2}))^2}}{4}
\Bigr] dx
\Bigr)$$
\begin{equation}
\label{lehmernumberdefinition}
 = 1.15411\ldots ,
\qquad \mbox{a value slightly below Lehmer's number
} 1.17628\ldots
\end{equation}

\begin{theorem}[Dobrowolski type minoration] 
\label{mainDOBROWOLSLItypetheorem}
Let $\alpha$ be a nonzero 
reciprocal
algebraic integer which is not a root of unity
such that
$\dyg(\alpha) 
\geq 260$,
with ${\rm M}(\alpha)
< 1.176280\ldots$. Then
\begin{equation}
\label{dodobrobro}
{\rm M}(\alpha) \geq 
\Lambda_r \mu_r \, 
-
\Lambda_r \mu_r \,\frac{S}{2 \pi}\, 
\Bigl(
\frac{1}{\lo (\dyg(\alpha))}
\Bigr)
\end{equation}
\end{theorem}
In terms of the Weil height $h$, 
using Theorem \ref{nfonctionBETA}, 
the asymptotics of the
minoration \eqref{dodobrobro}
takes the following form:
\begin{equation}
\label{dodobrobroWEIL}
\deg(\alpha) h(\alpha) 
~\geq~ 
\lo (\Lambda_r \mu_r)
+
\frac{S}{2 \pi}\, 
\frac{1}{\lo (\house{\alpha}-1)} .
\end{equation}
The minoration
\eqref{dodobrobro} can also be 
restated
in terms of the usual degree.
Let $B > 0$.
Let us consider the subset 
$\mathcal{F}_B$ of 
all nonzero reciprocal
algebraic integers
$\alpha$ not being a root of unity
such that $\house{\alpha} < \theta_{259}^{-1}$
satisfying
$n := \deg(\alpha) \leq (\dyg(\alpha))^B$.  
Then
\begin{equation}
\label{dodobrobrousualdegree}
{\rm M}(\alpha) \geq 
\Lambda_r \mu_r \, 
-
\Lambda_r \mu_r \,\frac{S B}{2 \pi}\, 
\Bigl(
\frac{1}{\lo n}
\Bigr) , \qquad 
\alpha \in \mathcal{F}_B .
\end{equation}
Comparatively, in 1979, 
Dobrowolski \cite{dobrowolski2}, using an auxiliary function,
obtained the asymptotic minoration,
with $n = \deg(\alpha)$,
\begin{equation}
\label{dobrowolski79inequality}
{\rm M}(\alpha) > 1 + (1-\epsilon)
\left(
\frac{\lo \lo n}{\lo n}
\right)^3 , \qquad n > n_0,
\end{equation}
with  
$1-\epsilon$
replaced by $1/1200$ for 
$n \geq 2$, for an effective version of the minoration.
In \eqref{dodobrobro} or
\eqref{dodobrobrousualdegree}, 
the constant in the minorant 
is not any more $1$
but $1.15411\ldots$
and the sign 
of the $n$-dependent term
is negative, with an appreciable
gain of
$(\lo n)^2$ in the denominator. 

The minoration 
\eqref{dodobrobro}
is general and admits a better 
lower bound, in a similar
formulation, 
when $\alpha$ only runs over the set
of Perron numbers
$(\theta_{n}^{-1})_{n \geq 2}$. 
Indeed, in
\cite{vergergaugry6} ,
it is shown that
\begin{equation}
\label{minoVG2015__}
{\rm M}(\theta_{n}^{-1})
~>
~\Lambda - \frac{\Lambda}{6} \, 
\Bigl(\frac{1}{\lo n}\Bigr)  ~
, \qquad ~n \geq 2,
\end{equation}
holds with the following constant of the minorant 
\begin{equation}
\label{limitMahlGn}
\Lambda:=
{\rm exp}\bigl(
\frac{3 \, \sqrt{3}}{4 \, \pi} {\rm L}(2, \chi_3)\bigr) ~=~ 
{\rm exp}\Bigl(
\frac{-1}{\pi} \,
\int_{0}^{\pi/3} \lo \bigl(2 \, \sin\bigl(\frac{x}{2}\bigr) \bigr) dx 
\Bigr)
~=~ 1.38135\ldots,
\end{equation}
higher than $1.1541\ldots$,
and $L(s,\chi_3):= 
\sum_{m \geq 1} \frac{\chi_{3}(m)}{m^s}$
the Dirichlet L-series for the 
character $\chi_{3}$,
with $\chi_3$ the uniquely specified odd
character of conductor $3$
($\chi_{3}(m) = 0, 1$ or $-1$ according to 
whether $m \equiv 0, \,1$ or
$2 ~({\rm mod}~ 3)$, equivalently
$\chi_{3}(m)$
$ = \left(\frac{m}{3}\right)$ 
the Jacobi symbol).
The two constants in
\eqref{dodobrobro}
and
\eqref{minoVG2015__}
are:
$\Lambda_r \mu_r S/(2 \pi) =
0.0315536\ldots$ in \eqref{dodobrobro}   
and
$\Lambda/6 = 0.230225\ldots$ in
\eqref{minoVG2015__}.

The Mahler measure
${\rm M}(G_n)$ of 
the trinomial
$G_n$ is equal to the lenticular Mahler measure
${\rm M}_{r}(G_n)$ itself, with limit
$\lim_{n \to +\infty} {\rm M}(G_n) 
= \lim_{n \to +\infty} {\rm M}_{r}(G_n)  
= \Lambda$, having
asymptotic expansion
\begin{equation}
\label{mahlermeasureGn_asympto}
{\rm M}(G_n) ~=~ \Lambda 
\Bigl( 1 + r(n) 
\, \frac{1}{\lo n}
+ O\left(\frac{\lo \lo n}{\lo n}\right)^2 \Bigr)
\end{equation}
with $r(n)$ real,
$|r(n)| \leq 1/6$. In the case of the trinomials
$G_n$ the characterization of the 
roots of modulus $< 1$ can be readily obtained
(Section $\S$\ref{S4}) and does not require
the detection method of Rouch\'e.
In the general case,
with $\beta \in 
(\theta_{n}^{-1}, \theta_{n-1}^{-1})$,
$n= \dyg(\beta)$ large enough,
applying the method of Rouch\'e
only leads to
the following asymptotic lower bound of 
the lenticular minorant,
similarly
as in \eqref{mahlermeasureGn_asympto},
as (Section $\S$\ref{S6.2}):
\begin{equation}
\label{lenticularminorationMr_intro}
{\rm M}_{r}(\beta) \geq\Lambda_r \mu_r 
(1 + 
\frac{\rc}{\lo n}
+
O\bigl(
\bigl(
\frac{\lo \lo n}{\lo n}
\bigr)^2
\bigr),
\end{equation}
$$
\mbox{with}\qquad 
|~\rc ~+~ O\bigl(
\frac{(\lo \lo n)^2}{\lo n}
\bigr)~| < \frac{\arcsin(\kappa/2)}{\pi}.
$$
Denote by 
$M_{\inf} := \liminf_{\house{\alpha} \to 1^{+}} 
{\rm M}(\alpha)$ 
the limit infimum
of the Mahler measures ${\rm M}(\alpha)$,
$\alpha \in \mathcal{O}_{\overline{\qb}}$,
when $\house{\alpha} > 1$ tends to $1^+$.
Then
\begin{equation}
\label{limitinfimumMahlermeasure}
\Lambda_r \mu_r \leq M_{\inf} \leq \Lambda.
\end{equation}
Because the $\beta$-shift is compact,
it seems reasonable
to formulate the following conjecture on the possible intermediate values
between $M_{\inf}$ and $\Lambda$.

\begin{conjecture}
For any $\nu_0 \in [M_{\inf},  \Lambda)$ there exists
a sequence of integer monic irreducible 
polynomials $(H_{m}(z))_{m}$ such that
$\lim_{m \to +\infty} {\rm M}(H_{m}) =\nu_0$.
\end{conjecture}

Lenticuli of conjugates lie in the cusp 
of Solomyak's
fractal (Section $\S$ \ref{S3.2})
\cite{dutykhvergergaugry}. The number
of elements of a lenticulus
$\lc_{\alpha}$ is an increasing
function of the dynamical degree
$\dyg(\alpha)$ as soon as 
$\dyg(\alpha)$ is large enough.
The existence of lenticuli composed
of three elements only
(one real, a pair of nonreal complex-conjugated
conjugates) 
is studied 
in Section $\S$ \ref{S7.1}. Such lenticuli appear
at small dynamical degrees. Since Salem
numbers have no nonreal complex conjugate
of modulus $< 1$, they should not possess
3-elements lenticuli of conjugates, therefore
they should possess
a small dynamical degree bounded from above. 
We obtain  31 as an upper bound as follows. 

\begin{theorem}[ex-Lehmer conjecture for Salem numbers]
\label{mainpetitSALEM}
Let $T$ denote the set 
of Salem numbers. Then $T$ is bounded from below:
$$\beta \in T
\qquad
\Longrightarrow
\qquad
\beta > 
\theta_{31}^{-1} = 1.08544\ldots$$
\end{theorem}
Lehmer's number $1.17628\ldots$ 
belongs to the interval 
$(\theta_{12}^{-1}, \theta_{11}^{-1})$
(Table 1).
This interval does not contain any other 
known Salem number. If there is another one,
its degree should be greater than 44 
\cite{mossinghoff}
\cite{mossinghoffrhinwu}. 
\begin{conjecture}
There is no Salem number in the interval
$$(\theta_{31}^{-1}, \theta_{12}^{-1})
= (1.08544\ldots, 1.17295\ldots).$$
\end{conjecture}
Parry Upper functions $f_{\beta}(z)$, 
with $\beta$ being an algebraic integer
of dynamical degree
$\dyg(\beta)$
$= 12$ to $16$,
do possess
3-elements lenticuli of zeroes
in the open unit disc (as in Figure \ref{perronselmerfig}).

The main obstruction in
Lehmer's problem arises from
the existence of 
lenticuli of conjugates in
a small angular sector containing 1 
in the complex plane.
These lenticuli come from the 
type of factorization
of the polynomial sections
of $f_{\house{\alpha}}$.
The importance of the angular
sectors containing 1 has already been  
guessed
by
Langevin in \cite{langevin} \cite{langevin2}
\cite{langevin3}
\cite{mignotte3},
by Dubickas and Smyth \cite{dubickassmyth},
by Rhin and Smyth \cite{rhinsmyth} 
and Rhin and Wu 
\cite{rhinwu}.
These lenticuli cannot be described
by these classical approaches, but become
visible by the present method.
Though the lenticuli of roots
lie inside and off the unit circle,
the complete set of conjugates remains fairly
regularly distributed in the sense that it
equidistributes on the unit circle at the limit,
once
the Mahler measures are small enough, as follows.

\begin{theorem}
\label{main_EquidistributionLimitethm}
Let $(\alpha_{q})_{q \geq 1}$ be
a sequence of  
algebraic integers 
such that $|\alpha_q| > 1$,
${\rm M}(\alpha_q) < 1.176280\ldots$,
$\dyg(\alpha_q) \geq 260$,
for $q \geq 1$,
with
$\lim_{q \to +\infty} 
\house{\alpha_q} = 1$.
Then the sequence 
$(\alpha_{q})_{q \geq 1}$
is strict and
\begin{equation}
\label{haarmeasurelimite}
\mu_{\alpha_q} 
~\to~ 
\mu_{\tb},
\qquad 
\dyg(\alpha_q) \to + \infty,
\qquad
\mbox{weakly},
\end{equation}
i.e. for all bounded, continuous
functions
$f: \cb^{\times} \to \cb$,
\begin{equation}
\label{haarlimitefunvtions}
\int f d \mu_{\alpha_q}
\to
\int f d \mu_{\tb},
\qquad
\dyg(\alpha_q) \to + \infty.
\end{equation}
\end{theorem}

Parry numbers are Perron numbers 
(Adler and Marcus \cite{adlermarcus});
the characterization of the set of 
Parry numbers 
(Definition 
\ref{parrynumberdefinition})
is a deep question addressed to
the dynamics of Perron numbers
(Bertrand-Mathis \cite{bertrandmathis},
Boyd \cite{boyd7},
Boyle and Handelman \cite{boyle2} 
\cite{boylehandelman}
\cite{boylehandelman2},
Brunotte \cite{brunotte},
Calegari and Huang \cite{calegarihuang},
Dubickas and Sha \cite{dubickassha},
Lind \cite{lind} \cite{lind2},
Lind and Marcus \cite{lindmarcus},
Thurston \cite{thurston2},
Verger-Gaugry \cite{vergergaugry2}
\cite{vergergaugry3}), 
associated with the rationality
of the dynamical zeta function 
of the $\beta$-shift 
\begin{equation}
\label{dynamicalfunction_}
\zeta_{\beta}(z) := \exp\left(
\sum_{n=1}^{\infty} \, \frac{\mathcal{P}_n}{n} \, z^n
\right),\quad
\mathcal{P}_n :=\#\{x \in [0,1] \mid 
T_{\beta}^{n}(x) = x\}\end{equation}
counting the number of periodic 
points of period dividing $n$.
For $\alpha$ a nonzero algebraic integer
which is not a root of unity, with
$\beta = \house{\alpha}$,
by Theorem \ref{parryupperdynamicalzeta},
$$\beta \quad 
\mbox{is a Parry number}
\quad \mbox{if and only if}\quad 
\zeta_{\beta}(z) \quad \mbox{is a rational function};$$
and
$|z|=1$ is the natural boundary of
the domain of fracturability of 
the minimal polynomial 
$P_{\alpha}$, 
in the sense of
Theorem \ref{divisibilityALPHA},
if and only if 
$\beta$ is not a Parry number,
as soon as $|\alpha|$ is close enough to 1, 
in the Carlson-Polya dichotomy.
Comparatively,
for complete nonsingular projective 
algebraic varieties
$X$ over the field of
$q$ elements, $q$ a prime power, 
the zeta function $\zeta_{X}(t)$
introduced by 
Weil \cite{weil} is only
a rational function
(Dwork \cite{dwork},
Tao \cite{tao}):
the first Weil's conjecture,
for which there exists a set of 
characteristic values was proved by Dwork
using $p$-adic methods
(Dwork \cite{dwork}),
and ``Weil II", the Riemann hypothesis, 
proved by Deligne using 
$l$-adic \'etale cohomology in 
characteristic $p \neq l$ 
(Deligne \cite{deligne}). 
It is defined as a dynamical zeta 
function with the action of the Frobenius.
The purely $p$-adic methods of Dwork
(Dwork \cite{dwork}), continued by
Kedlaya \cite{kedlaya} for ``Weil II"
in the need of numerically computing zeta functions
by explicit equations, 
allow an intrinsic computability,
as in Lauder and Wan \cite{lauderwan}, towards
a $p$-adic cohomology theory,
are linked to ``extrinsic geometry", to
the defining equations of the variety
itself.
They are 
in contrast with the relative version of crystalline 
cohomology developped by
Faltings 
\cite{faltings}, or
the Monsky-Washnitzer
constructions used by Lubkin \cite{lubkin}.
We refer the reader to Robba \cite{robba},
Kedlaya \cite{kedlaya}, Tao \cite{tao},
for a short survey on other developments.

After Weil \cite{weil}, and introduced
in general terms by 
Artin and Mazur in \cite{artinmazur},
the theory of dynamical zeta functions 
$\zeta(z)$ associated with dynamical systems,
based on an analogy with the number theory zeta functions,
developped under the impulsion
of Ruelle \cite{ruelle4}
in the direction of the thermodynamic formalism
and with Pollicott, Baladi and Keller \cite{baladikeller}
towards transfer operators and counting orbits
\cite{parrypollicott}.
The determination and the
existence of meromorphic extensions
or/and natural boundaries of 
dynamical zeta functions
is a deep problem.

In the present
proof of the conjecture of Lehmer,
the analytic extension of the dynamical 
zeta function of the $\beta$-shift
behaves as an analogue
of Weil's zeta function
(in the sense that both
are dynamical zeta functions).
But it generates questions 
beyond the analogues of
Weil's conjectures
since not only the rational case of
$\zeta_{\beta}$
contributes to the minoration
of the Mahler measure, but also
the nonrational case
with the unit circle
as natural boundary 
and lenticular poles arbitrarily close to it.
For instance,
a part of the analogue 
of ``Weil II" (Riemann Hypothesis)
would be
the determination of the
geometry of the beta-conjugates
in the rationality case. 
Beta-conjugates
are zeroes of Parry polynomials, 
whose factorization has been 
studied in the context of 
the theory
of Pinner and Vaaler \cite{pinnervaaler}
in \cite{vergergaugry3}.

An apparent difficulty for the 
computation of the minorant of
${\rm M}(\alpha)$
comes from the absence of 
complete characterization 
of the set of Parry numbers 
$\pb_P$, 
when $\beta = \house{\alpha}$ 
is close to one, since we never know
whether $\beta$ is a Parry number or not.
But the Mahler measure M$(\alpha)$
is independent of 
the Carlson-Polya dichotomy.
Indeed,
the two domains of definitions
of $\zeta_{\beta}$, 
``$\cb$" and ``$|z| < 1$", 
together with the corresponding splitting
 \eqref{decompozzz_main}, may occur
fairly ``randomly" 
when $\beta$ tends to one.
But $\{|z| < 1\}$ 
is a domain included in both,
M$(\alpha)$ ``reading" only the roots
in it and not taking care of the ``status" of 
the unit circle.
Whether $f_{\house{\alpha}}(z)$ can be 
continued analytically or not
beyond
the unit circle has no 
effect on the value of the Mahler measure
${\rm M}(\alpha)$.

%------------------------------------------------------
The paper is 
organized as follows:
in Section $\S$\ref{S2}  
we recall
the properties of the 
R\'enyi-Parry numeration
system of the $\beta$-shift.
The analytic functions,
in particular
the Parry Upper function 
$f_{\beta}(z)$
and the dynamical zeta function
$\zeta_{\beta}(z)$,
associated
to the dynamics of the $\beta$-shift, 
are introduced
in Section $\S$\ref{S3}. 
In Section $\S$\ref{S5}
the peculiar 
consequences of
the lexicographical ordering, 
induced
by the numeration system, 
on the zeroes of
$f_{\beta}(z)$ are developped, in
particular the lenticular zeroes
and their identification 
as Galois conjugates 
of the base of numeration $\beta$.
Coupling the knowledge
of the geometry of
the lenticular roots
with the method of
asymptotic expansions 
(recalled in Section
$\S$\ref{S4} for trinomials)
gives a continuous 
lenticular minorant of
M$(\beta)$ and a 
Dobrowolski type minoration.
The proofs of Lehmer's Conjecture 
and Schinzel-Zassenhaus's Conjecture
follow, for $\beta > 1$ any reciprocal
algebraic integer
(case {\bf (i)}).
These results are extended in 
Section $\S$\ref{S6} (case {\bf (ii)})
for 
any nonzero reciprocal 
algebraic integer $\alpha$ 
which is not a root of unity, 
$\arg(z) \in (0, \pi/2]$. 
 
This case is shown to
be deduced from the preceding case
(given by
Section $\S$\ref{S5}) by taking
$\beta:= \house{\alpha}$ which is real
and $> 1$.
  In Section $\S$\ref{S7} 
  the Conjecture of Lehmer 
  is proved for Salem numbers, 
  using another regime of 
  asymptotic expansions of the roots
  of the trinomials $G_n$, more adapted
  to the cusp in Solomyak's fractal. 
  Concomitantly to the 
  limit problem of Lehmer it is shown that 
  the conjugates of the base of numeration
  $\house{\alpha}$
  equidistribute on the unit circle in 
  the complex plane.
  Using a Theorem of Belotserkovski 
  a Theorem of limit     
  equidistribution 
  of the conjugates is formulated
  in Section $\S$\ref{S8}, when the 
  dynamical degree of $\house{\alpha}$ 
  tends to infinity. 
  A few consequences 
  are mentioned 
  in Section $\S$\ref{S9}, in particular
  a Conjecture of Margulis. 
  
\vspace{0.3cm}

\vspace{0.1cm}

The proof of Lehmer's Conjecture starts at Section 
\S \ref{S5}. It is preferable that
the reader be acquainted with the results of
\cite{vergergaugry6}.
Lenticuli of roots are investigated
in \cite{dutykhvergergaugry}. 
The method of 
identification of the lenticular poles of 
$\zeta_{\beta}(z)$ with zeroes of the minimal
polynomial of $\beta$ is published in
\cite{dutykhvergergaugry2}.

%-----------------------------------------------------------

\section{The $\beta$-shift, the R\'enyi-Parry dynamical system of numeration}
\label{S2}

\subsection{Towards the problem of Lehmer}
\label{S2.1}

The direction which is followed 
is the following:
it consists in using
the analytic functions
associated with 
the R\'enyi-Parry dynamical 
system of numeration,
the $\beta$-shift (i.e. with
the language in base $\beta$) with
$1 < \beta < 2$, 
first where $\beta$ is fixed
to formulate the properties of 
these functions,
and then vary continuously  
the basis of numeration 
$\beta$ 
taking the limit to $1^+$,
to use the limit properties of these functions
for solving the problem of Lehmer.
This is a method of variable basis,
where the variable $\beta$ 
runs over $\overline{\qb} \cap (1, +\infty)$,
more precisely over
the set of reciprocal algebraic integers.
This mathematics appears when
the base of numeration is not
an integer. The analytic functions 
are the Parry Upper function
$f_{\beta}(z)$  
and the dynamical zeta function
$\zeta_{\beta}(z)$. 
The Parry Upper function
$f_{\beta}(z)$ is the generalized
Freholm determinant of the transfer
operator $\mathcal{L}_{t\beta}$
of the $\beta$-transformation. 
Both analytic functions are
presented in Section $\S$\ref{S3}.
In Section $\S$\ref{S2} 
are recalled the properties
of the $\beta$-shift, in particular
the convertion of the inequality
``$<$" on $(1,2)$ to the lexicographical 
inequality ``$<_{lex}$" on sequences
of $0, 1$ digits.

Let us say a few words on the
$\beta$-shift.
In 1957 R\'enyi \cite{renyi}
introduced new representations 
of a real number
$x$, using a positive function
$y=f(x)$ and infinite iterations of it,
in the form of an ``$f$-expansion", as
$$x = \epsilon_0 +
f(\epsilon_1
+f(\epsilon_2 + \ldots))$$
with ``digits" $\epsilon_i$
in some alphabet 
and remainders
$f(\epsilon_n
+f(\epsilon_{n+1} + \ldots))$.
This approach considerably enlarged
the usual decimal numeration system,
that is internationally used today, 
and also numeration systems
in integer basis (recall
that the bases $5, 10, 12, 20 $ 
and $60$
were used in the Antiquity in several
countries as natural counting systems), 
by allowing
arbitrary real bases of 
numeration
(Fraenkel \cite{fraenkel}, Lothaire
\cite{lothaire}, Chap. 7): 
let $\beta > 1$ be not an integer
and consider
$f(x) = x/\beta$ if $0 \leq x \leq \beta$,
and $f(x) = 1$ if $\beta < x$. 
Then the
$f$-expansion of $x$ is the
representation of $x$ in base $\beta$
as
$$x = \epsilon_0 + \frac{\epsilon_1}{\beta}
+\frac{\epsilon_2}{\beta^2}
+\ldots
+ \frac{\epsilon_n}{\beta^n} + \ldots.$$
In terms of
dynamical systems, in 1960, 
Parry \cite{parry} 
\cite{parry2}
\cite{parry3}
has reconsidered
and studied the ergodic properties of
such representations of real numbers
in base $\beta$, in particular the conditions of
faithfullness of the map:
$x \longleftrightarrow (\epsilon_i)_i$
and the complete set of 
admissible sequences 
$(\epsilon_0 , \epsilon_1 , \epsilon_2 , \ldots)$
for all real numbers.
This complete set is called
the language in base $\beta$.
Let us note
that $\beta > 1$ will run over
the set of reciprocal algebraic integers
in Section $\S$\ref{S5} and 
$\S$\ref{S6}, but 
that the $\beta$-shift is defined 
in general for any real number $\beta$, 
algebraic or transcendental.

\subsection{The $\beta$-shift, $\beta$-expansions, lacunarity and symbolic dynamics}
\label{S2.2}

Let $\beta > 1$ be a real number
and let
$\mathcal{A}_{\beta} := \{0, 1, 2, \ldots, 
\lceil \beta - 1 \rceil \}$. If $\beta$ is not
an integer, then 
$\lceil \beta - 1 \rceil
= \lfloor \beta \rfloor$.
Let $x$ be a real number
in the interval $[0,1]$.
A representation in base $\beta$ 
(or a $\beta$-representation; or a $\beta$-ary representation if 
$\beta$ is an integer) of $x$ 
is an infinite word
$(x_i)_{i \geq 1}$ of
$\mathcal{A}_{\beta}^{\nb}$
such that
$$x = \sum_{i \geq 1} \, x_i \beta^{-i} \, .$$
The main difference with the case
where $\beta$ is an integer
is that $x$ may have several representations.
A particular $\beta$-representation,
called the $\beta$-expansion, 
or the greedy $\beta$-expansion,
and denoted
by $d_{\beta}(x)$, of $x$
can be computed either by the 
greedy algorithm, or equivalently by the
$\beta$-transformation
$$T_{\beta} : \,\, x \, \mapsto \,  \beta x \quad(\hspace{-0.3cm}\mod 1) 
= \{\beta x \}.$$
The dynamical system 
$([0,1], T_{\beta})$ is called the
R\'enyi-Parry numeration system in base $\beta$,
the iterates of $T_{\beta}$ providing
the successive digits $x_i$ of $d_{\beta}(x)$
\cite{liyorke}.
Denoting $T_{\beta}^{0} := {\rm Id}, 
T_{\beta}^{1} := T_{\beta}, 
T_{\beta}^{i} := T_{\beta} (T_{\beta}^{i-1})$
for all $i \geq 1$, we have:
$$d_{\beta}(x) = (x_i)_{i \geq 1}
\qquad \mbox{{\rm if and only if}} \qquad
x_ i = \lfloor \beta T_{\beta}^{i-1}(x) \rfloor$$
and we write the $\beta$-expansion of $x$ as
\begin{equation}
\label{xexpansion}
x \, = \, \cdot x_1 x_2 x_3 \ldots 
\qquad \mbox{instead of}
\qquad x = \frac{x_1}{\beta} +\frac{x_2}{\beta^2} +
\frac{x_3}{\beta^3} +
\ldots .
\end{equation}
The digits are
$x_1 = \lfloor \beta x \rfloor$,
$x_2 = \lfloor \beta \{\beta x \} \rfloor$,
$x_3 = \lfloor \beta \{\beta \{\beta x \} \} \rfloor, \, 
\ldots$ \,, depend upon $\beta$.

The R\'enyi-Parry numeration dynamical 
system in base $\beta$
allows the coding, as a (positional) $\beta$-expansion, 
of any real number $x$. 
Indeed, if $x > 0$, there exists $k \in \zb$ such that
$\beta^k \leq x < \beta^{k+1}$. Hence
$1/\beta \leq x/\beta^{k+1} <1$; thus it is enough to 
deal with representations and $\beta$-expansions
of numbers in the interval $[1/\beta,1]$. 
In the case where $k \geq 1$, 
the $\beta$-expansion of $x$ is
$$x \, = \, x_1 x_2 \ldots x_k \cdot x_{k+1} x_{k+2} \ldots ,$$
with $x_1 = \lfloor \beta (x/\beta^{k+1}) \rfloor$,
$x_2 = \lfloor \beta \{\beta  (x/\beta^{k+1})\} \rfloor$,
$x_3 = \lfloor \beta \{\beta \{\beta  (x/\beta^{k+1})\} \} \rfloor$, etc.
If $x < 0$, by definition: $d_{\beta}(x) = - d_{\beta}(-x)$.
The part $ x_1 x_2 \ldots x_k $
is called the $\beta$-integer part of
the $\beta$-expansion of $x$, and
the terminant $ \cdot x_{k+1} x_{k+2} \ldots$ is called the $\beta$-fractional part of $d_{\beta}(x)$.

A $\beta$-integer is a real number $x$ such that
the $\beta$-integer part of $d_{\beta}(x)$ is equal to
$d_{\beta}(x)$ itself (all the digits $x_{k+j}$
being equal to 0 for $j \geq 1$): in this case, 
if $x > 0$ for instance, $x$ is the polynomial
$$x\, = \, \sum_{i=1}^{k} x_i \beta^{k-i} \, ,
\qquad 0 \leq x_i \leq \lceil \beta - 1 \rceil$$
and the set of $\beta$-integers is denoted by
$\zb_{\beta}$. For all $\beta > 1$
$\zb_{\beta} \subset \rb$ is discrete and   
$\zb_{\beta} = \zb$
if $\beta$ is an integer $\neq 0,1$.

The set $\mathcal{A}_{\beta}^{\nb}$ is endowed with
the lexicographical order
(not usual in number theory),
and the product topology. The one-sided
shift 
$\sigma : (x_i)_{i \geq 1} \mapsto (x_{i+1})_{i \geq 1}$
leaves invariant the subset $D_{\beta}$ 
of the $\beta$-expansions of real numbers
in $[0,1)$. The closure of $D_{\beta}$ in
$\mathcal{A}_{\beta}^{\nb}$
is called
the $\beta$-shift, and is denoted by $S_{\beta}$.
The $\beta$-shift is a subshift of
$\mathcal{A}_{\beta}^{\nb}$, for which
$$d_{\beta} \circ T_{\beta} = \sigma \circ d_{\beta}$$
holds on the interval $[0,1]$ (Lothaire \cite{lothaire},
Lemma 7.2.7). In other terms, $S_{\beta}$ is such that
\begin{equation}
\label{betashitcorrespondance}
x \in [0,1] \qquad 
\longleftrightarrow
\qquad
(x_i)_{i \geq 1} \in S_{\beta}
\end{equation}
is bijective. This one-to-one correspondence between
the totally ordered interval
$[0, 1]$ and the totally lexicographically ordered
$\beta$-shift $S_{\beta}$
is fundamental. Parry (\cite{parry} Theorem 3)
has shown 
that only one sequence of digits
entirely controls the $\beta$-shift
$S_{\beta}$, and that the ordering is preserved
when dealing with the greedy $\beta$-expansions. 
Let us precise how the usual 
inequality ``$<$" on the real line is
transformed into the inequality
``$<_{lex}$", meaning ``lexicographically 
smaller with all its shifts". 

\noindent
{\it The greatest element of $S_{\beta}$}:
it comes from $x=1$ and is given either by
the R\'enyi $\beta$-expansion of $1$, or 
by a slight modification of it 
in case of finiteness. Let us precise it.
The greedy $\beta$-expansion of $1$ is by 
definition
denoted by
\begin{equation}
\label{renyidef}
d_{\beta}(1) = 0.t_1 t_2 t_3 \ldots
\qquad
{\rm and ~uniquely ~corresponds~ to}
\qquad 1 = \sum_{i=1}^{+\infty} t_i \beta^{-i}\, ,
\end{equation}
where 
\begin{equation}
\label{digits__ti}
t_1 = \lfloor \beta \rfloor,
t_2 = \lfloor \beta \{\beta\}\rfloor = \lfloor \beta T_{\beta}(1)\rfloor,
t_3 = \lfloor \beta \{\beta \{\beta\}\}\rfloor = \lfloor \beta T_{\beta}^{2}(1)\rfloor,
\ldots
\end{equation} 
The sequence $(t_i)_{i \geq 1}$
is given by the orbit of one 
$(T_{\beta}^{j}(1))_{j \geq 0}$ by
\begin{equation}
\label{polyTbeta}
T_{\beta}^{0}(1)=1, ~T_{\beta}^{j}(1) = \beta^j - t_1 \beta^{j-1} - t_2 \beta^{j-2} - \ldots - t_j
\in \mathbb{Z}[\beta] \cap [0,1]
\end{equation}
for all $j \geq 1$.
The digits
$t_i$ belong to 
$\mathcal{A}_{\beta}$.  
We say that $d_{\beta}(1)$ is finite if it ends in infinitely many zeros.

\begin{definition}
\label{parrynumberdefinition}
If $d_{\beta}(1)$ is finite or ultimately periodic (i.e. eventually
periodic), then the real number $\beta > 1$ 
is said to be a {\it Parry number}.
In particular, a Parry number $\beta$ is 
said to be {\it simple} if $d_{\beta}(1)$ is finite.
\end{definition}

The greedy $\beta$-expansion of\, $1/\beta$ is
\begin{equation}
\label{renyidef_unsurbeta}
d_{\beta}(\frac{1}{\beta}) = 0. 0 \, t_1 t_2 t_3 \ldots
\qquad
{\rm and ~uniquely ~corresponds~ to}
\qquad \frac{1}{\beta} = \sum_{i=1}^{+\infty} t_i \beta^{-i-1}.
\end{equation}
From 
$(t_i)_{i \geq 1} \in \mathcal{A}_{\beta}^{\nb}$
is built  
$(c_i)_{i \geq 1} \in \mathcal{A}_{\beta}^{\nb}$,
defined by
$$
c_1 c_2 c_3 \ldots := \left\{
\begin{array}{ll}
t_1 t_2 t_3 \ldots & \quad \mbox{if ~$d_{\beta}(1) = 0.t_1 t_2 \ldots$~ is infinite,}\\
(t_1 t_2 \ldots t_{q-1} (t_q - 1))^{\omega}
& \quad \mbox{if ~$d_{\beta}(1)$~ is finite,
~$= 0. t_1 t_2 \ldots t_q$,}
\end{array}
\right.
$$
where $( \, )^{\omega}$ means that the word within $(\, )$ is indefinitely repeated.
The sequence $(c_i)_{i \geq 1}$ is the unique
element of $\mathcal{A}_{\beta}^{\nb}$
which allows to obtain all the admissible
$\beta$-expansions of all the elements of
$[0,1)$.

\begin{definition}[Conditions of Parry]
A sequence $(y_i)_{i \geq 0}$ of elements of
$\mathcal{A}_{\beta}$ (finite or not) is said
{\it admissible} if 
\begin{equation}
\label{conditionsParry}
\sigma^{j}(y_0, y_{1}, y_{2}, \ldots)=
(y_j, y_{j+1}, y_{j+2}, \ldots) <_{lex}
~(c_1, \, c_2, \, c_3, \,  \ldots) 
\quad \mbox{for all}~ j \geq 0,
\end{equation}
where $<_{lex}$ means {\it lexicographically smaller}.
\end{definition}
\begin{definition}
A sequence 
$(a_i)_{i \geq 0} \in \mathcal{A}_{\beta}^{\nb}$
satisfying  \eqref{lyndonEQ} is said to be
{\it Lyndon (or self-admissible)}:
\begin{equation}
\label{lyndonEQ}
\sigma^{n}(a_0, a_{1}, a_{2}, \ldots)
=
(a_n, a_{n+1}, a_{n+2}, \ldots) <_{lex} (a_0, a_1, a_2, \ldots) 
\qquad \mbox{for all}~ n \geq 1.
\end{equation} 
\end{definition}
The terminology comes from the introduction
of such words by Lyndon in \cite{lyndon}, in honour of his work. Other orderings
are reviewed in Ngu\'ema Ndong
\cite{nguemandong} \cite{nguemandong2}.
The present Lyndon ordering is
reported in
\cite{nguemandong}, ex. 2 in subsection 1.2
and
in \cite{nguemandong2}, subsection 4.1, Theorem
5 and ex. 4 for applications
to the dynamical zeta function of
negative $\beta$-shift.

Any admissible representation $(x_i)_{i \geq 1}
\in \mathcal{A}_{\beta}^{\nb}$
corresponds, by \eqref{xexpansion}, 
to a real number $x \in [0,1)$
and conversely the greedy
$\beta$-expansion of $x$
is  $(x_i)_{i \geq 1}$ itself.
For an infinite admissible sequence 
$(y_i)_{i \geq 0}$ of elements of
$\mathcal{A}_{\beta}$ the 
(strict) lexicographical inequalities
\eqref{conditionsParry} constitute an infinite number of inequalities which are
unusual in number theory 
\cite{blanchard} 
\cite{frougny} \cite{frougny2}
\cite{lothaire} \cite{parry}.

In number theory, inequalities are
often associated to collections of half-spaces in
euclidean or adelic Geometry of Numbers
(Minkowski's Theorem, etc).
The conditions of Parry are of totally different nature
since they refer to 
a reasonable control, order-preserving, 
of the gappiness (lacunarity)
of the coefficient vectors of the
power series which are
the generalized Fredholm determinants of the transfer 
operators of the $\beta$-transformations
(cf Section $\S$\ref{S3}).

In the correspondence
$[0,1] \longleftrightarrow
S_{\beta}$, the element $x=1$ 
admits the maximal element $d_{\beta}(1)$
as counterpart. The uniqueness 
of the $\beta$-expansion
$d_{\beta}(1)$
and its property to be Lyndon
characterize the base of numeration $\beta$ 
as follows.

\begin{proposition}
\label{betacharacterized}
Let $(a_0, a_1, a_2, \ldots)$  be a sequence of non-negative
integers where $a_0 \geq 1$ and $a_n \leq a_0$ for all $n \geq 0$.
The unique solution $\beta > 1$ of 
\begin{equation}
\label{equabase}
1 = \frac{a_0}{x} + \frac{a_1}{x^2} + \frac{a_2}{x^3} + \ldots
\end{equation}
is such that $d_{\beta}(1) = 0. a_0 a_1 a_2 \ldots$ if and only if
\begin{equation}
\label{self}
\sigma^{n}(a_0, a_{1}, a_{2}, \ldots)
=
(a_n, a_{n+1}, a_{n+2}, \ldots) <_{lex} (a_0, a_1, a_2, \ldots) 
\qquad \mbox{for all}~ n \geq 1.
\end{equation} 
\end{proposition}

\begin{proof}
Corollary 1 of Theorem 3 in Parry \cite{parry} (Corollary
7.2.10 in Frougny \cite{frougny2}).
\end{proof}

If $1 < \beta < 2$, then 
the condition
``$a_0 \geq 1$ and $a_n \leq a_0$ for all $n \geq 0$"
amounts to
``$a_0 = 1$"; in this case the
$\beta$-integer part of $\beta$ is equal to $a_0 = 1$
and its $\beta$-fractional part is
$a_1 \beta^{-1} + a_2 \beta^{-2} + a_3 \beta^{-3} + \ldots$.
The base of numeration
$\beta = 1$ would correspond to the sequence
$(1, 0, 0, 0, \ldots)$ in \eqref{equabase} but
this sequence has its first digit $1$
outside the alphabet
$\mathcal{A}_{1} = \{0\}$: it cannot be considered
as a $1$-expansion. Fortunately 
numeration in base one is not often used.
The base of numeration
$\beta = 2$ would correspond to
the constant sequence $(1, 1, 1, 1, \ldots)$
in \eqref{equabase}
but this sequence is not self-admissible.
When $\beta = 2$, $2$ being an integer,
$2$-ary representations
differ and $(2, 0, 0, 0, \ldots)$ is taken
instead of $(1, 1, 1, 1, \ldots)$
(Frougny and Sakarovitch \cite{frougnysakarovitch},
Lothaire \cite{lothaire}).

Infinitely many cases of 
lacunarity, between $(1,0, 0, 0, \ldots)$ and
$(1, 1, 1, 1, \ldots)$, may occur in the sequence
$(a_0 , a_1 , a_2 , \ldots)$ in \eqref{equabase}.
If $\beta \in (1, 2)$ is fixed,
with $d_{\beta}(1)=0. t_1 t_2 t_3 \ldots$
then any $x$, $1/\beta < x < 1$, admits a
$\beta$-expansion 
$d_{\beta}(x)$ which lies
lexicographically (Parry \cite{parry}, Lemma 1)
between those of the extremities:
\begin{equation}
\label{dbetaxencadrement}
d_{\beta}(\frac{1}{\beta}) =
0. 0 \, t_1 t_2 t_3 \ldots
~~
<_{lex}
~~
d_{\beta}(x)
~~
<_{lex}
~~ 
d_{\beta}(1) =
0. t_1 t_2 t_3 \ldots.
\end{equation}

Let $1 < \beta < 2$ be a real number, with
$d_{\beta}(1) = 0. t_1 t_2 t_3 \ldots$.
If $\beta$ is a simple Parry number, then
there exists $n \geq 2$, depending upon $\beta$, 
such that
$t_n \neq 0$ and
$t_j = 0 , j \geq n+1$. Parry 
\cite{parry} has shown that
the set of simple Parry numbers is dense in
the half-line
$(1,+\infty)$. If $\beta$ is a Parry number 
which is not simple, the sequence $(t_i)_{i \geq 1}$
is eventually periodic: there exists an integer
$m \geq 1$, the preperiod length, 
and an integer $p \geq 1$, the period length,
such that
$$d_{\beta}(1) = 0. t_1 t_2 \ldots t_m
(t_{m+1} t_{m+2} \ldots t_{m+p})^{\omega},
$$
$m$ and $p$ depending upon $\beta$, with at least
one nonzero 
digit $t_j$, with
$j \in \{m+1, m+2, \ldots, m+p\}$. 
The gaps of successive zeroes in $(t_i)_{i \geq 1}$
are those of the preperiod 
$(t_1 , t_2 , \ldots,  t_m)$ then those
of the period 
$(t_{m+1} , t_{m+2} , \ldots, t_{m+p})$, then
occur periodically up till infinity. The length
of such gaps of zeroes is at most
$\max\{m-2, p-1\}$. The asymptotic lacunarity
is controlled by the periodicity in this case.

If $1 < \beta < 2$ is an algebraic number 
which is not a Parry number, the sequences of 
gaps of zeroes in  $(t_i)_{i \geq 1}$
remain asymptotically moderate 
and controlled by the Mahler measure
${\rm M}(\beta)$ of $\beta$, as follows.

\begin{theorem}[Verger-Gaugry]
\label{lacunarityVG06}
Let $\beta > 1$ be an algebraic number such that 
$d_{\beta}(1)$ is infinite
and gappy in the sense that there exist
two infinite sequences
$\{m_n\}_{n \geq 1}$ 
and $\{s_n\}_{n \geq 0}$
such that
$$1 = s_0 \leq m_1 < s_1 \leq m_2 < s_2 \leq \ldots
\leq m_n < s_n \leq m_{n+1}
< s_{n+1} \leq \ldots$$
with $(s_n - m_n) \geq 2$, $t_{m_n} \neq 0$,
$t_{s_n} \neq 0$
and $t_i = 0$
if
$m_n < i < s_n$ for all $n \geq 1$. 
Then
\begin{equation}
\label{gappiness}
\limsup_{n \to +\infty} \frac{s_n}{m_n}
\leq \frac{\lo ({\rm M}(\beta))}{\lo \beta}
\end{equation}
\end{theorem}

\begin{proof}
\cite{vergergaugry}, Theorem 1.1.
\end{proof}

Theorem \ref{lacunarityVG06} 
also became a consequence of 
Theorem 2 in \cite{adamczewskibugeaud}.
In Theorem \ref{lacunarityVG06} the quotient
$s_n/m_n , n \geq 1$, is called the 
$n$-th Ostrowski quotient of the sequence
$(t_i)_{i \geq 1}$. For a given algebraic number 
$\beta> 1$, whether the upper bound
\eqref{gappiness} is exactly 
the limsup of the sequence of the 
Ostrowski quotient of $(t_i)_{i \geq 1}$
is unknown. For Salem numbers,
this equality always holds since
${\rm M}(\beta) = \beta$, and the upper bound
\eqref{gappiness} is 1.
\vspace{0.2cm}

\noindent
{\it Varying the base of numeration $\beta$
in the interval $(1,2)$:}
for all $\beta \in (1, 2)$, 
being an algebraic number or 
a transcendental number, the alphabet
$\mathcal{A}_{\beta}$
of the $\beta$-shift is
always the same: $\{0,1\}$.
All the digits of all $\beta$-expansions
$d_{\beta}(1)$ are zeroes or ones.
Parry (\cite{parry})
has proved that
the relation of order 
$1 < \alpha < \beta < 2$
is preserved
on the corresponding
greedy $\alpha$- and $\beta$- expansions
$d_{\alpha}(1)$ and
$d_{\beta}(1)$
as follows.

\begin{proposition}
\label{variationbasebeta}
Let $\alpha > 1$ and $\beta > 1$. 
If the R\'enyi $\alpha$-expansion of 1 is
$$d_{\alpha}(1) = 0. t'_1 t'_2 t'_3\ldots, 
\qquad ~i.e.
\quad
1 ~=~ \frac{t'_1}{\alpha} + \frac{t'_2}{\alpha^2} + \frac{t'_3}{\alpha^3} + \ldots$$
and the R\'enyi $\beta$-expansion of 1 is
$$d_{\beta}(1) = 0. t_1 t_2 t_3\ldots, 
\qquad ~i.e. 
\quad
1 ~=~ \frac{t_1}{\beta} + \frac{t_2}{\beta^2} + \frac{t_3}{\beta^3} + \ldots,$$
then $\alpha < \beta$ if and only if $(t'_1, t'_2, t'_3, \ldots) 
<_{lex} (t_1, t_2, t_3, \ldots)$. 
\end{proposition}

\begin{proof}
Lemma 3 in Parry \cite{parry}.
\end{proof}

For any integer $n \geq 1$ the sequence
of digits $1 0^{n-1} 1$, with $n-1$ times ``$0$"
between the two ones, is self-admissible. By
Proposition \ref{betacharacterized} 
it defines an unique
solution $\beta \in (1,2)$ of \eqref{equabase}.
Denote by $\theta_{n+1}^{-1}$ this solution.
From Proposition \ref{variationbasebeta}  
we deduce that the sequence 
$(\theta_{n}^{-1})_{n \geq 2}$ 
is (strictly) decreasing and tends to $1$ when
$n$ tends to infinity.

From \eqref{equabase} 
the real number $\theta_{2}^{-1}$
is the unique root $> 1$ of
the equation $1 = 1/x + 1/x^2$, that is
of $X^2 - X -1$.
Therefore it is the
Pisot number (golden mean)
 $= \frac{1+\sqrt{5}}{2} = 1.618\ldots$. 
Being interested in bases $\beta > 1$ 
close to 1 
tending to $1^+$, we will focus on
the interval $( 1, \frac{1+\sqrt{5}}{2}\, ]$ 
in the sequel.
This interval is
partitioned by
the decreasing sequence
$(\theta_{n}^{-1})_{n \geq 2}$ as
\begin{equation}
\label{jalonnement}
\bigl( 1, \frac{1+\sqrt{5}}{2}\, \bigr] ~=~
\bigcup_{n=2}^{\infty}
\left[ \, \theta_{n+1}^{-1} , \theta_{n}^{-1}
\, 
\right)
~~ \bigcup ~~\left\{
\theta_{2}^{-1}
\right\}.
\end{equation}

Theorem \ref{lacunarityVG06} gives 
an upper bound of
the asymptotic behaviour of the
Ostrowski quotients of the
$\beta$-expansion $(t_i)_{i \geq 1}$ of 1, 
due to the fact that
$\beta > 1$ is an algebraic number. 
The following theorem shows that the gappiness 
of $(t_i)_{i \geq 1}$ also admits some
uniform lower bound, for all
gaps of zeroes.
The condition of minimality on the 
length of the gaps of zeroes 
in $(t_i)_{i \geq 1}$
is only a function of the interval  
$\left[ \, \theta_{n+1}^{-1} , \theta_{n}^{-1}
\, 
\right)$
to which $\beta$ belongs,
when $\beta$ tends to 1.

\begin{theorem}
\label{zeronzeron}
Let $n \geq 2$. A real number 
$\beta \in ( 1, \frac{1+\sqrt{5}}{2}\, ]$ 
belongs to   
$[\theta_{n+1}^{-1} , \theta_{n}^{-1})$ if and only if the 
R\'enyi $\beta$-expansion of unity is of the form
\begin{equation}
\label{dbeta1nnn}
d_{\beta}(1) = 0.1 0^{n-1} 1 0^{n_1} 1 0^{n_2} 1 0^{n_3} \ldots,
\end{equation}
with $n_k \geq n-1$ for all $k \geq 1$.  
\end{theorem}

\begin{proof} 
Since $d_{\theta_{n+1}^{-1}}(1) = 0.1 0^{n-1} 1$ and
$d_{\theta_{n}^{-1}}(1) = 0.1 0^{n-2} 1$, 
Proposition \ref{variationbasebeta} implies that
the condition is sufficient. It is also necessary:
$d_{\beta}(1)$ begins as $0.1 0^{n-1} 1$
for all $\beta$ such that
$\theta_{n+1}^{-1} \leq \beta < \theta_{n}^{-1}$.
For such $\beta$s 
we write $d_{\beta}(1) = 0.1 0^{n-1} 1 u$~  
with digits in the alphabet
$\mathcal{A}_{\beta}
=\{0, 1\}$ common to all $\beta$s, that is
$$u= 1^{h_0} 0^{n_1} 1^{h_1} 0^{n_2} 1^{h_2} \ldots$$
and $h_0, n_1, h_1, n_2, h_2, \ldots$ integers $\geq 0$.
The self-admissibility lexicographic condition
\eqref{self} applied to the sequence
$(1, 0^{n-1}, 1^{1+h_0}, 0^{n_1}, 1^{h_1}, 0^{n_2}, 1^{h_3}, \ldots)$,
which characterizes uniquely the base of numeration $\beta$, 
readily implies
$h_0 = 0$ and 
$~h_k = 1$ and 
$n_k \geq n-1$ for all $k \geq 1$.
\end{proof}

\begin{remark}
The case $n_1 = +\infty$ in \eqref{dbeta1nnn}
corresponds to the simple Parry number
$\beta = \theta_{n+1}^{-1}$. The value
$+\infty$ is not excluded from the set
$(n_k)_{k \geq 1}$ in the following sense:
if there exists $j \geq 2$ such that
$n-1 \leq n_k < +\infty, ~k < j$, with
$n_j = +\infty$, then $\beta$ is a simple Parry number in
$[\theta_{n+1}^{-1} , \theta_{n}^{-1})$
characterized by (cf Section $\S$\ref{S2}):
$$d_{\beta}(1) = 0 . 1 0^{n-1} 1 0^{n_1} 1 0^{n_2} 1 \ldots
1 0^{n_{j} - 1} 1.$$
All the simple Parry numbers lying
in the interval 
$[\theta_{n+1}^{-1} , \theta_{n}^{-1})$
are obtained in this way. 
On the contrary,
the transcendental numbers $\beta$ in
$[\theta_{n+1}^{-1} , \theta_{n}^{-1})$
have all R\'enyi $\beta$-expansions
$d_{\beta}(1) = 0 . t_1 t_2 t_3 \ldots$ of 1
such that the sequence of exponents 
$(n_k)_{k \geq 1}$ of the successive zeroes, 
corresponding to the sequence
of the lengths of the gaps of zeroes, 
never takes the value $+\infty$.
\end{remark}

\begin{definition}
Let $\beta \in ( 1, \frac{1+\sqrt{5}}{2}\, ]$ be a real number.
The integer $n \geq 3$ such that
$\theta_{n}^{-1} \leq \beta < \theta_{n-1}^{-1}$
is called the {\it dynamical degree} 
of $\beta$, and
is denoted by ${\rm dyg}(\beta)$.
By convention we put:
${\rm dyg}(\frac{1+\sqrt{5}}{2}) = 2$. 
\end{definition}

The function $n={\rm dyg}(\beta)$ is locally
constant on the interval 
$(1, \frac{1+\sqrt{5}}{2}]$, 
is decreasing, 
takes all values in $\mathbb{N} \setminus \{0,1\}$,
and satisfies:
$\lim_{\beta > 1, \beta \to 1} \dyg(\beta) = +\infty$.
The relations between %the restriction of 
the 
dynamical degree ${\rm dyg}(\beta)$ 
%to
%$\overline{\qb} \cap 
%(1, \frac{1+\sqrt{5}}{2}]$ 
and
the (usual) degree $\deg(\beta)$ will be 
investigated later (Theorem
\ref{dygdeginequality}; $\S$ \ref{S5}, $\S$ \ref{S6.5}).
%defined on the set of all algebraic numbers $\beta$
%in $( 1, \frac{1+\sqrt{5}}{2}\, ]$, and will be related 
%to the degree deg$(\beta)$ of the minimal polynomial of
%$\beta$ in Section \ref{S6}. 
Let us observe that
the equality
deg$(\beta) =$ dyg$(\beta) = 2$ 
holds if $\beta = \frac{1+\sqrt{5}}{2}$, but 
the equality case is not the case in general.

\begin{definition}
\label{selfadmissiblepowerseries}
A power series $\sum_{j=0}^{+\infty} a_j z^j$,
with $a_j \in \{0, 1\}$ for all $j \geq 0$,
$z$ the complex variable,
is said to be {\it Lyndon (or self-admissible)}
if its
coefficient vector $(a_i)_{i \geq 0}$ 
is Lyndon.
\end{definition}

\section{Generalized Fredholm theory, dynamical zeta function, Perron-Frobenius operator, transfer operator, Parry Upper function, and the $\beta$-shift}
\label{S3}

\subsection{The Parry Upper function, the Parry polynomial}
\label{S3.1}

\begin{definition}
\label{parryupperfunction}
Let $\beta \in (1, (1+\sqrt{5})/2]$ be a real
number, and $d_{\beta}(1) = 0. t_1 t_2 t_3 \ldots$ its
R\'enyi $\beta$-expansion of 1.
The power series $f_{\beta}(z) :=
-1 + \sum_{i \geq 1} t_i z^i$
of the complex variable $z$
is called the {\it Parry Upper function} 
at $\beta$.
\end{definition}

In this paragraph a presentation of
the Parry Upper function is given from 
the side of generalized Fredholm Theory,
to show the relations between
Fredholm determinants, 
generalized Fredholm determinants,
and (weighted) dynamical zeta functions
(introduced by Ruelle
\cite{ruelle} \cite{ruelle2} \cite{ruelle3} \cite{ruelle4}
\cite{ruelle5} \cite{ruelle6}
\cite{ruelle7} \cite{ruelle8}
\cite{ruelle9}
by analogy with the
thermodynamic formalism
of statistical mechanism \cite{ruelle4},
and recently developped e.g. by
Baladi \cite{baladi} \cite{baladi2}
\cite{baladi3},
Baladi and Keller \cite{baladikeller},
Hofbauer \cite{hofbauer},
Hofbauer and Keller \cite{hofbauerkeller},
Milnor and Thurston \cite{milnorthurston},
Parry and Pollicott \cite{parrypollicott},
Pollicott \cite{pollicott} \cite{pollicott2},
Preston \cite{preston},
Takahashi \cite{takahashi3}
\cite{takahashi4}).

\begin{proposition}
\label{fbetainfinie}
For $1 < \beta <(1+\sqrt{5})/2$ any real
number,
with $d_{\beta}(1) = 0. t_1 t_2 t_3 \ldots$,
the Parry Upper function $f_{\beta}(z)$
is such that $f_{\beta}(1/\beta) = 0$. 
It is such that
$f_{\beta}(z) + 1$ has coefficients 
in the alphabet
$\mathcal{A}_{\beta} =\{0,1\}$ and is 
Lyndon.
It takes the form
\begin{equation}
\label{fbetalyndon}
f_{\beta}(z) = G_{\dyg(\beta)} + z^{m_1} + 
z^{m_2} + \ldots + z^{m_q} + z^{m_{q+1}} + \ldots
\end{equation}
with~ $m_1 - \dyg(\beta) \geq  \dyg(\beta) -1$,
$m_{q+1} - m_q \geq  \dyg(\beta) -1$ for
$q \geq 1$.
Conversely, given a power series 
\begin{equation}
\label{fbetalyndonconverse}
-1 + z + z^n  + z^{m_1} + 
z^{m_2} + \ldots + z^{m_q} + z^{m_{q+1}} + \ldots
\end{equation}
with $n \geq 3$,
$m_1 - n \geq  n -1$,
$m_{q+1} - m_q \geq  n -1$ for
$q \geq 1$, then there exists an unique
$\beta \in (1, (1+\sqrt{5})/2)$ for which
$n = \dyg(\beta)$ with
$f_{\beta}(z)$ equal to
\eqref{fbetalyndonconverse}.
 
Moreover, if $\beta$,
$1 < \beta <(1+\sqrt{5})/2$,
is a reciprocal
algebraic integer, the power series
\eqref{fbetalyndon} is never a polynomial.
\end{proposition}

\begin{proof}
The expression of 
$f_{\beta}(z)$ readily comes
from Theorem \ref{zeronzeron}.
Let us prove the last claim.
Assume that $\beta$ is a reciprocal
algebraic integer and that 
$f_{\beta}(z)$ is a polynomial. The
polynomial $f_{\beta}(z)$ would vanish at
the two real zeroes
$\beta$ and $1/\beta$. But the sequence
$-1\, t_1\, t_2\, t_3 \ldots$ has only 
one sign change. 
By Descartes's rule 
we obtain a contradiction.
\end{proof}

The lacunarity of $f_{\beta}(z)$ is 
moderate since the Ostrowki quotients
have an asymptotic upper bound, by
Theorem \ref{lacunarityVG06}.
By Theorem \ref{zeronzeron}, 
any gap of missing monomials in
$f_{\beta}(z)$ has
a length greater than or equal to 
%the dynamical degree
$\dyg(\beta) - 1$ what 
controls the lacunarity a minima.

The definition of 
$f_{\beta}(z)$ seems simple
since
the vector coefficient of
$f_{\beta}(z) + 1$
is only a sequence of integers 
deduced from the
orbit of 1 under the iterates of the
$\beta$-transformation
$T_{\beta}$, by \eqref{digits__ti} 
and \eqref{polyTbeta}
\cite{flattolagarias}
\cite{flattolagariaspoonen};
nevertheless it is deeply related 
to the Artin-Mazur dynamical zeta function
$\zeta_{\beta}(z)$
(given by \eqref{dynamicalfunction})
of the R\'enyi-Parry dynamical system
$([0, 1], T_{\beta})$, to
the Perron-Frobenius operator $P_{T_{\beta}}$
associated with $T_{\beta}$,
to the transfer operator of $T_{\beta}$
 and to the
generalized ``Fredholm determinant" \eqref{fredholmdeterminantgeneralized} of this operator. 
In the kneading theory of Milnor and Thurston
\cite{milnorthurston} it is a kneading determinant.
Let us recall these links, knowing
that the theory of Fredholm 
(Grothendieck
\cite{grothendieck} \cite{grothendieck2}, 
Riesz and Nagy \cite{riesznagy} Chap. IV) is done for
compact operators while the Perron-Frobenius
operators associated with the $\beta$-transformations
$T_{\beta}$
are noncompact by nature (Mori \cite{mori}
\cite{mori2}, Takahashi \cite{takahashi}
\cite{takahashi2} \cite{takahashi5}).

Let $(X,\Sigma, \mu)$ be a 
$\sigma$-finite measure space and let 
$T:X \to X$ be a nonsingular transformation, i.e. 
$T$ is measurable and satisfies: for all $A \in \Sigma$,
$\mu(A) = 0 \Longrightarrow 
\mu(T^{-1}(A)) = 0$.
In ergodic theory, by the Radon-Nikodym theorem, 
the operator
$P_T : L^{1}(X,\Sigma, \mu)
\to L^{1}(X,\Sigma, \mu)$ defined by
\begin{equation}
\label{perronfrobeniusdefinition}
\int_A P_{T} f d \mu ~=~
\int_{T^{-1}(A)} f d\mu 
\end{equation}
is called the Perron-Frobenius operator associated with
$T$. 
Let $\beta \in (1, \theta_{2}^{-1})$,
$X = [0, 1]$, $\Sigma$ the Borel 
$\sigma$-algebra and 
$T_{\beta}$ the $\beta$-transformation.
The $T_{\beta}$-invariant probability measure
$\mu = \mu_{\beta}$ of the 
$\beta$-shift, on $\Sigma$, 
is unique (R\'enyi \cite{renyi}), ergodic
(Parry \cite{parry}), 
maximal (Hofbauer \cite{hofbauer})
and absolutely 
continuous with respect to the Lebesgue measure
$dt$, with Radon-Nikodym derivative 
(Lasota and Yorke \cite{lasotayorke},
Parry \cite{parry}, 
Takahashi \cite{takahashi}) :
$$h_{\beta} 
~=~
C \,\sum_{n: x < T_{\beta}^{n}(1)} 
\frac{1}{\beta^{n+1}} ,\qquad 
\mbox{ so that}
\qquad d\mu_{\beta} = h_{\beta} dt,
$$ 
for some constant $C > 0$.
These results were independently discovered by
A.O. Gelfond
\cite{fallerpfister}.
We denote 
by $P_{T_{\beta}}$ the Perron-Frobenius operator
associated with $T_{\beta}$.

The $\beta$-transformation 
$T_{\beta}$ is a 
piecewise monotone map of the interval
$[0, 1]$ with 
weight function
$g = 1$. 

In the context of
noncompact operators, the objective consists in 
giving a sense to
\begin{equation}
\label{fredholmdeterminantgeneralized}
` {\rm det}` \, (Id - z \, L_{t})
=
\exp\Bigl(
- \sum_{n \geq 1} \frac{`tr` \,L_{t}^{n}}{n}\, z^n
\Bigr),
 \end{equation}
where $L_{t}$ is 
a dynamically defined weighted transfer operator
acting on a suitable Banach space
\cite{baladi2}  \cite{hofbauerkeller2}. 

Let $1 < \beta < (1 + \sqrt{5})/2$ be a real number
and $0 = a_0 < a_1 = \frac{1}{\beta} < a_2 =1$
be the finite partition of $[0, 1]$.
The map $T_{\beta}$ is strictly monotone and continuous
on $[a_0 , a_1)$ and $[a_1 , 1]$.
The $\beta$-transformation $T_{\beta}$
is one of the simplest transformations among
piecewise monotone intervals maps
(Baladi and Ruelle
\cite{baladiruelle},
Milnor and Thurston \cite{milnorthurston},
Pollicott \cite{pollicott}). 
For each function
$f: [0, 1] \to \cb$, let
$${\rm var}(f) :=
\sup\Bigl\{ \,\sum_{i=1}^{n}\, |f(e_i) - f(e_{i-1})|
\mid
n \geq 1 , 0 \leq e_1 \leq e_2 \leq \ldots
\leq e_n \leq 1 \Bigr\},$$
$$\|f\|_{BV} :=
{\rm var}(f) + \sup(|f|),$$
and denote by
BV the Banach space of functions 
with bounded variation
\cite{keller}
\cite{keller2}:
$$BV := \{f: [0, 1] \to \cb \mid
\|f\|_{BV} < \infty \}.$$
For $g \in BV$, one 
can define the following transfer operator
$$L_{t \beta, g}: BV \to BV,
\quad 
L_{t \beta, g} f(x) := 
\sum_{y , T_{\beta}(y) = x} g(y) f(y) .$$
We will only consider the case 
$g \equiv 1$ in the sequel and
put $L_{t \beta} :=
L_{t \beta,1}$.

\begin{theorem}
\label{dynamicalzetatransferoperator}
Let $\beta \in (1,\theta_{2}^{-1})$. Then,
\begin{itemize}
\item[(i)] the Artin-Mazur
dynamical zeta function
$\zeta_{\beta}(z)$ defined 
by 
\begin{equation}
\label{dynamicalfunction}
\zeta_{\beta}(z) := \exp\Bigl(
\sum_{n=1}^{\infty} \, 
\frac{\#\{x \in [0,1] \mid 
T_{\beta}^{n}(x) = x\}}{n} \, z^n
\Bigr),
\end{equation}
counting the number of periodic 
points of period dividing $n$,
is nonzero and meromorphic
in $\{|z| < 1\}$, and such that
$1/\zeta_{\beta}(z)$ is holomorphic
in $\{|z| < 1\}$,
\item[(ii)] suppose 
$|z| < 1$. Then $z$ is a pole of
$\zeta_{\beta}(z)$ of multiplicity $k$
if and only if
$z^{-1}$ is an eigenvalue of $L_{t \beta}$ of multiplicity 
$k$.
\end{itemize}
\end{theorem}

\begin{proof}
Theorem 2 in 
\cite{baladikeller}, 
assuming that
the set of intervals 
$([0, a_1),[a_1 , 1])$
forming the partition of
$[0,1]$ is generating;
In \cite{ruelle5} \cite{ruelle8} Ruelle shows that this assumption is not necessary, showing how to remove this obstruction.
\end{proof}

Theorem \ref{dynamicalzetatransferoperator}
was %proved and 
stated in Baladi and Keller
\cite{baladikeller} under more general
assumptions.
Theorem \ref{dynamicalzetatransferoperator} 
has been conjectured by Hofbauer and Keller
\cite{hofbauerkeller} 
for piecewise monotone maps, for the case
where the function $g$ is piecewise constant.
(cf also Mori \cite{mori} \cite{mori2}).
The case $g=1$ in the transfer operators
was studied by Milnor and Thurston \cite{milnorthurston},
Hofbauer \cite{hofbauer}, 
Preston \cite{preston}.

The fact 
(Theorem \ref{dynamicalzetatransferoperator} (ii)) 
that the poles of $\zeta_{\beta}(z)$,
lying in the open unit disc,
are of the same multiplicity of the inverses of the 
eigenvalues 
of the transfer
operator $L_{t \beta}$ is a 
extension of Theorem 2, Theorem 3 and 
Theorem 4 in
Grothendieck \cite{grothendieck2} in the 
context of the Fredholm theory 
with compact operators.

When $\beta > 1$ is a reciprocal 
algebraic integer and
tends to $1^+$,
we will prove in Section $\S$ \ref{S5}
that the multiplicity $k$ is equal to
$1$ for the first pole $1/\beta$ 
of $\zeta_{\beta}(z)$ and for
a subcollection of Galois conjugates
of $1/\beta$ in an angular sector.

The relations between 
the poles of the dynamical
zeta function $\zeta_{\beta}(z)$,
the zeroes of the Parry Upper function
$f_{\beta}(z)$ and the eigenvalues of the
transfer operator $L_{t \beta}$
come from Theorem \ref{dynamicalzetatransferoperator}
and from the following theorem.

\begin{theorem}
\label{parryupperdynamicalzeta}
Let $\beta > 1$ be a real number.
Then the Parry Upper function
$f_{\beta}(z)$ satisfies
\begin{equation}
\label{parryupperdynamicalzeta_i}
(i)\qquad f_{\beta}(z) = - \frac{1}{\zeta_{\beta}(z)} 
\qquad \mbox{if}~ \beta ~\mbox{is not a simple Parry number},
\end{equation}
and
\begin{equation}
\label{parryupperdynamicalzeta_ii}
(ii)\quad
f_{\beta}(z) = - \frac{1 - z^N}{\zeta_{\beta}(z)}
\qquad
\mbox{if $\beta$ is a simple Parry number}
\end{equation}
where $N$, 
which depends upon $\beta$, is the minimal positive integer such that $T_{\beta}^{N}(1) = 0$.
It is holomorphic in the open unit disk
$\{|z| < 1\}$.
It has no zero in
$|z| \leq 1/\beta$ except $z=1/\beta$
which is a simple zero.
The Taylor series of 
$f_{\beta}(z)$ at $z=1/\beta$
is $f_{\beta}(z)
=
c_{\beta, 1} \bigl(z - \frac{1}{\beta}\bigr)
+ c_{\beta, 2} \bigl(z - \frac{1}{\beta}\bigr)^2+ \ldots$
with
\begin{equation}
\label{coefficientsfbetaatbetainverse}
c_{\beta, m} = 
\sum_{n=m}^{\infty} \frac{n !}{(n-m) ! ~m !}%n (n-1)\ldots(n-m+1)
\lfloor \beta T_{\beta}^{n-1}(1) \rfloor 
\bigl(\frac{1}{\beta}\bigr)^{n-m}
 ~> 0,
\qquad \mbox{for all}~ m \geq 1.
\end{equation}
\end{theorem}

\begin{proof}
Theorem 2.3 and Appendix A
in Flatto, Lagarias and Poonen 
\cite{flattolagariaspoonen}; 
Theorem 1.2 in
Flatto and Lagarias 
\cite{flattolagarias}, I;
Theorem 3.2
in Lagarias \cite{lagarias}.
From Takahashi \cite{takahashi},
Ito and Takahashi
\cite{itotakahashi}, these authors
deduce 
\begin{equation}
\label{zetafunctionfraction}
\zeta_{\beta}(z) = \frac{1 - z^N}{(1 - \beta z)
\Bigl(
\sum_{n=0}^{\infty}T_{\beta}^{n}(1) \, z^n
\Bigr)}
\end{equation}
where ``$z^N$" has to be replaced by ``$0$"
if $\beta$ is not a simple Parry number.
Since $\beta T_{\beta}^{n}(1) =
\lfloor \beta T_{\beta}^{n}(1)
\rfloor +
\{\beta T_{\beta}^{n}(1)\} = 
t_{n+1} + T_{\beta}^{n+1}(1)$
by \eqref{digits__ti}, for $n \geq 1$,
expanding the power series of the denominator
\eqref{zetafunctionfraction} readily gives:
\begin{equation}
\label{fbetazetabeta}
-1 +t_1 z + t_2 z^2 +\ldots
= f_{\beta}(z) 
= 
-(1 - \beta z)
\Bigl(
\sum_{n=0}^{\infty}T_{\beta}^{n}(1) \, z^n
\Bigr).
\end{equation}
The zeroes of smallest modulus are
characterized in Lemma 5.2, Lemma 5.3 and Lemma 5.4
in \cite{flattolagariaspoonen}. The coefficients
$c_{\beta,m}$ readily come from the derivatives of
$f_{\beta}(z)$.
\end{proof}

From Theorem \ref{parryupperdynamicalzeta},
since the roots of cyclotomic polynomials
are of modulus 1, the zeroes of 
$f_{\beta}(z)$ within $|z|< 1$ are always
exactly
the poles of $\zeta_{\beta}(z)$ in this domain,
whatever the R\'enyi-Parry dynamics of $\beta > 1$
is.

\begin{definition}
\label{parrypolynomial}
If $\beta$ is a simple Parry number, 
with $d_{\beta}(1) = 0 . t_1 t_2 \ldots t_m$,
$t_m \neq 0$,
the polynomial
\begin{equation}
\label{parrypolynomesimple}
P_{\beta, P}(X) :=
X^m - t_1 X^{m-1} - t_2 X^{m-2} - \ldots t_m
\end{equation}
is called the {\it Parry polynomial} of $\beta$.
If $\beta$ is a Parry number which is not simple,
with
$d_{\beta}(1) = 0 . t_1 t_2 \ldots t_m
(t_{m+1} t_{m+2} \ldots t_{m+p+1})^{\omega}$
and not purely periodic 
($m$ is $\neq 0$), then
$$P_{\beta, P}(X) :=
X^{m+p+1} - t_1 X^{m+p} - t_2 X^{m+p-1}
- \ldots - t_{m+p} X - t_{m+p+1}
\qquad \qquad\mbox{ }$$
\begin{equation}
\label{parrypolynomeperiodicity}
\mbox{ } \qquad \qquad \qquad \qquad
-X^m + t_1 X^{m-1} + t_2 X^{m-2}
+ \ldots + t_{m-1} X + t_{m}
\end{equation}
is the Parry polynomial of $\beta$.
If $\beta$ is a nonsimple Parry number
such that
$d_{\beta}(1) = 0 . 
(t_{1} t_{2} \ldots t_{p+1})^{\omega}$
is purely periodic (i.e. $m=0$),
then
\begin{equation}
\label{parrypolynomepureperiodicity}
P_{\beta, P}(X) :=
X^{p+1} - t_1 X^{p} - t_2 X^{p-1}
- \ldots - t_{p} X - (1 + t_{p+1})
\end{equation}
is the Parry polynomial of $\beta$.
By definition the degree $d_P$ of
$P_{\beta,P}(X)$
is 
respectively $m, m + p +1, p+1$
in the three cases.
\end{definition}

If $\beta$ is a Parry number,
the Parry polynomial $P_{\beta, P}(X)$,
belonging to the ideal 
$P_{\beta}(X) \zb[X]$,
admits $\beta$ as simple root
and is often not irreducible
\cite{vergergaugry2}
\cite{bertrandmathis2}.
The polynomial 
$\frac{P_{\beta, P}(X)}{P_{\beta}(X)}$ 
has been called {\em complementary factor} 
by Boyd. 
For the two cases
\eqref{parrypolynomesimple}
and \eqref{parrypolynomepureperiodicity}
the constant term is $\neq 0$; hence
$\deg(P^{*}_{\beta,P}) =
\deg(P_{\beta,P})$.
In the case of \eqref{parrypolynomeperiodicity}
denote by
$q_{\beta} := 0$ if 
$t_{m} \neq t_{m+p+1}$
and, if $t_{m} = t_{m+p+1}$,
$q_{\beta} := 1 +
\max\{r \in \{0, 1, m-1\} 
\mid t_{m-l} = t_{m+p+1-l} \mbox{~for all}~
0 \leq l \leq r \}$.
Then
$p+1 \leq \deg(P_{\beta,P}) - q_{\beta}
= \deg(P^{*}_{\beta,P}) \leq \deg(P_{\beta,P})$.

Applying the Carlson-Polya dichotomy 
(Bell and Chen \cite{bellchen},
Bell, Miles and Ward
\cite{bellmilesward},
Carlson \cite{carlson} \cite{carlson2}, 
Dienes \cite{dienes},
P\'olya \cite{polya}, Robinson \cite{robinson},
Szeg\H{o} \cite{szego}) to
the power series $f_{\beta}(z)$, for which the coefficients
belong to the finite set
$\mathcal {A}_{\beta} \cup \{-1\}$, gives the following
equivalence.

\begin{theorem}
\label{carlsonpolya}
The real number $\beta> 1$
is a Parry number if and only if
the Parry Upper function
$f_{\beta}(z)$ is a rational function, 
equivalently if and only
if $\zeta_{\beta}(z)$ is a rational function.

The set of Parry numbers, resp.
of nonParry numbers,
in $(1, \infty)$, is not empty.
If $\beta$ is not a Parry number, 
then 
$|z| = 1$ is the natural boundary
of  
$f_{\beta}(z)$. 
If $\beta$ is a Parry number,
with R\'enyi $\beta$-expansion of $1$ given by
$$d_{\beta}(1) = 0 . t_1 t_2 \ldots t_m (t_{m+1} t_{m+2} \ldots
t_{m+p+1})^{\omega}, \qquad t_1 = \lfloor \beta \rfloor,
~t_i \in \mathcal{A}_{\beta}, i \geq 2,$$
the preperiod length being
$m \geq 0$ and the period length $p+1 \geq 1$,
$f_{\beta}(z)$ admits an analytic 
meromorphic extension over $\cb$, of
the following form:
$$
f_{\beta}(z) =
- P_{\beta, P}^{*}(z) \qquad
\mbox{if $\beta$ is simple,}
$$
$$
f_{\beta}(z) =
\frac{- P_{\beta, P}^{*}(z)}{1 - z^{p+1}}
\qquad 
\mbox{if $\beta$ is nonsimple,}
$$
where the Parry polynomial is given by
\eqref{parrypolynomesimple},
\eqref{parrypolynomeperiodicity}
or
\eqref{parrypolynomepureperiodicity}.

If $\beta \in (1, 2)$ is a Parry number,
the (na\"ive) height 
${\rm H}(P_{\beta, P})$
of $P_{\beta, P}$
is equal to 1 except when: 
$\beta$ is nonsimple 
and that 
$t_{p+1} = \lfloor \beta T_{\beta}^{p}(1) \rfloor = 1$,
in which case the Parry polynomial of $\beta$
has na\"ive height
${\rm H}(P_{\beta, P}) = 2$.
\end{theorem}

\begin{proof}
Verger-Gaugry \cite{vergergaugry6}.
The set of nonParry numbers $\beta$
in $(1, \infty)$ is not empty
as a consequence of Fekete-Szeg\H{o}'s Theorem 
\cite{feketeszego} 
since the radius of convergence
of $f_{\beta}(z)$ is equal to 1 in any case
(whatever the R\'enyi-Parry dynamics
of $\beta$ is),
and that its domain of definition always contains
the open unit disk which has a transfinite
diameter equal to 1.
The set of Parry numbers $\beta$
in $(1, \infty)$ is also nonempty.
Indeed Pisot numbers, of degree $\geq 2$,
are Parry numbers
(Schmidt \cite{schmidt}, 
Bertrand-Mathis \cite{bertrandmathis}).
Therefore the dichotomy between
Parry and nonParry numbers in
$(1,\infty)$ has a sense.
\end{proof}

\begin{definition}
\label{betaconjugate}
Let $\beta > 1$ be a Parry number.
If the Parry polynomial $P_{\beta,P}(z)$
of $\beta$ is not irreducible,
the roots of 
$P_{\beta,P}(z)$ which are not Galois conjugates
of $\beta$ are called the {\it 
beta-conjugates} of $\beta$. 
\end{definition}
Beta-conjugates were studied in
\cite{vergergaugry4} 
\cite{vergergaugry5}
in terms of Puiseux theory and
in association with germs of curves.

\subsection{Distribution of zeroes of Parry Upper functions $f_{\beta}(z)$ in Solomyak's fractal}
\label{S3.2}

Let $\beta> 1$ be a real number (algebraic 
or transcendental).
The Parry Upper function $f_{\beta}(z)$
has its zeroes of modulus $< 1$
in a region of the open unit disk, called
Solomyak's factal, 
whose construction
is given
in \cite{solomyak} \S 3. 
Let us recall it and
summarize its arithmetic properties
in 
Theorem \ref{solomyakfractalGBordG}.
From Theorem \ref{parryupperdynamicalzeta}
and \eqref{zetafunctionfraction},
the zeroes of $f_{\beta}(z)$ in $|z| < 1$,
which are $\neq 1/\beta$,
are the zeroes of modulus $< 1$
of the power series 
$1 + \sum_{j=1}^{\infty} T_{\beta}^{j}(1) \, z^j$
where the coefficients are
real numbers in the interval
$[0,1]$.
Then, in full generality, let 
$$\mathcal{B} :=
\{h(z) = 1 + \sum_{j=1}^{\infty} a_j z^j
\mid a_j \in [0, 1] \}$$
be the class of power series defined on
$|z| < 1$ equipped with the 
topology of uniform convergence
on compacts sets of $|z| < 1$.  
The subclass $\mathcal{B}_{0,1}$
of $\mathcal{B}$ denotes functions whose 
coefficients are all zeros or ones. 
The space 
$\mathcal{B}$ is compact and convex.
Let
$$\mathcal{G} := \{\lambda \mid |\lambda| < 1,
\exists ~h(z) \in \mathcal{B} ~\mbox{such that}~ 
h(\lambda)=0 \}
\quad
\subset ~D(0,1)$$
be the set of zeroes of the power series 
belonging to $\mathcal{B}$.
The zeroes
gather within the unit circle and
curves in $|z| < 1$
given in polar coordinates, by
\cite{vergergaugry2}.
The complement 
$D(0,1) \setminus (\mathcal{G} \cup \{\frac{1}{\beta}\})$
is a zero-free region for 
$f_{\beta}(z)$; the
domain $D(0,1) \setminus \mathcal{G}$
is 
star-convex due to the fact that:
$h(z) \in \mathcal{B} \Longrightarrow
h(z/r) \in \mathcal{B}$, for
any $r > 1$ (\cite{solomyak}, \S 3),
and that $1/\beta$ is the unique root
of $f_{\beta}(z)$ in $(0,1)$.

For every $\phi \in (0, 2 \pi)$, there exists
$\lambda = r e^{i \phi} \in \mathcal{G}$;
the point of minimal modulus
with argument $\phi$ is denoted 
$\lambda_{\phi} =
\rho_{\phi} e^{i \phi}
\in \mathcal{G}$,
$\rho_{\phi} < 1$. 
A function
$h \in \mathcal{B}$ is called $\phi$-optimal if
$h(\lambda_{\phi})  = 0$.
Denote by $\mathcal{K}$ the subset of
$(0, \pi)$ for which there exists a
$\phi$-optimal function belonging to
$\mathcal{B}_{0,1}$.
Denote by $\partial \mathcal{G}_{S}$ the
``spike": $[-1, \frac{1}{2}(1 - \sqrt{5})]$
on the negative real axis.

\begin{theorem}[Solomyak]
\label{solomyakfractalGBordG}
(i) The union $\mathcal{G} \cup \mathbb{T}
\cup \partial \mathcal{G}_{S}$ is closed,
symmetrical with respect to the real axis, has
a cusp at $z=1$ with logarithmic tangency 
(Figure 1 in {\rm \cite{solomyak}}),
 
(ii) the boundary $\partial \mathcal{G}$ 
is a continuous curve, given by
$\phi \to |\lambda_{\phi}|$ on 
$[0, \pi)$, taking its values in
$[\frac{\sqrt{5} - 1}{2}, 1)$,
with 
$|\lambda_{\phi}| = 1$ if and only if $\phi = 0$. 
It admits
a left-limit at $\pi^{-}$,
$1 > \lim_{\phi \to \pi^{-}} |\lambda_{\phi}|
> |\lambda_{\pi}| = \frac{1}{2}(-1 + \sqrt{5})$,
the left-discontinuity at $\pi$ corresponding 
to the extremity of $\partial \mathcal{G}_{S}$.

(iii) at all points $\rho_{\phi} e^{i \phi} \in \mathcal{G}$
such that $\phi/\pi$ is rational in an open dense
subset of $(0,2)$, $\partial \mathcal{G}$
is non-smooth,

(iv) there exists a nonempty subset
of transcendental numbers $L_{tr}$, of 
Hausdorff dimension zero, such that
$\phi \in (0, \pi)$ and
$\phi \not\in \mathcal{K} \cup~ \pi \qb~ \cup~ \pi L_{tr}$ implies that the boundary curve $\partial \mathcal{G}$
has a tangent at $\rho_{\phi} e^{i \phi}$ (smooth point).
\end{theorem}

\begin{proof}
\cite{solomyak}, $\S$ 3 and $\S$ 4.
\end{proof}

\begin{definition}
\label{solomyakfractaldefinition}
The set $\mathcal{G} \cup \mathbb{T}
\cup \partial \mathcal{G}_{S}$
is called Solomyak's fractal.
\end{definition}

Solomyak's fractal
contains the set $\overline{W}$,
where $W$ 
consists of the zeroes $\lambda$, 
$|\lambda|< 1$, of the polynomials
$1 + \sum_{j=1}^{q} a_j z^j$ having all coefficients
$a_j$ zeroes and ones, studied by
Odlyzko and Poonen \cite{odlyzkopoonen}.

Let $1<\beta<(1+\sqrt{5})/2$ be
a reciprocal algebraic integer and
apply the Carlson-Polya dichotomy
to $f_{\beta}(z)$ :

- if 
$\beta$ is a Parry number,
then $f_{\beta}(z)$ has a 
finite number of zeroes by 
Theorem \ref{carlsonpolya},
$\mathcal{G}$
contains all the Galois-conjugates 
and the inverses of the beta-conjugates
(if any) of $\beta$ of modulus $< 1$,
and the unit circle $|z|=1$
is not a natural boundary of
$f_{\beta}(z)$.
Note that the Galois conjugates of
$\beta$ are the Galois conjugates of
$1/\beta$, but the $\beta$-conjugates
of $\beta$ could be roots of
non-reciprocal factors.

- if $\beta$ is not a
Parry number the zeroes of
$f_{\beta}(z)$ in 
the subfractal
$\mathcal{G}$
is more difficult to describe. 
In this case, the unit circle
$|z|=1$ is the natural boundary
of $f_{\beta}(z)$.
The study of the zeroes of 
power series
which lie 
very close to natural boundaries, 
having moderate lacunarity, 
is a difficult problem. This problem is more
difficult than for sparse
power series having
Hadamard lacunarity for instance
(Fuchs \cite{fuchs}, Levinson \cite{levinson} Chap. VI, Robinson
\cite{robinson}).
Let us view the problem from the side of
dynamical zeta functions since
$\zeta_{\beta}(z) = -1/f_{\beta}(z)$.
It is classical to study the analytic
behaviour of the dynamical zeta 
function on its natural 
disk of convergence
which is centered at 0 and of 
radius of convergence
$\exp(-\mathcal{H})$, 
where $\mathcal{H}$ is the 
topological entropy of the dynamical system
\cite{parrypollicott}.
In the case of the 
$\beta$-shift
the 
topological entropy
is $\mathcal{H} = \lo \beta$
(Proposition 5.1 in \cite{parrypollicott}).
Therefore
the important subregion of  
$\mathcal{G} \cup \mathbb{T}
\cup \partial \mathcal{G}_{S}$, in 
the open unit disk,
to be investigated for the existence, 
the number (eventually infinite) and 
the geometry
of zeroes of $f_{\beta}(z)$
is the annular region
$\{z \mid \exp(-\mathcal{H}) = \beta^{-1} 
< |z| < 1\}$, 
in particular when $\beta$ tends to $1^+$.
This problem is the general problem of the
extension of the meromorphy 
of $\zeta_{\beta}(z)$.
Then the main objective is
the extended research of zeroes 
of $f_{\beta}(z)$ in
$\{\beta^{-1} < |z| < 1 \}$, which
is mostly
concerned with
(i) the meromorphic extension
of the dynamical zeta function of a dynamical system
outside the disk of convergence whose radius is
$\exp(-\mathcal{H})$ with $\mathcal{H}$
the topological entropy, the pressure, etc, 
of the dynamical system,
and their poles in this annular region
(Haydn \cite{haydn},
Hilgert and Rilke \cite{hilgertrilke},
Parry and Pollicott \cite{parrypollicott},
Pollicott \cite{pollicott}, Ruelle \cite{ruelle7}),
(ii) the structure theorems of orthogonal 
decomposition of the transfer operators
(eventually the Perron-Frobenius operators), 
with the geometry of their 
isolated eigenvalues  (e.g.
Theorem 1 in Baladi and Keller \cite{baladikeller}).

In the sequel 
(in Section $\S$ \ref{S5.3} and 
Section $\S$ \ref{S6.1}) 
we will not solve
the problem of the exact
determination of the zeroes 
of $f_{\beta}(z)$ in
$\{\beta^{-1} < |z| < 1 \}$.
We will overcome this difficulty 
for the problem of Lehmer.
Instead, we will use
approximate values
of the zeroes, by
using \`a la Poincar\'e 
divergent series.
It will be sufficient
to observe the separation 
of the collection of zeroes
into two categories: - those which are
very close to the unit circle, called
{\it nonlenticular zeroes}, 
- those
which lie off the unit circle, 
called {\it lenticular zeroes}.
The lenticular zeroes,
spreading inside the cusp region of
$\mathcal{G} \cup \mathbb{T}
\cup \partial \mathcal{G}_{S}$, 
stemming from $\beta^{-1}$, 
in the neighbourhood 
of $z=1$ towards $e^{\pm i \frac{\pi}{3}}$,
are identified as Galois conjugates
of $\beta^{-1}$ in Section $\S$\ref{S5.4}.
Lenticuli of zeroes
are exemplified in
\cite{dutykhvergergaugry}.
Then we will obtain the
asymptotic expansion
of the minorant ${\rm M}_{r}(\beta)$
of the Mahler measure 
${\rm M}(\beta)$,
for $\beta > 1$ being a reciprocal 
algebraic integer,
from the lenticular roots
of $f_{\beta}(z)$.

\subsection{Carlson-Polya dichotomy of reciprocal algebraic integers $\beta > 1$ close to one}
\label{S3.3}

The set $\pb$ of Perron numbers is dense
in $(1,+\infty)$. It
contains the subset $\pb_P$ of
Parry numbers by a result of Lind 
\cite{lind} (Blanchard \cite{blanchard},
Boyle \cite{boyle}, 
Denker, Grillenberger and Sigmund 
\cite{denkergrillenbergersigmund}, 
Frougny in 
\cite{lothaire} chap.7). 
The set $\pb \setminus \pb_P$
is not empty (by
Akiyama's Theorem \ref{akiyamathm}
recalled below, also as a
consequence of Fekete-Szeg\H{o}'s Theorem 
\cite{feketeszego})
;
it would contain 
all Salem numbers of large degrees,
by Thurston \cite{thurston2} p. 11. 
Parry (\cite{parry}, Theorem 5) proved that
the subcollection 
of simple Parry numbers is dense
in $[1, \infty)$.
Simple Parry numbers $\beta$, 
as nonreciprocal algebraic integers,
satisfy the minoration
${\rm M}(\beta) \geq \Theta$.
In the opposite direction 
a Conjecture of K. Schmidt \cite{schmidt}
asserts that Salem numbers are all Parry 
numbers. 
For Salem numbers $\beta$
of degree $\geq 6$,
Boyd 
\cite{boyd15} established
a simple probabilistic model, based on 
the frequencies of digits
occurring in the R\'enyi $\beta$-expansions of unity,
to conjecture that, more realistically,
Salem numbers are dispatched into 
the two sets of Parry numbers and nonParry numbers,
each of them with densities $> 0$. 
This model, coherent with Thurston's one 
(\cite{thurston2}, p. 11),
is in contradiction with the conjecture of 
K. Schmidt.
This dichotomy of Salem numbers
was verified
by Hichri \cite{hichri} 
\cite{hichri2} \cite{hichri3} 
for Salem numbers of 
degree 8.
The coding of small neighbourhoods
of Salem numbers
by Stieltjes
continued fractions has been 
investigated
in \cite{guichardvergergaugry},
in view of characterizing
this dichotomy locally.
Salem numbers of degree 4
are Parry numbers \cite{boyd13}.
Boyd's model covers the set of
Salem numbers smaller than 
Lehmer's number, if any.
Few examples
of nonParry algebraic numbers $> 1$ exist;
Solomyak (\cite{solomyak} p. 483)
gives $\frac{1}{2}(1 + \sqrt{13})$.

\begin{theorem}[Akiyama]
\label{akiyamathm}
The dominant root $\gamma_n > 1$
of 
$-1 -z +z^n$, for $n \geq 2$, 
is a Perron number which is a Parry number if and only
if $n=2, 3$. If $n=2, 3$, 
$\gamma_2 = \theta_{2}^{-1}$ and
$\gamma_3 = \theta_{5}^{-1} = \Theta$
are Pisot numbers which are simple Parry numbers. 
\end{theorem}

\begin{proof}
Theorem 1.1 and Lemma 2.2
in \cite{akiyama} 
using Lagrange inversion formula.
%Let us recall the dynamics of the 
%Perron numbers 
%$\theta_{n}^{-1}$ for $n \geq 2$: 
%...
\end{proof}

This dichotomy, due to the 
R\'enyi-Parry dynamics, would separate the
set of real reciprocal algebraic 
integers $ \beta > 1$
into two disjoint nonempty subsets.
Since it corresponds exactly to the
dichotomy of Parry Upper functions
$f_{\beta}(z)$, we speak of 
{\it Carlson-Polya dichotomy of
real reciprocal algebraic 
integers $ \beta > 1$}.
The small Salem numbers found 
by Lehmer in \cite{lehmer},
reported in the Survey 
\cite{vergergaugrySurvey},
either given 
by their minimal polynomial or 
equivalently by their $\beta$-expansion,
are Parry numbers. 
The two smallest ones Lehmer \cite{lehmer}
has found:

\vspace{0.2cm}
\noindent
\begin{tabular}{c|c|c|l}
\label{smallSalemexpansions}
$\deg(\beta)$ & 
$\beta =
{\rm M}(\beta)$
&
minimal pol. of $\beta$
& 
$d_{\beta}(1)$\\
\hline

8 & 
$1.2806\ldots$
&
$X^8 - X^5 - X^4 -X^3 +1$
&
$0 . 1 (0^5 1 0^5 1 0^{7})^{\omega}$\\
10 & 
$1.17628\ldots$
&
$X^{10} + X^9 - X^7 - X^6 - X^5 $
&
$0 . 1 (0^{10} 1 0^{18} 1 0^{12} 1 0^{18} 1 0^{12})^{\omega}$\\
 & {\tiny ``Lehmer's number"}&
 $- X^4 -X^3 + X + 1$
 & \\

\end{tabular}

\vspace{0.3cm}

\noindent
of respective dynamical degrees $\dyg(\beta)$
7 and 12,
with Parry polynomials of respective 
degrees 20 and 75, given 
by \eqref{parrypolynomeperiodicity},  are
such that their respective
Parry Upper functions take the form given
in Proposition \ref{fbetainfinie}, namely
$f_{\beta}(z) =$
\begin{equation}
\label{salempetitN2}
-\frac{z^{20} -z^{19} -z^{13} -z^{7} -z +1}{1 - z^{19}} = -1 + z +z^{7}+ z^{13} + \ldots = G_{7}(z) +
z^{13}+\ldots ,
\end{equation}
resp. 
$$f_{\beta}(z) =
-\frac{z^{75} -z^{74} -z^{63} -z^{44} -z^{31} -z^{12} -z +1}{1 - z^{74}}$$
\begin{equation}
\label{salempetitN1Lehmernumber}
= -1 + z +z^{12} + z^{31}+ \ldots 
= G_{12}(z) + z^{31}+\ldots .
\end{equation}

The relations between the digits $(t_i)$
in the R\'enyi $\beta$-expansion of 
unity of an algebraic integer $\beta > 1$
and the coefficient vector
of its minimal polynomial are 
still obscure in general, except 
in a few cases: e.g.
for Salem numbers of degree 4 and 6
(Boyd \cite{boyd11} \cite{boyd12}
\cite{boyd14}),
for Salem numbers of degree 8 
(Hichri \cite{hichri} \cite{hichri2}
\cite{hichri3}), 
for Pisot numbers
(Boyd \cite{boyd15}, Frougny and Solomyak
\cite{frougnysolomyak},
Bassino \cite{bassino} in the cubic case,
Hare \cite{hare} \cite{haretweedle},
Panju \cite{panju} for regular Pisot numbers).
A more abstract ergodic
viewpoint is developped
in Schmidt \cite{schmidt2}, with 
potential applications
to limit Mahler measures. 

\vspace{0.3cm}

By \eqref{fbetazetabeta},
the topological properties
of the set $\{T_{\beta}^{n}(1)\}$
control the Carlson-Polya dichotomy
of reciprocal algebraic integers 
$> 1$.
On this basis Blanchard \cite{blanchard} 
proposed a classification
of real numbers $\beta > 1$ into five classes;
Verger-Gaugry in \cite{vergergaugry} refined it 
in terms of asymptotic gappiness in the direction of
more enlighting the algebraicity of $\beta$: 

\noindent
\begin{tabular}{ll}
Class C1: & $d_{\beta}(1)$ is finite,\\

Class C2: & $d_{\beta}(1)$ is ultimately periodic but not finite,\\
Class C3: & $d_{\beta}(1)$ contains bounded strings of zeroes, but is not ultimately\\ 
& periodic (0 is not an accumulation point of $\{T_{\beta}^{n}(1)\})$,\\
Class C4: & $\{T_{\beta}^{n}(1)\}$ is not dense in $[0, 1]$, but
admits 0 as an accumulation point,\\
Class C5: & $\{T_{\beta}^{n}(1)\}$ is dense in $[0, 1]$.

\end{tabular}

Apart from C1, resp. C2, 
which is exactly 
the set of simple, resp. nonsimple, Parry numbers,
how the remaining algebraic numbers $> 1$
are dispatched in 
the classes
C3, C4 and C5 is obscure.
The specification property, meaning
that 0 is not an accumulation point for $\{T_{\beta}^{n}(1)\}$, 
was weakened
by Pfister and Sullivan \cite{pfistersullivan}
and Thompson \cite{thompson}.
For unique $q$-expansions the specification and
synchronization properties were studied
by Alcaraz Barrera 
\cite{alcarazbarrera}
\cite{alcarazbarrera2}
\cite{alcarazbarrerabakerkong}.
Schmeling \cite{schmeling} proved that the class C3 has full Hausdorff dimension
and that the class C5,
probably  mostly occupied 
by transcendental numbers, is of full 
Lebesgue measure 1.
Lacunarity and Diophantine approximation
were investigated by 
Bugeaud and Liao \cite{bugeaudliao},
Hu, Tong and Yu \cite{hutongyu}, 
Li, Persson, Wang and Wu \cite{liperssonwangwu}.
For any $x_0 \in [0,1]$ the asymptotic distance
$\liminf_{n \to \infty} |T_{\beta}^{n}(1) - x_0|$, 
for almost all
$\beta > 1$ (for the Lebesgue measure),
was studied by Persson and Schmeling
\cite{perssonschmeling} \cite{schmeling}, 
Ban and Li \cite{banli},
Cao \cite{cao}, Fang, Wu and Li \cite{fangwuli},
Li and Chen \cite{lichen},
L\"u and Wu \cite{luwu}, Tan and Wang \cite{tanwang}.
Kwon \cite{kwon} studies the 
subset of Parry numbers whose conjugates
lie close to the unit circle, using 
technics of combinatorics of words.
The seperation between algebraic numbers
and transcendental numbers 
was studied by Dubickas \cite{dubickas12}.
Bugeaud \cite{bugeaud}
investigates  Diophantine
approximation properties and
$\beta$-representations in algebraic
bases.
Adamczewski and Bugeaud 
(\cite{adamczewskibugeaud}, Theorem 4)
show that the class C4 contains self-lacunary numbers,
all transcendental, from Schmidt's
Subspace Theorem
and results of Corvaja and Zannier.

\subsection{Cyclotomic jumps in families of Parry Upper functions, right-continuity}
\label{S3.4}

Allowing the real 
base $\beta$ to vary continuously
in the neighbourhood $[1, \theta_{2}^{-1})$ of 1, 
except 1,
asks the question whether it has a sense to
consider the continuity of the bivariate
Parry Upper function
$(\beta, z) \to f_{\beta}(z)$, and, if it is the case, 
on which subsets, in $z$, of the complex plane.

Theorem \ref{convergencecompactsetsUNITDISK} 
and its Corollary
show that the open unit disk is a
domain where the continuity of the roots of
$f_{\beta}(z)$ in $|z| < 1$
holds though the
functions f$_{\beta}(z)$ are only right continuous
in $\beta$,
with infinitely many cyclotomic jumps,
while, in the complement $|z| \geq 1$,
either the Parry Upper functions are not defined,
or may exhibit drastic changes 
on the unit circle. In the present 
attack of the Conjecture of Lehmer
only the open unit disk is of interest.

\begin{lemma}
\label{continuity}
Let $1 < \beta < \theta_{2}^{-1}$
and
$0 < x < 1$. Then
(i) the bivariate $\beta$-transformation map 
$(\beta,x) \to
T_{\beta}(x) =\{\beta x\} = 
\beta x -\lfloor \beta x \rfloor$ 
is continuous, 
in $\beta$ and $x$, when 
$\beta x$ is not a positive integer.
If $\beta x$ is a positive integer,
$x = 1/\beta$ and
\begin{equation}
\label{continuitegauche}
\lim_{y \to \frac{1}{\beta}^{-}, \,\gamma \to \beta^{-}}
T_{\gamma}(y) = 1,
\qquad
\lim_{y \to \frac{1}{\beta}^{+}, \, \gamma \to \beta^{+}}
T_{\gamma}(y) = 
T_{\beta}(\frac{1}{\beta}) = 0;
\end{equation}
(ii) for any $(\beta, x)$,
there exists $\epsilon = \epsilon_{\beta, x}$
such that $T_{\gamma}(y)$ is increasing
both in $\gamma
\in [\beta, \beta + \epsilon)$ 
and in $y \in [x, x+\epsilon)$.
\end{lemma}

\begin{proof}
Lemma 3.1 in \cite{flattolagariaspoonen}.
(i) If $\beta x$ is an integer, this integer 
is 1 necessarily. The value $x=1/\beta$ is
a negative power of $\beta$. 
The R\'enyi $\beta$-expansion of $1/\beta$
is deduced from $d_{\beta}(1)$ by a shift,
given in \eqref{renyidef} and 
\eqref{renyidef_unsurbeta};
the sequence
$(T_{\beta}^{n}(\frac{1}{\beta}))_{n \geq 1}$
is directly obtained from
$(T_{\beta}^{n}(1))_{n \geq 1}$.
The fractional part
$\gamma \to \{\gamma\}= T_{\gamma}(1)$ 
is right continuous, 
hence the result; (ii) obvious.
\end{proof}

\begin{lemma}
\label{continuitysimpleParrynumber}
Let $\beta \in (1, \theta_{2}^{-1})$.
(i) If $\beta$ 
is a simple Parry number,
then, for all $n \geq 1$, the map
$\gamma \to T_{\gamma}^{n}(1)$
is right continuous
at $\beta$:
\begin{equation}
\label{continuitedroite}
\lim_{\gamma \to \beta^{+}} T_{\gamma}^{n}(1)
= T_{\beta}^{n}(1),
\end{equation}

(ii) if $\beta$ is a simple Parry number,
such that $T_{\beta}^{N}(1) = 0$
with $T_{\beta}^{k}(1) \neq 0$, $1 \leq k < N$, 
then, for all $n \geq 1$,
\begin{equation}
\label{discontinuitegauche}
\lim_{\gamma \to \beta^{-}} T_{\gamma}^{n}(1)
= 
\left\{
\begin{array}{lc}
T_{\beta}^{n}(1),& \qquad n < N \quad(\mbox{left continuity})\\
T_{\beta}^{n_N}(1),& \qquad
n \geq N,
\end{array}
\right.
\end{equation}
where $n_N \in \{0, 1, \ldots N-1\}$
is the residue of $n$ modulo $N$,

(iii) if $\beta$ 
is a nonsimple Parry number,
then
$\gamma \to T_{\gamma}^{n}(1)$
is continuous 
at $\beta$:
\begin{equation}
\label{continuitegauchedroite}
\lim_{\gamma \to \beta^{-}} T_{\gamma}^{n}(1)
= T_{\beta}^{n}(1) = 
\lim_{\gamma \to \beta^{+}} T_{\gamma}^{n}(1),
\qquad \mbox{for all }~ n \geq 1.
\end{equation}
\end{lemma}

\begin{proof} Lemma 3.2 in 
\cite{flattolagariaspoonen}.
\end{proof}

Denote by
$$\mathcal{F}:=
\{{f_{\beta}}_{|_{|z| < 1}}(z) \mid
1 < \beta < \theta_{2}^{-1}
\}$$
the set of the restrictions of the 
Parry Upper functions $f_{\beta}(z)$, 
$1 < \beta < \theta_{2}^{-1}$,
to the open unit disk. The set $\mathcal{F}$ is 
equipped with
the topology of the uniform convergence on 
compact subsets of $|z| < 1$.

\begin{theorem}
\label{convergencecompactsetsUNITDISK}
In $\mathcal{F}$
the following right and left
limits hold:
(i) if  $\beta$ be a nonsimple Parry number,
then continuity occurs as:
\begin{equation}
\label{ParryUpperfunctioncontinuity}
\lim_{\gamma \to \beta^-}
f_{\gamma}(z) = 
f_{\beta}(z)
=
\lim_{\gamma \to \beta^+}
f_{\gamma}(z),
\end{equation}
(ii) if $\beta$ is a simple Parry number,
and $N$ the minimal value for which
$T_{\beta}^{N}(1) = 0$, then
\begin{equation}
\label{ParryUpperfunctionrightcontinuity}
\lim_{\gamma \to \beta^+}
f_{\gamma}(z) = 
f_{\beta}(z) ,
\end{equation}
\begin{equation}
\label{ParryUpperfunctionleftCycloJumpcontinuity}
\lim_{\gamma \to \beta^-}
f_{\gamma}(z) = 
\frac{f_{\beta}(z)}{(1 - z^N)}.
\end{equation}
\end{theorem}

\begin{proof}
Let $\gamma, \beta \in (1, \theta_{2}^{-1})$
with $|\gamma - \beta| \leq \epsilon$,
$\epsilon > 0$,
$d_{\gamma}(1) = 0 . t'_1 t'_2 \ldots$
and $d_{\beta}(1) = 0 . t_1 t_2 \ldots$. 
Any compact subset
of $|z| < 1$ is included
in a closed disk centered at 0 of radius $r$
for some $0 < r < 1$.
Assume $|z| \leq r$. 
(i) Assume $\beta$ nonsimple. 
Since
$|T_{\gamma}^{m}(1) -
T_{\beta}^{m}(1)| \leq 2$ for $m \geq 1$,
then
$$|f_{\gamma}(z) - f_{\beta}(z)|
=\!
\bigl| \sum_{n \geq 1} 
(t'_n -t_n) z^n
\bigr|
\!=\!
\bigl|\sum_{n \geq 1} 
[(\gamma T_{\gamma}^{n-1}(1)
- \beta T_{\beta}^{n-1}(1))
-
(T_{\gamma}^{n}(1) -
T_{\beta}^{n}(1))]
 z^n \Bigr|
$$
\begin{equation}
\label{lebesgueconvergencedominated}
\leq
\sum_{n \geq 1} 
\Bigl|(\gamma T_{\gamma}^{n-1}(1)
- \beta T_{\beta}^{n-1}(1))
-
(T_{\gamma}^{n}(1) -
T_{\beta}^{n}(1))\Bigr|
 r^n
\leq 2 (\epsilon + \beta + 1) \sum_{n \geq 1} r^n ,
\end{equation}
which is convergent.
By \eqref{continuitegauchedroite} and 
the Lebesgue dominated convergence theorem,
taking the limit termwise in the summation,
$$\lim_{\gamma \to \beta}
|f_{\gamma}(z) - f_{\beta}(z)| = 0,
\qquad
\mbox{uniformly for}~ |z| \leq r.$$
(ii) 
By \eqref{continuitedroite}
and
\eqref{discontinuitegauche}, the iterates
of $1$ under the $\gamma$-transformation
$T_{\gamma}^n(1)$ behave differently at $\beta$
if $\gamma < \beta$ or resp. 
$\gamma > \beta$ when $\gamma$ tends to $\beta$:
if $\gamma > \beta$, we apply
the Lebesgue dominated convergence theorem
in \eqref{lebesgueconvergencedominated}
to obtain the
right continuity at $\beta$, i.e.
\eqref{ParryUpperfunctionrightcontinuity}; 
if
$\gamma \to \beta^-$,
\eqref{ParryUpperfunctionleftCycloJumpcontinuity}
comes from the dominated convergence theorem
applied to
$$
f_{\gamma}(z) - \frac{1}{1 - z^N}f_{\beta}(z)
= 
(\beta z - 1)
\left[
\Bigl( \sum_{n =0}^{\infty} 
T_{\gamma}^{n}(1) 
\frac{\gamma z - 1}{\beta z - 1} 
z^n
\Bigr) 
- 
\Bigl( \sum_{q=0}^{\infty}
\sum_{m = 0}^{N-1} T_{\beta}^{m}(1) z^{m + qN}
\Bigr)
\right].
$$
\end{proof}

Theorem B in Mori \cite{mori}, on 
the continuity properties of 
spectra of Fredholm matrices,
admits the following counterpart
in terms of the Parry Upper functions:

\begin{corollary}
\label{zeroesParryUpperfunctionContinuity}
The root functions of $f_{\beta}(z)$ valued in
$|z| < 1$ are all continuous, 
as functions of $\beta \in (1, \theta_{2}^{-1})
\setminus \bigcup_{n \geq 3} \{\theta_{n}^{-1}\}$.
\end{corollary}

\begin{proof}
Let $(\gamma_{i})_{i \geq 1}$ be a sequence
of real numbers tending to $\beta$.
The (restrictions, to the open unit disk, of
the) functions $f_{\gamma_{i}}(z)$ 
constitute a convergent sequence in
$\mathcal{F}$,
tending
either to $f_{\beta}(z)$ or
$f_{\beta}(z)/(1 - z^N)$ for some integer $N \geq 1$.
By Hurwitz's Theorem (\cite{sakszygmund} (11.1))
any disk in $|z| < 1$, whose closure does not intersect
the unit circle,
which contains a zero $w(\beta)$ of 
$f_{\beta}(z)$ also contains a zero of 
$f_{\gamma_{i}}(z)$ for all $i \geq i_0$,
for some $i_0$.
The multiplicity of $w(\beta)$ is equal 
to the number of zeroes $w(\gamma_{i})$, 
counted with multiplicities, in this disk.
\end{proof}

Another consequence, in $\mathcal{F}$,
is the disappearance of the cyclotomic jumps
of the left-discontinuities at 
the reciprocal algebraic integers
$\beta > 1$ close to $1^+$.

\begin{corollary}
\label{disappearanceJumps}
If $\beta \in (1, \theta_{2}^{-1})$ 
is a reciprocal algebraic integer,
then the
left and right continuity of the Parry Upper function
occurs in $\mathcal{F}$ as:
\begin{equation}
\label{ParryUpperfunctioncontinuitysmallMahler}
\lim_{\gamma \to \beta^-}
f_{\gamma}(z) = 
f_{\beta}(z)
=
\lim_{\gamma \to \beta^+}
f_{\gamma}(z).
\end{equation}
\end{corollary}

\begin{proof}
From Theorem \ref{convergencecompactsetsUNITDISK} 
the case
\eqref{ParryUpperfunctionleftCycloJumpcontinuity}
cannot occur since a reciprocal
algebraic integer $> 1$ cannot be a 
simple Parry number.
\end{proof}

\section{Asymptotic expansions of the Mahler measures ~${\rm M}(-1 + X + X^n)$, a Dobrowolski type minoration}
\label{S4}

\subsection{Factorization of the trinomials $-1+X+X^n$, lenticuli of roots}
\label{S4.1}

The notations used throughout this note
come from
the factorization of $G_{n}(X) := -1 +X +X^n$ 
(Selmer \cite{selmer}, Verger-Gaugry 
\cite{vergergaugry6} Section 2). 
Summing in pairs over complex conjugated imaginary roots, the
indexation of the roots and
the factorization of $G_{n}(X)$ 
are taken as follows:
\begin{equation}
\label{factoGG}
G_{n}(X) = (X - \theta_n) \, 
\left(
\prod_{j=1}^{\lfloor \frac{n}{6} \rfloor} (X - z_{j,n}) ( X - \overline{z_{j,n}}) 
\right)
\times q_{n}(X),
\end{equation}
where
$\theta_n$ is the only (real) root 
of $G_{n}(X)$ in the interval $(0, 1)$, where
$$q_{n}(X) ~=~
\left\{
\begin{array}{ll}
\displaystyle
\left(
\prod_{j=1+\lfloor \frac{n}{6} \rfloor}^{\frac{n-2}{2}} (X- z_{j,n}) (X - \overline{z_{j,n}})
\right) \times (X- z_{\frac{n}{2},n})
&
\mbox{if $n$ is even, with}\\
 & \mbox{$z_{\frac{n}{2},n}$ real $< -1$},\\
\displaystyle
\prod_{j=1+\lfloor \frac{n}{6} \rfloor}^{\frac{n-1}{2}} (X- z_{j,n}) (X - \overline{z_{j,n}})
&
\mbox{if $n$ is odd,}
\end{array}
\right.
$$
where the index $j = 1, 2, \ldots$ is such that 
$z_{j,n}$ is a (nonreal) complex zero of $G_{n}(X)$, except if 
$n$ is even and $j=n/2$, such that
the argument $\arg (z_{j,n})$ of $z_{j,n}$
is roughly equal to $2 \pi j/n$
(Proposition \ref{zedeJIargumentsORDRE1})
and that the family of arguments 
$(\arg(z_{j,n}))_{1 \leq j < \lfloor n/2 \rfloor}$ 
forms a strictly increasing sequence with $j$:
$$0 < \arg(z_{1,n}) < \arg(z_{2,n}) < \ldots 
< \arg(z_{\lfloor \frac{n}{2} \rfloor,n}) \leq \pi.$$
For $n \geq 2$ all the roots of $G_{n}(X)$ 
are simple, and
the roots of $G_{n}^{*}(X) = 1 + X^{n-1} - X^n$, 
as inverses of the roots
of $G_{n}(X)$, are classified in the 
reversed order (Figure \ref{perronselmerfig}).

\begin{proposition}
\label{irredGn}
Let $n \geq 2$. 
If $n \not\equiv 5 ~({\rm mod}~ 6)$, 
then $G_{n}(X)$ is irreducible over $\qb$. 
If $n \equiv 5 ~({\rm mod}~ 6)$, then 
the polynomial $G_{n}(X)$ admits 
$X^2 - X +1$ as irreducible factor
in its factorization and $G_{n}(X)/(X^2 - X +1)$ is irreducible.
\end{proposition}

\begin{proof}
Selmer \cite{selmer}.
\end{proof}

\begin{figure}
\begin{center}
\includegraphics[width=8cm]{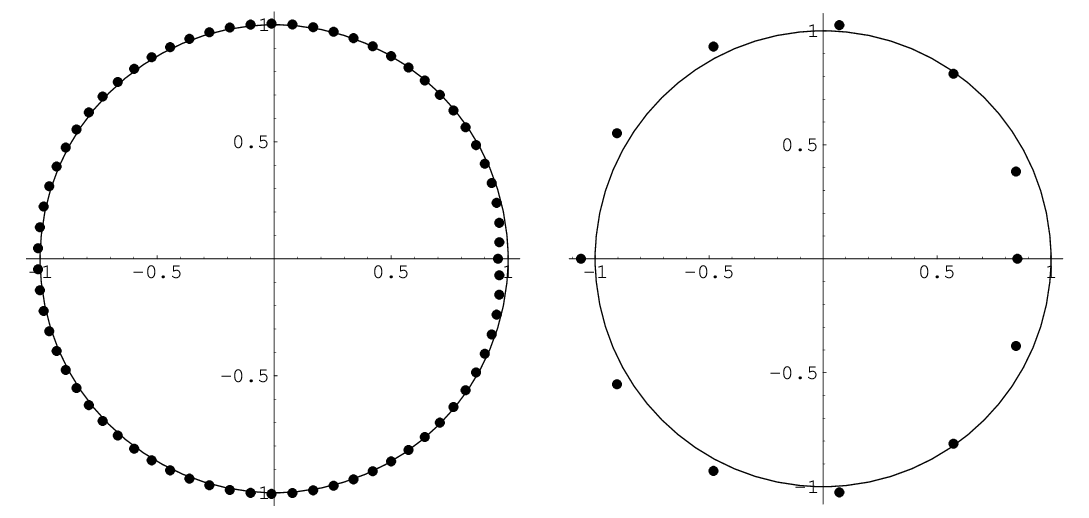}
\end{center}
\caption{
The roots (black bullets) of $G_{n}(z)$
(represented here with $n=71$ and $n=12$)
are uniformly distributed near $|z|=1$
according to the theory of
Erd\H{o}s-Tur\'an-Amoroso-Mignotte.
A slight bump appears in the half-plane
$\Re(z) > 1/2$ in the neighbourhood of $1$, at 
the origin of the different regimes of asymptotic expansions.
The dominant root of
$G_{n}^{*}(z)$ is the Perron number $\theta_{n}^{-1} > 1$, with
$\theta_n$ the unique root of $G_n$
in the interval $(0,1)$.}
\label{perronselmerfig}
\end{figure}

\begin{proposition}
\label{closetoouane}
For all $n \geq 2$, all zeros $z_{j,n}$ and
$\theta_n$
of the polynomials
$G_{n}(X)$ have a modulus in the interval
\begin{equation}
\label{bounboun}
\Bigl[ ~1- \frac{2 \,\lo n}{ n}, ~1 + \frac{2 \,\lo 2}{n} ~\Bigr] ,
\end{equation}

(ii)~ the trinomial $G_{n}(X)$ admits a unique real root $\theta_n$ in
the interval $(0,1)$.
The sequence $(\theta_n)_{n \geq 2}$ is strictly increasing, $\lim_{n \to +\infty} \theta_n = 1$, with
$\theta_2 = \frac{2}{1+\sqrt{5}} = 0.618\ldots$,

(iii)~ 
the root $\theta_n$ 
is the unique root of smallest modulus among all the roots
of $G_{n}(X)$; 
if $n \geq 6$, 
the roots of modulus $< 1$ of $G_{n}(z)$
in the closed upper half-plane have the following properties:

(iii-1)~~ $\theta_n ~<~ |z_{1,n}|$,

(iii-2) ~ for any pair of successive indices
$j, j+1$ in $\{1, 2, \ldots, \lfloor n/6 \rfloor\}$,
$$| z_{j,n} | < |z_{j+1,n} | .$$
\end{proposition}

\begin{proof}
(i)(ii) Selmer \cite{selmer}, pp 291--292; (iii-1) Flatto, Lagarias and Poonen
\cite{flattolagariaspoonen},
(iii-2) Verger-Gaugry \cite{vergergaugry6}.
\end{proof}

The Pisot number (golden mean)
$\theta_{2}^{-1} = \frac{1+\sqrt{5}}{2} = 1.618\ldots$
is the largest
Perron number in the family $(\theta_{n}^{-1})_{n \geq 2}$. 
The interval $( 1, \frac{1+\sqrt{5}}{2}\, ]$ is
partitioned by
the strictly decreasing sequence
of Perron numbers
$(\theta_{n}^{-1})$ as
\begin{equation}
\label{decoupage}
( 1, \frac{1+\sqrt{5}}{2}\, ] ~=~
\left(
\bigcup_{n=2}^{\infty}
\left[ \, \theta_{n+1}^{-1} , \theta_{n}^{-1}
\, \right)
\right)
~ \bigcup ~\left\{
\theta_{2}^{-1}
\right\}.
\end{equation}

By the direct method of asymptotic 
expansions of the roots, as in 
\cite{vergergaugry6}, or by Smyth's Theorem 
\cite{smyth} 
(Dubickas \cite{dubickas}),
since the trinomials $G_{n}(X)$ are not 
reciprocal, 
the Mahler measure of $G_n$
satisfies 
\begin{equation}
\label{smyththeorem}
{\rm M}(\theta_{n}) ~=~ {\rm M}(G_n) ~\geq~ \Theta = 1.3247\ldots, \qquad n \geq 2,
\end{equation}
where $\Theta = \theta_{5}^{-1}$ is the smallest Pisot number, dominant root of
the Pisot polynomial 
$X^3 - X - 1 = -G_{5}^{*}(X) / (X^2 - X + 1)$.

\begin{proposition}
\label{rootsdistrib}
Let $n \geq 2$. 
Then (i) the number $p_n$ of roots of $G_{n}(X)$
which lie inside the open sector $\mathcal{S} =
\{ z \mid |\arg (z)| < \pi/3 \}$ is equal to
\begin{equation}
\label{pennn}
1 + 2 \lfloor \frac{n}{6} \rfloor,
\end{equation} 

(ii) the correlation between the
geometry of the roots of $G_{n}(X)$ 
which lie inside the unit disk
and the upper half-plane and their indexation is given by:
\begin{equation}
\label{rootsinside}
j \in \{1, 2, \ldots, \lfloor \frac{n}{6} \rfloor \}
~\Longleftrightarrow~ \Re(z_{j,n}) > \frac{1}{2} ~\Longleftrightarrow~ |z_{j,n}| < 1,
\end{equation} 
and
the Mahler measure
M$(G_n)$ of the trinomial $G_{n}(X)$ is
\begin{equation}
\label{mahlerGG}
{\rm M}(G_{n}) ~=~ {\rm M}(G_{n}^{*}) ~=~ \theta_{n}^{-1} \, 
\prod_{j=1}^{\lfloor n/6 \rfloor} |z_{j,n}|^{-2}.
\end{equation}
\end{proposition}
\begin{proof}
Verger-Gaugry \cite{vergergaugry6}, Proposition 3.7.
\end{proof}

\subsection{Asymptotic expansions: roots of $G_n$ and relations}
\label{S4.2}

The (Poincar\'e) asymptotic expansions of the roots of $G_n$
(and $G_{n}^{*}$)
are generically written:
${\rm Re}(z_{j,n}) = 
{\rm D}({\rm Re}(z_{j,n})) + 
{\rm tl}({\rm Re}(z_{j,n}))$,
${\rm Im}(z_{j,n}) = {\rm D}({\rm Im}(z_{j,n})) 
+ {\rm tl}({\rm Im}(z_{j,n}))$,
$\theta_n = {\rm D}(\theta_n) + 
{\rm tl}(\theta_n)$,
where "D" and "tl" stands for
 {\it ``development"}
(or {\it ``limited expansion"}, 
or {\it ``lowest order terms"})
and
"tl" for {\it ``tail"} 
(or ``remainder", or {\it ``terminant"} in \cite{dingle}).
They are given at a sufficiently high order
allowing to deduce the asymptotic expansions 
of the Mahler measures ${\rm M}(G_n)$.
The terminology {\it order} 
comes from the general theory
(Borel \cite{borel}, Copson \cite{copson}, 
Dingle \cite{dingle}, 
Erd\'elyi \cite{erdelyi}); the 
approximant solutions of a polynomial
equation say $G(z) = 0$ which arise naturally correspond to
{\it order} $1$. The solutions corresponding to {\it order} $2$
are obtained by inserting the {\it order $1$ approximant solutions}
into the equation $G(z) = 0$, 
for getting {\it order $2$ approximant solutions}. 
And so on, as a function of $\deg{G}$. 
The order is
the number of steps in this iterative process.
Poincar\'e \cite{poincare}
introduced
this method of divergent series
for the $N$ - body problem in celestial mechanics;
this method does not appear in number theory
in the book ``Divergent series" of Hardy.
The equivalent of the variable time $t$ (in celestial 
mechanics) will be the dynamical degree
$\dyg(\house{\alpha})$
of the house of the algebraic integer
$\alpha$ in number theory 
(with $|\alpha| > 1$), 
a new ``variable concept" introduced in the present study; 
for the trinomials
$G_n$ it will be $n$.

The asymptotic expansions of 
$\theta_n$ and those roots $z_{j,n}$
of $G_{n}(z)$ which lie 
in the first quadrant
are (divergent) 
sums of functions of only {\em one} 
variable, which is $n$, while
those of the other roots $z_{j,n}$
are functions of a couple of 
{\em two} variables which is: \begin{itemize}
\item[$\bullet$] $(n, j/n)$
in the angular sector
$\pi / 4 > \arg z  > 2 \pi \, \lo n / n$, 
and 
\item[$\bullet$] $(n, j/\lo n)$
in the angular sector
$2 \pi \, \lo n / n > \arg z > 0$.
\end{itemize}
The first sector is the main angular sector.
The second sector if the bump angular sector.
A unique regime of asymptotic expansion
exists in the main angular sector, whereas
two regimes of asymptotic expansions
do exist in the bump sector
(Appendix, and \cite{vergergaugry6}).
These regimes 
are separated by two sequences
$(u_n)$ and $(v_n)$, to which 
the second variable
$j/n$, resp. $j/\lo n$, is compared.
Details can be found in the Appendix.

%--------------------------------------------------------

\begin{proposition}
\label{thetanExpression}
Let $n \geq 2$.
The root $\theta_n$ can be expressed as: 
$\theta_n = {\rm D}(\theta_n) + {\rm tl}(\theta_n)$ with
${\rm D}(\theta_n) = 1 -$
\begin{equation}
\label{DthetanExpression}
\frac{\lo n}{n}
\left(
1 - \bigl(
\frac{n - \lo n}{n \, \lo n + n - \lo n}
\bigr)
\Bigl(
\lo \lo n - n 
\lo \Bigl(1 - \frac{\lo n}{n}\Bigr)
- {\rm Log} n
\Bigr)
\right)
\end{equation}
and
\begin{equation}
\label{tailthetanExpression}
{\rm tl}(\theta_n) ~=~ \frac{1}{n} \, O \left( \left(\frac{\lo \lo n}{\lo n}\right)^2 \right),
\end{equation}
with the constant $1/2$ involved in $O \left(~\right)$.
\end{proposition}

\begin{proof}
\cite{vergergaugry6} Proposition 3.1.
\end{proof}

\begin{lemma}
\label{remarkthetan}
Given the limited expansion D$(\theta_n)$ of
$\theta_n$ as in \eqref{DthetanExpression}, denote
$$\lambda_n := 1 - (1 - {\rm D}(\theta_{n}))\frac{n}{\lo n}.$$ 
Then $\lambda_n = {\rm D}(\lambda_{n}) + {\rm tl}(\lambda_{n})$, with
\begin{equation}
\label{eqqww}
{\rm D}(\lambda_n) = \frac{\lo \lo n}{\lo n} \left(\frac{1}{1+\frac{1}{\lo n}}\right), \qquad
{\rm tl}(\lambda_n) = O\left( \frac{\lo \lo n}{n}  \right)
\end{equation}
with the constant 1 in the Big O.
\end{lemma}

\begin{proof}
\cite{vergergaugry6} Lemma 3.2.
\end{proof}

In the sequel, for short, we write $\lambda_n$ instead of
${\rm D}(\lambda_n)$.

\begin{proposition}
\label{zjjnnExpression}
Let $n \geq n_0 = 18$ and $1 \leq j \leq \lfloor \frac{n-1}{4} \rfloor$. 
The roots $z_{j,n}$ of $G_{n}(X)$ have the following asymptotic expansions:
$z_{j,n} = {\rm D}(z_{j,n}) + {\rm tl}(z_{j,n})$
in the following angular sectors:

\begin{itemize}
\item[(i)] \underline{
Sector $\frac{\pi}{2} > \arg z > 2 \pi \frac{\lo n}{n}$ (main sector):}
$${\rm D}(\Re(z_{j,n})) = \cos\bigl(2 \pi \frac{j}{n}\bigr) + 
\frac{\lo \bigl(2 \, \sin\bigl(\pi \frac{j}{n}\bigr)
\bigr)}{n},$$
$${\rm D}(\Im(z_{j,n})) ~=~ \sin\bigl(2 \pi \frac{j}{n}\bigr) 
+ 
\tan\bigl(\pi \frac{j}{n}\bigr)
\, 
\frac{\lo \bigl(2 \, \sin\bigl(\pi \frac{j}{n}\bigr)
\bigr)}{n},$$
with
$${\rm tl}(\Re(z_{j,n})) ~=~  {\rm tl}(\Im(z_{j,n})) ~=~
\frac{1}{n} \, O \left( \left(\frac{\lo \lo n}{\lo n}\right)^2 \right)$$
and the constant 1 in the Big $O$,
\vspace{0.1cm}

\item[(ii)] \underline{
``Bump" sector $2 \pi \frac{\lo n}{n} > \arg z > 0$ : }

\subitem \underline{$\bullet$ 
Subsector $2 \pi \frac{\sqrt{(\lo n) (\lo \lo n)}}{n} > \arg z > 0$:}
$${\rm D}(\Re(z_{j,n})) = \theta_n + \frac{2 \pi^2}{n} \left(\frac{j}{\lo n}\right)^2 \bigl( 1+2 \lambda_n \bigr), 
$$
$${\rm D}(\Im(z_{j,n})) = \frac{2 \pi \lo n}{n}  \left(\frac{j}{\lo n}\right)
\left[1 - \frac{1}{\lo n} (1 + \lambda_n)\right],$$

with
$${\rm tl}(\Re(z_{j,n})) = \frac{1}{n \lo n} \left(\frac{j}{\lo n}\right)^2 O\left(
\left(\frac{\lo \lo n}{\lo n}\right)^2\right),
$$
$$
{\rm tl}(\Im(z_{j,n})) = \frac{1}{n \lo n} \left(\frac{j}{\lo n}\right)
O\left(
\left(\frac{\lo \lo n}{\lo n}\right)^2\right),$$

\subitem \underline{$\bullet$
Subsector $2 \pi \frac{\lo n}{n} > \arg z >
2 \pi \frac{\sqrt{(\lo n) (\lo \lo n)}}{n}$:}
$${\rm D}(\Re(z_{j,n})) = \theta_n +
\frac{2 \pi^2}{n} \left(\frac{j}{\lo n}\right)^2
\left(
1 + \frac{2 \pi^2}{3} \left(\frac{j}{\lo n}\right)^2
\left(
1+\lambda_n
\right)
\right)
$$
$$  {\rm D}(\Im(z_{j,n})) ~=~ $$
$$\frac{2 \pi \lo n}{n} \left(\frac{j}{\lo n}\right)
\left[1 - \frac{1}{\lo n}
\left(
1 - \frac{4 \pi^2}{3} \left(\frac{j}{\lo n}\right)^2
\left( 1 - \frac{1}{\lo n}
(
1 - \lambda_n
) \right)
\right)\right],
$$
with
$${\rm tl}(\Re(z_{j,n})) = \frac{1}{n} O \left(
\left(\frac{j}{\lo n}
\right)^6
\right),
{\rm tl}(\Im(z_{j,n})) = \frac{1}{n} O\left(
\left(\frac{j}{\lo n}\right)^5
\right).$$
\end{itemize} 
\end{proposition}

\begin{proof}
\cite{vergergaugry6} Proposition 3.4.
\end{proof}

Outside the ``bump sector"
the moduli of the roots $z_{j,n}$
are readily obtained as (Proposition 3.5 in
\cite{vergergaugry6}):
\begin{equation}
\label{zjnModulusFirst}
|z_{j,n}| =
 1 + \frac{1}{n} \, \lo \bigl(2 \, \sin\bigl(\frac{\pi j}{n}\bigr)  \bigr) +
\frac{1}{n} O
\left( \frac{(\lo \lo n)^2}{(\lo n)^2}
\right),
\end{equation}
with the constant 1 in the Big $O$ (independent of $j$).
The following expansions of the $|z_{j,n}|$s 
at the order 3 will be needed in the method of Rouch\'e.

\begin{proposition}
\label{zedeJIargumentsORDRE1}
$$\arg(z_{j,n}) = 2 \pi (\frac{j}{n} 
+ {A}_{j,n})
\quad
\mbox{with}\quad 
A_{j,n}
=
-\frac{1}{2 \pi n}
\left[
\frac{1 - \cos(\frac{2 \pi j}{n})}
{\sin(\frac{2 \pi j}{n})}
\lo (2 \sin(\frac{\pi j}{n}))
\right]$$
$$\mbox{and}\qquad\qquad
{\rm tl}(\arg(z_{j,n}))=
\frac{1}{n} O\Bigl(\left(
\frac{\lo \lo n}{\lo n}
\right)^2 \Bigr).
$$
\end{proposition}

\begin{proof}
\S 6 in \cite{vergergaugry6}.
\end{proof}

\begin{proposition}
\label{zedeJImodulesORDRE3}
For all $j$ such that $\pi / 3 \geq \arg z_{j,n} > 
2 \pi \frac{\lceil v_n \rceil}{n}$,
the asymptotic expansions of the
moduli of the roots $z_{j,n}$ are
$$|z_{j,n}| = {\rm D}(|z_{j,n}|) +
{\rm tl}(|z_{j,n}|)$$ 
with
\begin{equation}
\label{absolzjjnn_mieux}
{\rm D}(|z_{j,n}|) = 1 + \frac{1}{n} \, \lo \bigl(2 \, \sin\bigl(\frac{\pi j}{n}\bigr)  \bigr) +
 \frac{1}{2 n}\left( \frac{\lo \lo n}{\lo n} \right)^2
\end{equation}
and
\begin{equation}
\label{absolzjjnntail_mieux}
{\rm tl}(|z_{j,n}|) = \frac{1}{n} O
\left( \frac{(\lo \lo n)^2}{(\lo n)^3}
\right)
\end{equation}
where the constant involved in $O(~)$ is 1
(does not depend upon $j$).
\end{proposition}

\begin{proof}
\cite{vergergaugry6} Section 5.1.
\end{proof}

The following asymptotic expansions 
in Proposition \ref{zedeUNmodule}, 
Proposition \ref{zedeJImoduleMoinsUN} and
Proposition \ref{zedeJIMoinsUNzedeJI}
will be used in the
method of Rouch\'e in \S 5. 

\begin{proposition}
\label{zedeUNmodule}
For $n \geq 18$, the modulus of the first root
$z_{1,n}$ of $G_{n}(z) = -1 +z+z^n$ is
\begin{equation}
\label{zedeUN}
|z_{1,n}| = 1 - \frac{\lo n - \lo \lo n}{n}  
+ \frac{1}{n} O\left(\frac{\lo \lo n}{\lo n}\right)
\end{equation}
and
\begin{equation}
\label{UNmoinszedeUN}
|-1 + z_{1,n}| =
\frac{\lo n - \lo \lo n}{n}
+ \frac{1}{n} O\left(
\frac{\lo \lo n}{\lo n}\right)
\end{equation}
with the constant 1 in the two Big Os.
\end{proposition}

\begin{proof}
The root
$z_{1,n}$ belongs to the
subsector $2 \pi \frac{\sqrt{(\lo n) (\lo \lo n)}}{n} > \arg z > 0$:
first, from Lemma \ref{remarkthetan}, 
the asymptotic expansion of
$\lambda_n$ is
$$\lambda_n = 
\frac{\lo \lo n}{\lo n}
+ O(\frac{\lo \lo n}{(\lo n)^2})
$$
with the constant 1 in the Big $O$.
Since
${\rm D}(|z_{1,n}|) =
{\rm D}(\Re(z_{1,n}))
(1 + \bigl(\frac{{\rm D}(\Im(z_{1,n}))}
{{\rm D}(\Re(z_{1,n}))}\bigr)^2)^{1/2}$,
that
$${\rm D}(\Re(z_{1,n})) = \theta_n + 
\frac{2 \pi^2}{n} 
\bigl(\frac{1}{\lo n}\bigr)^2 
\bigl( 1+2 \lambda_n \bigr), \,
{\rm D}(\Im(z_{1,n})) = \frac{2 \pi}{n}
\bigl[1 - \frac{1}{\lo n} (1 + \lambda_n)\bigr]$$
(Proposition \ref{zjjnnExpression}) and
$$\theta_{n} = 1 - \frac{\lo n}{n}(1 - \lambda_n) 
+ \frac{1}{n} O\left(\left(\frac{\lo \lo n}{\lo n}\right)^2\right)
$$
(Proposition \ref{thetanExpression}) we deduce 
\eqref{zedeUN} and the expansion
\eqref{UNmoinszedeUN} from
the expansion of $\lambda_n$.
\end{proof}

\begin{proposition}
\label{zedeJImoduleMoinsUN}
For $n \geq 18$, the modulus of 
$-1 + z_{j,n}$, where  
$z_{j,n}$ is the $j$-th root
of $G_{n}(z) = -1 +z+z^n$, 
$\lceil v_n \rceil \leq j \leq 
\lfloor n/6 \rfloor$, is
\begin{equation}
\label{UNmoinszedeJI}
|-1 + z_{j,n}| 
~=~
2 \sin(\frac{\pi j}{n})
+ \frac{1}{n} O\Bigl(
\Bigl(
\frac{\lo \lo n}{\lo n}
\Bigr)^2
\Bigr)
\end{equation}
with the constant 1 in the Big O.
\end{proposition}

\begin{proof}
From \eqref{zjnModulusFirst}, 
Proposition
\ref{zjjnnExpression}
and Proposition \ref{zedeJImodulesORDRE3}, 
the identity
$$
|-1 + z_{j,n}|^2
=
(-1 + \Re(z_{j,n}))^2
+ (\Im(z_{j,n}))^2
=
1 + |z_{j,n}|^2
- 2 \Re(z_{j,n})
$$
implies:
$|-1 + z_{j,n}|^2
=$
$$2 
-
2 \cos\bigl(2 \pi \frac{j}{n}\bigr)
+
\frac{1}{n} O\Bigl(\Bigl(
\frac{\lo \lo n}{\lo n}
\Bigr)^2\Bigr)
=
4 \sin^2 \bigl(\frac{\pi j}{n}\bigr)
+
\frac{1}{n} O\Bigl(\Bigl(
\frac{\lo \lo n}{\lo n}
\Bigr)^2\Bigr)
$$
with the constant 4
in the Big $O$.
We deduce \eqref{UNmoinszedeJI}.
\end{proof}

\begin{proposition}
\label{zedeJIMoinsUNzedeJI}
For $n \geq 18$, the modulus of 
$(-1 + z_{j,n})/z_{j,n}$, where  
$z_{j,n}$ is the $j$-th root
of $G_{n}(z) = -1 +z+z^n$, 
$\lceil v_n \rceil \leq j \leq 
\lfloor n/6 \rfloor$, is
\begin{equation}
\label{UNmoinszedeJIzedeJI}
\frac{|-1 + z_{j,n}|}{|z_{j,n}|} 
~=~
2 \sin(\frac{\pi j}{n})
\Bigl(
1 -
\frac{1}{n} \lo (2 \sin(\frac{\pi j}{n}))
\Bigr)
+
\frac{1}{n} O\left(
\left(
\frac{\lo \lo n}{\lo n}
\right)^2
\right)
\end{equation}
with the constant 2
in the Big O.
\end{proposition}

\begin{proof}
The expansion \eqref{UNmoinszedeJIzedeJI}
readily comes 
\eqref{UNmoinszedeJI}
and
$|z_{j,n}|$ given by
Proposition \ref{zedeJImodulesORDRE3}. 
\end{proof}

\subsection{Minoration of the Mahler measure}
\label{S4.3}

In \cite{vergergaugry6} 
``\`a la Poincar\'e"
asymptotic expansions
are shown to
give  ``controlled" approximants
of
the set of the values of the Mahler measures
M$(G_n)$ and an exact value of
its limit point. Compared to several methods 
(Amoroso \cite{amoroso2} \cite{amoroso3},
Boyd and Mossinghoff \cite{boydmossinghoff},
Dixon and Dubickas \cite{dixondubickas},
Langevin \cite{langevin}, Smyth \cite{smyth5}),
the present approach is new
in the sense that the use
of auxiliary functions by
Dobrowolski \cite{dobrowolski2}
is replaced by the R\'enyi-Parry
dynamics of the
Perron numbers
$(\theta_{n}^{-1})_{n \geq 2}$.
Let us briefly mention the results.
The product
\begin{equation}
\label{mmapprox}
\Pi_{G_n} ~:=~
D({\rm M}(G_n)) ~=~ D(\theta_n)^{-1} \times 
\prod_{\stackrel{
z_{j,n} ~\mbox{{\tiny in}}~ |z|<1}{{\rm {\small outside ~bump}}}}
{\rm D}(|z_{j,n}|)^{-2}
\end{equation}
is considered, instead of 
\begin{equation}
\label{mmapproxM}
{\rm M}(G_n) ~=~
\theta_{n}^{-1} \,
\prod_{j=1}^{\lfloor n/6 \rfloor} |z_{j,n}|^{-2}
 = 
 \prod_{\lc_{\theta_{n}^{-1}}} |z|^{-1} 
\end{equation}
as approximant value of ${\rm M}(G_n)$.
In \eqref{mmapprox}
the zeroes $z_{j,n}$
present in the bump sector are 
discarded since they do not contribute to the
limited asymptotic expansions, as shown in 
\cite{vergergaugry6} Section 4.2.

In \cite{vergergaugry6} Section 4,
the two limits
$\displaystyle \lim_{n \to +\infty} \Pi_{G_n}$ and 
$\displaystyle \lim_{n \to +\infty} {\rm M}(G_n)$ 
are shown to exist, to be equal (and greater 
than $\Theta$).

\begin{theorem}
\label{main1}
Let $\chi_3$ be the uniquely specified odd
character of conductor $3$
($\chi_{3}(m) = 0, 1$ or $-1$ according to whether $m \equiv 0, \,1$ or
$2 ~({\rm mod}~ 3)$, equivalently
$\chi_{3}(m) = \left(\frac{m}{3}\right)$ the Jacobi symbol),
and denote $L(s,\chi_3) = \sum_{m \geq 1} \frac{\chi_{3}(m)}{m^s}$
the Dirichlet L-series for the character $\chi_{3}$. Then,
with $\Lambda$ given by
\eqref{limitMahlGn},
$\lim_{n \to +\infty} {\rm M}(G_n) ~=~ 
{\rm M}(-1+z+y)=\Lambda = 1.38135\ldots$

\end{theorem}

\begin{proof}
\cite{vergergaugry6} Theorem 1.1; Smyth \cite{smyth4}, using
Boyd-Smyth's method
of bivariate Mahler measures
(\cite{vergergaugry6} Section 4.1).
\end{proof}

Introduced in the product \eqref{mmapproxM}, the
terminants of the asymptotic expansions 
of the moduli of the roots $z_{j,n}$ and of
$\theta_n$
provide the higher-order terms of the 
asymptotic expansion of  ${\rm M}(G_n)$.

\begin{theorem}
\label{mahlerGntrinomial}
Let $n_0$ be an integer such that 
$\frac{\pi}{3} > 2 \pi \frac{\lo n_0}{n_0}$,
and let $n \geq n_0$.
Then,
\begin{equation}
\label{mggnfluctu_}
{\rm M}(G_n) ~=~ 
\Lambda ~ \Bigl( 1 + r(n) 
\, \frac{1}{\lo n}
+ O\left(\frac{\lo \lo n}{\lo n}\right)^2 \bigr)
\Bigr)
\end{equation}
with
the constant 
$1/6$
involved in the Big O, and
with $r(n)$ real,
$|r(n)| \leq 1/6$.
\end{theorem}

\begin{proof}
\cite{vergergaugry6} Theorem 1.2.
\end{proof}

In Theorem \ref{mahlerGntrinomial} we take
$n_0 = 18$. For the small values of $n$, we have:
$${\rm M}(G_2) = \theta_{2}^{-1} = 
\frac{1+\sqrt{5}}{2} = 1.618\ldots$$ 
and the following lower bound.

\begin{proposition}
\label{maincoro3}
{\rm M}$(G_n) \geq {\rm M}(G_5) = \theta_{5}^{-1} = 
\Theta = 1.3247\ldots$ for all $n \geq 3$,
with equality if and only if $n=5$.
\end{proposition}

\begin{proof}
\cite{vergergaugry6} Corollary 1.4.
\end{proof}

The minoration of the
residual distance between the two
algebraic integers $1$ and 
$\theta_{n}^{-1}$ is deduced from the
Zhang-Zagier height and Doche's improvement.

\begin{proposition}
\label{maincoro7}
Let $u=0$ except if $n \equiv 5$ mod $6$ in which case $u = -2$. Then,
\begin{equation}
\label{maincoro7equation}
{\rm M}(\theta_{n}^{-1} -1) ~\geq~ 
\frac{\eta^{n+u}}
{\Lambda} \bigl(1 - 
\frac{1}{ 6 \, \lo n}
\bigr), \qquad n \geq 2,
\end{equation}
with $\eta = 1.2817770214$.
\end{proposition} 

\begin{proof}
Except for a finite subset of algebraic numbers,
the minoration ${\rm M}(\alpha)
{\rm M}(1-\alpha) \geq 
(\theta_{2}^{-1/2})^{\deg(\alpha)}$ was 
established by Zagier \cite{zagier} and improved by
Doche \cite{doche}, with the lower bound
$\eta > \theta_{2}^{-1/2}$
instead of $\theta_{2}^{-1/2}$ itself
.
The minorant
\eqref{maincoro7equation} follows from
\eqref{mggnfluctu_}.
\end{proof}

Recall Dobrowolski's minoration 
 \cite{dobrowolski2}:
\begin{equation}
\label{minoDOBRO1979}
\mbox{{\rm M}}(\alpha) ~>~ 
1 + (1-\epsilon) \left(\frac{\lo \lo d}{\lo d}\right)^3 
\, , \quad d > d_{1}(\epsilon),
\end{equation}
for any nonzero algebraic number
$\alpha$ of degree $d$. 
Voutier in \cite{voutier} 
obtained other effective minorations,
improving \eqref{minoDOBRO1979} for
$d \geq 2$. A survey
on effective minorations is given in
\cite{vergergaugrySurvey}.
Theorem
\ref{mahlerGntrinomial} implies the following
minoration of ${\rm M}(\theta_{n}^{-1})$ 
which is better than \eqref{minoDOBRO1979}
in the sense that the constant term of the minorant is $> 1$. 
\begin{theorem}
\label{maincoro5}
\begin{equation}
\label{minoVG2015}
\mbox{{\rm M}}(\theta_{n}^{-1}) ~>~ \Lambda 
- \frac{\Lambda}{6} 
\left(\frac{1}{\lo n}\right)
\, , \quad n \geq n_{1} = 2.
\end{equation}
\end{theorem}

\begin{proof}
\cite{vergergaugry6} Corollary 1.6.
\end{proof}
The problem of extremality of Perron
numbers is still open
(cf Boyd \cite{boyd7},
\cite{vergergaugrySurvey}).
The extremality of
the Perron numbers
$\theta_{n}^{-1}$ occurs only for
$n=2, 3$.
In general, if extremality holds,
by Lind-Boyd's Conjecture, 
it would be
associated with a lenticular distribution of roots
(of modulus $> 1$)
which admits a proportion asymptotically equal to
$\frac{2}{3} n$. 
For the trinomials $G_{n}^{*}, \,n \geq 4$,
this proportion is only $\frac{1}{3} n$, 
for $n$ large.

\section{The lenticular minorant of the Mahler measure M$(\beta)$ for $\beta > 1$ a real reciprocal algebraic integer close to one}
\label{S5}

The starting point of the proof of Lehmer's
Conjecture is the investigation of the
zeroes of the Parry Upper functions 
$f_{\beta}(z)$
in the annular
region $\{\beta^{-1} < |z| <1\}
\cup\{\beta^{-1}\}$ for
$\beta > 1$ any reciprocal algebraic integer
tending to one. The annular domain 
$\{\beta^{-1} < |z| <1\}$
is 
the extended domain of meromorphy of
the dynamical zeta function $\zeta_{\beta}(z)$
in the open unit disk. 
The case of the zero $\beta^{-1}$,
as $= {\rm D}(\beta^{-1}) 
+ {\rm tl}(\beta^{-1})$,
is readily deduced from
Section $\S$\ref{S5.1} where 
the asymptotic expansion of
a real number
$\beta > 1$ close to $1$ 
is obtained 
as a function
of its dynamical degree
$\dyg(\beta)$. 
The geometry of the other zeroes
is studied in Section
$\S$\ref{S5.2} and $\S$\ref{S5.3}.
The lenticular zeroes
are identified as Galois conjugates 
of $\beta^{-1}$ in Section $\S$\ref{S5.4}.

\subsection{Asymptotic expansions of a real number $\beta > 1$ close to one and of the dynamical degree $\dyg(\beta)$}
\label{S5.1}

\begin{lemma}
\label{longueurintervalthetann}
Let $n \geq 6$. 
The difference
$\theta_{n} - \theta_{n-1} > 0$
admits the following
asymptotic expansion, reduced to its terminant:

\begin{equation}
\label{asymthetanthetan}
\theta_{n} - \theta_{n-1}  ~=~ \frac{1}{n} 
O\left(
\left(
\frac{\lo \lo n}{\lo n}
\right)^2
\right), 
\end{equation}
with the constant $1$ involved in $O\left( ~\right)$.
\end{lemma}

\begin{proof} 
From \eqref{DthetanExpression} and Lemma
\ref{remarkthetan}, we have
$$\theta_n = 1 - \frac{\lo n}{n}(1 - \lambda_n) +
\frac{1}{n} O\left(
\left(
\frac{\lo \lo n}{\lo n}
\right)^2 
\right)
$$
with
the constant $1/2$ involved in $O\left(~ \right)$, and
$$\lambda_n = \frac{\lo \lo n}{\lo n} \left( 
\frac{1}{1+ \frac{1}{\lo n}}\right) +
O\left( \frac{\lo \lo n}{n}  \right)$$
with the constant 1 in the Big O.
Then we deduce
$${\rm D}(\theta_n) - {\rm D}(\theta_{n-1}) = 
\frac{\lo n}{n^2} + O\left(\frac{\lo \lo n}{n^2} \right).$$
The real function $x^{-2} \lo x$ on $(1,+\infty)$
is decreasing for $x \geq \sqrt{e}$.
Hence 
the sequence 
$({\rm D}(\theta_n) - {\rm D}(\theta_{n-1}))$
is decreasing for $n$ large enough.
By Proposition \ref{closetoouane} 
$(\theta_n - \theta_{n-1})_n$
is already known to tend to $0$. 

Since ${\rm tl}(\theta_n) = 
\frac{1}{n} O\left(
\left(
\frac{\lo \lo n}{\lo n}
\right)^2
\right)$, we have
$$\theta_n - \theta_{n-1} = \left(\theta_n - {\rm D}(\theta_n)\right) + 
[{\rm D}(\theta_n) - {\rm D}(\theta_{n-1})]
- \left(\theta_{n-1} - {\rm D}(\theta_{n-1})\right)$$
$$= {\rm tl}(\theta_n) + 
\left(\frac{\lo n}{n^2} + O\left(\frac{\lo \lo n}{n^2} \right)\right) 
- {\rm tl}(\theta_{n-1})$$
\begin{equation}
\label{diffTH}
= \frac{1}{n} O\left(
\left(
\frac{\lo \lo n}{\lo n}
\right)^2
\right)
\end{equation}
where the constant involved in $O\left(~\right)$ is now 
$1 = 1/2 + 1/2$. Hence the claim. 
\end{proof}

\begin{theorem}
\label{betaAsymptoticExpression}
Let $n \geq 6$. Let $\beta > 1$ be a real
number of dynamical degree $\dyg(\beta) = n$.
Then 
$\beta^{-1}$ can be expressed as: 
$\beta^{-1} = {\rm D}(\beta^{-1}) 
+ {\rm tl}(\beta^{-1})$ 
with
${\rm D}(\beta^{-1}) = 1 -$
\begin{equation}
\label{DbetaAsymptoticExpression}
\frac{\lo n}{n}
\left(
1 - \bigl(
\frac{n - \lo n}{n \, \lo n + n - \lo n}
\bigr)
\Bigl(
\lo \lo n - n 
\lo \Bigl(1 - \frac{\lo n}{n}\Bigr)
- {\rm Log} n
\Bigr)
\right)
\end{equation}
and
\begin{equation}
\label{tailbetaAsymptoticExpression} 
{\rm tl}(\beta^{-1})
~=~ \frac{1}{n} \, O \left( \left(\frac{\lo \lo n}{\lo n}\right)^2 \right),
\end{equation}
with the constant $1$ involved in $O \left(~\right)$.
\end{theorem}

\begin{proof}
By definition  $\theta_{n} \leq \beta^{-1} < \theta_{n-1}$.
The
development term of $\beta^{-1}$ is
$D(\beta^{-1}) = D(\beta^{-1} - \theta_{n})+
D(\theta_{n})$,
with 
$|\beta^{-1} - \theta_{n}| < \theta_{n} -\theta_{n-1}$.
By Lemma \ref{longueurintervalthetann},
$D(\theta_{n} -\theta_{n-1}) = 0$.
Therefore
$\beta^{-1} = D(\beta^{-1})
+{\rm tl}(\beta^{-1})$
is deduced from
$D(\theta_{n})$ in 
\eqref{DthetanExpression}.
\end{proof}

\begin{theorem}
\label{nfonctionBETA} 
Let $\beta \in (1, \theta_{6}^{-1})$ be a real number. 
The asymptotic expansion 
of the locally constant function 
$n = {\rm dyg}(\beta)$, as a function of the variable 
$\beta -1$,
is
\begin{equation}
\label{dygExpression}
n = - \frac{\lo (\beta - 1)}{\beta - 1}
\Bigl[1+
O\Bigl(
\Bigl(
\frac{\lo (-\lo (\beta - 1))}
{\lo (\beta - 1)}
\Bigr)^2
\Bigr)
\Bigr]
\end{equation}
with the constant 1 in $O(~)$.
\end{theorem}

\begin{proof}
Inverting \eqref{DbetaAsymptoticExpression} 
gives the asymptotic expansion of 
$n$ as a function of $\beta$:
from \eqref{DbetaAsymptoticExpression} 
readily comes
\begin{equation}
\label{dygAsymptoticExpansion}
n = \frac{\beta}{\beta - 1}
\lo \Bigl(\frac{\beta}{\beta - 1}\Bigr)
\Bigl[ 1+
O\Bigl(
\Bigl(
\frac{\lo \lo \Bigl(\frac{\beta}{\beta - 1}\Bigr)}
{\lo \bigl(\frac{\beta}{\beta - 1} \bigr)}
\Bigr)^2
\Bigr)
\Bigr]
\end{equation}
then \eqref{dygExpression} 
as $\beta \to 1$.
\end{proof}

\begin{remark}
\label{dygvalue}
({\it Simplified forms}): 
If $\beta$ runs over the 
set of Perron numbers
$\theta_{n}^{-1}, \, n = 5, 6, \ldots, 12$,
and over 
the smallest
Parry - Salem numbers 
$\beta \leq 1.240726\ldots$, 
the dynamical degree of
$\beta$ (cf Table 1, Column 1, in \cite{vergergaugrySurvey}) 
is the integer part
of $D(n)$ in \eqref{dygAsymptoticExpansion}:
\begin{equation}
\label{dygintegerpart}
\dyg(\beta) = \lfloor \frac{\beta}{\beta - 1}
\lo \Bigl(\frac{\beta}{\beta - 1}\Bigr) 
\rfloor.
\end{equation}
\noindent
From \eqref{DbetaAsymptoticExpression} 
a (approximate) simplified
form of $\beta^{-1}$ is deduced:
\begin{equation}
\label{DbetaAsymptoticExpressionsimplerinverse}
{\rm D}(\beta^{-1}) = 1 -\frac{1}{n}
\left(
\lo n - \lo \lo n 
+ \frac{\lo \lo n}{\lo n}
\right).
\end{equation}

\end{remark}

\subsection{Fracturability of the minimal polynomial by the Parry Upper function}
\label{S5.2}

The algebraic integer 
$\beta > 1$ is assumed 
reciprocal and close to one.
The relations between
the Parry Upper function $f_{\beta}(z)$ and
the minimal polynomial $P_{\beta}(z)$, as 
analytical functions, 
are given in Theorem
\ref{splitBETAdivisibility}
and will be complemented by
Proposition \ref{splitBETAdivisibility+++}, 
after 
the characterization of the common lenticular 
zeroes, for $n = \dyg(\beta)$
large enough.

\begin{theorem}
\label{splitBETAdivisibility}
Let $1 < \beta < (1+\sqrt{5})/2$ 
be a reciprocal 
algebraic integer. 
The following formal decomposition of 
the minimal polynomial
\begin{equation}
\label{decompo}
P_{\beta}(X) = P_{\beta}^{*}(X)
=
U_{\beta}(X) \times
f_{\beta}(X),
\end{equation}
holds, as a product
of the Parry Upper function 
\begin{equation}
\label{formlacu}
f_{\beta}(X)= G_{\dyg(\beta)}(X)
+ X^{m_1} + X^{m_2} + X^{m_3}+ \ldots.
\end{equation}
with
$m_0:= \dyg(\beta)$,
$m_{q+1} - m_q \geq \dyg(\beta)-1$ for $q \geq 0$,
and the invertible formal series 
$U_{\beta}(X) 
=
- \zeta_{\beta}(X) P_{\beta}(X)
\in \zb[[X]]$.
The specialization $X \to z$ of the formal variable
to the complex variable leads to
the identity between analytic functions,
obeying the Carlson-Polya dichotomy:
\begin{equation}
\label{decompozzz}
P_{\beta}(z) = U_{\beta}(z) \times
f_{\beta}(z)
\quad
\left\{
\begin{array}{cc}
\mbox{on}~ \cb & 
\mbox{if $\beta$ is a Parry number, with}\\
& U_{\beta} ~\mbox{and}~ f_{\beta}~ \mbox{both meromorphic},\\
&\\
\mbox{on}~ |z| < 1 & 
\mbox{if~} \beta \mbox{~is a
nonParry number, with $|z|=1$}\\
& \mbox{as natural boundary for both $U_{\beta}$ and $f_{\beta}$.}
\end{array}
\right.
\end{equation}
In both cases, the domain of holomorphy
of the function $U_{\beta}(z)$   
contains 
the open  
disk
$D(0, \theta_{\dyg(\beta) - 1})$. 
\end{theorem}

\begin{proof}
The reciprocal algebraic integer $\beta$
lies between two successive Perron
numbers of the family
$(\theta_{n}^{-1})_{n \geq 5}$, as
$
\theta_{n}^{-1} < \beta
< \theta_{n-1}^{-1}, \,
\dyg(\beta) = n \geq 6$. 
By Proposition 
\ref{fbetainfinie}
the Parry Upper function
$f_{\beta}(z)$ at 
$\beta$
has the form (from \eqref{formlacu}):
\begin{equation}
f_{\beta}(z) = 
-1 + z + z^n + z^{m_1} + z^{m_2}
+ z^{m_{3}}+ \ldots
\end{equation}
The algebraic integer $\beta$ is a Parry number or 
a nonParry number. 
In both cases, 
$f_{\beta}(\beta^{-1}) = 0$.
If
$f_{\beta}(z) = - 1 + \sum_{j \geq 1} t_j z^j$, 
the zero
$\beta^{-1}$ of $f_{\beta}(z)$
is simple since
the derivative of $f_{\beta}(z)$
satisfies
$f'_{\beta}(\beta^{-1}) 
= \sum_{j \geq 1} j \, t_j \, \beta^{-j+1} > 0$.
The other zeroes of
$f_{\beta}(z)$ of modulus $< 1$
lie in
$1/\beta \leq |z| < 1$. 
Therefore the poles, if any,
of $U_{\beta}(z) = P_{\beta}(z)/f_{\beta}(z)$
of modulus $< 1$
all lie in the annular region
$\theta_{\dyg(\beta)-1} < |z| < 1$.

The formal decomposition
\eqref{decompo}, in
$\zb[[X]]$, is always possible. 
Indeed, if
we put
$U_{\beta}(X) = 
-1 + \sum_{j \geq 1} b_j X^j$, and
$P_{\beta}(X) = 1 + a_1 X + a_2 X^2 + \ldots
a_{d-1} X^{d-1}
+ X^d $, (with $a_j = a_{d-j}$),
the formal identity 
$P_{\beta}(X) = U_{\beta}(X) \times 
f_{\beta}(X)$
leads to the existence of the coefficient 
vector $(b_j)_{j \geq 1}$ of
$U_{\beta}(X)$, as a function
of $(t_j)_{j \geq 1}$ and
$(a_i)_{i = 1,\ldots, d-1}$, as:
$b_1 = -(a_1 + t_1)$,
and, for $r = 2, \ldots, d-1$,
\begin{equation}
\label{bedeRecurrencedebut}
b_r = -(t_r + a_r - 
\sum_{j=1}^{r-1}b_j t_{r-j}) 
\quad
\mbox{with} \quad
b_d = -(t_d + 1 - 
\sum_{j=1}^{d-1} b_j t_{r-j}),
\end{equation}
\begin{equation}
\label{bedeRecurrence}
b_r = -t_r + 
\sum_{j=1}^{r-1} b_j t_{r-j} \quad \mbox{for}~~  r > d.
\end{equation}
Then
$b_j \in \zb, \, j \geq 1$;
the integers
$b_r , r > d$,
are determined recursively
by \eqref{bedeRecurrence} 
by
the sequence $(t_i)$ and
from the finite subset
$\{b_0 = -1, b_1 , b_2 , \ldots , b_d \}$,
itself determined from $P_{\beta}(X)$
using \eqref{bedeRecurrencedebut}.
They inherit the asymptotic
properties
of the asymptotic lacunarity 
of $(t_i)$ when $r$ is very large
\cite{vergergaugry}.
If $R_{\beta}$ denotes the radius 
of convergence of $U_{\beta}(z)$
the inequality 
$R_{\beta}
\geq 
\theta_{\dyg(\beta)-1}$
can be directly obtained
using Hadamard's formula
$R_{\beta}^{-1}
= \limsup_{r \to \infty} \, |b_r|^{1/r}$
and the following Lemma \ref{bound_br_exp} 
(in which $n = \dyg(\beta)$) whose proof
is immediate.

\begin{lemma}
\label{bound_br_exp}
Let $\epsilon > 0$ such that
$\theta_{n-1}^{-1} < \exp(\epsilon)$. 
There exists a constant $C = C(\epsilon) \geq 
\max\{1, \exp(\epsilon (n-1)) - 1\}$
such that:
\begin{equation}
\label{br_expgeneral}
|b_r| \leq C \times \exp(\epsilon \, r),
\qquad \quad \mbox{for all}~ r \geq 0.
\end{equation}
\end{lemma}

\end{proof}

\subsection{A lenticulus of zeroes of $f_{\beta}(z)$ in the cusp of Solomyak's fractal}
\label{S5.3}

In this subsection $\beta \in (1, \theta_{6}^{-1})$  
is assumed to be a real number 
(algebraic or transcendental) such that
$\beta \not\in
\{\theta_{n}^{-1} \mid n \geq 7\}$. 
In Theorem \ref{omegajnexistence} 
it will be proved that, to such a $\beta$, 
is associated a lenticulus of zeroes 
of $f_{\beta}(z)$ in the cusp
of Solomyak's fractal $\mathcal{G}$
(cf Figure 9
in \cite{dutykhvergergaugry}), located
in the angular sector
$$|\arg(z)| < \pi/18.2880.$$
For $\beta \in
\{\theta_{n}^{-1} \mid n \geq 12\}$
the
lenticuli of zeroes of $f_{\beta}(z)$
have been investigated in the larger sector
$|\arg(z)| < \pi/3$
(cf Section $\S$ \ref{S4}).
Examples of lenticuli can be visualized
in \cite{dutykhvergergaugry}.

The method which will be used
to
detect the lenticuli of zeroes of
$f_{\beta}(z)$
is the method of  Rouch\'e. 
This method will be shown to be powerful
enough to reach relevant minorants 
of the Mahler measure
M$(\beta)$ for $\beta > 1$ any reciprocal
algebraic integer (cf $\S$ \ref{S5.6}).

Let $n :=\dyg(\beta)$.
The algebraic integers 
$z_{j,n}, 1 \leq j < \lfloor n/6 \rfloor$,
which constitute the lenticulus
$\lc_{\theta_{n}^{-1}}$ in the (upper) 
Poincar\'e half-plane
satisfy
(cf \S \ref{S4}):
$$f_{\theta_{n}^{-1}}(\theta_{n}) 
= 
f_{\theta_{n}^{-1}}(z_{1,n}) 
=
f_{\theta_{n}^{-1}}(z_{2,n})
=
f_{\theta_{n}^{-1}}(z_{3,n})
=\ldots = 
f_{\theta_{n}^{-1}}(z_{\lfloor n/6 \rfloor, n})
=
0
,$$ 
with $f_{\theta_{n}^{-1}}(z) 
= -1 + z +z^n$.
\,The Parry Upper function at $\beta$ 
is 
characterized by the sequence 
of exponents $(m_q)_{q \geq 0}$:
\begin{equation}
\label{fbet}
f_{\beta}(z) = -1 + z + z^n + z^{m_1} + z^{m_2}
+ z^{m_3} + \ldots = G_{n}(z) +
\sum_{q \geq 1} z^{m_q} ,
\end{equation}
where $m_0 := n$, with the fundamental
minimal gappiness condition:
\begin{equation}
\label{minimalgappiness}
m_{q+1} - m_q \geq n-1 \qquad 
\mbox{for all}~ q \geq 0.
\end{equation}

The R\'enyi $\beta$-expansion 
$d_{\beta}(1)$ of 1
is infinite or not, namely 
the sequence of exponents
$(m_{q})_{q \geq 0}$
is either infinite or finite:
if it is infinite
the integers $m_{q}$
never take  
the value
$+\infty$;
if not
the power series
$f_{\beta}(z)$ is a polynomial
of degree $m_{q}$ for some integer
$m_{q} , q \geq 2$. In both cases,
the integer $m_1 \geq 2 n - 1$ is finite.

We will compute
real 
numbers $t_{j,n} \in (0,1)$ 
such that the small circles
$C_{j,n} := \{ z \mid |z - z_{j,n}|= 
\frac{t_{j,n}}{n}\}$
of respective centers $z_{j,n}$,
$|z_{j,n}| < 1$, 
all satisfy the Rouch\'e conditions:
\begin{equation}
\label{rourouche}
\left|f_{\beta}(z) - G_{n}(z)\right| =
\left|\sum_{q \geq 1} z^{m_q} \right|
< 
\left| G_{n}(z) \right|
\qquad 
\mbox{for} ~z \in C_{j,n},
~\mbox{for}~ j = 1, 2, \ldots, J_n ,
\end{equation}
are pairwise disjoint,
are small enough to
avoid to intersect $|z| = 1$, 
with $J_n \leq \lfloor \frac{n}{6}\rfloor$
the largest possible integer 
(in the sense of Definition
\ref{Jndefinition} 
and Proposition \ref{argumentlastrootJn}).
As a consequence, the number of zeroes
of $f_{\beta}(z)$ and
$G_{n}(z)$ in the open disk
$D_{j,n} := \{ z \mid |z - z_{j,n}|< 
\frac{t_{j,n}}{n}\}$
will be equal, implying
the existence of a simple zero
of the Parry Upper function
$f_{\beta}(z)$ in each disk
$D_{j,n}$.
The maximality of $J_n$ means that
the conditions of Rouch\'e cannot be satisfied
as soon as 
$J_n < j \leq \lfloor \frac{n}{6}\rfloor$
for the reason that the circles 
$C_{j,n}$ are too close to 
$|z|=1$.

The values $t_{j,n}$ are
necessarily smaller than $\pi$ 
in order to  avoid any overlap
between two successive circles
$C_{j,n}$ and $C_{j+1,n}$. Indeed,
since the argument 
$\arg z_{j,n}$ of the $j$-th root
$z_{j,n}$ is roughly equal to
$2 \pi j/n$ 
(Proposition \ref{zedeJIargumentsORDRE1}), 
the distance  $|z_{j,n} - 
z_{j+1,n}|$ is approximately
$2 \pi /n$.

The problem of the choice of the radius
$t_{j,n}/n$ is
a true problem. On one hand, a too small
radius would lead to make impossible
the application of the Rouch\'e conditions,
in particular
for those disks $C_{j,n}$ located
very near the unit circle. Indeed, we do 
not know a priori whether the unit circle is 
a natural boundary or not
for $f_{\beta}(z)$;
locating zeroes close to a natural boundary 
is a difficult problem in general.
On the other hand,
taking larger values of $t_{j,n}$ 
readily leads to
a bad localization of the zeroes of
$f_{\beta}(z)$, and hence, for
algebraic integers $\beta > 1$, to a trivial
minoration of the Mahler measure
${\rm M}(\beta)$. The sequel reports
a compromise, after many trials of the 
author, which works (cf $\S$ \ref{S5.6}).

For any
real number
$\beta \in (1, \theta_{6}^{-1})$  
such that
$\beta \not\in
\{\theta_{n}^{-1} \mid n \geq 7\}$
let us denote by
$\omega_{j,n} \in D_{j,n}$
the simple zero of
$f_{\beta}(z)$; then 
$|\omega_{j,n}|<1$ and
\begin{equation}
\label{localisationZEROomegajn}
|\omega_{j,n}| ~\leq~ 
|z_{j,n}| 
+ \frac{t_{j,n}}{n}
\qquad 
\mbox{with}
\quad
z_{j,n} ~\neq~
\omega_{j,n},\quad j = 1, 2, \ldots, J_n ;
\end{equation}
if, in addition, $\beta > 1$
is a reciprocal algebraic integer, 
the strategy
for obtaining a minorant
of ${\rm M}(\beta)$ will be 
the following: to
identify the 
zeroes
$\omega_{j,n}$ as 
roots of the minimal polynomial
$P_{\beta}^{*}(z)= P_{\beta}(z)$, 
then to obtain 
a lower bound of the Mahler measure
${\rm M}(\beta)$
(will be made explicit 
in \S \ref{S5.6}) from these roots by
\begin{equation}
\label{minoMM}
\beta \times \prod 
\, |\omega_{j,n}|^{-2}
~~\geq~~
\theta_{n}^{-1} \times
\prod_{j} \, (|z_{j,n}| 
+ \frac{t_{j,n}}{n})^{-2} ,
\end{equation}
where $j$ runs over 
$\{1, 2, \ldots, J_n
\}$.

In general, 
for
any real number
$\beta \in (1, \theta_{6}^{-1})$
such that
$\beta \not\in
\{\theta_{n}^{-1} \mid n \geq 7\}$,
the quantities $t_{j,n}$ will be 
estimated by the following inequalities:
\begin{equation}
\label{rourouchesimple}
\frac{|z|^{2 (n - 1)+1}}{1-|z|^{n-1}} =
\frac{|z|^{2n - 1}}{1-|z|^{n-1}}
~<~ 
\left| G_{n}(z) \right|
\quad 
\mbox{for} ~z \in C_{j,n},
\quad j = 1, 2, \ldots, J_n
\end{equation}
instead of \eqref{rourouche}, 
too complicated to handle.
In \eqref{rourouchesimple} 
the exponent
``$n-1$" comes from the minimal
gappiness condition \eqref{minimalgappiness}, 
that is from the
dynamical degree $n$
of $\beta$ 
itself, as unique variable.
Indeed,
due to the great variety of possible 
infinite admissible 
sequences $(m_q)_{q \geq 1}$ in
the power series $f_{\beta}(z)$ 
in \eqref{fbet},
for which
$m_q \geq q (n-1) + n$ for all $q \geq 1$,
we will proceed by taking 
the upper bound condition
\eqref{rourouchesimple}
which comes from the general inequality:
$$\left|f_{\beta}(z) - G_{n}(z)\right|
 =
\left|\sum_{q \geq 1} z^{m_q} \right|
\leq \sum_{q \geq 1}   | z^{m_q}|
\leq \frac{|z|^{2n - 1}}{1-|z|^{n-1}} ,
\qquad \quad |z| < 1.
$$

The radius $t_{0,n}/n$ of the first circle 
$C_{0,n} := \{ z \mid |z - \theta_{n}| 
= \frac{t_{0,n}}{n}\}$, 
which contains 
$\beta^{-1}$, is readily obtained without
the method of Rouch\'e.

\begin{lemma}
\label{tzero}
Let $n \geq 7$.
$$t_{0,n} := 
\left(
\frac{\lo \lo n}{\lo n}
\right)^2 .
$$
\end{lemma}
\begin{proof} 
Since $\beta^{-1}$ runs over
the open interval
$(\theta_{n-1}, \theta_{n})$, this
interval
$(\theta_{n-1}, \theta_{n})$
is necessarily completely included in
$D_{0,n}$, and
the radius
of $C_{0,n}$ is 
$\theta_{n} - \theta_{n-1}$.
We deduce the result 
from Lemma \ref{longueurintervalthetann}. 
From Proposition \ref{zjjnnExpression} 
the root
$z_{1,n}$ admits 
$\Im(z_{1,n}) =
\frac{2 \pi}{n} (1 - \frac{1}{\lo n} + \ldots)$
as imaginary part. Then, 
for any 
$t_{1,n} \in (0,1)$, 
the circle $C_{0,n}$, of radius
$t_{0,n}/n$, and 
$C_{1,n}$ are
disjoint and do not intersect $|z|=1$.
\end{proof}

By Proposition \ref{rootsdistrib}
the only angular sector to be considered
for the roots $z_{j,n}$ of $G_n$
and the
Rouch\'e circles
$C_{j,n}$, up to complex-conjugation,
is $0 \leq \arg(z) \leq +\frac{\pi}{3}$.
In this sector the ``bump" angular sector,
$\arg z \in 
(0, 2 \pi (\lo n)/n)$
(cf Appendix; Remark 3.3 in \cite{vergergaugry6}), 
will be shown to contribute 
negligibly.

The existence of the roots $\omega_{j,n}$
in the main subsector  
is proved in Theorem \ref{cercleoptiMM}, 
then in Theorem 
\ref{omegajnexistence} in a refined version.
Proposition \ref{cercleoptiMMBUMP}
completes the proof of their existence
in the bump angular sector.
In the complement
of the family of the
adjustable Rouch\'e disks
Theorem \ref{absencezeroesOutside} 
asserts the existence of
a zerofree region
depending upon 
the dynamical degree of $\beta$.

In the following
we consider the problem
of the parametrization of the radii
$t_{j,n}/n$ by a unique 
real number $a \geq 1$,
allowing to adjust continuously and uniformly
the size of each circle $C_{j,n}$.
We solve it by finding an optimal value.

\begin{theorem}
\label{cercleoptiMM}
Let $n \geq n_1 = 195$,
$a \geq 1$, and
$j \in \{\lceil v_n \rceil, 
\lceil v_n \rceil + 1, \ldots, \lfloor n/6 \rfloor\}$ .
Denote by $C_{j,n}
:= \{z \mid |z-z_{j,n}| = \frac{t_{j,n}}{n} \}$
the circle centered at the $j$-th root
$z_{j,n}$ of $-1 + X + X^n$, with
$t_{j,n} = \frac{\pi |z_{j,n}|}{a}$. 
Then the 
condition of Rouch\'e 
\begin{equation}
\label{rouchecercle}
\frac{
\left|z\right|^{2 n -1}}{1 - |z|^{n-1}}
~<~
\left|-1 + z + z^n \right| , 
\qquad \mbox{for all}~ z \in C_{j,n},
\end{equation}
holds true on the circle
$C_{j,n}$
for which the center $z_{j,n}$
satisfies
\begin{equation}
\label{petitsecteur18}
\frac{|-1+z_{j,n}|}{|z_{j,n}|} 
~~<~~ \frac{
1 -
\exp\bigl(
\frac{- \pi}{a}
\bigr)
}{2 \exp\bigl(
\frac{\pi}{a}
\bigr) -1} .
\end{equation}
The condition $n \geq 195$
ensures the existence of such roots $z_{j,n}$.
Taking the value $a = a_{\max}
= 5.87433\ldots$
for which the upper bound of
\eqref{petitsecteur18} 
is maximal, equal to 
$0.171573\ldots$, the roots
$z_{j,n}$ which satisfy 
\eqref{petitsecteur18}
all belong to the angular sector,
independent of $n$:
\begin{equation}
\label{angularsectoramax}
\arg(z) ~\in ~\bigl[\, 0, 
+ \frac{\pi}{18.2880}  \,\bigr] .
\end{equation}
For any real number $\beta > 1$ having 
$\dyg(\beta) = n$,
$f_{\beta}(z)$ admits a
simple
zero $\omega_{j,n}$ in 
$D_{j,n}$
for which the center $z_{j,n}$
satisfies 
\eqref{petitsecteur18}
with $a=a_{{\max}}$, and $j$ in the range
$ \{\lceil v_n \rceil, 
\lceil v_n \rceil + 1, \ldots, \lfloor n/6 \rfloor\}$. 

\end{theorem}

\begin{proof}
Denote by $\varphi := \arg (z_{j,n})$
the argument of the $j$-th root
$z_{j,n}$.

Since $-1 + z_{j,n} + z_{j,n}^n = 0$, 
we have $|z_{j,n}|^n = |-1 + z_{j,n}|$.
Let us write $z= z_{j,n}+ \frac{t_{j,n}}{n} e^{i \psi}
=
z_{j,n}(1 + \frac{\pi}{a \, n} e^{i (\psi - \varphi)})$
the generic element belonging to $C_{j,n}$, with
$\psi \in [0, 2 \pi]$.

Let $X := \cos(\psi - \varphi)$.
Let us show that if the inequality
\eqref{rouchecercle} of Rouch\'e 
holds true for $X =+1$, for a certain circle
$C_{j,n}$,
then it holds true
for all $X \in [-1,+1]$, that is for 
every argument $\psi \in [0, 2 \pi]$,
i.e. for every 
$z \in C_{j,n}$.
Let us show 
$$
\left|1 + \frac{\pi}{a \,n} e^{i (\psi - \varphi)}
\right|^{n}
=
\exp\Bigl(
\frac{\pi \, X}{a}\Bigr)
\times 
\left(
1 - \frac{\pi^2}{2 a^2 \, n} (2 X^2 -1) 
+O(\frac{1}{n^2})
\right) 
$$
and
$$
\arg\left(
\Bigl(1 + \frac{\pi}{a \, n} e^{i (\psi - \varphi)}
\Bigr)^{n}\right)
=
sgn(\sin(\psi - \varphi))
\times
\left( ~\frac{\pi \, \sqrt{1-X^2}}{a}
[1 -
\frac{\pi \, X}{a \, n}
]
+O(\frac{1}{n^2})
\right)
.$$
Indeed, since 
$\sin(\psi - \varphi) = \pm \sqrt{1 - X^2}$,
then
$$
\Bigl(1 + \frac{\pi}{a \, n} e^{i (\psi - \varphi)}
\Bigr)^{n}
=
\exp\left(
n \, \lo (1 + \frac{\pi}{a \, n} e^{i (\psi - \varphi))})
\right)
$$
$$=
\exp\left(
\frac{\pi}{a} (
X \pm i \sqrt{1-X^2}
)
+\left[- \frac{n}{2} (\frac{\pi}{a \, n} (
X \pm i \sqrt{1-X^2}
))^2
+ O(\frac{1}{n^2})
\right]
\right)
$$
$$=
\exp\Bigl(
\frac{\pi \, X}{a}
- \frac{\pi^2}{2 a^2 \, n} (2 X^2 -1) 
+O(\frac{1}{n^2})\Bigr)
\times
\exp\left(
\pm \, i ~\frac{\pi \, \sqrt{1-X^2}}{a}
[1 -
\frac{\pi \, X}{a \,n}
]
+O(\frac{1}{n^2})
\right)
. $$
Moreover,
$$
\left|1 + \frac{\pi}{a \, n} 
e^{i (\psi - \varphi)}
\right|
=
\left|1 + \frac{\pi}{a \, n} 
(X \pm i \sqrt{1-X^2})
\right|
=1 + \frac{\pi \, X}{a \, n} + O(\frac{1}{n^2}).
$$
with
$$\arg(1 + \frac{\pi}{a \, n} e^{i (\psi - \varphi)})
= 
sgn(\sin(\psi - \varphi)) \times
\frac{\pi \sqrt{1 - X^2}}{a \, n} 
+ O(\frac{1}{n^2}).
$$
For all $n \geq 18$ (Proposition 3.5
in \cite{vergergaugry6}), 
let us recall that
\begin{equation}
\label{devopomain}
|z_{j,n}|
=
1 + \frac{1}{n} \lo (2 \sin \frac{\pi j}{n})
+ \frac{1}{n} O \left(
\frac{\lo \lo n}{\lo n}\right)^2 .
\end{equation}
Then the left-hand side term of \eqref{rouchecercle}
is
$$\frac{
\left|z\right|^{2 n -1}}{1 - |z|^{n-1}}
=
\frac{|-1 + z_{j,n}|^2 
\left|1 + \frac{\pi}{a \, n} e^{i (\psi - \varphi)}
\right|^{2 n}}
{|z_{j,n}| \, 
\left|1 + \frac{\pi}{a \, n} e^{i (\psi - \varphi)}\right|
-
|-1 + z_{j,n}| \,
\left|1 + \frac{\pi}{a \, n} e^{i (\psi - \varphi)}\right|^{n}}$$

\begin{equation}
\label{rouchegauche}
=
\frac{|-1 + z_{j,n}|^2 
\left(
1 - \frac{\pi^2}{a \, n} (2 X^2 -1) 
\right)
\exp\bigl(
\frac{2 \pi \, X}{a}\bigr)
}
{\bigl(\,
1 + \frac{1}{n} \lo (2 \sin \frac{\pi j}{n})
+ \frac{\pi \, X}{a \, n} 
\bigr)
-
|-1 + z_{j,n}| \,
\left(
1 - \frac{\pi^2}{2 a \, n} (2 X^2 -1) 
\right)
\exp(
\frac{\pi \, X}{a})
}
\end{equation}
up to
$\frac{1}{n} 
O \left(
\frac{\lo \lo n}{\lo n}
\right)^2$
-terms (in the terminant).
The right-hand side term of 
\eqref{rouchecercle} is 
$$\left|-1 + z + z^n \right|
=
\left|
-1 + z_{j,n}\Bigl(1 + \frac{\pi}{n \, a} e^{i (\psi - \varphi)}\Bigr) +
z_{j,n}^{n}
\Bigl(1 + \frac{\pi}{n \, a} e^{i (\psi - \varphi)}
\Bigr)^n
\right|
$$
$$=
\left|
-1 + z_{j,n}
(1 \pm i \frac{\pi \sqrt{1 - X^2}}{a \, n})
(1 + \frac{\pi \, X}{a \, n}) +
\hspace{5cm} \mbox{} \right.
$$
\begin{equation}
\label{rouchedroite}
\left.
\mbox{} \hspace{-0.2cm}
(1 - z_{j,n})
\bigl(
1 - \frac{\pi^2}{2 a^2 \, n} (2 X^2 -1) 
\bigr)
\exp\bigl(
\frac{\pi \, X}{a}
\bigr) \,
\exp\Bigl(
\pm \,
i \,
\Bigl( ~\frac{\pi \, \sqrt{1-X^2}}{a}
[1 - \frac{\pi \, X}{a \, n}] 
\Bigr)
\Bigr)
\!+ \!O(\frac{1}{n^2})
\right|
\end{equation}

Let us consider
\eqref{rouchegauche}
and
\eqref{rouchedroite}
at the first order for the asymptotic expansions, 
i.e. up to $O(1/n)$ - terms instead of
up to 
$O(\frac{1}{n}(\lo \lo n/ \lo n)^2)$ - terms or
$O(1/n^2)$ - terms.
\eqref{rouchegauche} becomes:
$$\frac{|-1+z_{j,n}|^2 \exp(\frac{2 \pi X}{a})}
{|z_{j,n}| - |-1+z_{j,n}| \exp(\frac{\pi X}{a})}$$
and \eqref{rouchedroite} is equal to:
$$|-1 + z_{j,n}|
\left|
1 -
\exp\bigl(
\frac{\pi \, X}{a}
\bigr) \,
\exp\Bigl(
\pm \,
i \,
\frac{\pi \, \sqrt{1-X^2}}{a} 
\Bigr)
\right|
$$
and is independent of the sign of 
$\sin(\psi - \varphi)$.
Then
the inequality \eqref{rouchecercle} is 
equivalent to
\begin{equation}
\label{roucheequiv1}
\frac{|-1+z_{j,n}|^2 \exp(\frac{2 \pi X}{a})}
{|z_{j,n}| - |-1+z_{j,n}| \exp(\frac{\pi X}{a})}
<
|-1+z_{j,n}| \, 
\left|
1 -
\exp\bigl(
\frac{\pi \, X}{a}
\bigr) \,
\exp\Bigl(
\pm \,
i \,
\frac{\pi \, \sqrt{1-X^2}}{a} 
\Bigr)
\right|
,
\end{equation}
and \eqref{roucheequiv1} to
\begin{equation}
\label{amaximalfunctionX}
\frac{|-1 + z_{j,n}|}{|z_{j,n}|}
~  < ~ \, 
\frac{\left|
1 -
\exp\bigl(
\frac{\pi \, X}{a}
\bigr) \,
\exp\Bigl(
 \,
i \,
\frac{\pi \, \sqrt{1-X^2}}{a} 
\Bigr)
\right|
\exp\bigl(
\frac{-\pi \, X}{a}
\bigr)}{\exp\bigl(
\frac{\pi \, X}{a}
\bigr) +\left|
1 -
\exp\bigl(
\frac{\pi \, X}{a}
\bigr) \,
\exp\Bigl(
 \,
i \,
\frac{\pi \, \sqrt{1-X^2}}{a} 
\Bigr)
\right|} =: \kappa(X,a).
\end{equation}

Denote by
$\kappa(X,a)$ the right-hand side term, 
as a function of $(X, a)$, 
on $[-1, +1] \times [1, +\infty)$.
It is routine to show that, for any
fixed $a$,
the partial derivative $\partial \kappa_X$
of $\kappa(X,a)$ with respect to $X$ 
is strictly negative
on the interior of the domain. 
The function $x \to 
\kappa(x,a)$ takes its minimum
at $X=1$, and this minimum is always 
strictly positive. 
Hence the inequality of Rouch\'e
\eqref{rouchecercle} will be satisfied
on $C_{j,n}$ once it is 
satisfied at $X = 1$.

For which range of values of $j/n$? 
Up to
$O(1/n)$-terms in 
\eqref{amaximalfunctionX},
it is given by the set of integers
$j$ for which $z_{j,n}$
satisfies:
\begin{equation}
\label{amaximalfunction}
\frac{|-1 + z_{j,n}|}{|z_{j,n}|} 
< \kappa(1,a) 
=
\frac{\left|
1 -
\exp\bigl(
\frac{\pi}{a}
\bigr)
\right|
\exp\bigl(
\frac{-\pi}{a}
\bigr)}{\exp\bigl(
\frac{\pi}{a}
\bigr) +\left|
1 -
\exp\bigl(
\frac{\pi}{a}
\bigr)
\right|} .
\end{equation}
In order to take into account
a collection of roots of $z_{j,n}$
as large as possible, i.e.
in order to have a minorant of the Mahler measure
${\rm M}(\beta)$ the largest possible,
the value of $a \geq 1$ has to be chosen
such that $a \to \kappa(1,a)$ is maximal
in \eqref{amaximalfunction}.

\begin{figure}
\begin{center}
\includegraphics[width=6cm]{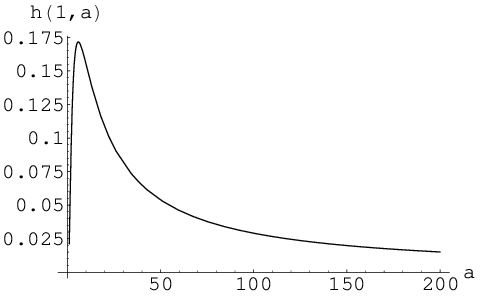}
\end{center}
\caption{
Curve of the Rouch\'e condition $a \to \kappa(1,a)$ 
(upper bound in \eqref{petitsecteur18}),
for the circles
$C_{j,n} = \{z \mid |z-z_{j,n}| =
\pi |z_{j,n}|/(a \, n)\}$
centered at the zeroes $z_{j,n}$
of the trinomial $-1 + X + X^n$,
as a function of
the size of the circles
$C_{j,n}$ parametrized by
the adjustable real number $a \geq 1$.}
\label{h1a}
\end{figure}

The function $a \to \kappa(1,a)$ 
reaches its maximum
$\kappa(1, a_{\max}) := 0.171573\ldots$
at $a_{\max} = 5.8743\ldots$.
(Figure \ref{h1a}).
Denote by $J_n$ the maximal integer $j$
for which
\eqref{amaximalfunction} is satisfied
and in which $a$ is taken equal to 
$a_{\max}$ 
(Definition \ref{Jndefinition} 
and Proposition \ref{argumentlastrootJn}).
From Proposition \ref{argumentlastrootJn},
in which are reported the asymptotic expansions
of $J_n$ and $\arg(z_{J_n , n})$,
we deduce 
\begin{equation}
\label{gammeindexj}
\arg(z_{j,n}) < \frac{\pi}{18.2880\ldots} =
0.171784\ldots
\quad
\mbox{for} ~
j= \lceil v_n \rceil, \lceil v_n \rceil + 1, \ldots,
J_n .
\end{equation}
\begin{remark}
\label{value195}
The minimal value $n_1 = 195$
is calculated by the condition 
$2 \pi \frac{v_{n}}{n}
<
\frac{\pi}{18.2880\ldots} 
$
$= 0.171784\ldots $,
for all
$n \geq n_1$,
for having a strict inclusion,
of the ``bump sector" inside
the angular sector defined by
the maximal opening angle
$0.171784\ldots$
(cf Appendix for the sequence $(v_n)$)
\end{remark}
This finishes the proof.
\end{proof}

Let us calculate the argument of the last root
$z_{j,n}$  for which 
\eqref{amaximalfunctionX} 
is an equality with $X=1$.

\begin{definition}
\label{Jndefinition}
Let $n \geq 195$.
Denote by $J_n$ the largest integer
$j \geq 1$ such that
the root
$z_{j, n}$ of $G_n$  
satisfies
\begin{equation}
\label{JJJn}
\frac{|-1 + z_{j ,n}|}{|z_{j ,n}|}
~  \leq ~ \, 
\kappa(1,a_{\max}) =
\frac{
1 -
\exp\bigl(
\frac{- \pi}{a_{\max}}
\bigr)
}{2 \exp\bigl(
\frac{\pi}{a_{\max}}
\bigr) 
- 1}
=
0.171573\ldots
\end{equation}
\end{definition}

Let us observe that the upper bound
$\kappa(1, a_{\max})$ is independent of $n$.
From this independence we deduce
the following ``limit" angular sector
in which the Rouch\'e conditions
can be applied.

\begin{proposition}
\label{argumentlastrootJn}
Let $n \geq 195$. Let us
put $\kappa:=\kappa(1, a_{\max})$ for short.
Then
\begin{equation}
\label{argzJJJn}
\arg(z_{J_n , n}) = 
2 \arcsin 
\bigl( \frac{\kappa}{2} \bigr)
+ \frac{\kappa \, \lo \kappa}
{n \sqrt{4 - \kappa^2}}
+
\frac{1}{n} O\bigl(
\bigl(
\frac{\lo \lo n}{\lo n}
\bigr)^2
\bigr),
\end{equation}
\begin{equation}
\label{Jnasymptotic}
J_n = 
\frac{n}{\pi}
\bigl(
\arcsin\bigl( \frac{\kappa}{2} \bigr) 
\bigr)
+
\frac{\kappa \, \lo \kappa}
{\pi \,\sqrt{4 - \kappa^2}}
+ 
O\bigl(
\bigl(
\frac{\lo \lo n}{\lo n}
\bigr)^2
\bigr)
\end{equation}
with, at the limit,
\begin{equation}
\label{philimite}
\lim _{n \to + \infty}
\arg(z_{J_n , n}) = 
\lim _{n \to + \infty}
2 \pi \frac{J_n}{n}
=
2 \arcsin 
\bigl( \frac{\kappa}{2} \bigr) = 0.171784\ldots 
\end{equation}
\end{proposition}

\begin{proof}
Since $\lim_{n \to +\infty} |z_{J_n , n}| = 1$,
we deduce from \eqref{JJJn} that
the limit argument
$\varphi_{lim}$ of $z_{J_n, n}$
satisfies $|-1 + \cos(\varphi_{lim})
+ i \sin(\varphi_{lim})|
= 2 \sin(\varphi_{lim}/2) = \kappa(1,a_{\max})$.
We deduce \eqref{philimite}, and the equality between
the two limits from
\eqref{argumentzjn}.
 
From \eqref{JJJn}, the inequality
$|-1 + z_{j ,n}| \leq |z_{j ,n}| \, \kappa(1,a_{\max})$
already implies that 
$\arg(z_{J_n , n})) < \varphi_{lim}$.
In the sequel, 
we will use the asymptotic expansions
of the roots $z_{J_n , n}$.
From Section 6 in \cite{vergergaugry6}
the argument of $z_{J_n , n}$ takes the following
form
\begin{equation}
\label{argumentzjn}
\arg(z_{J_n , n})) =
2 \pi (\frac{J_n}{n} + \Re)
\quad
\mbox{with}
\quad
\Re =
- \frac{1}{2 \pi n}
\left[
\frac{1 - \cos(\frac{2 \pi J_n}{n})}{\sin (\frac{2 \pi J_n}{n})} 
\lo (2 \sin(\frac{\pi J_n}{n}))
\right]
\end{equation}
with
$${\rm tl}(\arg(z_{J_n , n})))
=
+\frac{1}{n} O\left(
\left(
\frac{\lo \lo n}{\lo n}
\right)^2
\right) .
$$
Its modulus is
\begin{equation}
\label{devopomainCONSTANTEc}
|z_{J_n ,n}|
=
1 + \frac{1}{n} \lo (2 \sin \frac{\pi J_n}{n})
+ \frac{1}{n} O \left(
\frac{\lo \lo n}{\lo n}\right)^2 .
\end{equation}
Denote $\varphi := \arg(z_{J_n ,n})$.
Up to $\frac{1}{n} O\bigl(
\bigl(
\frac{\lo \lo n}{\lo n}
\bigr)^2
\bigr)$-terms, 
we have
$$|-1 + z_{J_n ,n}|^2 =
\Bigl|-1 + 
[1 + \frac{1}{n} \lo (2 \sin \frac{\pi J_n}{n})]
(\cos(\varphi) + i \sin(\varphi))
\Bigr|^2
$$
$$=
[-1 +[1 + \frac{1}{n} \lo (2 \sin \frac{\pi J_n}{n})]
(\cos(\varphi)]^2
+
[1 + \frac{1}{n} \lo (2 \sin \frac{\pi J_n}{n})]^2
(\sin(\varphi)^2
$$
$$=
1 +
[1 + \frac{1}{n} \lo (2 \sin \frac{\pi J_n}{n})]^2
-
2 [1 + \frac{1}{n} \lo (2 \sin \frac{\pi J_n}{n})]
\cos(\varphi)
$$
\begin{equation}
\label{JJJngauche}
=
4 (\sin(\frac{\varphi}{2}))^2
+
\frac{4}{n} (\sin(\frac{\varphi}{2}))^2 \, 
\lo (2 \sin \frac{\pi J_n}{n})
=
4 (\sin(\frac{\varphi}{2}))^2
[1 + \frac{1}{n} \lo (2 \sin \frac{\pi J_n}{n})] .
\end{equation}
Up to $\frac{1}{n} O\bigl(
\bigl(
\frac{\lo \lo n}{\lo n}
\bigr)^2
\bigr)$-terms, due to the definition of $J_n$,
let us consider \eqref{JJJn} as an equality;
hence, from \eqref{JJJngauche} and
\eqref{devopomainCONSTANTEc}, 
the following identity 
should be satisfied
\begin{equation}
\label{argumentINTER}
2 \sin(\frac{\varphi}{2})
=
\kappa \,[1 + \frac{1}{2 n} \lo (2 \sin \frac{\pi J_n}{n})]
\end{equation}
We now use \eqref{argumentINTER} to obtain
an asymptotic expansion of
$\psi_n := 2 \pi \frac{J_n}{n} - \varphi_{lim}$
as a function of $n$ and $\varphi_{lim}$
up to $\frac{1}{n} O\bigl(
\bigl(
\frac{\lo \lo n}{\lo n}
\bigr)^2
\bigr)$-terms.
First, at the first order in $\psi_n$,
$$
\sin(\frac{\pi J_n}{n}) = \frac{\psi_n}{2} \cos(\frac{\varphi_{lim}}{2})
+ \sin(\frac{\varphi_{lim}}{2}),
\quad
\cos(\frac{\pi J_n}{n}) = - \frac{\psi_n}{2} \sin(\frac{\varphi_{lim}}{2})
+ \cos(\frac{\varphi_{lim}}{2}),
$$
$$
\lo \bigl(2 \sin(\frac{\pi J_n}{n})\bigr)
=
\lo \bigl(2 \sin(\frac{\varphi_{lim}}{2})\bigr)
+
\psi_n \frac{\cos(\frac{\varphi_{lim}}{2})}
{2 \sin(\frac{\varphi_{lim}}{2})}
=
\lo \kappa + \psi_n \frac{\cos(\frac{\varphi_{lim}}{2})}{h}.
$$
Moreover,
$$
\Bigl[
\frac{1 - \cos(\frac{2 \pi J_n}{n})}{\sin (\frac{2 \pi J_n}{n})} 
\lo (2 \sin(\frac{\pi J_n}{n}))
\Bigr]
$$
\begin{equation}
\label{RezJJJn}
=
\tan (\frac{\varphi_{lim}}{2})
(\lo \kappa)
\Bigl[1 + \psi_n 
\left(
\frac{1}{\sin (\varphi_{lim})}
+
\frac{\cos (\frac{\varphi_{lim}}{2})}
{\kappa \, \lo \kappa}
\right)
\Bigr].
\end{equation}
Hence,
with $2 \sin(\varphi/2)
=
2 \sin(\pi J_n/n) \cos(\pi \Re)
+
2 \cos(\pi J_n/n) \sin(\pi \Re)
$,
and
from \eqref{argumentzjn},
up to $\frac{1}{n} O\bigl(
\bigl(
\frac{\lo \lo n}{\lo n}
\bigr)^2
\bigr)$-terms, the identity
\eqref{argumentINTER} becomes
$$
[\psi_n \cos(\frac{\varphi_{lim}}{2})
+ 2 \sin(\frac{\varphi_{lim}}{2})]
+(\frac{- 2 \cos(\frac{\varphi_{lim}}{2}) \tan(\frac{\varphi_{lim}}{2})
\, \lo \kappa}{2 n})
=
\kappa \,  \bigl[
1 + \frac{ \lo \kappa}{2 n}
\bigr] .
$$
We deduce 
\begin{equation}
\label{psinasymptotic}
\psi_n 
=
\frac{\kappa \, \lo \kappa}
{n \, \cos(\frac{\varphi_{lim}}{2})}
+
\frac{1}{n} O\bigl(
\bigl(
\frac{\lo \lo n}{\lo n}
\bigr)^2
\bigr)
,
\end{equation} 
then $2 \pi J_n / n = \psi_n + \varphi_{lim}$
and
\eqref{Jnasymptotic}.
With
$2 \pi J_n / n$, and
from \eqref{argumentzjn} and
\eqref{RezJJJn} we deduce
\eqref{argzJJJn}.
This finishes the proof.
\end{proof}

\begin{remark}
\label{openingangle_sin_quadratic_alginteger}
(i) The maximal half-opening angle 
of the sector in which
one can detect zeroes of $f_{\beta}(z)$,
for any $\beta$ such that
$\theta_{n-1} < \beta^{-1} < \theta_n$,
by the method of Rouch\'e, is
$0.17178... =
2 \arcsin(\frac{\kappa(1, a_{\max})}{2})$.
Remarkably
this upper bound $2 \arcsin(\frac{\kappa(1, a_{\max})}{2})$
is independent of $n$. By comparison it is
fairly small with respect to $\pi/3$ for
the Perron numbers $\theta_{n}^{-1}$.

(ii) The curve $a \to \kappa(1,a)$, given by Figure 
\ref{h1a},
is such that any value
in the interval $(0, \kappa(1,a_{max}))$
is reached by the function
$\kappa(1,a)$
from two values say $a_1$
and
$a_2$, of $a$, satisfying $a_1 < a_{max}
< a_2$. On the contrary, the correspondence
$a_{max} \leftrightarrow \kappa(1,a_{max})$ is unique, 
corresponding to a double root.
Denote $D:=\exp(\pi/a_{max})$ and 
$\kappa := \kappa(1,a_{max})$.
It means that the quadratic algebraic equation 
$2 \kappa D^2 - (\kappa+1) D +1 =0$ 
deduced from the upper bound
in \eqref{JJJn} has necessarily
a discriminant equal to zero.
The discriminant is
$\kappa^2 - 6 \kappa +1$. 
Therefore $D = (\kappa+1)/(4 \kappa)$
and the limit value
$x = 2 \arcsin(\kappa/2)$ in
\eqref{philimite} satisfies
the quadratic algebraic equation
$$4 (\sin(x/2))^2   -
12 \sin(x/2) + 1 = 0.$$ 
\end{remark}

\begin{proposition}
\label{cercleoptiMMBUMP}
Let $n \geq n_1 = 195$.
The circles
$C_{j,n}
:= \{z \mid |z-z_{j,n}| = 
\frac{\pi |z_{j,n}|}{n\, a_{\max}} \}$
centered at the roots 
$z_{j,n}$
of the trinomial
$-1 + z + z^n$ which belong to
the ``bump sector", namely 
for
$j \in \{1, 2, \ldots, \lfloor v_n \rfloor \}$,
are such that 
the 
conditions of Rouch\'e 
\begin{equation}
\label{rouchecercleBUMP}
\frac{
\left|z\right|^{2 n -1}}{1 - |z|^{n-1}}
~<~
\left|-1 + z + z^n \right| , 
\quad \mbox{for all}~ z \in C_{j,n},
~~\mbox{}~
1 \leq j \leq \lfloor v_n \rfloor ,
\end{equation}
hold true. 
For any real number $\beta > 1$ having 
$\dyg(\beta) = n$,
$f_{\beta}(z)$ admits a
simple
zero $\omega_{j,n}$ in 
$D_{j,n}$
(with $a=a_{{\max}}$), for
$j$ in the range 
$\{1, 2, \ldots,
\lfloor v_n \rfloor\}$.

\end{proposition}

\begin{proof}
The development terms $``D"$ of the
asymptotic expansions
of $|z_{j,n}|$ change 
from the main angular sector
$\arg z \in (2 \pi (\lo n)/n,
\pi/3)$ to
the first transition region
$\arg z \asymp 2 \pi (\lo n)/n$, the ``bump sector",
further to the second transition region\newline
$\arg z \asymp 2 \pi \sqrt{(\lo n)(\lo \lo n)}/n$,
and to a small neighbourhood of 
$\theta_n$
(Section \ref{S4.2}).

Then the proof of
\eqref{rouchecercleBUMP}
is the same as that
of Theorem \ref{cercleoptiMM} once 
\eqref{devopomain} is 
substituted by the suitable 
asymptotic expansions which correspond to 
the angular sector
of the ``bump". 
The terminants of the 
respective asymptotic expansions
of $|z_{j,n}|$
also change: this change imposes to 
reconsider
\eqref{rouchegauche} 
and
\eqref{rouchedroite}
up to $\lo n/n$ - terms, and not
up to $1/n$ - terms, as in the proof
of Theorem \ref{cercleoptiMM}. 
It is remarkable that 
the  inequality \eqref{amaximalfunctionX} 
remains the same, with the same upper bound function
$\kappa(X,a)$.
Then the equation of the 
curve of the Rouch\'e condition
$a \to \kappa(1,a)$, on $[1, +\infty)$,
is the same as in Theorem 
\ref{cercleoptiMM} for controlling
the conditions of Rouch\'e.
The optimal value 
$a_{\max}$ of $a$ also remains the same, and
\eqref{rouchecercle} also holds true
for those $z_{j,n}$ in the bump sector.
\end{proof}

From the inequalities 
\eqref{petitsecteur18} in
Theorem \ref{cercleoptiMM}, also used
in the proof of
Proposition \ref{cercleoptiMMBUMP},
we now obtain a finer localization
of  a subcollection of the
roots 
$\omega_{j,n}$
of the Parry Upper function
$f_{\beta}(z)$, and
a definition of the lenticulus 
$\lc_{\beta}$ of $\beta$, as follows.

\begin{theorem}
\label{omegajnexistence}
Let $n \geq n_1 = 195$.
Let 
$\beta > 1$ be 
any real number
having $\dyg(\beta) = n$.
The Parry Upper function
$f_{\beta}(z)$ has 
an unique simple zero
$\omega_{j,n}$ in each disk
$D_{j,n} := \newline \{z \mid
|z - z_{j,n}| < 
\frac{\pi \, |z_{j,n}|}{ n \, a_{\max}}\}$,
$j = 1, 2, \ldots, J_n$,
which satisfies the additional inequality:
\begin{equation}
\label{rootslesvoila}
|\omega_{j,n} - z_{j,n}| < \, 
\frac{\pi |z_{j,n}|}{ n \, a_{j,n}} 
\qquad \quad
\mbox{for}~~ j = \lceil v_n \rceil,
\lceil v_n \rceil+1,  \ldots, J_n ,
\end{equation}
where $a_{J_n , n} = a_{\max}$ and, 
for 
$j = \lceil v_n \rceil, \ldots, J_n - 1$,
the value $a_{j,n}$~, $> a_{\max}$, is
defined by
\begin{equation}
\label{petitcercle_jmainDBJN}
D\bigl(\frac{\pi}{a_{j,n}}
\bigr)
~=~
\lo \Bigl[
\frac{1 + B_{j,n}
-
\sqrt{1 - 6 B_{j,n} + B_{j,n}^{2} }}{4 B_{j,n}}
\Bigr]
\end{equation}
$$\mbox{with}\qquad\qquad
B_{j,n}:= 
2 \sin(\frac{\pi j}{n})
\Bigl(
1 - \frac{1}{n} \lo (2 \sin(\frac{\pi j}{n}))
\Bigr) ,
$$
and, putting 
$D:= D\bigl(\frac{\pi}{a_{j,n}}
\bigr)$ for short,
\begin{equation}
\label{petitcercle_jmainTL}
{\rm tl}(\frac{\pi}{a_{j,n}}
\bigr)
~=~
\frac{2}{n} \times
B_{j,n}^{-1} \,
(\frac{-3 + \exp(-D) + 2 \exp(D)}
{4 - \exp(-D) - 2 \exp(D)})
\times
\left(
\frac{\lo \lo n}{\lo n}
\right)^2 .
\end{equation}
An upper bound of the tails, independent of $j$,
is given by
\begin{equation}
\label{petitcercle_jmainTL_majorant}
O \Bigl(
\frac{(\lo \lo n)^2}{(\lo n)^3}
\Bigr)
\end{equation}
with the constant $\frac{1}{7 \pi}$ 
in the Big O.
The lenticulus
$\lc_{\beta}$ associated with $\beta$
is constituted by 
the following subset of roots of
$f_{\beta}(z)$:
\begin{equation}
\label{lenticulusBETA}
\lc_{\beta}
:=
\{1/\beta\}
~~\cup~~
\bigcup_{j=1}^{J_n}
\left(
\{\omega_{j,n} \}
\cup 
\{\overline{\omega_{j,n}} \}\right).
\end{equation}
\end{theorem}

\begin{proof}
The existence 
of the zeroes comes from
Proposition \ref{cercleoptiMMBUMP}
and Theorem \ref{cercleoptiMM}, with 
the maximal value $J_n$
of the
index $j$
given by Proposition \ref{Jndefinition}.
To refine the localization
of $\omega_{j,n}$
in the neighbourhood of
$z_{j,n}$, in the main angular sector, i.e.
for $j \in 
\{\lceil v_n \rceil,
\lceil v_n \rceil+1,  \ldots, J_n\}$,
the conditions of Rouch\'e 
\eqref{rouchecercle}
are now used to define the new radii.

The value $a_{j,n}$ 
is defined by the development term
${\rm D}(\frac{\pi}{a_{j,n}})$, 
itself defined
as follows:
\begin{equation}
\label{petitsecteur_aJI_D}
{\rm D}\left(
\frac{|-1+z_{j,n}|}{|z_{j,n}|}
\right) 
~~=:~~ \frac{
1 -
\exp\bigl(
-{\rm D}(\frac{\pi}{a_{j,n}})
\bigr)}
{2 \exp\bigl(
{\rm D}(\frac{\pi}{a_{j,n}})
\bigr) -1}
\end{equation}
and the tail
${\rm tl}(\frac{\pi}{a_{j,n}})$ 
calculated from 
${\rm tl} \left(
\frac{|-1+z_{j,n}|}{|z_{j,n}|}
\right)$ so that
the Rouch\'e condition
\begin{equation}
\label{petitsecteur_aJItail}
\frac{|-1+z_{j,n}|}{|z_{j,n}|}
=
{\rm D}\left(
\frac{|-1+z_{j,n}|}{|z_{j,n}|}
\right)
+ {\rm tl} \left(
\frac{|-1+z_{j,n}|}{|z_{j,n}|}
\right)
~~<~~ \frac{
1 -
\exp\bigl(
-\frac{\pi}{a_{j,n}}
\bigr)}
{2 \exp\bigl(
\frac{\pi}{a_{j,n}}
\bigr) -1}
\end{equation}
holds true.
From
Proposition \ref{zedeJIMoinsUNzedeJI},
denote
$$B_{j,n}:=
{\rm D}\left(
\frac{|-1+z_{j,n}|}{|z_{j,n}|}
\right)= 
2 \sin(\frac{\pi j}{n})
\Bigl(
1 - \frac{1}{n} \lo (2 \sin(\frac{\pi j}{n}))
\Bigr).$$
Let
$W := \exp({\rm D}(\frac{\pi}{a_{j,n}}))$.
The identity \eqref{petitsecteur_aJI_D}
transforms into the equation of degree 2:
\begin{equation}
\label{bjn_complet}
2 B_{j,n} \, W^2
- 
\Bigl(
B_{j,n}
+ 1
\Bigr) W + 1 ~=~0
\end{equation}
from which \eqref{petitcercle_jmainD} 
is deduced. For the calculation of
${\rm tl}(\frac{\pi}{a_{j,n}})$,
denote $D := {\rm D}(\frac{\pi}{a_{j,n}})$
and $tl_{j,n} := 
{\rm tl}(\frac{\pi}{a_{j,n}})$.
Then, at the first order,
$ \frac{1 -
\exp\bigl(
\frac{-\pi}{a_{j,n}}
\bigr)}
{2 \exp\bigl(
\frac{\pi}{a_{j,n}}
\bigr) -1}$
$$ =
\frac{1 -
\exp\bigl(
-D - tl_{j,n}
\bigr)}
{2 \exp\bigl(
D + tl_{j,n}
\bigr) -1}
= B_{j,n}
[1+
 tl_{j,n} \times
(\frac{4 - \exp(-D) - 2 \exp(D)}
{-3 + \exp(-D) + 2 \exp(D)})].
$$
From \eqref{petitsecteur_aJItail} and
\eqref{UNmoinszedeJIzedeJI}
the following
inequality should be satisfied, 
with the constant 2 in the Big O,
$$
\frac{1}{n} O\Bigl(
\bigl(
\frac{\lo \lo n}{\lo n}
\bigr)^2
\Bigr)
=
{\rm tl} \Bigl(
\frac{|-1+z_{j,n}|}{|z_{j,n}|}
\Bigr)
~<~
tl_{j,n}
\times
B_{j,n} \,
(\frac{4 - \exp(-D) - 2 \exp(D)}
{-3 + \exp(-D) + 2 \exp(D)})] .
$$
The expression of $tl_{j,n}$
in \eqref{petitcercle_jmainTL}
follows, to obtain a strict 
inequality in \eqref{petitsecteur_aJItail}.
By \eqref{petitcercle_jmainDBJN} 
the quantity $\exp(D)$ is a function 
of $B_{j,n}$, which tends to 
$\frac{3}{4}$ when 
$B_{j,n}$ tends to $0$;
hence, at the first order, 
a lower bound of 
the function $B_{j,n} \to
| B_{j,n} \,
(\frac{4 - \exp(-D) - 2 \exp(D)}
{-3 + \exp(-D) + 2 \exp(D)})|$
is obtained for
$j = \lceil v_n \rceil$, and given by
$2 \pi \frac{\lo n}{n} \times 7$.
Then it suffices to take
$$tl_{j,n} =
cste \Bigl(
\frac{(\lo \lo n)^2}{(\lo n)^3}
\Bigr)
$$
with $cste = 1/(7 \pi)$, to obtain
a tail independent of $j$,
and therefore the
conditions of Rouch\'e 
\eqref{petitsecteur_aJItail}
 satisfied
with these new smaller radii and tails
in the main 
angular sector.
\end{proof}

\begin{remark}
\label{petitcercleapproxinfini}
For $n$ very large, up to second-order terms, \eqref{bjn_complet}
reduces to
$$4 \sin(\frac{\pi j}{n}) \, W^2
- 
\Bigl(
2 \sin(\frac{\pi j}{n})
+ 1
\Bigr) W + 1 ~=~0
$$
and \eqref{petitcercle_jmainDBJN}
to
\begin{equation}
\label{petitcercle_jmainD}
D\bigl(\frac{\pi}{a_{j,n}}
\bigr)
~=~
\lo \Bigl[\frac{1 + 2 \sin(\frac{\pi j}{n})
-
\sqrt{1 - 12 \sin(\frac{\pi j}{n}) 
+ 4 (\sin(\frac{\pi j}{n}))^2}}{8 \sin(\frac{\pi j}{n})}
\Bigr].
\end{equation}
\end{remark}

\begin{lemma}
\label{cnroucheasymptotic}
Let $n \geq 195$
and
$c_n$ defined by
$|z_{J_n , n}| = 1 - \frac{c_n}{n}$.
Let us put
$\kappa:=\kappa(1,a_{\max})$ for short.
Then
\begin{equation}
\label{cncnlimit}
c_n = - (\lo \kappa) \, (1 + \frac{1}{n})
+
\frac{1}{n} 
O\bigl(
\bigl(\frac{\lo \lo n}{\lo n}\bigr)^2
\bigr) ,
\end{equation} 
with $c = \lim_{n \to +\infty} c_n = 
- \lo \kappa = 1.76274\ldots$,
and, up to  $O(\frac{1}{n}\bigl( 
\bigl(
\frac{\lo \lo n}{\lo n}
\bigr)^2\bigr))$-terms,
\begin{equation}
\label{cnroucheminilimit}
\frac{(1 - \frac{c_n}{n})^{2 n}}
{(1 - \frac{c_n}{n}) - (1 - \frac{c_n}{n})^{n}}
=
\frac{e^{-2 c}}{1 - e^{- c}}
\Bigl(
1 +
\frac{c}{2 n (1 - e^{- c})}
\bigl[
2 - c e^{-c} - 2 c
\bigr]
\Bigr) 
\end{equation}
with $e^{-2 c}/(1 - e^{- c}) = 0.0355344\ldots$
\end{lemma}

\begin{proof}
The asymptotic expansion 
\eqref{cncnlimit}
of $c_n$ 
is deduced from the asymptotic expansions
of $\psi_n$ 
and $z_{J_n , n}$ given
by \eqref{psinasymptotic} and
\eqref{devopomainCONSTANTEc}
(Proposition 3.5
in \cite{vergergaugry6}). 
We deduce
the limit 
$
c ~:=~ - \lo (\kappa(1,a_{\max})) =
1.76274\ldots
$
and then \eqref{cnroucheminilimit}
follows.
\end{proof}

\begin{definition}
\label{definitionHn}
Let $n \geq n_2 := 260$.
We denote by $H_n$ the 
largest integer $j \geq \lceil v_n \rceil$
such that
\begin{equation}
\label{HnJnmini}
\arg(z_{J_n , n}) - \arg(z_{j, n}) 
\geq 
\frac{(1 - \frac{c_n}{n})^{2 n}}
{(1 - \frac{c_n}{n}) - (1 - \frac{c_n}{n})^{n}}.
\end{equation}
\end{definition}

\begin{proposition}
\label{hnasymptotic}
Let $n \geq 260$. Let denote
$\kappa := \kappa(1,a_{\max})$
for short.
Then
$$
\arg(z_{H_n , n})
~=~
2 \arcsin\bigl(\frac{\kappa}{2}\bigr) ~-~ 
\frac{\kappa^2}{1 - \kappa} \hspace{6.5cm}\mbox{}
$$
\begin{equation}
\label{hhnasymptotic}
+ \frac{\lo \kappa}{n}
\left[
\frac{\kappa}
{\sqrt{4 - \kappa^2}}
+
\frac{2 + \kappa \, \lo(\kappa)+ 2 \, \lo(\kappa)}
{2 ( 1 - \kappa )}
\right]
+ \frac{1}{n} O\bigl(
\Bigl(
\frac{\lo \lo n}{\lo n}
\Bigr)^2
\bigr),
\end{equation}
with, at the limit,
$$\lim_{n \to +\infty}
\arg(z_{H_n , n})
=
2 \arcsin\bigl(\frac{\kappa}{2}\bigr) ~-~ 
\frac{\kappa^2}{1 - \kappa} =
0.13625 .
$$
\end{proposition}

\begin{proof}
The asymptotic expansion 
of the right-hand side
term of \eqref{HnJnmini} is
\begin{equation}
\label{cielrouche}
\frac{(1 - \frac{c_n}{n})^{2 n}}
{(1 - \frac{c_n}{n}) - (1 - \frac{c_n}{n})^{n}}
=
\frac{e^{-2 c}}{1 - e^{- c}} 
\Bigl(
1 +
\frac{c (2 - c e^{-c} - 2 c)}{2 n (1 - e^{- c})}
\Bigr) + \ldots
\end{equation}
Then the
asymptotic expansion
of $\arg(z_{H_n , n})$ comes from
\eqref{HnJnmini}
in which the inequality is 
replaced by an equality,
and from the
asymptotic expansion \eqref{argzJJJn} of
$\arg(z_{J_n , n})$
(Proposition \ref{argumentlastrootJn}).
\end{proof}

For $n$ large enough, $\arg(z_{H_n , n})$
is equal to $2 \pi \frac{H_n}{n}$, up 
to higher order - terms, and a definition of
$H_n$ in terms of asymptotic expansions could be:
\begin{equation}
\label{hhhnasymptotic}
H_n\! =\! \lfloor
\frac{n}{2 \pi}
\bigl(
 2 \arcsin\bigl(\frac{\kappa}{2}\bigr) - 
\frac{\kappa^2}{1 - \kappa}\bigr)
\!-\!
\lo(\kappa)
\Bigl[
\frac{\kappa}
{\sqrt{4 - \kappa^2}}
+
\frac{2 + \kappa \, \lo(\kappa)+ 2 \, \lo(\kappa)}
{2 ( 1 - \kappa )}
\Bigr]
\rfloor,
\end{equation}
For simplicity's sake, we
will take the following 
definition of $H_n$
\begin{equation}
\label{hhhndefinition}
H_n := \lfloor
\frac{n}{2 \pi}
\bigl(
 2 \arcsin\bigl(\frac{\kappa}{2}\bigr) ~-~ 
\frac{\kappa^2}{1 - \kappa}\bigr) - 1
\rfloor .
\end{equation}
\begin{remark}
\label{value260}
The value $n_2 = 260$ is calculated
by the inequality
$ \frac{2 \pi v_n}{n} < \arg(z_{H_n , n})$
which should be valid 
for all $n \geq 260$, 
where $H_n$ is given by
\eqref{hhhndefinition},
$\arg(z_{H_n , n})$ by
\eqref{hhnasymptotic},
where $(v_n)$ is the
delimiting sequence (cf Appendix) of the
transition region of the boundary of the bump sector.
A first minimal value of $n$ 
is first estimated by
$2 \pi \frac{\lo n}{n} < D(\arg(z_{H_n , n}))$
using \eqref{hhnasymptotic}.
Then it is corrected so that
the numerical value
of the tail of the asymptotic expansion
in \eqref{hhnasymptotic} be taken into account
in this inequality.
\end{remark}

\begin{theorem}
\label{absencezeroesOutside}
Let $n \geq n_2 := 260$. 
Denote by
$\dc_{n}$ the subdomain of the open 
unit disk, symmetrical with respect to the real axis, 
defined by the conditions:

\begin{equation}
\label{cerclescalottes_0}
 |z| < 1 - \frac{c_n}{n}, 
\qquad \frac{1}{n}
\left(\frac{\lo \lo n}{\lo n}\right)^2 < 
|z - \theta_n|,
\end{equation}
\begin{equation}
\label{cerclescalottes_1}
\frac{\pi |z_{j,n}|}{n \, a_{\max}} < |z - z_{j,n}|,
\quad 
\qquad 
\mbox{for}~~ j= 1, 2, \ldots, J_n ,
\end{equation}
and,
for $\displaystyle j = 
J_n + 1,
\ldots, 2 J_n - H_n + 1$, 

\begin{equation}
\label{cerclescalottes}
\displaystyle \frac{\pi |z_{j,n}|}
{n \, s_{j,n}} < |z - z_{j,n}| ,
\qquad \mbox{with}~~
s_{j,n} = a_{\max}
\Bigl[
1 +
\frac{a_{\max}^{2}
(j - J_n)^2}{\pi^2 \,
J_{n}^{2}}
\Bigr]^{-1/2} . 
\end{equation}

Then, for any real number 
$\beta > 1$ 
having $\dyg(\beta) = n$,  
the Parry Upper function
$f_{\beta}(z)$ does not vanish at
any point $z$ in $\dc_{n}$.
\end{theorem}

\begin{proof}
Assume $\beta > 1$ such that
$\theta_{n-1} < \beta^{-1} < \theta_n$.
We will apply the general form of the 
Theorem of Rouch\'e to the compact
$\kc_n$ which is 
the adherence of the domain $\dc_{n}$,
i.e. we will show that
the inequality (and symmetrically with 
respect to the real axis)
\begin{equation}
\label{roucheKn}
|f_{\beta}(z) - G_{n}(z)| < |G_{n}(z)|,
\quad \quad z~ \in~ \partial \kc_{n}^{ext} \,  
\cup \, C_{1,n} \, 
\cup \, C_{2,n} \cup
\ldots \cup \, C_{J_n , n} 
\end{equation}
holds, with $z \in {\rm Im}(z) \geq 0$,
where $\partial \kc_n$
is the union of: (i) the arcs of the
circles defined by the equalities in
\eqref{cerclescalottes_1}
and
\eqref{cerclescalottes},
arcs which lie in $|z| \leq 1 - c_n/n$,  
and circles
for which the intersection
with $|z| = 1 - c_n/n$
is not empty,
(ii) the arcs of 
$C(0,1-c_n/n)$ which have empty intersections
with the interiors
of the disks defined by
the inequalities ``$>$", instead of
``$<$", in
\eqref{cerclescalottes_1}
and
\eqref{cerclescalottes},
which join two successive circles.
The two functions
$f_{\beta}(z)$
and $G_{n}(z)$ are continuous
on the compact
$\kc_n$, holomorphic
in its interior $\dc_n$, and
$G_n$ has no zero 
in $\kc_n$.
As a consequence
the function $f_{\beta}(z)$ will have no zero in
the interior $\dc_n$ of $\kc_n$.

Instead of using $f_{\beta}(z)$ itself
in \eqref{roucheKn},
we will show
that the following inequality
holds true 
\begin{equation}
\label{rouchecercleNONZERO}
\frac{
\left|z\right|^{2 n -1}}{1 - |z|^{n-1}}
~<~
\left|-1 + z + z^n \right| , 
\quad ~\mbox{for all}~ 
z 
\in \partial \kc_{n}^{ext}
\end{equation}
what will imply the claim.

The Rouch\'e inequalities
\eqref{roucheKn} 
\eqref{rouchecercleNONZERO} 
hold true
on the (complete) 
circles $C_{j,n}, 1 \leq j \leq J_n$
by Theorem \ref{cercleoptiMM}
and Proposition \ref{cercleoptiMMBUMP};
these conditions 
become out of reach
for $j$ taking higher values (i.e.
in $\{J_n + 1, \ldots,
\lfloor n/6\rfloor\}$), but we will show that
they remain true
on the arcs defined by
the equalities
in \eqref{cerclescalottes}.

The domain $\dc_n$ only depends 
upon the dynamical degree $n$ of
$\beta$, not of $\beta$ itself.

Let us prove that 
the external Rouch\'e circle
$|z| = 1 - c_n/n$ intersects
all the circles 
$C_{J_n - k,n}, k = 0 , 1, \ldots,
k_{\max}$,
with
$k_{\max} := \lfloor 
J_{n} (\frac{\pi}{a_{\max}})\rfloor$.
Indeed, up to 
$\frac{1}{n} O 
\bigl(\bigl(\frac{\lo \lo n}{\lo n}\bigr)^2 \bigr)$-
terms, 
from Proposition
\ref{argumentlastrootJn},
$$\lo (2 \sin(\pi \frac{J_n}{n}))
= \lo (2 \sin(\pi \frac{(J_n - k) + k}{n}))
= \lo \bigl(
2 \pi \frac{J_n - k}{n}(1 + \frac{k}{ J_n - k})
\bigr)$$
\begin{equation}
\label{intero}
=\lo \bigl(2 \sin(\pi \frac{J_n - k}{n})\bigr) 
+ \frac{k}{J_n} . 
\end{equation}
Since $|z_{J_n , n}|=
1 - c_n/n = 
1 +
\frac{1}{n}
\lo (2 \sin(\pi \frac{J_n}{n}))
+
\frac{1}{n} O 
\bigl(\bigl(\frac{\lo \lo n}{\lo n}\bigr)^2 \bigr)$, 
we deduce from
\eqref{intero}, 
with 
$k \leq  k_{\max}$,
that the point 
$z \in C(0,1 - c_n/n )$ 
for which
$\arg(z) =  \arg(z_{J_n - k,n})$
is such that
$$|z_{J_n - k ,n} - z | 
~=~
\frac{k}{n \, J_n}
~\leq~ 
 \, \frac{\lfloor J_n (\frac{\pi}{a_{\max}})
\rfloor}{n \, J_n}
~\leq~ \frac{\pi}{n \, a_{\max}}
$$
up to 
$\frac{1}{n} O 
\bigl(\bigl(\frac{\lo \lo n}{\lo n}\bigr)^2 \bigr)$-
terms.  
As soon as $n$ is large enough, 
we deduce
that $z$ lies in the interior of
$D_{J_n - k,n}$. 
Since the function
$x \to \lo (2 \sin(\pi x))$ is negative and
strictly increasing
on $(0, 1/6)$, the sequence
$(|z_{j,n}|)_{j=H_n , \ldots, J_n}$
is strictly increasing,
by \eqref{devopomainCONSTANTEc}. 
Hence we deduce that the
circle $|z| = 1 - c_n/n$
intersects 
all the circles 
$C_{j,n}$
for $j=J_n - k_{\max}, \ldots, J_n$.

The same arguments show that
the external Rouch\'e circle
$|z|= 1 - c_n/n$ intersects all the circles
$C(z_{j,n}, \frac{\pi |z_{j,n}|}
{n \, s_{j,n}})$
for $j = J_n + 1, J_n + 2, \ldots, 2 J_n - H_n + 1$.

The quantities 
$s_{j,n}$, for 
$j = J_n + 1, \ldots , 2 J_n - H_n +1$,
are easily calculated (left to the reader)
so that
the distance (length of the $j$-th circle segment)
$$\left|
~\frac{z_{j,n}}{|z_{j,n}|} (1 - \frac{c_n}{n}) - y_j
~\right|
=
\left|
~\frac{z_{j,n}}{|z_{j,n}|} (1 - \frac{c_n}{n}) - y'_j 
~\right|$$
for $y_j, y'_j \in 
C(z_{j,n}, \frac{\pi |z_{j,n}|}{n \, s_{j,n}})
\cap C(0, 1 - \frac{c_n}{n}), y_j \neq y'_j$,
be independent of 
$j$ in the interval 
$\{J_n +1 , \ldots, 2 J_n - H_n + 1\}$
and equal to
\begin{equation}
\label{chord}
\frac{\pi |z_{J_n , n}|}{n \, a_{\max}}.
\end{equation}
Then the two sequences
of moduli of centers
$(|z_{j,n}|)_{j=J_n + 1 , \ldots, 2 J_n - H_n +1}$
and of radii\newline
$(\frac{\pi |z_{j,n}|}
{n \, s_{j,n}})_{j=J_n + 1 , \ldots, 2 J_n - H_n +1}$
are both increasing, with the fact that
the corresponding disks
$D(z_{j,n}, \frac{\pi |z_{j,n}|}
{n \, s_{j,n}})$
keep constant the 
intersection chord  
$\arg(y_j) - \arg(y'_j)
= 
\frac{\pi |z_{J_n , n}|}{n \, a_{\max}}$
with the external Rouch\'e 
circle $|z| = 1 - c_n/n$.

\begin{figure}
\begin{center}
\includegraphics[width=8cm]{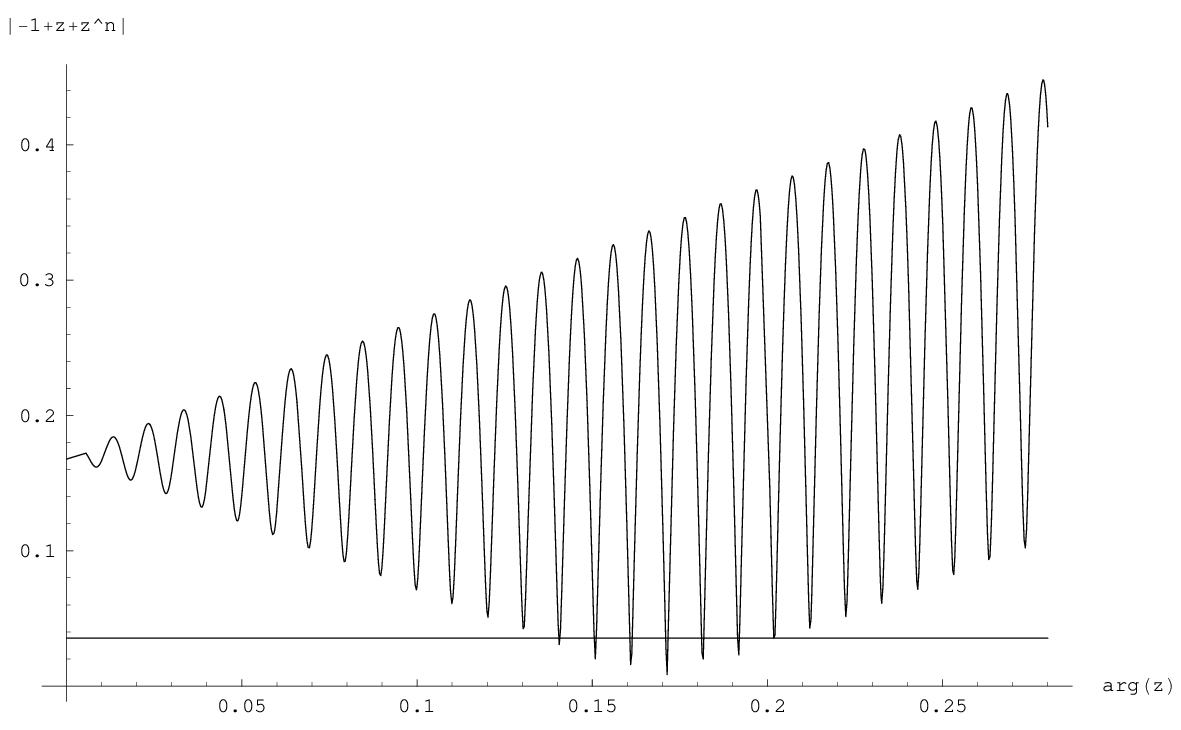}
\end{center}
\caption{
Oscillations of the upper bound
$|-1 + z + z^n|$ of
the Rouch\'e inequality \eqref{rouchecercle},
for $z$ running over the curve
$|z|=1 - c_n / n$ (here represented with $n = 615$)
as a function of $\arg(z)$ in 
$[0, 0.28]$.
The minima correspond to the angular positions
of the zeroes $z_{j,n}$
of the trinomial $-1 + X + X^n$,
for $j = 1, 2, \ldots, H_n,
\ldots, J_n , \ldots, 2 J_n - H_n + 1, \ldots$
($J_{615} = 17 , H_{615} = 12$).
The angular separation between two successive minima
is $\approx 2 \pi/n$. The difference between
two successive minima is
$\approx 2 \pi/n$. For $n=615$,
the arguments
$2 \pi (\lo n)/n$ (limiting the bump sector),
$\arg(z_{H_n , n})$ and
$\arg(z_{J_n , n})$ are respectively equal to
$0.0656\ldots, 0.12189\ldots, 0.17129\ldots$.
The horizontal line at the $y$-coordinate 
$0.0354...$ is the value of the 
left-hand side term of
the Rouch\'e inequality \eqref{rouchecercle}
(Proposition \ref{roucheKn}); 
it is always strictly smaller than
the minimal
value of
the oscillating function $|-1 + z + z^n|$
on the external boundary
$\partial \kc_{n}^{ext}$, 
whose geometry surrounds the roots
$z_{j,n}$ for $j$ 
between $H_n + 1$ and $2 J_n - H_n + 1$.
}
\label{oscill}
\end{figure}

Let $z \in C(0, 1 - \frac{c_n}{n})$,
$\varphi := \arg(z) \in [0, \pi]$. 
Denote by
$Z(\varphi) := 
| G_{n}((1 - \frac{c_n}{n}) e^{i \varphi})|^2
= \bigl|-1  + (1 - \frac{c_n}{n}) e^{i \varphi}
+ (1 -\frac{c_n}{n})^n e^{i \, n \, \varphi}\bigr|^2$.
The expansion of the function
$Z(\varphi)$ as a function
of $\varphi$, up to
$O(1/n)$- terms, is the following:
$Z(\varphi) =$
$$(-1 + (1 - \frac{c_n}{n}) \cos(\varphi)
\!+\!
(1 - \frac{c_n}{n})^n \cos(n \varphi))^2
\!+\!
((1 - \frac{c_n}{n}) \sin(\varphi)
)\!+\!
(1 - \frac{c_n}{n})^n \sin(n \varphi))^2
$$
$$
=2 + e^{- 2 c} - 2 \cos(\varphi)
- 2 e^{-c} \cos(n \, \varphi)
+ 2 e^{- c} \cos(\varphi) \cos(n \, \varphi)
+ 2 e^{- c} \sin(\varphi) \sin(n \, \varphi)
$$
$$
= 2 + e^{- 2 c} - 2 \cos(\varphi)
- 4 e^{-c} \sin(\frac{\varphi}{2})
\left(\cos(n \, \varphi) \sin(\frac{\varphi}{2})
- \sin(n \, \varphi) \cos(\frac{\varphi}{2})
\right)
$$
\begin{equation}
\label{zedefi}
= 2 + e^{- 2 c} - 2 \cos(\varphi)
+ 4 e^{-c} \sin(\frac{\varphi}{2})
\sin(n \varphi - \frac{\varphi}{2}).
\end{equation}
The function $Z(\varphi)$,
defined on $[0, \pi/3]$,
is almost-periodic (in the sense of Besicovitch
and Bohr), 
takes the value $0$  at
$\varphi = \arg(z_{J_n ,n})$, and
therefore, up to 
$O(1/n)$-terms, has its minima 
at the
successive arguments 
$\arg(z_{J_n ,n}) + \frac{2 k \pi}{n}$
for
$|k|= 0, 1, 2, \ldots, J_n - H_n + 1, \ldots$
(Figure \ref{oscill}).
For such integers $k$,
from \eqref{zedefi}, we deduce
the successive minima
\begin{equation}
\label{successiveminima}
|-1 + z_{J_n ,n} e^{- 2 i k \pi/n}
+ (z_{J_n ,n} e^{- 2 i k \pi/n})^n|
=
|G_{n}(z_{J_n ,n})| 
+ \frac{2 |k| \pi}{n} = \frac{2 |k| \pi}{n}
\end{equation}
up to 
$\frac{1}{n} O \bigl(\bigl(\frac{\lo \lo n}{\lo n}
\bigr)^2\bigr)$- terms, with
$\arg(z_{J_n , n} e^{- 2 i k \pi/n})
=
\arg(z_{J_n - k,n})$
up to
$O(1/n)$-terms.

With the above notations, 
denote by $y_j , y'_j$ the two points
of $C(0, 1 - \frac{c_n}{n})$
which belong to
$C_{j,n}$ for $2 H_n - J_n \leq j \leq J_n$,
to
$C(z_{j,n}, \frac{\pi |z_{j,n}|}{n \, s_{j,n}})$
for $J_n + 1 \leq j \leq 2 J_n - H_n + 1$.
Writing them by increasing argument, we
have:
\begin{equation}
\label{listeyjyprimej}
y_{2 H_n - J_n}, y'_{2 H_n - J_n}, \ldots, 
y_{H_{n}}, y'_{H_{n}}, \ldots,
y_{J_{n}}, y'_{J_{n}},~
y_{J_{n} + 1}, y'_{J_{n} + 1},
\ldots,
y_{2 J_{n} - H_n + 1},~ y'_{2 J_{n} - H_n + 1}.
\end{equation}
The Rouch\'e inequality
\eqref{roucheKn} is obviously satisfied
at each $y_j$ and $y'_j$ for $j= 2 H_n - J_n, \ldots, J_n$.
Let us show that 
this inequality holds
at each point
$y_j$ and $y'_j$ for $j= J_n + 1, \ldots, 
2 J_n - H_n +1$.
Indeed, for such a point, say
$y_j$, 
there exists
$$\xi_j
=
w_j \, z_{J_n , n} e^{2 i (j - J_n) \pi/n} 
+ (1-w_j) \, y_j
, \quad \mbox{for some}~ w_j \in [0,1],$$  
lying in the segment
$\left[z_{J_n , n} e^{2 i (j - J_n) \pi/n}, y_j
\right]$
such that
$$G_{n}(y_j) = G_{n}(z_{J_n , n} e^{2 i (j-J_n) \pi/n})
+ (y_j - z_{J_n , n} e^{2 i (j-J_n) \pi/n}) \, G'_{n}(\xi_j) 
$$
with, using \eqref{chord},
$$|G_{n}(y_j) - G_{n}(z_{J_n , n} e^{2 i (j - J_n) \pi/n})|
=
|y_j - z_{J_n , n} e^{2 i (j - J_n) \pi/n}| |G'_{n}(\xi_j)|
=
\frac{\pi |z_{J_n , n}|}{n \, a_{\max}}
|G'_{n}(\xi_j)|.$$
The derivative of $G_{n}(z)$
is $G'_{n}(z) = 1 + n z^{n-1}$.
Up to $O(1/n)$-terms, 
the line generated by the
segment $\left[z_{J_n , n} e^{2 i (j - J_n) \pi/n}, 
y_j \right]$
is tangent to the circle
$C(0, 1-c_n/n)$, and the
modulus $\frac{1}{n}|G'_{n}(\xi_j)|$ 
satisfies
$$\frac{1}{n}|G'_{n}(\xi_j)| 
\,=\,
\frac{1}{n}|G'_{n}(z_{J_n , n} e^{2 i (j - J_n) \pi/n})|
\,=\,
\frac{1}{n}|G'_{n}(z_{J_n , n})|
\,=\,
\lim_{n \to +\infty}
\frac{1}{n}|G'_{n}(z_{J_n , n})|
\,=\,
e^{-c} .$$ 
From
$|G_{n}(y_j)| \geq
\bigl||G_{n}(y_j) - G_{n}(z_{J_n , n} e^{2 i (j - J_n) \pi/n})|
- |G_{n}(z_{J_n , n} e^{2 i (j - J_n) \pi/n})|\bigr|
$ 
and \eqref{chord}
we deduce
\begin{equation}
\label{gagao_0}
|G_{n}(y_j)| \geq
\,  \frac{\pi |z_{J_n , n}|}{a_{\max}}
 e^{-c} - \frac{2 \pi |j - J_n| }{n} .
\end{equation}
But, by definition of $H_n$, still
up to $O(1/n)$-terms,
for $|j - J_n| \leq J_n - H_n - 1$,
\begin{equation}
\label{gagao_1}
\frac{2 \pi |j - J_n|}{n} \leq
\frac{2 \pi \, (J_n - H_n - 1)}{n} =
\arg(z_{J_n , n}) -\arg(z_{H_n +1, n}) 
\leq \frac{e^{-2 c}}{1 - e^{-c}} .
\end{equation}
This inequality is in particular
satisfied for the last two
values of $| j - J_n |$ which are
$J_n - H_n$ and $J_n - H_n +1$
up to $O(1/n)$-terms.

Since the inequality
\begin{equation}
\label{gagao_2}
0.0710\ldots = 
~~2 \frac{e^{-2 c}}{1 - e^{-c}} 
~<~
\frac{\pi |z_{J_n , n}|}{a_{\max}} e^{-c}
~~=~ 0.0914\ldots
\end{equation}
holds, 
from \eqref{gagao_0}, \eqref{gagao_1}
and \eqref{gagao_2},
as soon as $n$ is large enough,
we deduce the Rouch\'e inequality
$$|G_{n}(y_j)| ~\geq~
\, 
\frac{\pi |z_{J_n , n}|}{a_{\max}} 
- \frac{e^{-2 c}}{1 - e^{-c}} 
~\geq~
\frac{e^{-2 c}}{1 - e^{-c}} .$$
Therefore
the conditions of Rouch\'e 
\eqref{rouchecercleNONZERO} 
hold at all the points $y_j$ and $y'_j$ of
\eqref{listeyjyprimej}.

Let us prove that the conditions of 
Rouch\'e \eqref{rouchecercleNONZERO} 
hold on each arc $y'_j~y_{j+1}$ of
the circle $|z| = 1 - c_n/n$,
for $j = 2 H_n - J_n, 2 H_n - J_n + 1 , 
\ldots, 2 J_n - H_n$.
Indeed, from \eqref{zedefi}, 
the derivative
$Z'(\varphi)$ takes a positive value at
the extremity $y'_j$ 
while it takes a negative value
at the other extremity $y_{j+1}$.
$Z(\varphi)$ is almost-periodic
of almost-period $2 \pi/n$. 
The function $\sqrt{Z(\varphi)}$ is increasing
on
$(\arg(z_{j , n}),
 \arg(z_{j , n})+  \frac{\pi}{n})$
and decreasing on 
$(\arg(z_{j , n}) +  \frac{\pi}{n},
\arg(z_{j , n}) + 2 \frac{\pi}{n})$;
on the arc $y'_j~y_{j+1}$ it takes
the value $|G_{n}(y'_j)| 
\geq \frac{e^{-2 c}}{1 - e^{-c}}$,
admits a maximum, and decreases
to 
$|G_{n}(y_{j+1})|
\geq \frac{e^{-2 c}}{1 - e^{-c}}$. 
Hence, \eqref{rouchecercle} holds true for
all $z \in C(0, 1 - c_n/n)$ with
$\arg(y'_j) \leq \arg(z) \leq \arg(y_{j+1})$.

Let us now prove that
the condition of Rouch\'e 
\eqref{rouchecercle} is
satisfied in the angular sector
$0 \leq \arg(z) \leq \arg(z_{H_n , n})$.
Indeed, in this angular sector,
the successive minima of
$\sqrt{Z(\varphi)}$ are all
above 
$\frac{e^{-2 c}}{1 - e^{-c}}$
by the definition of $H_n$ and
\eqref{successiveminima}. Hence the claim.

Let us prove that
the condition of Rouch\'e 
\eqref{rouchecercle} is
satisfied in the angular sector
\newline$\arg(z_{2 J_n - H_n + 1, n}) \leq \arg(z) 
\leq \frac{\pi}{2}$.
In this angular sector,
the oscillations of $\sqrt{Z(\varphi)}$
still occur by the form
of \eqref{zedefi} and
the successive minima of
$\sqrt{Z(\varphi)}$ are all
above 
$\frac{e^{-2 c}}{1 - e^{-c}}$
for $\frac{2 J_n - H_n + 2}{J_n} \leq \arg(z)
\leq \pi/2$, 
by 
\eqref{successiveminima}
for $k \geq J_n - H_n + 1$. 
We deduce the claim.

The condition of Rouch\'e 
\eqref{rouchecercle} is also
satisfied in the angular sector
$\pi \leq \arg(z) \leq \pi/2$,
since then $\cos(\varphi) \leq 0$
and therefore
$\sqrt{Z(\varphi)} \geq 
\sqrt{2 + e^{-2 c} - 4 e^{- c}} =1.15\ldots$.
Since this lower bound is
greater than the
value
$\frac{e^{-2 c}}{1 - e^{-c}} = 0.0354\ldots$
we deduce the claim.

Let us show that 
the conditions of Rouch\'e 
\eqref{rouchecercle} are also
satisfied on the arcs
\newline$C(z_{j,n}, \frac{\pi |z_{j,n}|}{n \, s_{j,n}})
\cap \overline{D}(0, 1 - \frac{c_n}{n})$
for $j= J_n + 1 , \ldots, 2 J_n - H_n +1$.
For such an integer $j$, 
let us denote such an arc
by $y_j~y'_j$.
The two extremities $y_j$ and $y'_j$
of  the arc $y_j~y'_j$ of
the circle
$C(z_{j,n}, \frac{\pi |z_{j,n}|}{n \, s_{j,n}})$
define the same value of the difference cosine,
say 
$X_j := \newline\cos(\arg(y_j - z_{j,n}) - \arg(z_{j,n}))
=
\cos(\arg(y'_j - z_{j,n}) - \arg(z_{j,n}))$,
by \eqref{chord}.
The conditions of Rouch\'e are 
already satisfied at the points
$y_j$ and $y'_j$ by the above.
Recall that, 
for any fixed 
$a \geq 1$, the function
$\kappa(X,a)$,
defined in \eqref{amaximalfunctionX},
is such that
the partial derivative $\partial \kappa_X$
of $\kappa(X,a)$ 
is strictly negative
on the interior of 
$[-1, +1] \times [1, +\infty)$. 
In particular
the function $\kappa(X,s_{j,n})$
is decreasing.
For any point $\Omega$ of the arc
$y_j~y'_j$, we denote by
$X =\cos(\arg(\Omega - z_{j,n}) - \arg(z_{j,n}))$.
We deduce, up to $O(1/n)$-terms,
$$
\frac{e^{-2 c}}{1 - e^{-c}} 
~\leq~ \kappa(X_j,s_{j,n}) ~\leq~ \kappa(X,s_{j,n}),
\qquad \mbox{for all}~~X \in [-1, X_j],$$
hence the result.
\end{proof}

\begin{remark}
\label{lenticulusforalgebraicintegers}
In the case where
$\beta \in (1, \theta_{6}^{-1})$ is an 
algebraic integer
such that
$\beta \not\in
\{\theta_{n}^{-1} \mid n \geq 6\}$,
the lenticulus $\mathcal{L}_{\beta}$ of 
Galois conjugates of $1/\beta$ 
in the angular sector
$\arg z \in \{-\frac{\pi}{3}, +\frac{\pi}{3}\}$ 
is obtained by truncation and 
a slight deformation
of $\lc_{\theta_{\dyg(\beta)}^{-1}}$.
The asymptotic 
expansion of the minorant of the Mahler measure
${\rm M}(\beta)$ will
be obtained from this lenticulus
as a function of the
dynamical degree $\dyg(\beta)$.
\end{remark}

\subsection{Identification of the lenticuli of roots as conjugates}
\label{S5.4}

\mbox{}\newline In this paragraph,
$\beta \in (1,(1+\sqrt{5})/2)$
is assumed to be
a reciprocal algebraic integer 
which is fixed,
with $\dyg(\beta) = n \geq 260$,
${\rm M}(\beta) \leq 1.176280\ldots$
(we continue to be under the asssumptions of 
the Important Nota (N) in the Introduction - 
cf Section \ref{S5.9} also
to discard some non-relevant families of reciprocal
algebraic integers $\beta > 1$ tending to 1).
Since $\beta$ is reciprocal
the Parry Upper function
$f_{\beta}(x)$ is a power series which is
never a polynomial.
The Mahler measure
${\rm M}(\beta)$ is a function of
the roots of $P^{*}_{\beta}(z)=P_{\beta}(z)$
as
\begin{equation}
\label{mahlermeasurelent}
{\rm M}(\beta)
=
\prod_{\gamma ~
\mbox{{\tiny conjugate of}}~ \beta^{-1},
|\gamma|<1} |\gamma|^{-1}.
\end{equation}
This product is over
the set of conjugates $\gamma$
of $\beta^{-1}$,
$|\gamma|<1$.
Up till now, 
the only element of this set, coming
from the zeroes of $f_{\beta}(z)$, is
$\beta^{-1}$. It is a common
(simple) zero of  $P^{*}_{\beta}(z)$ and
$f_{\beta}(z)$. We will show that 
this set contains other zeroes
of $f_{\beta}(z)$.
In $\S$\ref{S5.4.1} 
we prove that the 
separation of the roots of $f_{\beta}(z)$
of modulus $< 1$
into two distinct
subcollections occurs inside the angular
sector, roughly (``the cusp region")
$$
\arg(z) ~\in ~\bigl[\,  
- \frac{\pi}{18},
+ \frac{\pi}{18}  \,\bigr] 
$$
given by \eqref{angularsectoramax}: 
those which are
off the unit circle form a lenticular shape,
those which are very close to the unit circle
lie in a thin annular neighbourhood of
$|z|=1$. 
In $\S$\ref{S5.4.2} 
we use the 
R\'enyi-Parry numeration system to
construct sequences of rewriting polynomials
``between" the minimal polynomial
$P_{\beta}(z) = P^{*}_{\beta}(z)$
and $f_{\beta}(z)$.
Two sequences of polynomial functions
are used to identify
the conjugates of $\beta^{-1}$ 
with the zeroes of $f_{\beta}(z)$:  
the sequence of polynomial
sections of $f_{\beta}(z)$,
the sequence of rewriting polynomials.
As a consequence, the product
\eqref{mahlermeasurelent}
will be over a set containing
the collection of
lenticular zeroes of $f_{\beta}(z)$.

\

\subsubsection{Lenticular roots of $f_{\beta}(x)$ and its polynomial sections}
\label{S5.4.1}

For $n \geq 3$, 
the polynomial sections of 
$f_{\beta}(x)$, $\beta \in (1,(1+\sqrt{5})/2)$, 
are of the type
\begin{equation}
\label{fclasseB}
-1 + x + x^n +
x^{m_1} + x^{m_2} + \ldots + x^{m_s}
\end{equation}
where $s \geq 0$, $m_1 - n \geq n-1$, 
$m_{q+1}-m_q \geq n-1$ for $1 \leq q < s$.
Denote by
$\mathcal{B}$ the class of the polynomials
defined by \eqref{fclasseB},
and by
$\mathcal{B}_n$ those whose third monomial is 
exactly
$x^n$, so that
$$\mathcal{B} =
\cup_{n \geq 3} \mathcal{B}_n .$$
The case ``$s=0$" corresponds to 
the trinomials $G_{n}(z) :=
-1+z+z^n$ (Selmer \cite{selmer}).

\begin{theorem}
\label{thm2lenticuli}
Let $c_{lent} := \min_{n \geq 260}
(c_n - \frac{\pi}{a_{max}})$.
Let $n \geq 260$ and
$\beta > 1$ be a reciprocal algebraic 
integer such that $\dyg(\beta) = n$,
${\rm M}(\beta) \leq 1.176280\ldots$.
Denote by
$$f_{\beta}(z) =
-1 + z + z^n +
z^{m_1} + z^{m_2} + \ldots + z^{m_j}+
z^{m_{j+1}}+\ldots,$$
where 
$m_1 - n \geq n-1$, 
$m_{j+1}-m_j \geq n-1$ for $j \geq 1$,
the Parry Upper function at $\beta$.
Then   
the zeroes  
of  $f_{\beta}(z)$ of modulus $< 1$ 
which lie in
$- \arg(z_{J_n , n}) - 
\frac{\pi}{n a_{max}}
< 
\arg z 
< 
+ \arg(z_{J_n , n})
+
\frac{\pi}{n a_{max}}$
either belong to
$$\Bigl\{z \mid \bigl||z|-1\bigr| < \frac{1}{3}
\frac{c_{lent}}{n}
\Bigr\}\quad
\mbox{or to}
\quad \bigl\{z \mid ||z|-1| > \frac{c_{lent}}{n}
\bigr\} 
.$$
In the second class of zeroes, all the zeroes are simple, and lie in the union 
$$D_{0,n} ~\cup~
\bigcup_{j=1}^{J_n}
(\overline{D_{j,n}} \cup D_{j,n});$$
there is one zero per disk $D_{j,n}$, 
$\overline{D_{j,n}}$, the disk
$D_{0,n}$ containing the element $\beta^{-1}$.
\end{theorem}

\begin{proof} 
Denote by
$$\mathcal{S}_{n}
:= \Bigl\{z \mid \theta_{n-1} \leq |z| < 1,
- \arg(z_{J_n , n}) - 
\frac{\pi}{n a_{max}}
\leq 
\arg z \leq + \arg(z_{J_n , n}) 
+ \frac{\pi}{n a_{max}}
\Bigr\}$$
the truncated angular sector and let
$$\widehat{\mathcal{S}_{n}}
:= \stackrel{o}{\mathcal{S}_{n}} 
\, \setminus
\, 
\Bigl(
\bigcup_{j=1}^{J_n} \left(
D_{j,n} \cup \overline{ D_{j,n}}
\right)
\,
\cup \, 
D(\theta_n , \theta_n - \theta_{n-1})
\Bigr)^{cl}
$$
the open truncated
angular sector obtained from
$\mathcal{S}_{n}$
by removing the closure of the Rouch\'e disks
$\overline{D_{j,n}}, D_{j,n}$ centered at
the zeroes $z_{j,n}$ of $G_{n}(z)$
in $\mathcal{S}_{n}$ of
respective radius
$\frac{\pi |z_{j,n}|}{n a_{max}}$,
and of
$D(\theta_n , \theta_n - \theta_{n-1})$.
The argument $\arg(z_{J_n , n})$ is defined
in \eqref{argzJJJn}.
The analytic function 
$G_{n}(z)$ has no zero
in the adherence 
$\overline{\widehat{\mathcal{S}_{n}}}$
of $\widehat{\mathcal{S}_{n}}$
and
reaches its infimum 
$\inf_{z \in \widehat{\mathcal{S}_n}} 
|-1+z+z^n| > 0$
on the boundary
$\partial \widehat{\mathcal{S}_{n}}$
of $\widehat{\mathcal{S}_{n}}$.
On the Rouch\'e circles
$C_{j,n}, \overline{C_{j,n}}$, 
$j=1, \ldots, J_n$,
using \eqref{rouchecercle}
and
\eqref{rouchecercleBUMP}, 
this infimum
is bounded from below by
$$\frac{
\left|Z\right|^{2 n -1}}{1 - |Z|^{n-1}}$$
where $Z$ is the point of $C_{1,n}$
of smallest modulus, which is
such that
$|Z| = |\theta_{n} - 
\frac{\pi}{n a_{max}} \theta_n|
$ at the first order.
Putting aside the Rouch\'e circles,
using the inequality
$|-1+z+z^n| \geq ||-1+z|- |z^n||$,
the minimum of
$|-1+z+z^n|$ on the arcs
$|z|=1,
|z|=\theta_{n-1}$, the segments
$\arg x = \pm 
(\arg(z_{J_n , n}) + \frac{\pi}{n a_{max}})$
and the circle $C(\theta_n , 
\theta_n - \theta_{n-1})$
on $\partial \widehat{\mathcal{S}_{n}}$
is bounded from below by
$$|-1+\theta_{n-1}| - |\theta_{n-1}^{n}|
=
(1-\theta_{n-1})^2 .
$$
Denote
$$\delta_n := \min \Bigl\{
(1-\theta_{n-1})^2 , 
\frac{
\left|Z\right|^{2 n -1}}{1 - |Z|^{n-1}}
\Bigr\}.$$
We have:
$0 < \delta_n \leq 
\inf_{z \in \widehat{\mathcal{S}_n}} 
|-1+z+z^n|$ and
$\lim_{n \to \infty}
\delta_n = 0$.
It is easy to show that
$$\lim_{n \to \infty}\frac{\lo \delta_n}{n} =0.$$
Using \S 5.3 in \cite{dutykhvergergaugry}
this limit condition allows to calculate
a first-order estimate of the thickness
of the annular neighbourhood
of the unit circle, in
$\widehat{\mathcal{S}_{n}}$,
which contains the roots of a polynomial
section
$-1 + z + z^n +
z^{m_1} + z^{m_2} + \ldots + z^{m_s}$
of $f_{\beta}(z)$;
this estimate is
\begin{equation}
\label{epaisseurss}
e(s) = 1 -
\left(
1 - 2 \frac{(n -1) (s - \delta_n)}{
(n -1) (s^2 +  s) + 2 (m_s - n) }
\right)^{1/(n-1)} .
\end{equation}
In the expression
\eqref{epaisseurss}
$n$ is fixed, as well as the 
sequence $(m_j)_{j \geq 1}$
since $\beta$ is fixed, therefore
$f_{\beta}(z)$ also;
the integer $m_s$ tends to infinity, if $s$
tends to infinity,
since 
$m_s - n
\geq 
(m_1 - n) +\sum_{j=2}^{s} (m_{j} - m_{j-1})
\geq s (n-1)$;
the integer $s$ is large enough
(at least to have
$s - \delta_n > 0$)
and
$\lim_{s \to \infty} e(s) =0$.

Among all the Rouch\'e disks
$D_{j,n}$, $1 \leq j \leq J_n$,
the $J_{n}$th
Rouch\'e disk 
$D_{J_n , n}$ is the closest
to the unit circle
(by (iii-2) in
Proposition \ref{closetoouane}).
By Lemma \ref{cnroucheasymptotic} 
its center is $z_{J_n , n}$,
of modulus 
$|z_{J_n , n}| = 1 - \frac{c_n}{n}$,
and its radius is
$\frac{\pi |z_{J_n , n}|}{n a_{max}}
< \frac{\pi}{n a_{max}}$.
By Lemma \ref{cnroucheasymptotic}
the limit $c = \lim_{n \to \infty}
c_n$ exists, is positive, 
and, from a numerical viewpoint,
$c - \frac{\pi}{a_{max}} = 1.76274
- 0.53479 = 1.22794\ldots$.
By the asymptotic
expansion of $c_n$
in Lemma \ref{cnroucheasymptotic},
the constant 
$c_{lent} := \min_{n \geq 260}
(c_n - \frac{\pi}{a_{max}})$ 
is positive.
The disk
$\{z \mid |z| < 1 - \frac{c_{lent}}{n}\}$
contains all the Rouch\'e disks
$D_{j,n}, \overline{D_{j,n}}$,
$1 \leq j \leq J_n$, and $D_{0,n}$.

Let assume that $f_{\beta}(z)$ has a zero
in
$$\widehat{\mathcal{S}_{n}}
\cap
\Bigl\{z \mid |z| < 1 - \frac{1}{3} 
\frac{c_{lent}}{n}
\Bigr\}
.$$
Denote it by
$z$, counted with 
multiplicity. 
There exists $r > 0$ small enough such that
the open disk $D(z,r)$ be included
in $\widehat{\mathcal{S}_{n}}
\cap
\{z \mid |z| < 1 - \frac{1}{3} 
\frac{c_{lent}}{n}\}$ and only contains
the zero $z$ of $f_{\beta}(z)$.
By Hurwitz Theorem (for instance
cf \S 11 in Chap. 2 in \cite{sakszygmund}) 
the number of zeroes of any polynomial
section 
$-1 + z + z^n +
z^{m_1} + z^{m_2} + \ldots + z^{m_s}$
of $f_{\beta}(z)$ 
in $D(z,r)$ should be equal to 
the multiplicity $\geq 1$ of $z$, as soon as
$s$ is large enough, say
$s \geq s_0$ for some $s_0$.

Since $\lim_{s \to 0} e(s) = 0$,
we obtain a contradiction by taking
$s_0$ such that
$e(s) \leq \frac{c_{lent}}{10 \,n}$ for all
$s \geq s_0$. The constant 10, 
at the denominator, is 
arbitrary and may be taken eventually
larger.
This means that all the zeroes 
of all the polynomial 
sections of $f_{\beta}(z)$,
in $\widehat{\mathcal{S}_{n}}$, 
for all $s \geq s_0$,
are contained in
$$1 - \frac{c_{lent}}{3 \,n}
< |z| < 1.
$$
But 
$\{z \mid 1 - \frac{c_{lent}}{3 \,n}
< |z| < 1\}
\cap
D(z,r) = \emptyset$. Contradiction.

Therefore the 
zeroes of
$f_{\beta}(z)$ which lie in the open
angular sector
$$\{z \in \cb : |z| < 1,
- \arg(z_{J_n , n}) - 
\frac{\pi}{n a_{max}}
<
\arg z 
< 
\arg(z_{J_n , n}) 
+ \frac{\pi}{n a_{max}}
\}$$
are located either
in the Rouch\'e disks by Theorem
\ref{cercleoptiMM}
and Theorem \ref{omegajnexistence},
or in a small neighbourhood
of the unit circle included
in $\{z \mid 1 - \frac{c_{lent}}{3 \,n}
< |z| < 1\}$. This dichotomy
naturally 
extends to the zeroes of
any polynomial section
of $f_{\beta}(z)$ 
(cf the proofs of 
Theorem
\ref{cercleoptiMM},
Theorem \ref{omegajnexistence}
and Theorem \ref{thm2lenticuli}).
\end{proof}

\begin{definition}
\label{lenticularzerodefinition}
Let $n \geq 260$. Let
$\beta > 1$ be a reciprocal algebraic 
integer such that $\dyg(\beta) = n$,
with ${\rm M}(\alpha)
< 1.176280\ldots$.
The zeroes of $f_{\beta}(z)
:= 
-1 + x + x^n +
x^{m_1} + x^{m_2} + \ldots + x^{m_s}+\ldots$,
resp. of any of its polynomial
section
$f(x) := 
-1 + x + x^n +
x^{m_1} + x^{m_2} + \ldots + x^{m_s}
\in \mathcal{B}_n$,
which belong to the angular sector
\begin{equation}
\label{angularsectornnn}
\Bigl\{z \in \cb : |z| < 1 - 
\frac{c_{lent}}{n} , \,
|\arg z| \leq \arg(z_{J_n , n}) 
+ \frac{\pi}{n a_{max}}
\Bigr\}
\end{equation}
are called the {\em lenticular zeroes}
of $f_{\beta}(z)$, resp. of $f$.
\end{definition}

\begin{theorem}
\label{thm1factorization}
For any $f \in \mathcal{B}_n$, $n \geq 3$,
denote by
$$f(x) =
-1 + x + x^n +
x^{m_1} + x^{m_2} + \ldots + x^{m_s} =
A(x) B(x) C(x),$$
where $s \geq 1$, $m_1 - n \geq n-1$, 
$m_{j+1}-m_j \geq n-1$ for $1 \leq j < s$,
the factorization of  $f$ where $A$ is 
the cyclotomic part, 
$B$ the reciprocal noncyclotomic part,
$C$ the nonreciprocal part.
Then 
\begin{itemize}
\item[(i)] the  nonreciprocal part $C$ is
nontrivial, irreducible, 
and never vanishes on the unit circle,
\item[(ii)] if $\gamma_{s} > 1$
denotes the real algebraic integer
uniquely determined 
by the sequence
$(n, m_1 , m_2 , \ldots , m_s)$ 
such that
$1/\gamma_s$ is the unique
real root of $f$ in $(0,1)$, 
the polynomial
$-C^*(X)$, opposite of
the reciprocal polynomial  of
$C(X)$,
is the minimal polynomial
of $\gamma_s$, and 
$\gamma_s$ is 
a nonreciprocal algebraic integer.
\end{itemize}
\end{theorem}

\begin{proof}
Theorem 3 in \cite{dutykhvergergaugry}.
\end{proof}

Let us precise the behaviour
of the reciprocal noncyclotomic parts $B$
of $f(x)$ in 
Theorem \ref{thm1factorization}
on the lenticular roots of $f$.

\begin{proposition}
\label{Bcyclolenticules}
Let $\beta > 1$ be a 
reciprocal algebraic integer having
$\dyg(\beta) \geq 260$.
Let 
$f_{\beta}(x) =
-1 + x + x^n +
x^{m_1} + x^{m_2} + \ldots + x^{m_s} +\ldots$
be the Parry Upper function at $\beta$
and, for $s \geq 0$, denote its $s$-th 
polynomial section by
$$f(x) =
-1 + x + x^n +
x^{m_1} + x^{m_2} + \ldots + x^{m_s}
\quad \mbox{factorized as} ~= A(x) B(x) C(x),$$
where $s \geq 1$, $m_1 - n \geq n-1$, 
$m_{j+1}-m_j \geq n-1$ for $1 \leq j < s$,
where $A$ is the cyclotomic part, 
$B$ the reciprocal noncyclotomic part,
$C$ the nonreciprocal part of $f$.

There exists $s_0$ (depending upon $n$)
such that
the reciprocal
noncyclotomic part $B$
of $f(x)$, if any,
does not vanish on the lenticular roots
of $f$, as soon as 
$s \geq s_0$.
\end{proposition} 

\begin{proof}
The result is true
in the case 
``$s=0$"  
since a trinomial
$G_{n}(X)$ is either nonreciprocal
and irreducible or the 
product of a nonreciprocal
irreducible polynomial by a
cyclotomic polynomial
(by Proposition \ref{irredGn}) 
\cite{selmer}. 
 Let us assume 
$n = \dyg(\beta) \geq 260$ ($n$ is fixed)
and make it explicit.
Let $z_{j,n}$ be a zero of
$G_{n}(x) = -1 + x + x^n$ which
belongs to the angular sector
\eqref{angularsectornnn}.
Then
the reciprocal trinomial
$G_{n}^{*}(z_{j,n}) =
-z_{j,n}^n + z_{j,n}^{n-1} + 1$
at $z_{j,n}$ is such that
$$|G_{n}^{*}(z_{j,n})| =
|z_{j,n} ( 1+z_{j,n}^{n-2})|
= |z_{j,n} (1 + (1-z_{j,n})/z_{j,n}^2)|$$
$$
\geq 
\theta_{n-1}
(1 -
|\frac{1-z_{J_n , n}}{z_{J_n , n} z}|
)
\geq 
\theta_{n-1} (1 - \kappa(1, a_{max}))
$$
by Proposition \ref{closetoouane} and
Definition \ref{Jndefinition}. 
Obviously 
$\min_{n \geq 260}
\theta_{n-1} (1 - \kappa(1, a_{max}))$
is $> 0$ since
$\lim_{n \to \infty}
\theta_{n} (1 - \kappa(1, a_{max}))
= 1 - \kappa(1, a_{max}) = 1 - 0.171573\ldots =
0.828427\ldots$.
In other terms $G_{n}^{*}(x)$ 
does not vanish on the lenticular zeroes
$z$ of $G_{n}(x)$ which lie 
in  \eqref{angularsectornnn}, 
with the lower bound
$\min_{n \geq 260}
\theta_{n-1} (1 - \kappa(1, a_{max}))$, 
independent of $n$ and uniform on all the
lenticular roots of $f$ in \eqref{angularsectornnn}.

We now consider the general case
``$s \geq 1$" and extend the previous case. 
Let $D_{j,n}$ (defined in Theorem
\ref{omegajnexistence}, with
$D_{-j,n} = \overline{D_{j,n}}$) be a Rouch\'e disk
in the open angular sector
\eqref{angularsectornnn}, 
that is for $j \in \{0, 1, \ldots, J_n\}$. 
Denote by $r_s$ the unique zero in
$D_{j,n}$ of
the $s$th polynomial section
$f(x) =
-1 + x + x^n +
x^{m_1} + x^{m_2} + \ldots + x^{m_s}
\in 
\mathcal{B}_n$. 
Recall that in $D_{j,n}$
the unique zero of
$f_{\beta}(z)$ is
$\omega_{j,n}$ (with
$\omega_{0,n} := \beta^{-1}$)
so that
$\lim_{s \to \infty} r_s = 
\omega_{j,n}$. 
We will show that
the polynomial
$f^{*}(x) = A^{*}(x) B^{*}(x) C^{*}(x)
= x^{m_s} f(1/x)
$, 
reciprocal of
$f(x)$, does not vanish on
the lenticular zero $r_s$ of $f(x)$. 

Assume the contrary.
Then 
\begin{equation}
\label{zerodureciproque}
0 = f^{*}(r_s)
=
G_{n}^{*}(r_s)\, 
r_{s}^{m_s - n}
+
\bigl(1 + r_{s}^{m_s - m_{s-1}}
+
r_{s}^{m_s - m_{s-2}}
+
\ldots
+
r_{s}^{m_s - m_{1}}
\bigr).
\end{equation}
From the polynomial $f^*$, let us construct
the associated Parry Upper function
$$f_{\beta}^{[s]}(z) :=
G_{n}(z)
+
z^{m_1} \bigl(
f^{*}(z) - G_{n}^{*}(z) \, z^{m_s - n}
\bigr)
+ (f_{\beta}(z) - f(z))$$
$$
=
-1 + z + z^n + z^{m_1}
+ z^{m_1 + m_s - m_{s-1}}
+
z^{m_1 + m_s - m_{s-2}}
+\ldots
+
z^{ m_s }+
\mbox{tail of}~ f_{\beta}(z).
$$
For
$1 \leq q \leq s-1$ the 
conditions 
$(m_s - m_{s-(q+1)})
-
(m_s - m_{s-q}) = 
m_{s-q} - m_{s-(q+1)} \geq n-1$  imply that
Theorem \ref{omegajnexistence} applies.
For every $s \geq 1$ the Parry Upper function
$f_{\beta}^{[s]}(z)$ has a zero in $D_{j,n}$.
But, from \eqref{zerodureciproque},
we would have
\begin{equation}
\label{zerodureciproquemodules}
\left|G_{n}^{*}(r_s)\, 
r_{s}^{m_s - n}\right|
=
\Bigl|
\frac{f_{\beta}^{[s]}(r_s) -
G_{n}(r_s) - 
\sum_{q=s+1}^{\infty} r_{s}^{m_q}}{r_{s}^{m_1}}
\Bigr|.
\end{equation}
The limit
$\lim_{s \to \infty}
\sum_{q=s+1}^{\infty} r_{s}^{m_q}
=0$ holds and
$$\liminf_{s \to \infty}
\Bigl|
\frac{f_{\beta}^{[s]}(r_s) -
G_{n}(r_s) - 
\sum_{q=s+1}^{\infty} r_{s}^{m_q}}{r_{s}^{m_1}}
\Bigr|
=
\liminf_{s \to \infty}
\Bigl|
\frac{f_{\beta}^{[s]}(\omega_{j,n}) -
G_{n}(\omega_{j,n})}{\omega_{j,n}^{m_1}}
\Bigr|
$$
admits a lower bound
which is strictly positive.
Indeed,
this lower bound can be computed
from any series
$\sum_{q=0}^{\infty} z^{d_q}$, $d_0 = 0,
d_{q+1}-d_q \geq n-1, q \geq 0$,
by
$$
|1 + \omega_{j,n}^{d_1}
+
\omega_{j,n}^{d_2}
+
\ldots
+
\omega_{j,n}^{d_q}
+\ldots|
\geq
1 - |\omega_{j,n}|^{n-1}
\frac{1}{1 - |\omega_{j,n}|^{n-1}}
$$
and approximated by
\begin{equation}
\label{minototal}
1 - |z_{j,n}|^{n-1}
\frac{1}{1 - |z_{j,n}|^{n-1}}
=
1-\frac{|1-z_{j,n}|}{|z_{j,n}| - |1-z_{j,n}|}
\geq
1-\frac{|1-z_{J_n,n}|}{|z_{J_n,n}| - |1-z_{J_n,n}|} > 0.
\end{equation}
The contradiction comes from the fact
that the lhs of \eqref{zerodureciproquemodules}
tends to 0 when $s$ tends to infinity, whereas
the rhs of \eqref{zerodureciproquemodules}
has a positive liminf by
\eqref{minototal}. We deduce
that $B(r_s)=B^{*}(r_s) = 0$ cannot hold as soon as $s$
is large enough.
\end{proof}

\begin{corollary}
\label{Bcyclolenticulescoro}
Let $\beta > 1$ be a 
reciprocal algebraic integer having
$n =
\dyg(\beta) \geq 260$, 
with ${\rm M}(\alpha)
< 1.176280\ldots$.
For $s$ large enough, 
if $\gamma_s > 1$ denotes the 
real root of the polynomial
$C^{*}(z)$, all the lenticular
zeroes of the $s$-th polynomial section
$f(z) = A(z)B(z)C(z)$ 
of $f_{\beta}(z)$
are conjugates of
$\gamma_{s}^{-1}$ where conjugation is 
relative to the irreducible nonreciprocal
polynomial part
$C(z)$. 
The degree of the irreducible
nonreciprocal factor
$C(X)$ of $f$ satisfies
\begin{equation}
\label{degreeCsection}
\deg(C) \geq 1 + 
2 J_n = 1+
\frac{n}{\pi}
\bigl(
2 \arcsin\bigl( \frac{\kappa}{2} \bigr) 
\bigr)
+
\frac{2 \kappa \, \lo \kappa}
{\pi \,\sqrt{4 - \kappa^2}}
+ 
O\bigl(
\bigl(
\frac{\lo \lo n}{\lo n}
\bigr)^2
\bigr).
\end{equation}
\end{corollary}

\begin{proof}
The irreducible nonreciprocal part
$C$ of $f$ has at least $1 + 2 J_n$
zeroes, where $J_n$ is given
by \eqref{Jnasymptotic}. The
constant $
2 \arcsin\bigl( \frac{\kappa}{2} \bigr)
= 0.171784\ldots$ is given
by \eqref{philimite}.
\end{proof}

\

\subsubsection{Rewriting trail and rewriting polynomials from ``$P_{\beta}$" to ``$f_{\beta}$"}
\label{S5.4.2}

(cf \cite{dutykhvergergaugry2} for more details)
As mentioned above $\beta$ is a reciprocal
algebraic integer which is fixed
such that $\dyg(\beta)=n \geq 260$,
${\rm M}(\beta) \leq 1.176280\ldots$, as well as
its minimal polynomial $P_{\beta}$ and
its Parry Upper function $f_{\beta}$.
Let $P_{\beta}(X)=P_{\beta}^{*}(X)
= 1 + a_1 X + a_2 X^2
+ a_3 X^3 + \ldots + a_{d-1} X^{d-1} + X^d$,
$a_i \in \zb$, $a_{d-j}=a_j, d \geq 1$, be the minimal
polynomial of $\beta$. 
Let $f_{\beta}(z)
= -1 + t_1 z + t_2 z^2 + t_3 z^3 + \ldots$
be the Parry Upper function at $\beta$,
written
in the generic form 
(with $t_1 = 1$, 
$t_2 = t_3 = \ldots = t_{n-1} =0$,
$t_n = 1$, etc).
Recall that the sequence
$(t_i)_{i \geq 1}$ is unique, 
entirely characterizes
$\beta$ and is Lyndon (self-admissible). 
Polynomial sections of $f_{\beta}(X)$
are denoted by $S_q$: for $q \geq 1$,
$S_{q}(z) = 
-1 + t_1 z + t_2 z^2 + \ldots
+ t_q z^q$, $|z| < 1$.
Since $f_{\beta}(\beta^{-1})=0$
and that, for any
$q \geq 1$, $f_{\beta}(\beta^{-1})-
S_{q}(\beta^{-1})$ is a sum of positive terms,
we can permute them and group them in order
to obtain $d$ components
in the $\qb$-basis
$\{1, \beta^{-1}, \beta^{-2}, \ldots,
\beta^{-d+1}\}$.
With
$$g_{q,j}(z):=
~z^{-j} \times
\sum_{i=q+1 \atop i \, \equiv \,q+1+j ~\mbox{{\tiny mod}}~ d}^{\infty}
t_i z^{i- (q+1)}\,
\qquad
q \geq 1, j=0, 1, \ldots, d-1,
$$
we obtain
the existence of
$d$ power series
$g_{q,0}(z), g_{q,1}(z), \ldots,
g_{q,d-1}(z)$, all defined 
on the open unit disk, such that
\begin{equation}
\label{fbetaSbetadecompo} 
f_{\beta}(\beta^{-1})-
S_{q}(\beta^{-1}) =
\frac{1}{\beta^{q+1}}\,
\Bigl(
\sum_{j=0}^{d-1}
g_{q,j}(\beta^{-1}) \beta^{-j}
\Bigr),
\qquad q \geq 1.
\end{equation}
Let us restrict
$z$ to the open angular sector 
\eqref{angularsectornnn}:
$$z \in 
\bigl\{z \mid |z| < 1 - 
\frac{c_{lent}}{n} , \,
|\arg z| \leq \arg(z_{J_n , n}) 
+ \frac{\pi}{n a_{max}}
\bigr\}.
$$
Then all the power series 
$g_{q,j}(z)$
are 
absolutely convergent in this sector,
having the same uniform upper bound
in modulus
\begin{equation}
\label{gqj_top}
|g_{q,j}(z)| \leq 
\frac{1}{1-(1-c_{lent}/n)^d},
\qquad \mbox{for all}~
q \geq 1, j=0, 1, \ldots, d-1,
\end{equation}
which is 
independent of $q$ and $j$.

Now, let us use the numeration
system in base $\beta$
in order to obtain alternate expressions
of the $d$ components
$g_{q,j}(\beta^{-1})$
in \eqref{fbetaSbetadecompo}.
It will allow to ``restore" the digits
$t_i$ of $f_{\beta}$ one after the other.
The 
identities 
$P_{\beta}(\beta^{-1})=0$
and
$f_{\beta}(\beta^{-1})=0$
give two
$\beta$-representations of 1, the second 
one being the R\'enyi $\beta$-expansion of 1:
\begin{equation}
\label{equa1P}
1= -a_1 \beta^{-1} - a_2 \beta^{-2}
- a_3 \beta^{-3} + \ldots - a_{d-1} 
\beta^{-(d-1)} - \beta^{-d}
=
1 - P_{\beta}(\beta^{-1})
,
\end{equation}
\begin{equation}
\label{equa1f}
1= t_1 \beta^{-1} + t_2 \beta^{-2}
+
t_3 \beta^{-3}+ \ldots = 1 + 
f_{\beta}(\beta^{-1}).
\end{equation}
Let us construct an infinite chain
of intermediate
$\beta$-representations of 1 between them.
Let us show that,
for
every
$q \geq 1$,
there exists
a polynomial
$A_q \in \zb[X]$, with
$\deg(A_q) \leq q$
and $A_{q}(0) = -1$,
and a
$d$-tuple
$(h_{q,0} , h_{q,1} , \ldots,
h_{q, d-1})$ of integers
such that
\begin{equation}
\label{AqPbetadecompo}
A_{q}(z)P_{\beta}(z) = S_{q}(z) +
z^{q+1}
\Bigl(
\sum_{j=0}^{d-1}
h_{q,j} \,z^{j}
\Bigr),
\end{equation}
satisfying
\begin{equation}
\label{hacheqj_gqj}
h_{q,j}=
g_{q,j}(\beta^{-1}) , \qquad
\mbox{for}\,
j = 0, 1, 2, \ldots, d-1. 
\end{equation}
For obtaining $A_1$ the quantity
$0 = \beta^{-1}(a_1 + 1) P_{\beta}(\beta^{-1})$
is added to \eqref{equa1P}.
Then
$$1=1+(-1 + (1+a_1)\beta^{-1})
P_{\beta}(\beta^{-1})
=
1 + S_{1}(\beta^{-1})
+
\beta^{-2}
\Bigl(
\sum_{j=0}^{d-1}
h_{1,j} \,z^{j}
\Bigr)$$
with $h_{1,j}
=
-a_{j+2} + a_{j+1} (a_1 + 1)$
for $j=0,1, \ldots, d-3$,
$h_{1,d-2}
=
-1 + a_{d-1} (a_1 + 1)$
and
$h_{1,d-1}
=
(a_1 + 1)$.
We deduce $A_{1}(z) = -1 + (a_1 + 1) z$.
From \eqref{fbetaSbetadecompo} and 
the fact that the lhs of 
\eqref{AqPbetadecompo} is equal to 0
for $z = \beta^{-1}$
we also deduce
$$h_{1,j}=
g_{1,j}(\beta^{-1}) , \qquad
\mbox{for}\,
j = 0, 1, 2, \ldots, d-1.$$
Now let us proceed recursively. 
Let us assume that
$A_1 , A_2 , \ldots, A_q$ are already constructed,
with
\eqref{AqPbetadecompo} and
\eqref{hacheqj_gqj} satisfied.
For obtaining $A_{q+1}$ let us first observe,
from \eqref{AqPbetadecompo}, 
that
$h_{q,0} \in \zb$ is the coefficient
of $\beta^{-(q+1)}$ in the 
$\beta$-representation of 1 which is
\begin{equation}
\label{qe_betarepresentation}
1 = 1+A_{q}(\beta^{-1}) P_{\beta}(\beta^{-1}).
\end{equation}
Let us add the quantity
$0= \beta^{-(q+1)} (t_{q+1} - h_{q,0})
P_{\beta}(\beta^{-1})$
to \eqref{qe_betarepresentation}
and consider the polynomial $A_{q+1}(z) = 
A_{q}(z) +(t_{q+1} - h_{q,0}) z^{q+1}$.
Then we obtain the  
$\beta$-representation of 1:
$$1 = 1 + A_{q+1}(\beta^{-1})P_{\beta}(\beta^{-1})
=
1 + S_{q+1}(\beta^{-1})
+
\beta^{-(q+1)}
\Bigl(
\sum_{j=0}^{d-1}
h_{q+1,j} \,\beta^{-j}
\Bigr)
$$
where the $d$-tuple
$(h_{q+1,0} , h_{q+1,1} , \ldots,
h_{q+1, d-1})$ of integers
can be readily computed from
$(h_{q,0} , h_{q,1} , \ldots,
h_{q, d-1})$ 
and $t_{q+1}$.
From \eqref{fbetaSbetadecompo} and 
the fact that the lhs of 
\eqref{AqPbetadecompo} is equal to 0
for $z = \beta^{-1}$
we deduce
$$h_{q+1,j}=
g_{q+1,j}(\beta^{-1}) , \qquad
\mbox{for}\,
j = 0, 1, 2, \ldots, d-1.$$  
  
\begin{definition}
\label{rewritingtrailandpolynomials}  
The sequence
of $\beta$-representations of 1 
$$(1 = 1+A_{q}(\beta^{-1})
P_{\beta}(\beta^{-1}))_{q \geq 1}$$
is called the
{\em rewriting trail from ``$P_{\beta}$" to
``$f_{\beta}$"}.   
The polynomial 
$$A_{q}(X) = -1 + (a_1 + 1) X
+
\sum_{j=1}^{q-1} (t_{j+1} - h_{j,0}) X^{j+1} 
\quad \in \zb[X]$$ 
is called
the {\em $q$th rewriting polynomial}
of the
rewriting trail from ``$P_{\beta}$" to
``$f_{\beta}$".
\end{definition}

\begin{proposition}
\label{APversfbeta}
Let $\beta > 1$ be a reciprocal
algebraic integer, $\dyg(\beta)=n \geq 260$,
and denote by $P_{\beta}$
its minimal polynomial and by $f_{\beta}$
its Parry Upper function at $\beta$.
If $x \neq \beta^{-1}$ 
is a zero of $P_{\beta}(z)$ in
the open angular sector 
\begin{equation}
\label{angularsectornnnbis}
z \in 
\Bigl\{z \mid |z| < 1 - 
\frac{c_{lent}}{n} , \,
|\arg z| \leq \arg(z_{J_n , n}) 
+ \frac{\pi}{n a_{max}}
\Bigr\},
\end{equation}
then $x$ is a lenticular zero 
of $f_{\beta}(z)$.
\end{proposition}

\begin{proof}
Let us use the
rewriting trail and 
the rewriting polynomials 
from ``$P_{\beta}$" to
``$f_{\beta}$". Let us assume that
$x \neq \beta^{-1}$ is
a zero of $P_{\beta}$
in the open angular sector
\eqref{angularsectornnnbis}, and
denote by $\sigma: \beta^{-1} \to x$
the conjugation. The image by $\sigma$
of the $\qb$-basis
$\{1, \beta^{-1}, \ldots,
\beta^{-d+1}\}$
is the $\qb$-basis
$\{1, x, \ldots,
x^{d-1}\}$.
From
\eqref{AqPbetadecompo} and
\eqref{hacheqj_gqj} we have
\begin{equation}
\label{interAPS}
0 = A_{q}(x)P_{\beta}(x) = S_{q}(x) +
x^{q+1}
\Bigl(
\sum_{j=0}^{d-1}
h_{q,j} \,x^{j}
\Bigr),
\end{equation}
with
$$\sigma(h_{q,j})=
h_{q,j}=
g_{q,j}(\sigma(\beta^{-1})) , \qquad
\mbox{for}\,\,
q \geq 1, j = 0, 1, 2, \ldots, d-1. 
$$
Since
$|h_{q,j}| \leq 
(1-(1-c_{lent}/n)^d)^{-1}
$
for all $q \geq 1$,
$j=0,1, \ldots, d-1$
by \eqref{gqj_top},
and that
$|x|<1$, 
we have
$$\lim_{q \to \infty}
x^{q+1}
\Bigl(
\sum_{j=0}^{d-1}
h_{q,j} \,x^{j}
\Bigr) = 0.$$
Therefore, from \eqref{interAPS}, 
we deduce
$\lim_{q \to \infty}
S_{q}(x)=0$. As a consequence 
the zero
$x$ necessarily belongs to the
set of limit points 
of the zeroes of
the polynomial sections
$S_{q}(z)$,
zeroes which lie in the angular sector
\eqref{angularsectornnnbis}. This
set of limit points is the
set of lenticular zeroes
of $f_{\beta}$, by Hurwitz's Theorem and
the uniform convergence of
$(S_q)$ to $f_{\beta}$ on every compact
of $D(0,1)$.
Hence the result.
\end{proof}

In the following Proposition
we continue the investigation
of the relative positioning of the zeroes of
$f_{\beta}(z)$ with respect to
those of the minimal polynomial
$P_{\beta}(z)$
in
the open angular sector 
\begin{equation}
\label{angularsectornnnter_02} 
\bigl\{z \mid |z| < 1 - 
\frac{c_{lent}}{n} , \,
|\arg z| \leq \arg(z_{J_n , n}) 
+ \frac{\pi}{n a_{max}}
\bigr\}.
\end{equation}
Indeed, since $\beta$ is fixed, both
collections of zeroes
are fixed. The purpose of 
Proposition \ref{ASversPbeta} is only
to discriminate the possible locations
from the impossible ones.
More explicitely,
from Proposition
\ref{APversfbeta} the zeroes 
of $P_{\beta}(z)$ which
lie the 
angular sector
\eqref{angularsectornnnter_02}
form a subcollection of the
lenticular zeroes of $f_{\beta}(z)$.
A natural question is whether
the remaining zeroes of
$f_{\beta}(z)$ in this sector
are also
zeroes of $P_{\beta}(z)$, or not?
The following Proposition
answers this question, 
using 
the tails of the 
rewriting polynomials constructed
from the polynomial sections
``$S_{s}$" of $f_{\beta}$
to ``$P_{\beta}$" (with $s$ large enough),
and the structure of
the type of factorization of the
polynomial sections of
$f_{\beta}(z)$.

\begin{proposition}
\label{ASversPbeta}
Let $\beta > 1$ be a reciprocal
algebraic integer, $\dyg(\beta)=n \geq 260$,
M$(\beta) \leq 1.176280\ldots$,
and denote by $P_{\beta}$
its minimal (reciprocal) 
polynomial and by $f_{\beta}$
its Parry Upper function.
The minimal polynomial
$P_{\beta}(X)$ of  $\beta > 1$ is 
such that there does not exist
an
irreducible integer polynomial  
$\widetilde{P_{\beta}}(X)$
and an
integer $r \geq 2$ ($r$ being the largest
integer having this property)
such that it is written under the form 
\begin{equation}
\label{requalityPP_}
P_{\beta}(X) = \widetilde{P_{\beta}}(X^r).
\end{equation}

If $x \neq \beta^{-1}$ 
is a lenticular zero of $f_{\beta}(z)$,
then 
$x$ is a zero 
of $P_{\beta}$.
\end{proposition}

\begin{proof}
Let us assume that
there exists a lenticular zero
$x$ of $f_{\beta}(z)$, then
such that $f_{\beta}(x)=0$, having the 
property $P_{\beta}(x)\neq 0$,
and show the contradiction.
We can add the assumption that
its imaginary part $\Im(x)$
is $ > 0$.
Denote $\nu := |P_{\beta}(x)| > 0$.
Let us consider
the $s$-th polynomial section
$S_{s}(z)=-1 + \sum_{j=1}^{s} t_j z^j$ 
of $f_{\beta}(z)$, where the integer
$s$, taken large enough, will be
fixed below.

The $s$-th polynomial section
$S_{s}(z)$ admits a unique real zero
in $(0,1)$.
Indeed $S_{s}(0)=-1$,
$S_{s}(1) > 1$, and
the derivative  of the restriction
of $S_{s}(z)$ to $[0,1]$ is positive
on $[0,1]$.
The polynomial
$S^{*}_{s}(z)$, reciprocal polynomial
of $S_{s}(z)$, admits a unique
real zero, say 
$\gamma_s, > 1$. 
We have: $\deg(S_s) \leq s$
and $\lim_{s \to \infty} 
\gamma_{s}^{-1} = \beta^{-1}$.
The real number
$\gamma_{s}$
is a nonreciprocal
algebraic integer which is such that 
$1 < \gamma_s < \beta$:
indeed, $y \to S_{s}(y)$ is strictly increasing
on $(0,1)$ and 
$S_{s}(\beta^{-1}) = 
-1 + \sum_{j=1}^{s} t_j \beta^{-j}
=
f_{\beta}(\beta^{-1}) 
-\sum_{j=s+1}^{\infty} t_j \beta^{-j}
=-\sum_{j=s+1}^{\infty} t_j \beta^{-j} <0
$,
so that $\beta^{-1} < \gamma_{s}^{-1}$.
There exists an integer, say $W_{\nu}$,
such that: (i) 
$\gamma_{s}^{-1}$ belongs 
to the open angular sector
\eqref{angularsectornnnbis} 
for all $s \geq W_{\nu}$
and (ii)
$s \geq W_{\nu} \Longrightarrow
|P_{\beta}(\gamma_{s}^{-1})| < 
\min\{1, \nu/2\}.$
In the following we take
$s \geq W_{\nu}$.

There exists
$j \in \{1, 2, \ldots, J_n\}$ such that
$x = \omega_{j,n}$ in the
$j$-th Rouch\'e disk
$D_{j,n}$. In this disk
$D_{j,n}$ the polynomial section
$S_{s}(z)$ has a unique (lenticular) zero;
let us denote it by
$r_s$. We have:
$\lim_{s \to \infty} r_s = 
\omega_{j,n}$
and $r_s$ is equal to
the conjugate
$\sigma_{s}(\gamma_{s}^{-1})$
of $\gamma_{s}^{-1}$
for some $\sigma_s$ which is the conjugation
relative to the irreducible nonreciprocal
(never trivial) part $C$ of $S_s$ by 
Proposition \ref{Bcyclolenticules}
and Corollary \ref{Bcyclolenticulescoro}.
We have: $C(\gamma_{s}^{-1}) =
C(r_s) =0$.
Denote by $s_c = \deg(C)$ the degree
of the component $C$.
The irreducible 
polynomials $P_{\beta}(X)$ and
$C(X)$ are coprime: indeed 
the first one is reciprocal
while the second one is nonreciprocal.
The integer $s_c$ is a function of $s$.
Denote
$$d := \deg (P_{\beta})
\qquad {\rm and} \qquad
H:= \max_{j =1, \ldots, d-1} \{|a_j|\} \geq 1$$
the (na\"ive) height of $P_{\beta}(X)
= 1 + \sum_{j=1}^{d-1} a_j X^j + X^d$.

\begin{rewritingtrail}
\label{heightP}
\cite{dutykhvergergaugry2}
In the same way as above 
let us construct the
rewriting trail  
from ``$S_{s}$" to
``$P_{\beta}$", at $\gamma_{s}^{-1}$.
The starting point is the identity
$1 = 1$, to which we add 
$0=  S_{\gamma_{s}}(\gamma_{s}^{-1})$
in the (rhs) right hand side.
Then we define the rewriting trail from
the R\'enyi 
$\gamma_{s}^{-1}$-expansion of 1
\begin{equation}
\label{gammasexpansionSs}
1=1+S_{\gamma_{s}}(\gamma_{s}^{-1})
=
t_{1}\gamma_{s}^{-1}  +t_2 \gamma_{s}^{-2} + \ldots 
+ t_{s-1} \gamma_{s}^{-(s-1)} +t_s \gamma_{s}^{-s}
\end{equation}
(with $t_1 = 1,
t_2 = t_3 = \ldots =
t_{n-1} = 0,
t_n = 1,$ etc) to
\begin{equation}
\label{equa1Pgammas}
- a_1 \gamma_{s}^{-1} - a_2 \gamma_{s}^{-2} + \ldots - 
a_{d-1} \gamma_{s}^{-(d-1)} - \gamma_{s}^{-d}
= 1 - P_{\beta}(\gamma_{s}^{-1}),
\end{equation}
by ``restoring" the digits
of $1 - P_{\beta}(X)$ one after the other,
from the left.
We obtain a sequence
$(A'_{q}(X))_{q \geq 1}$ of
rewriting polynomials involved
in this rewriting trail; 
for $q \geq 1$,
$A'_q \in \zb[X]$,
$\deg(A'_q) \leq q$
and $A'_{q}(0) = 1$.
At the first step we add $0=
-(-a_1 - t_1) \gamma_{s}^{-1} S_{\gamma_{s}}^{*}(\gamma_{s}^{-1})$;
and we obtain 
$$1= -a_1 \gamma_{s}^{-1}$$
$$+(-(-a_1 -t_1) t_1 + t_2) \gamma_{s}^{-2}
+(-(-a_1 -t_1) t_2 + t_3) \gamma_{s}^{-3} + \ldots
$$
so that the height of the polynomial
$$(-(-a_1 -t_1) t_1 + t_2) X^{2}
+(-(-a_1 -t_1) t_2 + t_3) X^{3} + \ldots
$$
is $\leq H+2$.
At the second step we add
$0=
-(-a_2 - (-(-a_1 -t_1) t_1 + t_2)) \gamma_{s}^{-2} 
S_{\gamma_{s}}^{*}(\gamma_{s}^{-1})
$.
Then we obtain
$$1= -a_1 \gamma_{s}^{-1} - a_2 \gamma_{s}^{-2}$$
$$-
[(-a_2 - (-(-a_1 -t_1) t_1 + t_2))t_1
+ (-(-a_1 -t_1) t_2 + t_3)] \gamma_{s}^{-3}+\ldots
$$
where the height of the polynomial
$$-[(-a_2 - (-(-a_1 -t_1) t_1 + t_2))t_1
+ (-(-a_1 -t_1) t_2 + t_3)] X^{3}+\ldots
$$
is $\leq H + (H+2)+(H+2)=3 H+4$.
Iterating this process $d$ times 
we obtain
$$1= -a_1 \gamma_{s}^{-1} - 
a_2 \gamma_{s}^{-2} -\ldots
- a_d \gamma_{s}^{-d}$$
$$+~~
polynomial ~~remainder~~ in~~ \gamma_{s}^{-1}.
$$
Denote by $V(\gamma_{s}^{-1})$
this polynomial remainder in $\gamma_{s}^{-1}$,
for some $V(X) \in \zb[X]$,
and $X$ specializing in $\gamma_{s}^{-1}$.
If we denote the upper bound of the
height of the polynomial remainder
$V(X)$, 
at step $q$, by $\lambda_q H + v_q$, 
we readily
deduce: $v_q = 2^q$, and
$\lambda_{q+1} = 2 \lambda_{q} +1$, $q \geq 1$,
with $\lambda_1 = 1$; then 
$\lambda_{q} = 2^{q}-1$.

To summarize,
the first 
rewriting polynomials of the
sequence
$(A'_{q}(X))_{q \geq 1}$ 
involved in this rewriting trail
are
$$A'_{1}(X) = 
-1 - (-a_1 - t_1) X,$$ 
$$A'_{2}(X) 
= 
-1 - (-a_1 - t_1) X -
(-a_2 - (-(-a_1 -t_1) t_1 + t_2)) X^2 , \quad {\rm etc}.$$

\

For $q \geq \deg(P_{\beta})$, all the coefficients 
of $P_{\beta}$ are ``restored"; denote
by
$(h_{q,j})_{j=0,1,\ldots,s-1}$ the $s$-tuple
of integers produced by this rewriting trail,
at step $q$. It is such that
\begin{equation}
\label{AprimeSPreste}
A'_{q}(\gamma_{s}^{-1}) S_{\gamma_{s}}^{*}(\gamma_{s}^{-1})
=
-P(\gamma_{s}^{-1}) + \gamma_{s}^{-q-1}
\Bigl(
\sum_{j=0}^{s-1} h_{q,j} \gamma_{s}^{-j}
\Bigr).
\end{equation}
Then take $q=d$.
The (lhs) left-and side 
of \eqref{AprimeSPreste} is equal to 
$0$.
Thus 
$$P(\gamma_{s}^{-1}) =
 \gamma_{s}^{-d-1}
\Bigl(
\sum_{j=0}^{s-1} h_{d,j} \gamma_{s}^{-j}
\Bigr)
\qquad
\Longrightarrow
\qquad
P(\gamma_{s}) =
\sum_{j=0}^{s-1} h_{d,j} \gamma_{s}^{-j-1}.
$$
The height
of the polynomial
\begin{equation}
\label{Wpolynomial}
W(X) :=\sum_{j=0}^{s-1} h_{d,j} X^{j+1}
\qquad {\rm is}\qquad \leq
 (2^d -1) H + 2^d,
 \end{equation} 
 and is independent of $s
 \geq W_v$.
\end{rewritingtrail}

\

For any $s \geq W_{\nu},$ 
let us observe that
$- P_{\beta}(\gamma_{s}^{-1})
$ is $> 0$, and that the sequence
$(\gamma_{s}^{-1})_s$ is decreasing.
Indeed, by Proposition
\ref{APversfbeta}, the polynomial function
$x \to P_{\beta}(x)$ is positive
on $(0, \beta^{-1})$, vanishes
at $\beta^{-1}$, 
and changes its sign
for
$x > \beta^{-1}$, 
so that 
$P_{\beta}(\gamma_{s}^{-1}) < 0$.
We have: $\lim_{s \to \infty}
P_{\beta}(\gamma_{s}^{-1}) =
P_{\beta}(\beta^{-1})=0$.

Let us use the 
$\gamma_{s}$-shift and the
greedy (R\'enyi) $\gamma_{s}$-expansion of
$- P_{\beta}(\gamma_{s}^{-1})$: 
there exists
an unique sequence of integers
$(\widehat{t_i})_{i \geq 1} \not\equiv (0)$ 
in the alphabet
$\{0, 1\}$ of the
$\gamma_{s}$-shift, such that
\begin{equation}
\label{boutgammashift}
- P_{\beta}(\gamma_{s}^{-1})
= \frac{\widehat{t_1}}{\gamma_s} +
\frac{\widehat{t_2}}{\gamma_{s}^{2}} +
\frac{\widehat{t_3}}{\gamma_{s}^{3}} +
\ldots .
\end{equation}
The integers $\widehat{t_i}$
are given by the
$\gamma_{s}$-transformation
$T_{\gamma_{s}}: [0,1] \to [0,1],
x \to \{\gamma_{s} x\}$, as in
\eqref{xexpansion}.
Explicitely, 
the digits, all in the alphabet
$\{0,1\}$, are

$\widehat{t_1} = \lfloor \gamma_s (- P_{\beta}(\gamma_{s}^{-1})) \rfloor$,

$\widehat{t_2} 
= \lfloor \gamma_s \{\gamma_s (- P_{\beta}(\gamma_{s}^{-1})) \} \rfloor$,

$\widehat{t_3} 
= \lfloor \gamma_s \{\gamma_s \{\gamma_s (- P_{\beta}(\gamma_{s}^{-1})) \} \} \rfloor, \, 
\ldots$ \,, and depend upon $\gamma_{s}$.

\noindent
Since $\lim_{s \to \infty}
P_{\beta}(\gamma_{s}^{-1})
=
P_{\beta}(\beta^{-1}) = 0$
there exists an increasing  sequence 
$(u_s)_{s \geq W_{\nu}}$ of positive integers,
satisfying
$\widehat{t_1}= \widehat{t_2}
=\ldots = \widehat{t_{u_s -1}} = 0$,
$\widehat{t_{u_s}}=1$,
such that 
the
identity 
between $- P_{\beta}(\gamma_{s}^{-1})$ and
its greedy expansion holds, as:
\begin{equation}
\label{boutgammashift_uZERO}
- P_{\beta}(\gamma_{s}^{-1})
= \frac{\widehat{t_{u_s}}}{\gamma_{s}^{u_s}} +
\frac{\widehat{t_{u_s + 1}}}{\gamma_{s}^{u_s + 1}} +
\frac{\widehat{t_{u_s + 2}}}{\gamma_{s}^{u_s + 2}} +
\ldots .
\end{equation}
The sequence $(u_s)$ is defined by the bounds
\begin{equation}
\label{bound_us}
\bigl|
\beta^{u_s} (P_{\beta}(\gamma_{s}^{-1}))
\bigr|
\geq
\bigl|
\gamma_{s}^{u_s} (P_{\beta}(\gamma_{s}^{-1}))
\bigr| \geq 1
\end{equation}
and
$$\bigl|
\gamma_{s}^{u_s} (P_{\beta}(\gamma_{s}^{-1}))
\bigr| \leq 
\frac{1}{1-\gamma_{s}^{-1}}.$$
Therefore, in \eqref{boutgammashift_uZERO},

\begin{equation}
\label{hgammasSeries}
\widehat{t_i} \in \{0,1\}, \quad i \geq 1,
\qquad
{\rm and}
\qquad
\lim_{s \to +\infty} u_s = +\infty .
\end{equation}

Now the lhs of \eqref{boutgammashift_uZERO} belongs to
$\mathbb{Q}(\gamma_{s})$. 
For conjugating \eqref{boutgammashift_uZERO}
by $\sigma_s$, if 
the image by
$\sigma_s$ of the lhs of
\eqref{boutgammashift_uZERO}
belongs to $\mathbb{Q}(r_s)$,
there are three cases for
the conjugation of the rhs of \eqref{boutgammashift_uZERO}:
\begin{enumerate}
\item[(i)] the rhs of \eqref{boutgammashift_uZERO} is finite (ends in infinitely many zeroes),
\item[(ii-1)] 
the rhs of \eqref{boutgammashift_uZERO} is eventually periodic (infinite and
ultimately periodic),
\item[(ii-2)] the rhs of \eqref{boutgammashift_uZERO} is infinite
and not eventually periodic.
\end{enumerate}
\

{\bf Case (i)}:
say  that
$\frac{\widehat{t_{u_s}}}{\gamma_{s}^{u_s}} +
\frac{\widehat{t_{u_s + 1}}}{\gamma_{s}^{u_s + 1}} +
\ldots+
\frac{\widehat{t_{u_s + N}}}{\gamma_{s}^{u_s + N}}$
is the rhs of \eqref{boutgammashift_uZERO}.
Then its image by $\sigma_s$ is
$\widehat{t_{u_s}} \,r_{s}^{u_s} +
\widehat{t_{u_s + 1}} \,r_{s}^{u_s + 1} +
\ldots+
\widehat{t_{u_s + N}} \, r_{s}^{u_s + N}$
and we have the equality
$$P_{\beta}(r_{s})
=
\sigma_{s}\left(- \sum_{j=0}^{s_c-1}
h''_{j} \,\gamma_{s}^{-j-d-2}
\right)=
- \sum_{j=0}^{s_c-1}
h''_{j} \,r_{s}^{j+d+2}
$$
$$=
\sigma_{s} \left( 
\frac{\widehat{t_{u_s}}}{\gamma_{s}^{u_s}} +
\frac{\widehat{t_{u_s + 1}}}{\gamma_{s}^{u_s + 1}} +
\ldots +
\frac{\widehat{t_{u_s + N}}}{\gamma_{s}^{u_s + N}} 
 \right)
=
\widehat{t_{u_s}} r_{s}^{u_s} +
\widehat{t_{u_s + 1}} r_{s}^{u_s + 1} +
\ldots +
\widehat{t_{u_s + N}} r_{s}^{u_s + N} .$$ 
Conjugation by $\sigma_s$ is done 
term by term.

\

{\bf Case (ii-1)}: 
the rhs of \eqref{boutgammashift_uZERO} 
is eventually periodic. Let us write it
$$
=
\frac{\widehat{t_{u_s}}}{\gamma_{s}^{u_s}} +
\frac{\widehat{t_{u_s + 1}}}{\gamma_{s}^{u_s + 1}} +
\ldots+
\frac{\widehat{t_{u_s + N}}}{\gamma_{s}^{u_s + N}}
+
\sum_{i=0}^{\infty}
\left(
\frac{\widehat{t_{u_s + N + 1}}}{\gamma_{s}^{u_s + N + i q + 1}} +
\frac{\widehat{t_{u_s + N + 2}}}{\gamma_{s}^{u_s + N + i q  +  2}} +
\ldots+
\frac{\widehat{t_{u_s + N + q}}}{\gamma_{s}^{u_s + N+ i q +q}}
\right).$$
The period is not equal to zero. The period length is
$q$.
We have: $|r_s|=|\sigma_{s}(\gamma_{s}^{-1})| < 1$. 
Then it is equal to
$$=
\frac{\widehat{t_{u_s}}}{\gamma_{s}^{u_s}} +
\ldots+
\frac{\widehat{t_{u_s + N}}}{\gamma_{s}^{u_s + N}}
+
\sum_{i=0}^{\infty}
\gamma_{s}^{- i q }
\left(
\frac{\widehat{t_{u_s + N + 1}}}{\gamma_{s}^{u_s + N  + 1}} +
\ldots+
\frac{\widehat{t_{u_s + N + q}}}{\gamma_{s}^{u_s + N +q}}
\right)$$
$$=
\frac{\widehat{t_{u_s}}}{\gamma_{s}^{u_s}} +
\ldots+
\frac{\widehat{t_{u_s + N}}}{\gamma_{s}^{u_s + N}}
+
\frac{1}{1-
\gamma_{s}^{- q }}
\left(
\frac{\widehat{t_{u_s + N + 1}}}{\gamma_{s}^{u_s + N  + 1}} +
\ldots+
\frac{\widehat{t_{u_s + N + q}}}{\gamma_{s}^{u_s + N +q}}
\right)$$
and its image by $\sigma_s$ is
$$=
\widehat{t_{u_s}} r_{s}^{u_s} +
\ldots+
\widehat{t_{u_s + N}} r_{s}^{u_s + N}
+
\frac{1}{1-
r_{s}^{q }}
\left(
\widehat{t_{u_s + N + 1}} r_{s}^{u_s + N  + 1} 
+
\ldots+
\widehat{t_{u_s + N + q}} r_{s}^{u_s + N +q}
\right)$$
$$=
\widehat{t_{u_s}} r_{s}^{u_s} +
\ldots+
\widehat{t_{u_s + N}} r_{s}^{u_s + N}
+
\sum_{i=0}^{\infty}
\left(
\widehat{t_{u_s + N + 1}} r_{s}^{u_s + N+ i q + 1} 
+
\ldots+
\widehat{t_{u_s + N + q}} r_{s}^{u_s + N + i q +q}
\right) .$$
The series can be conjugated term by term
by $\sigma_s$;
in this case we have the identity
$$P_{\beta}(r_{s})
=
\sigma_{s} \left( 
\frac{\widehat{t_{u_s}}}{\gamma_{s}^{u_s}} +
\frac{\widehat{t_{u_s + 1}}}{\gamma_{s}^{u_s + 1}} +
\ldots +
 \right)
=
\widehat{t_{u_s}} r_{s}^{u_s} +
\widehat{t_{u_s + 1}} r_{s}^{u_s + 1} +
\ldots  
.$$

\

{\bf Case (ii-2)}: (the reader may have interest
in \cite{dutykhvergergaugry2} for more details) if the rhs
of \eqref{boutgammashift_uZERO}
is 
not eventually periodic its conjugation
by $\sigma_s$ cannot be done term by term.
This difficulty
is overcome by enlarging 
the alphabet $\{0,1\}$ to a bigger alphabet
$\mathcal{A}$ and by replacing
the R\'enyi expansion
by a
$(\gamma_{s}, \mathcal{A})$-eventually periodic 
representation of $- P_{\beta}(\gamma_{s}^{-1})$.

Let us recall
the definitions.
The $(\delta, \mathcal{A})$-representations for
a given $\delta \in \mathbb{C}$,
$|\delta| > 1$ and
a given alphabet
$\mathcal{A} \subset \mathbb{C}$ finite,
are expressions of the form
$\sum_{k \geq -L} a_k \delta^{-k}$,
$a_k \in \mathcal{A}$, 
for some integer $L$. We denote
$$
{\rm Per}_{\mathcal{A}}(\delta)
:=
\{x \in \mathbb{C}
:
x \,\,{\rm has \,an \,eventually \,periodic}\,
(\delta, \mathcal{A}){\rm -representation}\}.
$$

\begin{theorem}[Kala - Vavra \cite{kalavavra}]
\label{kalavavra}
Let $\delta \in \mathbb{C}$ be an 
algebraic number of degree $d$, 
$|\delta|>1$, 
and
$a_d x^d - a_{d-1} x^{d-1}
- \ldots - a_1 x -a_0 \in \mathbb{Z}[x]$,
$a_0 a_d \neq 0$, be its minimal
polynomial. Suppose that
$|\delta'| \neq 1$ for any conjugate 
$\delta'$ of $\delta$,
Then there exists a finite alphabet $\mathcal{A} \subset \mathbb{Z}$ such that
$\mathbb{Q}(\delta) = {\rm Per}_{\mathcal{A}}(\delta)$.
\end{theorem}

Let us apply Theorem \ref{kalavavra}
to $\delta = \gamma_{s}$. By Proposition 5
in \cite{dutykhvergergaugry} $\gamma_s$
has no conjugate of modulus 1. 
Therefore there exists a finite alphabet
$\mathcal{A} \subset \mathbb{Z}$ 
such that the lhs of \eqref{boutgammashift_uZERO}
be identified with a
$(\gamma_{s}, \mathcal{A})$- representation
which is eventually periodic, for some
integer $u_s \in \zb$:
\begin{equation}
\label{hgammasSeries_alphabetA}
- P_{\beta}(\gamma_{s}^{-1})
= 
\frac{\widehat{t_{u_s }}}{\gamma_{s}^{u_s }} +
\frac{\widehat{t_{u_s + 1 }}}{\gamma_{s}^{u_s + 1 }} +
\frac{\widehat{t_{u_s + 2 }}}{\gamma_{s}^{u_s + 2 }} +
\ldots ~~.
\end{equation}
Being eventually periodic,
the representation
\eqref{hgammasSeries_alphabetA}
can now be
conjugated term by term by $\sigma_s$, 
since $|\sigma_{s}(\gamma_{s}^{-1})|<1$,
as in case (ii-1).
In \eqref{hgammasSeries_alphabetA} the digits
$\widehat{t_{i }}$ belong to 
a symmetrical alphabet
$\mathcal{A} =\{-m, \ldots, 0, \ldots, m\}$;
the integer $m$ is provided
by the rewriting trail, given in
``{\it (Rewriting trail)}. \ref{heightP}",
and
\eqref{Wpolynomial}: we have
$m = \lceil 2((2^d -1) H + 2^d)/3 \rceil$.

\

Indeed, by 
Theorem \ref{kalavavra}
there exist 
a preperiod $R(X) \in \mathcal{A}[X]$, 
a period
$T(X) \in \mathcal{A}[X]$ such that
$$\widehat{W}(\gamma_{s}^{-1})
:=
-P_{\beta}(\gamma_{s}) 
= R(\gamma_{s}^{-1})
+ \gamma_{s}^{-\deg R -1} 
\sum_{j=0}^{\infty} 
\frac{1}{\gamma_{s}^{j (\deg T + 1)}}
T(\gamma_{s}^{-1}),
$$
both polynomials $R$ and $T$ depending upon $s$.
Since the relation
$S_{\gamma_{s}}(\gamma_{s}^{-1})
=
-1 +t_{1} \gamma_{s}^{-1} 
+t_2 \gamma_{s}^{-2} + \ldots 
+ t_{s-1} \gamma_{s}^{-s+1} 
+t_s \gamma^{-s} = 0$ holds, we may assume
$\deg R \leq s-1$,
$\deg T \leq s-1$.
Then, for $X$ specialized at $\gamma_{s}^{-1}$,
we have the identity
\begin{equation}
\label{Wrepresentation}
\widehat{W}(X) = R(X) +
X^{L}\frac{T(X)}{1-X^r}
\end{equation}
for some positive integers $L, r$.
The height of $(1-X^r) \widehat{W}(X)$ is 
$\leq 2 ((2^d -1) H + 2^d)$
and, with 
$\mathcal{A}$
assumed $=\{-m, \ldots,0,\ldots,+m\}$,
the height of 
$(1-X^r) R(X) + X^L T(X)$ is less than
$3 m$. Therefore $m$ is 
$\leq 2 ((2^d -1) H + 2^d)/3$.
We can take 
$m = \lceil 2((2^d -1) H + 2^d)/3 \rceil$.
The alphabet 
$\mathcal{A} = \{-m, \ldots, m\}$
only
depends upon the degree $d$ and the height
$H$ of the polynomial $P_{\beta}$, and
does not depend upon $s$.

\

We now assume
$0 \neq |P(\gamma_s)| \ll 1$.
The
$(\gamma_{s}, \mathcal{A})$-eventually periodic
representation
of $-P(\gamma_{s})$
starts as
$$-P(\gamma_{s})= \widehat{W}(\gamma_{s}^{-1})
=
\frac{\widehat{t_{u_s }}}{\gamma_{s}^{u_s }}
+\frac{\widehat{t_{u_s +1}}}{\gamma_{s}^{u_s +1}}
+\frac{\widehat{t_{u_s +2}}}{\gamma_{s}^{u_s +2}}
+\ldots,\qquad {\rm with}~
|\widehat{t_{j}}| \leq m, j=u_s, u_s +1, \ldots$$
with $\,\widehat{t_{u_s }} \neq 0$.
The exponent $u_s$ appearing
in the first term  
is 
defined in \cite{frougnypelantovasvobodova};
by
Theorem 4, Remarks 5 to 7,
in \cite{frougnypelantovasvobodova},
there exists a positive real
number $\kappa_{\gamma_s , \mathcal{A}} > 0$
such that $u_s$ is the minimal integer 
such that
$$\gamma_{s}^{u_s -1} \geq 
\frac{\kappa_{\gamma,_s  \mathcal{A}}}
{|P(\gamma_s )|} .$$
Since $\lim_{s \to \infty} \gamma_s
= \beta > 1$ and that the alphabet
$\mathcal{A}$ does not depend upon
$s$, from Theorem 4, Remarks 5 to 7,
in \cite{frougnypelantovasvobodova},
we can replace 
$\kappa_{\gamma,_s  \mathcal{A}}$
by a constant $\kappa > 0$,
independent of $s$
(cf also \cite{dutykhvergergaugry2}). 
Thus
$$\lim_{s \to \infty} u_s = +\infty.$$

The 
sequence $(u_{s })$ is here defined by the bounds
\begin{equation}
\label{bound_usA}
\bigl|
\beta^{u_{s} - 1} (P_{\beta}(\gamma_{s}^{-1}))
\bigr|
\geq
\bigl|
\gamma_{s}^{u_s -1} (P_{\beta}(\gamma_{s}^{-1}))
\bigr| \geq \kappa
\end{equation}
and
$$\bigl|
\gamma_{s}^{u_s} (P_{\beta}(\gamma_{s}^{-1}))
\bigr| \leq 
\frac{m}{1-\gamma_{s}^{-1}}.$$
To the collection $(\gamma_s)_{s \geq W_{\nu}}$
is associated the collection
$(\sigma_{s}: \gamma_{s} \to r_s)_{s \geq W_{\nu}}$ 
of $\qb$-automorphisms of $\cb$.
Now, for any $s \geq W_{\nu}$, 
let us conjugate the eventually periodic
representation of
$-P_{\beta}(\gamma_{s}^{-1})$
by $\sigma_{s}$, 
term by term.
For the cases (i) and (ii-1)
we consider
\eqref{boutgammashift_uZERO},
and in the case (ii-2) we consider
\eqref{hgammasSeries_alphabetA}.

Using
\eqref{angularsectornnnbis}
we deduce:

\noindent
{\bf Case (i) and (ii-1):} with the minimal alphabet
$\{-1,0,+1\}$,
\begin{equation}
\label{zerolenticularMAJO}
\bigl|
P_{\beta}(r_{s})
\bigr|
\leq
| r_{s} |^{u_s} \frac{1}{1 - |r_s|}
\leq
\frac{n}{c_{lent}}
(1-\frac{c_{lent}}{n})^{u_s}
,
\end{equation}

\noindent
{\bf Case (ii-2):} with the alphabet
$\mathcal{A} =
\{-m, \ldots, +m\}$,
\begin{equation}
\label{zerolenticularMAJO_A}
\bigl|
P_{\beta}(r_{s})
\bigr|
\leq
| r_{s} |^{u_s} \frac{m}{1 - |r_s|}
\leq
\frac{n \, m}{c_{lent}}
(1-\frac{c_{lent}}{n})^{u_s } .
\end{equation}
In both cases, $\lim_{s \to \infty} u_s =+\infty$.
The rhs of \eqref{zerolenticularMAJO},
resp.
of \eqref{zerolenticularMAJO_A},
tends to 0 if $s$ tends to infinity.
We have: $\lim_{s \to \infty}
P_{\beta}(r_{s}) = 0$.
But $x = \lim_{s \to \infty} r_s$
and $z \to
P_{\beta}(z)$ is continuous.
Therefore there exists $s_0 \geq W_{\nu}$ 
such that $s \geq s_0 \Longrightarrow
\bigl|
P_{\beta}(r_{s})
\bigr| < \nu/2$. Contradiction.

\

The only limit possibility
is $P_{\beta}(x)=0$, for all the lenticular
zeroes $x$ of $f_{\beta}$. 
This provides a collection
of lenticular zeroes
of $P_{\beta}(z)$ in the cusp, and
implies a complete
identification of these zeroes
with the lenticular poles
of $\zeta_{\beta}(z)$.
\end{proof}

To summarize,

\begin{proposition}
\label{identificationFULL}
Under the assumptions of
Definition \ref{lenticularzerodefinition},
if $\Omega$ is a lenticular zero of
$f_{\beta}(z)$, then
$$ f_{\beta}(\Omega) = 0 \qquad \Longrightarrow
\qquad
P_{\beta}(\Omega) =0,$$
where $P_{\beta}$ is the minimal polynomial
of $\beta$.
\end{proposition}

{\it Nota}: the reader will notice that there is
no assumption of continuity of the conjugation 
$\sigma: \beta \to x$ in the proof of 
Proposition \ref{ASversPbeta}.

A direct proof of Proposition
\ref{ASversPbeta} is given 
in \cite{dutykhvergergaugry2}.

\subsection{Continuity of the lenticular Galois conjugates in the cusp}
\label{S5.4bis}

Let $n \geq 260$. 
To $n$ is associated
the set of the lenticular
zeroes of the trinomial
$-1+x+x^n$ of imaginary part $\geq 0$, as:
$$\{\theta_{n},
z_{1,n}, z_{2,n}, \ldots,
z_{J_n , n}\},$$
and the set of Rouch\'e disks
$$\{D_{1,n} D_{2,n} , \ldots, D_{J_n , n}\},$$
the $j$th-disk $D_{j,n}$
being centered at $z_{j,n}$.
For $j=1, \ldots, J_n$,
the disks $D_{j,n}$ satisfy:
$\overline{D_{j,n}} \subset D(0,1)$,
$\overline{D_{j,n}} \cap C(0,1) = \emptyset$,
$z_{j,n-1} \in D_{j,n}$, and,
for any real number 
$\beta \in [\theta_{n}^{-1}
, \theta_{n-1}^{-1}]$, 
$f_{\beta}(z)$ admits an unique zero
$\omega_{j,n}$
in $D_{j,n}$, which is simple
(Theorem \ref{omegajnexistence}).
Since $\beta \to f_{\beta}(z)$ is
injective on $[\theta_{n}^{-1}
, \theta_{n-1}^{-1}]$, the map
$\beta \to \omega_{j,n}$ is well-defined.
Let us denote this map by
$\omega_{j,n}$ and by
$\omega_{j,n}(\beta)$
the image of
$\beta$ (instead of $\omega_{j,n}$).
Let
$\widetilde{\mathcal{S}_{j,n}}:=
\overline{\omega_{j,n}([\theta_{n}^{-1}
, \theta_{n-1}^{-1}])}$ be the adherence
of the image of the closed interval
$[\theta_{n}^{-1}
, \theta_{n-1}^{-1}]$; it is a compact subset
of $D_{j,n}$ such that
$\widetilde{\mathcal{S}_{j,n}}
\cap \partial D_{j,n} = \emptyset$, for which the image
of the left extremity of the interval is 
the center of
the disk $D_{j,n}$:
$\omega_{j,n}(\theta_{n}^{-1})= z_{j,n}$.

\begin{lemma}
\label{zeroestrinomialn}
Let $n \geq 3$.
The two analytic functions
$-1+z+z^{n-1}$
and
$-1+z+z^{n} + \sum_{q=1}^{\infty} z^{q (n-1)+n}$
have the same zeroes 
(with the same multiplicities $=1$)
inside the open unit disk. 
\end{lemma}
\begin{proof}
Let $x_0 , |x_0|<1$, be a zero of the trinomial
$-1+z+z^{n-1}$.
Then $$0=-1 + x_0 + x_{0}^{n-1}=
-1 + x_0 + x_{0}^{n}\frac{1}{x_0}
=-1 + x_0 + x_{0}^{n} (1 + x_{0}^{n-2})
=
-1 + x_0 + x_{0}^{n} + x_{0}^{n+n-2+1}\frac{1}{x_0}.
$$
Let us replace the last $\frac{1}{x_0}$ by
$1 + x_{0}^{n-2}$. And so on, iteratively.
Doing this operation infinitely many times
provides the identity
$$0=-1+x_0+x_{0}^{n-1}
=
-1+x_0+x_{0}^{n} + \sum_{q=1}^{\infty} x_{0}^{q (n-1)+n}.$$
The converse comes from
$$-1+z+z^{n} + \sum_{q=1}^{\infty} z^{q (n-1)+n}
=
\frac{-1+z+z^{n-1}}{1-z^{n-1}},
\qquad |z| < 1.$$
Multiplicities are equal to 1 by
\cite{selmer}.
\end{proof}

Denote
$\delta:= \min \{1-|z| \mid
z \in \cup_{j=1}^{J_n} \widetilde{\mathcal{S}_{j,n}}
\} > 0$.

\begin{proposition}
\label{conttinuity}
Let $n \geq 260$.
For all 
$1 \leq j \leq J_n$, 
the map 
$$\omega_{j,n}:
[\theta_{n}^{-1}, \theta_{n-1}^{-1}]
\to D_{j,n}, \quad
\beta \to \omega_{j,n}(\beta)$$
is continuous.
\end{proposition}

\begin{proof}
Let $\beta_1 , \beta_2$ be two real numbers
in the open interval
$(\theta_{n}^{-1}, \theta_{n-1}^{-1})$,
and assume $\beta_1 < \beta_2$.
To $\beta_1$ , resp. $\beta_2$,
is associated uniquely the
sequence $(t_i)_{i \geq 1} \in \{0,1\}$,
resp. $(t'_i)_{i \geq 1} \in \{0,1\}$,
of the coefficients of the R\'enyi $\beta_1$-expansion
of unity $d_{\beta_1}(1) =0.t_1 t_2 t_3 \ldots$,
resp. of
$d_{\beta_2}(1) =0.t'_1 t'_2 t'_3 \ldots$;
the two Parry Upper functions
$f_{\beta_1}(z)$ 
and $f_{\beta_2}(z)$ 
are constructed from these sequences, by
Definition \ref{parryupperfunction}.
When $\beta_2$ is close to
$\beta_1$, the inequality
`$<$' is translated by the lexicographical inequality
`$<_{lex}$' on the sequences 
$(t_i)_{i \geq 1}$,
resp. $(t'_i)$, by Proposition \ref{betacharacterized}.
The first digits are the same. We 
define the lexicographical
metric $\overline{d}$ by:

$$\overline{d}(\beta_1 , \beta_2 ):=
e^{-r}
\qquad {\rm iff}~\quad
t_1 = t'_1 , t_2 = t'_2 , \ldots,
t_{r} = t'_{r}, 
t_{r+1} \neq t'_{r+1}.$$
To prove the continuity of
$\omega_{j,n}$
it suffices to show that, 
for all $\epsilon > 0$, 
there exists $\eta > 0$ such that
$$\overline{d}(\beta_1 , \beta_2 ) < \eta \qquad
\Longrightarrow \qquad 
|\omega_{j,n}(\beta_1)
- \omega_{j,n}(\beta_2)| < \epsilon,$$
and to establish the continuity
at the extremity
$\theta_{n-1}^{-1}$
of the closed interval
$[\theta_{n}^{-1}, \theta_{n-1}^{-1}]$.

We have
$$f_{\beta_1}(\beta_{1}^{-1})=
f_{\beta_1}(\omega_{j,n}(\beta_1))=
f_{\beta_2}(\beta_{2}^{-1})=
f_{\beta_2}(\omega_{j,n}(\beta_2))=
0.$$ 
Since
$\beta_1 \neq \beta_2$ and that
the disks $D_{j,n}$ have the property that
$f_{\beta_2}(z)$ contains an unique zero in it
(cf Theorem \ref{omegajnexistence}),
\begin{align*}
0 \neq f_{\beta_2}(\omega_{j,n}(\beta_1))
&=
f_{\beta_2}(\omega_{j,n}(\beta_2)
+
(\omega_{j,n}(\beta_1) - \omega_{j,n}(\beta_2))  )\\
&= (\omega_{j,n}(\beta_1) - \omega_{j,n}(\beta_2))
\Bigl[
\displaystyle\sum_{q=1}^{\infty}
\frac{f_{\beta_2}^{(q)}(\omega_{j,n}(\beta_{2}))}{
q!} (\omega_{j,n}(\beta_1)
-
\omega_{j,n}(\beta_2))^{q-1}
\Bigr].
\end{align*}
Now by Theorem \ref{omegajnexistence} 
the unique zero
$\omega_{j,n}(\beta_{2}^{-1})$ in $D_{j,n}$
is simple.
Thus $f_{\beta_2}^{'}(\omega_{j,n}(\beta_{2}))
\neq 0$ and
the function
$$z \to f_{\beta_2}(z)/(z- \omega_{j,n}(\beta_2)),$$
extended by continuity by
$f_{\beta_2}^{'}(\omega_{j,n}(\beta_{2}))$
at
$\omega_{j,n}(\beta_2)$,
does not take the value $0$ on the compact
$\widetilde{\mathcal{S}_{j,n}}$.
Hence, the function
$$z \to 
\left|
\sum_{q=1}^{\infty}
\frac{f_{\beta_2}^{(q)}(\omega_{j,n}(\beta_{2}))}{
q!} (z
-
\omega_{j,n}(\beta_2))^{q-1}
\right|$$
admits an infimum, say $\widetilde{\mu_{j}}$
on $\widetilde{\mathcal{S}_{j,n}}$. 
Denote $\widetilde{\mu} := \min \{\widetilde{\mu_{j}}
\mid j = 1, 2, \ldots J_n\} > 0$ the
infimum, common to all the disks $D_{j,n}$.

Now, with
$\overline{d}(\beta_1 , \beta_2 ):=
e^{-r}$,
$$f_{\beta_2}(\omega_{j,n}(\beta_1))
=
(f_{\beta_2}-f_{\beta_1})(\omega_{j,n}(\beta_1))
=
\sum_{k=r+1}^{\infty}
(t_{k} - t'_{k}) 
(\omega_{j,n}(\beta_1))^k .
$$
Therefore
$$|f_{\beta_2}(\omega_{j,n}(\beta_1))|
\leq
\sum_{k=r+1}^{\infty} |\omega_{j,n}(\beta_1)|^k
\leq
\sum_{k=r+1}^{\infty} (1-\delta)^k 
\leq
(1-\delta)^{r+1} \frac{1}{\delta}.
$$
Hence
$$
\left|\omega_{j,n}(\beta_1) - 
\omega_{j,n}(\beta_2)
\right|
\leq
\frac{1}{\widetilde{\mu}} (1-\delta)^{r+1} 
\frac{1}{\delta}.
$$
It suffices to take $r$ large enough
to have
the property of continuity.

Let us consider
the case of the extremity 
$\beta_2 = \theta_{n-1}^{-1}$. 
By Lemma \ref{zeroestrinomialn}, 
since 
$\theta_{n-1}$  is root of 
$-1+x+x^{n-1}$ it
satisfies:
$$0=
-1+\theta_{n-1}+\theta_{n-1}^{n} + \sum_{q=1}^{\infty} 
\theta_{n-1}^{q (n-1)+n}.$$
The metric $\overline{d}$ can be extended
to $(\theta_{n}^{-1}, \theta_{n-1}^{-1}]
\times
(\theta_{n}^{-1}, \theta_{n-1}^{-1}]$
since the sequence of 
coefficients of the power series
$-1+z+z^n + \sum_{q \geq 1} z^{q (n-1)+n}$
is obviously 
in the adherence of the
set $$\{(t_i) \mid 
{\rm the ~digits~}
t_i {\rm ~being ~those ~of}  
~~d_{\beta}(1) 
{\rm ~~for~ all}~
\beta \in (\theta_{n}^{-1}, \theta_{n-1}^{-1})\},$$
and not an isolated point.
The case is simpler for the other extremity
$\beta_1 = \theta_{n}^{-1}$.

\end{proof}

\begin{remark}
For every $j=1,\ldots, J_n$,
the image
$\omega_{j,n}(
[\theta_{n}^{-1}, \theta_{n-1}^{-1}] \cap \overline{\qb})$
is dense in $\widetilde{\mathcal{S}_{j,n}}$.
The restriction of $\omega_{j,n}$ to the
collection of the algebraic
integers in 
$[\theta_{n}^{-1}, \theta_{n-1}^{-1}]$ is continuous,
i.e. the lenticular
Galois conjugates 
$\omega_{j,n}(\beta)$ of an algebraic integer
$\beta$
are continuous functions of $\beta$.
Comparatively
the degree $\deg \beta$ has good chances to
vary very chaotically with $\beta$, if
$\beta$ varies continuously.

Now, seeking very small Mahler measures,
below Lehmer's number,
calls for the existence of reciprocal
algebraic integers in the intervals
$(\theta_{n}^{-1}, \theta_{n-1}^{-1})$.

Suppose the existence of a reciprocal
algebraic integer
$\beta > 1$ in the interval
$(\theta_{n}^{-1}, \theta_{n-1}^{-1})$.
As a consequence of Proposition
\ref{APversfbeta} and
Proposition \ref{ASversPbeta},
the minimal polynomial
$P_{\beta}(X)$ of $\beta$
would be the product of two components,
the lenticular part, with 
roots inside and outside the unit circle,
say
$$
(X-\beta) (X-1/\beta) \times \!\!\!
\prod_{\stackrel{{\rm lenticular ~roots}}
{{\rm Im}\beta^{(j)} > 0 , \,
|\beta^{(j)}|<1}} (X-\beta^{(j)})(X-\overline{\beta^{(j)}})
(X-\beta^{(j)^{-1}})(X-\overline{(\beta^{(j)^{-1}}})$$
by the non-lenticular part
$$\prod_{{\rm non-lenticular ~roots}} (X-\beta^{(j)}),$$
with the identifications between the Galois conjugates and the lenticular
poles of $\zeta_{\beta}(z)$:
$$\beta^{(j)} = \omega_{j,n}(\beta),
\qquad j = 1, 2, \ldots, J_n .$$
The problem of the localization
of the lenticular Galois conjugates
$\beta^{(j)}$ of $\beta$ then
calls first for a geometrical charaterization
of the supports which contain 
these conjugates, that is
of the compact sets
$\widetilde{\mathcal{S}_{j,n}}$
in the Rouch\'e disks $D_{j,n}$.

In \cite{flattolagariaspoonen}
Flatto, Lagarias and Poonen study the
continuity of the
modulus of the first root
\begin{equation}
\label{mapflatto}
\beta \to |\omega_{1,n}(\beta)|,
\end{equation}
over the union of the intervals
$[\theta_{n}^{-1}, \theta_{n-1}^{-1}]$.
Their Figure 1 summarizes
the fluctuations of
 the map \eqref{mapflatto},
 for $1 <\beta < 2$.
 The curve given by
Figure 1 in \cite{flattolagariaspoonen}
suggests that
the map $\omega_{1,n}$ is injective on
$(\theta_{n}^{-1}, \theta_{n-1}^{-1})$
but that the
union
$\widetilde{\mathcal{S}_{1,n-1}} \cup
\widetilde{\mathcal{S}_{1,n}} \cup
\widetilde{\mathcal{S}_{1,n+1}}$,
for all $n$ large enough,
is a self-intersecting object
\cite{dutykhvergergaugry3}.
\end{remark}

\subsection{Minoration of the Mahler measure: a continuous lower bound}
\label{S5.5}

Let $n \geq 260$ and
$\beta > 1$ be an algebraic integer
such that
$\theta_{n}^{-1} < \beta <
\theta_{n-1}^{-1}$,
with ${\rm M}(\alpha)
< 1.176280\ldots$. The 
factorization of the
minimal polynomial of $\beta$
is
\begin{equation}
\label{polyminimalLENTICULUS}
P_{\beta}(z)
~=~
\prod_{\gamma \in \lc_{\beta}}
(z - \gamma) ~\times
~
\prod_{\gamma \not\in \lc_{\beta}}
(z - \gamma).
\end{equation}
From Proposition \ref{APversfbeta}
and Proposition \ref{ASversPbeta}
the lenticular zeroes $\omega_{j,n}$
of the Parry Upper function
$f_{\beta}(z)$ are zeroes of the 
minimal polynomial of $\beta$.
The
lenticulus of conjugates is 
$\lc_{\beta} := \{
\beta^{-1}\}
\cup 
\bigcup_{j=1}^{J_n}
(\{\omega_{j,n}\} 
\cup 
\{\overline{\omega_{j,n}}\}) 
$. 
In this Section we
investigate the product,
called {\em lenticular
Mahler measure of $\beta$}, defined by
\begin{equation}
\label{mmrDEF_}
{\rm M}_{r}(\beta) ~:=~
\prod_{\gamma ~
\mbox{{\tiny conjugate of}}~ \beta^{-1},
|\gamma|<1 \atop lenticular} |\gamma|^{-1}
\end{equation}
and its asymptotic expansion.
The subscript ``r" added to the ``M"
of the Mahler measure
stands for ``reduced to the lenticulus".
We have:  
${\rm M}_{r}(\beta) \leq {\rm M}(\beta)$.

First let us
complete 
Theorem \ref{splitBETAdivisibility}.

\begin{proposition}
\label{splitBETAdivisibility+++}
Let $n \geq 260$ and
$\beta > 1$ a reciprocal
algebraic integer
such that
$\dyg(\beta)=n$,
with ${\rm M}(\alpha)
< 1.176280\ldots$.
Let $\Omega_n$ be the 
subdomain of the open unit disk
defined by the union of
\begin{equation}
\label{cerclescalottes_0complet}
\Bigl\{z \mid |z| < 1 - \frac{c_n}{n},
|\arg z| > \arg z_{J_n , n}
+ \frac{\pi}{n \, a_{max}}
\Bigr\}, 
\end{equation}
\begin{equation}
\label{cerclescalottes_1complet}
\Bigl\{z \mid |z| < 1 - 
\frac{c_{lent}}{n} , \,
|\arg z| < \arg(z_{J_n , n}) 
+ \frac{\pi}{n a_{max}}
\Bigr\},
\end{equation}
and,
for $\displaystyle j = 
J_n + 1,
\ldots, 2 J_n - H_n + 1$, 
\begin{equation}
\label{cerclescalottescomplet}
\displaystyle \frac{\pi |z_{j,n}|}
{n \, s_{j,n}} < |z - z_{j,n}| ,
\qquad \mbox{with}~~
s_{j,n} = a_{\max}
\Bigl[
1 +
\frac{a_{\max}^{2}
(j - J_n)^2}{\pi^2 \,
J_{n}^{2}}
\Bigr]^{-1/2} ,
\end{equation}
and symmetrically by complex conjugation
with respect to the real axis.

Then the minimal polynomial
$P_{\beta}(X)$ of $\beta$
is fracturable in 
the domain $\Omega_n$ in the sense that
the invertible power series
$U_{\beta}(z) = 
- \zeta_{\beta}(z) P_{\beta}(z) \in \zb[[z]]$,
satisfying 
$P_{\beta}(z)= U_{\beta}(z)\, f_{\beta}(z)$,
is 
not constant, does not vanish and is
holomorphic
in $\Omega_n$.
It satisfies
\begin{equation}
\label{nevervanishes}
U_{\beta}(\omega_{j,n}) = 
\frac{P'_{\beta}(\omega_{j,n})}
{f'_{\beta}(\omega_{j,n})}
\neq 0, ~j = 1, 2, \ldots, J_n , \quad
U_{\beta}(\beta^{-1}) = 
\frac{P'_{\beta}(\beta^{-1})}
{f'_{\beta}(\beta^{-1})}
\neq 0,
\end{equation}
and obeys the Carlson-Polya dichotomy, 
simultaneously
with $f_{\beta}(z)$ as in Theorem
\ref{splitBETAdivisibility}.
\end{proposition}

\begin{proof}
The domain of holomorphy of
$U_{\beta}(z)$ contains
$D(0, \theta_{\dyg(\beta)-1})$
by Theorem \ref{splitBETAdivisibility}.
The roots of the minimal poynomial
$P_{\beta}(z)$ are simple.
The roots of $f_{\beta}(z)$
in 
$\Omega_n$
are also simple by Theorem 
\ref{omegajnexistence}
and Theorem \ref{thm2lenticuli}.
The lenticular roots
of $f_{\beta}(z)$
coincide with roots
of $P_{\beta}(z)$
by Proposition 
\ref{APversfbeta}
and Proposition 
\ref{ASversPbeta}. 
We deduce 
the fracturability
of $P_{\beta}(z)$ on
$\Omega_n$.
The relations
\eqref{nevervanishes} follow from
the derivatives of the
identity
$P_{\beta}(z) = 
U_{\beta}(z) \times f_{\beta}(z)$.
The unit circle is the natural boundary
of $U_{\beta}(z)$ if and only if
$\beta$ is not a Parry number, 
by Theorem \ref{splitBETAdivisibility}.
\end{proof}

The modulus of 
the second 
smallest root
of $f_{\beta}(z)$ is a continuous
function of 
$\beta$ 
\cite{flattolagariaspoonen}.
Theorem \ref{Qautocontinus} 
extends this result.

\begin{theorem}
\label{Qautocontinus}
Let $n \geq 260$. Let
$\beta > 1$ be a reciprocal 
algebraic integer
such that
$\dyg(\beta)
=
n$, 
with ${\rm M}(\alpha)
< 1.176280\ldots$.
The product, called lenticular
Mahler measure of $\beta$, defined by
\begin{equation}
\label{mmrDEF}
{\rm M}_{r}(\beta) ~:=~
\prod_{\omega \in \lc_{\beta}}
|\omega|^{-1}
\end{equation}
is a continuous function
of $\beta$ on the open interval
$(\theta_{n}^{-1}, \theta_{n-1}^{-1})$, 
which admits
the following left and right limits
\begin{equation}
\label{mmmr1}
\lim_{\beta \to \theta_{n-1}^{{-1}^{\mbox{}\, -}}}
{\rm M}_{r}(\beta)
=
\prod_{
{\omega \in \lc_{\theta_{n-1}^{-1}}}}
|\omega|^{-1}
=
\theta_{n-1}^{-1}   \times
\prod_{\stackrel{1 \leq j \leq J_{n}}
{z_{j,n-1} \in \lc_{\theta_{n-1}^{-1}}}}
|z_{j,n-1}|^{-2} ,
\end{equation}
\begin{equation}
\label{mmmr2}
\lim_{\beta \to \theta_{n}^{{-1}^{\mbox{}\, +}}}
{\rm M}_{r}(\beta)
=
\prod_{
{\omega \in \lc_{\theta_{n}^{-1}}}}
|\omega|^{-1}
= 
\theta_{n}^{-1}   \times
\prod_{\stackrel{1 \leq j \leq J_{n}}
{z_{j,n} \in \lc_{\theta_{n}^{-1}}}}
|z_{j,n}|^{-2} .
\end{equation}
The discontinuity (jump) of
${\rm M}_{r}(\beta)$ 
at the Perron number
$\theta_{n-1}^{-1}$, given in the multiplicative form by
\begin{equation}
\label{mmmr3}
\frac{\lim_{\beta \to \theta_{n-1}^{{-1}^{\mbox{}\, -}}}
{\rm M}_{r}(\beta)}{\lim_{\beta \to \theta_{n-1}^{{-1}^{\mbox{}\, +}}}
{\rm M}_{r}(\beta)} ~~=~~ |z_{J_n, n-1}|^{-2},
\end{equation}
tends to 1 (i.e. disappears at infinity)
when
$n = \dyg(\beta)$ tends to infinity.\end{theorem}

\begin{proof}
From Corollary
\ref{zeroesParryUpperfunctionContinuity}
in \S \ref{S3.4} all the maps
$\beta \to \omega(\beta) \in \lc_{\beta}$
are continuous.
Now the identification
of the zeroes of the Parry Upper function
$f_{\beta}(z)$ as conjugates of
$\beta$, from 
Theorem \ref{splitBETAdivisibility+++},
allows to consider this
continuity property
as a continuity property over the
conjugates of $\beta$ which
define the lenticulus
$\lc_{\beta}$.
As a consequence all the maps
$\beta \to |\omega(\beta)| \in \lc_{\beta}$, 
are continuous, as well as
their product \eqref{mmrDEF}.
Let $1 \leq j \leq J_n$.
Let us prove that
$z_{j,n-1} \in D_{j,n}
=
\{z \mid |z-z_{j,n}| < \frac{\pi |z_{j,n}|}{n \, a_{\max}}\}$.
Indeed,
$$|z_{j,n}| = 1 + 
\frac{1}{n} \lo (2 \sin(\frac{\pi \, j}{n}))
 +..., \qquad \arg(z_{j,n}) = ...$$
 and
 $$|z_{j,n-1}| = 1 + 
\frac{1}{n-1} \lo (2 \sin(\frac{\pi \, j}{n-1}))
 +..., \qquad \arg(z_{j,n-1}) =...$$
 so that, easily, 
\begin{equation}
\label{zedeJNmoinsUNdedansDjn}
|z_{j,n} - z_{j,n-1}|
 <~
\frac{\pi |z_{j,n}|}{n \, a_{\max}}.
\end{equation}
The image of the
interval $(\theta_{n}^{-1}, 
\theta_{n-1}^{-1}) \cap 
\mathcal{O}_{\overline{\qb}}$
by a map
$\beta \to \omega_{j,n}(\beta) \in \lc_{\beta}$
is a curve in $D_{j,n}$
over $\mathcal{O}_{\overline{\qb}}$
with extremities 
$z_{j,n}$ and $z_{j,n-1}$, both
in $D_{j,n}$ 
by \eqref{zedeJNmoinsUNdedansDjn}.  
This curve does not
intersect itself. Indeed, 
if it would be 
a self-intersecting curve we would have,
for two distinct algebraic integers
$\beta$ and $\beta'$, the same
conjugate in $D_{j,n}$,
what is impossible since $P_{\beta}$
and $P'_{\beta}$ are both irreducible, and
therefore 
they cannot have a root
in common. This curve does not ramify either
by the uniqueness property imposed locally
by 
the Theorem of Rouch\'e. 
We deduce the left limit \eqref{mmmr1}
and the right limit
\eqref{mmmr2}
by continuity.
\end{proof}

\begin{remark}
\label{measuremahlercontinuityAUXPerrons}
Decomposing the Mahler measure gives
$${\rm M}(\beta) =
\prod_{\omega \in \lc_{\beta}} |\omega|^{-1}
\times
\!\!\! \prod_{\stackrel{\omega \not\in \lc_{\beta}, |\omega|< 1}{P_{\beta}(\omega)=0} } 
|\omega|^{-1} .
$$
Theorem \ref{Qautocontinus}, 
for which the Rouch\'e method
has been applied, shows the continuity
of the partial
product
$\beta \to
\prod_{\omega \in \lc_{\beta}} |\omega|^{-1}$, associated with the
identified lenticulus of 
conjugates of $\beta$, 
with $\beta$ running over each open
interval of extremities 
two successive 
Perron numbers
$\theta_{n}^{-1}$.
It is very 
probable that a method finer than the method of
Rouch\'e would lead to a higher value of
$J_n$, to more zeroes of
$f_{\beta}(z)$
identified as conjugates of $\beta$,
and the 
disappearance of the discontinuities
(jumps) in \eqref{mmmr3}.
\end{remark}

\begin{theorem}
\label{MahlerMINORANTreal}
Let $\beta > 1$ be a reciprocal 
algebraic integer
such that 
$\dyg(\beta) \geq 260$, 
with ${\rm M}(\alpha)
< 1.176280\ldots$.
Denote $\kappa = \kappa(1,a_{\max})$.
The Mahler measure ${\rm M}(\beta)$
is bounded from below by
the lenticular Mahler measure of $\beta$
as
$${\rm M}(\beta) =
~ {\rm M}_{r}(\beta) \times
\!\!\!\prod_{\stackrel{\omega \not\in \lc_{\beta}, |\omega|< 1}{P_{\beta}(\omega)=0}} 
|\omega|^{-1}
\geq 
~ {\rm M}_{r}(\beta).$$
Denoting
$$
\Lambda_r := \exp\Bigl(
\frac{-1}{\pi}
\int_{0}^{2 \arcsin(\frac{\kappa}{2})}
 \, \lo\bigl(
2 \sin\bigl(
\frac{x}{2}
\bigr)
\bigr)
dx
\Bigr)
=
1.16302\ldots,$$
and
$$
\mu_r :=
\exp\Bigl(
\frac{-1}{\pi}
\int_{0}^{2 \arcsin(\frac{\kappa}{2})}
\! \!\lo\Bigl[\frac{1 + 2 \sin(\frac{x}{2})
-
\sqrt{1 - 12 \sin(\frac{x}{2}) 
+ 4 (\sin(\frac{x}{2}))^2}}{8 \sin(\frac{x}{2})}
\Bigr] dx
\Bigr)$$
$$
=
0.992337\ldots,
$$
the lenticular Mahler measure 
${\rm M}_{r}(\beta)$ of $\beta$
admits a liminf and a limsup 
when $\beta$ tends to $1^{+}$, 
equivalently
when $\dyg(\beta)$ tends to infinity,
respectively bounded from below 
and above as
\begin{equation}
\label{inekinf}
\liminf_{\dyg(\beta) \to +\infty} 
\!\!{\rm M}_{r}(\beta)
~\geq~
\, 
\Lambda_r \cdot \mu_r
~=~ 1.15411\ldots ,
\end{equation}
\begin{equation}
\label{ineksup}
\limsup_{\dyg(\beta) \to +\infty} 
\!\!{\rm M}_{r}(\beta)
~\leq~
\, 
\Lambda_r  \cdot \mu_{r}^{-1}
~=~
1.172\ldots
\end{equation}

Then the ``limit
minorant" of the Mahler
measure ${\rm M}(\beta)$ of $\beta$, 
$\beta > 1$ running over
$\mathcal{O}_{\overline{\qb}}$,
when
$\dyg(\beta)$ tends to infinity, is given by
\begin{equation}
\label{inekmahler}
\liminf_{\dyg(\beta) 
\to \infty}{\rm M}(\beta) 
~\geq~
\Lambda_r \cdot \mu_r = 
1.15411\ldots
\end{equation}
\end{theorem}

\begin{proof}
The value
$2 \arcsin(\kappa(1,a_{\max})/2)
= 0.171784\ldots$ is given
by Proposition
\ref{argumentlastrootJn},
and $a_{\max} = 5.8743\ldots$ by 
Theorem \ref{cercleoptiMM}. 
The variations of the 
Mahler measure ${\rm M}(\beta)$
of $\beta$ can be fairly large when
$\beta$ approaches $1^{+}$. 
On the contrary 
the lenticular Mahler measure
${\rm M}_{r}(\beta)$ is a continuous
function of $\beta$ on
$(1, \theta_{260}^{-1})$
except at the point discontinuities
which are
the Perron numbers
$\theta_{n}^{-1}$ by 
\eqref{mmmr1}, \eqref{mmmr2}
and \eqref{mmmr3},
writing
$n = \dyg(\beta)$ for short.

First, by Proposition
\ref{argumentlastrootJn}
let us observe that 
the Riemann-Stieltjes sum
$$
S(f,n) := - 2 \sum_{j = 1}^{J_n}
\frac{1}{n} \, \lo \bigl(2 \, \sin\bigl(\frac{\pi j}{n}\bigr)  \bigr)
=
\frac{-1}{\pi} \,
\sum_{j= 1}^{J_n}
(x_{j} - x_{j-1}) f(x_j)$$
with
$x_j = \frac{2 \pi j}{n}$
and
$f(x) := \lo \bigl(2 \, \sin\bigl(\frac{x}{2}\bigr)  \bigr)
$
converges to the limit
\begin{equation}
\label{SfnLAMBDAr}
\lim_{n \to \infty} S(f,n) ~=~ \frac{-1}{\pi} \,
\int_{0}^{0.171784\ldots} f(x) dx 
~=~ 
\lo \, \Lambda_r 
~=~ \lo (1.16302\ldots).
\end{equation}
This limit is a log-sine
integral 
\cite{borweinborweinstraubwan}
\cite{borweinstraub}.
Let us now show how $\Lambda_r$
is related to  
$\liminf_{\dyg(\beta) \to \infty}{\rm M}_{r}(\beta)$  
and 
$\limsup_{\dyg(\beta) \to \infty}{\rm M}_{r}(\beta)$
to deduce
\eqref{inekinf} and \eqref{ineksup}.

Taking only into account the lenticular 
zeroes of $P_{\beta}(z)$, which constitute
the lenticulus $\lc_{\beta}$,
from Theorem \ref{cercleoptiMM}
and
Proposition \ref{cercleoptiMMBUMP}, 
we obtain
$$
\lo {\rm M}_{r}(\beta) =
- \lo (\frac{1}{\beta}) - 
2 \sum_{j=1}^{J_n} \lo |\omega_{j,n}|
= \lo (\frac{1}{\beta}) - 
2 \sum_{j=1}^{J_n} \lo |(\omega_{j,n}-z_{j,n})
+z_{j,n}|
$$
\begin{equation}
\label{logmmmr}
=~  \lo (\beta) - 
2 \sum_{j=1}^{J_n} \lo |z_{j,n}|
- 2 \sum_{j=1}^{J_n} 
\lo \bigl|1 + \frac{\omega_{j,n}-z_{j,n}}
{z_{j,n}}\bigr|.
\end{equation}
Obviously the first term 
of \eqref{logmmmr} tends to 0 when
$\dyg(\beta)$ tends to $+\infty$
since $\displaystyle \lim_{n \to \infty} \theta_n = 1$
(Proposition \ref{thetanExpression}).
Let us turn to the 
third summation in \eqref{logmmmr}.
The $j$-th root 
$\omega_{j,n} \in \lc_{\beta}$ of
$f_{\beta}(z)$ 
is the unique root
of $f_{\beta}(z)$ in the disk
$D_{j,n}
=\{z \mid |\omega_{j,n}-z_{j,n}| < \frac{\pi |z_{j,n}|}{n \, a_{\max}}\}$.
From Theorem \ref{omegajnexistence} 
we have the more precise localization 
in $D_{j,n}$:
$|\omega_{j,n}-z_{j,n}|
< \frac{\pi |z_{j,n}|}{n \, a_{j, n}}$
for $j = \lceil v_n \rceil, 
\ldots, J_n$ (main angular sector),
with
$$D(\frac{\pi}{a_{j,n}}) 
=
\lo\Bigl[\frac{1 + B_{j,n}
-
\sqrt{1 - 6 B_{j,n} 
+ B_{j,n}^2}}{4 B_{j,n}}
\Bigr]
$$ 
and 
$B_{j,n} =
2 \sin(\frac{\pi j}{n})
\Bigl(
1 - \frac{1}{n} \lo (2 \sin(\frac{\pi j}{n}))
\Bigr)
$ (from \eqref{petitcercle_jmainDBJN}).
  
For $j = \lceil v_n \rceil, 
\ldots, J_n$
the following inequalities hold:
$$1 - \frac{1}{n} D(\frac{\pi}{a_{j,n}})
\leq
|1 + \frac{\omega_{j,n}-z_{j,n}}
{z_{j,n}}|
\leq 
1 + \frac{1}{n} D(\frac{\pi}{a_{j,n}}),$$
up to second order terms.  
Let us apply
the remainder Theorem
of alternating series:
for $x$ real,
$|x| < 1$,
$|\lo(1+x) - x | ~\leq~ \frac{x^2}{2}$. 
Then the third summation 
in \eqref{logmmmr} satisfies 
$$-2 \,\lim_{n \to \infty} 
\sum_{j=1}^{J_n}
\frac{1}{n}
\lo\Bigl[\frac{1 + 2 \sin(\frac{\pi j}{n})
-
\sqrt{1 - 12 \sin(\frac{\pi j}{n}) 
+ 4 (\sin(\frac{\pi j}{n}))^2}}{8 \sin(\frac{\pi j}{n})}
\Bigr]
$$ 
\begin{equation}
\label{inflenticulusMAHLERm}
\leq
\liminf_{n \to \infty}
\left(
- 2 \sum_{j=1}^{J_n} 
\lo \bigl|1 + \frac{\omega_{j,n}-z_{j,n}}
{z_{j,n}}\bigr|\right)
\end{equation}
and
$$
\limsup_{n \to \infty}
\left(- 2 \sum_{j=1}^{J_n} 
\lo \bigl|1 + \frac{\omega_{j,n}-z_{j,n}}
{z_{j,n}}\bigr|
\right) \leq
$$
\begin{equation}
\label{suplenticulusMAHLERm}
+ 2 \,\lim_{n \to \infty} 
\sum_{j=1}^{J_n}
\frac{1}{n}
\lo\Bigl[\frac{1 + 2 \sin(\frac{\pi j}{n})
-
\sqrt{1 - 12 \sin(\frac{\pi j}{n}) 
+ 4 (\sin(\frac{\pi j}{n}))^2}}{8 \sin(\frac{\pi j}{n})}
\Bigr]
\end{equation}
Let us convert the limits to
integrals. 
The Riemann-Stieltjes sum
$$
S(F,n) :=
- 2 \sum_{j = 1}^{J_n}
\frac{1}{n} \, \lo \Bigl[
\frac{1 + 2 \sin(\frac{\pi j}{n})
-
\sqrt{1 - 12 \sin(\frac{\pi j}{n}) 
+ 4 (\sin(\frac{\pi j}{n}))^2}}{8 \sin(\frac{\pi j}{n})}
\Bigr]
$$
$$
=
\frac{-1}{\pi} \,
\sum_{j= 1}^{J_n}
(x_{j} - x_{j-1}) F(x_j)$$
with
$x_j = \frac{2 \pi j}{n}$
and
$F(x) := \lo \Bigl[\frac{1 + 2 \sin(\frac{x}{2})
-
\sqrt{1 - 12 \sin(\frac{x}{2}) 
+ 4 (\sin(\frac{x}{2}))^2}}{8 \sin(\frac{x}{2})}
\Bigr]$
converges to the limit
\begin{equation}
\label{limiteF}
\lim_{n \to \infty} S(F,n) 
= \frac{-1}{\pi} \,
\int_{0}^{0.171784\ldots} F(x) dx 
= 
\lo \, \mu_r 
\quad \mbox{with}
\quad
\mu_r ~=~ 0.992337\ldots.
\end{equation}
From the inequalities 
\eqref{inflenticulusMAHLERm}
and
\eqref{suplenticulusMAHLERm},
with the limit
\eqref{limiteF} as an integral,
and 
by taking the exponential of
\eqref{logmmmr}, we
obtain 
the two multiplicative factors
$\mu_r$
and
$\mu_{r}^{-1}$
of $\Lambda_r$
in
\eqref{inekinf}, resp. in
\eqref{ineksup}.

Let us show that the second 
summation in \eqref{logmmmr}
gives the term
$\Lambda_r$ 
in the 
inequalities \eqref{inekinf} 
and \eqref{ineksup}, 
when $n$ tends to infinity.
From \eqref{SfnLAMBDAr} 
it will suffice to show
that
\begin{equation}
\label{limmm}
\lim_{n \to \infty} S(f,n) =   
-2 \, \lim_{n \to \infty} \sum_{j=1}^{J_n} \lo |z_{j,n}| 
\end{equation}
The identity \eqref{limmm}
only concerns the roots
of the trinomials $G_n$. It was already 
proved to be true, but 
with $\lfloor n/6\rfloor$ 
instead of $J_n$ as maximal index $j$, 
in the summation,
in \cite{vergergaugry6} \S 4.2, pp 111--115.
The arguments of the proof are the same,
the domain of integration being now
$(0, \lim_{n \to \infty} 2 \pi \frac{J_n}{n}]$
given by
Proposition \ref{argumentlastrootJn}. 
\end{proof}

\subsection{Poincar\'e asymptotic expansion of the lenticular Mahler measure}
\label{S5.6}

The aim of this subsection is to prove 
Theorem \ref{Lrasymptotictheorem}, 
in the continuation of the last paragraph.

The logarithm of the
lenticular Mahler measure 
${\rm M}_{r}(\beta)$
of
$\beta > 1$, with
$\dyg(\beta) \geq 260$, given 
by \eqref{logmmmr}, 
admits the lower bound 
\begin{equation}
\label{logmmmrMINORANT}
L_{r}(\beta) 
=  \lo (\beta) - 
2 \sum_{j=1}^{J_n} \lo |z_{j,n}|
- 2 \sum_{j=1}^{\lfloor v_n \rfloor} 
\lo (1 + \frac{\pi}{n \, a_{max}})
- 2 \sum_{j=\lceil v_n \rceil}^{J_n} 
\lo (1 + \frac{\pi}{n \, a_{j,n}})
\end{equation}
which 
is only a function
of $n = \dyg(\beta)$, where
$(a_{j,n})$ is given 
by Theorem \ref{omegajnexistence},
the sequence $(v_n)$ by the Appendix,
and $J_n$ by Definition
\ref{Jndefinition}
and Proposition \ref{argumentlastrootJn}. 
From \eqref{inflenticulusMAHLERm},
\eqref{limiteF}
and \eqref{limmm}, the limit is
$\lim_{\dyg(\beta) \to \infty}
L_{r}(\beta) = \lo \Lambda_r + \lo \mu_r$.
In Theorem \ref{Lrasymptotictheorem}, 
we will gather the asymptotic contributions 
of each term and obtain  
the asymptotic expansion 
of $L_{r}(\beta)$
as a function of $n$.

(i) First term in \eqref{logmmmrMINORANT}:
from Lemma \ref{remarkthetan}
and 
Theorem \ref{betaAsymptoticExpression},
\begin{equation}
\label{sz1}
\lo (\beta) = 
\frac{\lo n}{n}(1 - \lambda_n) +
\frac{1}{n} O\left(
\left(
\frac{\lo \lo n}{\lo n}
\right)^2 
\right)
=
O\left(
\frac{\lo n}{n}
\right);
\end{equation}

(ii) second term in \eqref{logmmmrMINORANT}: 
from Proposition
\ref{zedeJImodulesORDRE3},
$\displaystyle
\sum_{j= \lceil v_n \rceil}^{J_n}
\lo |z_{j,n}| =$
$$
\sum_{j= \lceil v_n \rceil}^{J_n} \lo \Bigl(
1 + \frac{1}{n} \, \lo \bigl(2 \, \sin\bigl(\frac{\pi j}{n}\bigr)  \bigr)
+ \frac{1}{2 n}
\left( \frac{\lo \lo n}{\lo n} \right)^2
+
 \frac{1}{n} O\left(
\frac{(\lo \lo n)^2}{(\lo n)^3}
\right) \Bigr)$$
with the constant 1 involved in the Big O. Let us 
apply the remainder Theorem of alternating series: 
for $x$ real, $|x|<1$,
$\left| \lo (1+x)  - x \right| \leq 
\frac{x^2}{2}$.
Then
$$\left|
\sum_{j= \lceil v_n \rceil}^{J_n}
\lo |z_{j,n}|
-
\sum_{j= \lceil v_n \rceil}^{J_n}
\frac{1}{n} \, \lo \bigl(2 \, \sin\bigl(\frac{\pi j}{n}\bigr)  \bigr)
- \sum_{j= \lceil v_n \rceil}^{J_n}
\frac{1}{2 n}
\left( \frac{\lo \lo n}{\lo n} \right)^2
\right| $$
$$\leq
\sum_{j= \lceil v_n \rceil}^{J_n}
\frac{1}{n} \left| O
\left( \frac{(\lo \lo n)^2}{(\lo n)^3} \right) \right|
$$ 
\begin{equation}
\label{mmm03_mieux}
+
\frac{1}{2} 
\sum_{j= \lceil v_n \rceil}^{J_n}
\frac{1}{n^2}
\left[
\lo \bigl(2 \, \sin\bigl(\frac{\pi j}{n}\bigr)  \bigr)
+ 
\frac{1}{2}
\left( \frac{\lo \lo n}{\lo n} \right)^2
+
O\left(
\frac{(\lo \lo n)^2}{(\lo n)^3} \right)
\right]^2 .
\end{equation}
For $1 \leq j \leq J_n$,
the inequalities
$0 < 2 \sin(\pi j/n)\leq 1$
and
$\lo(2 \sin(\pi j/n))<0$ hold.
Then 
$|\lo(2 \sin(\pi j/n)|
\leq 
|\lo(2 \sin(\pi/n))|=O(\lo n)
$.
On the other hand, 
the two $O(~)$s in the rhs of
\eqref{mmm03_mieux}
involve a constant which 
does not depend upon $j$.
Therefore, from Proposition 
\ref{argumentlastrootJn},
the rhs of
\eqref{mmm03_mieux} is
$$=
O\left(
\Bigl(\frac{(\lo \lo n)^2}{(\lo n)^3}
\Bigr)
\right)
+ O\left(
\frac{\lo^2 n}{n}
\right)
=
O\left(
\Bigl(\frac{(\lo \lo n)^2}{(\lo n)^3}
\Bigr)
\right).
$$ 
On the other hand, 
the two regimes of asymptotic
expansions
in the Bump give (Appendix)
$$  
\sum_{j= \lceil u_n \rceil}^{\lfloor v_n \rfloor}
\lo |z_{j,n}|
= O\left(\frac{(\lo n)^{2+\epsilon}}{n}
\right),\quad 
\sum_{j= 1}^{\lfloor u_n \rfloor}
\lo |z_{j,n}|
=
O\left(\frac{(\lo n)^2}{n}
\right)$$
and
$$\sum_{j=\lceil \lo n \rceil}^{\lceil v_n \rceil}
\frac{2}{n} \lo \Bigl(
2 \sin(\frac{\pi j}{n})
\Bigr)
=
O\left(
\frac{(\lo n)^{2+\epsilon}}{n}
\right).
$$
Therefore
\begin{equation}
\label{sz2}
-2
\sum_{j=1}^{J_n}
\lo |z_{j,n}|
=
-
\sum_{j=\lceil \lo n \rceil}^{J_n}
\frac{2}{n} \, 
\lo \bigl(2 \, \sin\bigl(\frac{\pi j}{n}\bigr)  
\bigr)
+ O\left(
\left( \frac{\lo \lo n}{\lo n} \right)^2
\right)
\end{equation}
with the constant 
$\frac{1}{2 \pi} \arcsin(\frac{\kappa}{2})$ 
(from Proposition \ref{argumentlastrootJn})
involved in the Big O. 

(iii) third term in \eqref{logmmmrMINORANT}: 
with the definition
of $\epsilon$ and $(v_n)$ (Appendix),
\begin{equation}
\label{sz3}
- 2 \sum_{j=1}^{\lfloor v_n \rfloor} 
\lo (1 + \frac{\pi}{n \, a_{max}})
=
O\left(
\frac{(\lo n)^{1+\epsilon}}{n}
\right);
\end{equation}

(iv) fourth term in \eqref{logmmmrMINORANT}: 
from the 
Theorem of alternating series,
\begin{equation}
\label{ineqmu_r}
|\sum_{j=\lceil v_n \rceil}^{J_n} 
\lo (1 + \frac{\pi}{n \, a_{j,n}})
- \sum_{j=\lceil v_n \rceil}^{J_n}
\frac{1}{n}
D(\frac{\pi}{a_{j,n}})
- \sum_{j=\lceil v_n \rceil}^{J_n}
\frac{1}{n}
{\rm tl}(\frac{\pi}{a_{j,n}})|
\leq
\frac{1}{2}
\sum_{j=\lceil v_n \rceil}^{J_n}
\left(
\frac{\pi}{n \, a_{j,n}}
\right)^2 .
\end{equation}
The terminant
${\rm tl}(\frac{\pi}{ a_{j,n}})
=
O \Bigl(
\frac{(\lo \lo n)^2}{(\lo n)^3}
\Bigr)$
is given by 
\eqref{petitcercle_jmainTL_majorant}.
From Theorem \ref{omegajnexistence}, 
with $B_{j,n} =
2 \sin(\frac{\pi j}{n})
\Bigl(
1 - \frac{1}{n} \lo (2 \sin(\frac{\pi j}{n}))
\Bigr)
$,
it is easy to show
$$
D(\frac{\pi}{a_{j,n}}) 
=
\lo\Bigl[\frac{1 + B_{j,n}
-
\sqrt{1 - 6 B_{j,n} 
+ B_{j,n}^2}}{4 B_{j,n}}
\Bigr]$$
$$
= 
\lo\Bigl[\frac{1 + 
2 \sin(\frac{\pi j}{n})
-
\sqrt{1 - 12 \sin(\frac{\pi j}{n})
+ 4 \sin(\frac{\pi j}{n})^2}}
{8 \sin(\frac{\pi j}{n})}
\Bigr]
+O\left(
\frac{\lo n}{n}
\right).
$$ 
The rhs of \eqref{ineqmu_r} is
$= O\left(\frac{1}{n}
\right)$.
Then~
$- 2 \, \sum_{j=\lceil v_n \rceil}^{J_n} 
\lo (1 + \frac{\pi}{n \, a_{j,n}})
=$
\begin{equation}
\label{sz4}
\sum_{j=\lceil v_n \rceil}^{J_n}
\frac{-2}{n}\lo\Bigl[\frac{1 + 
2 \sin(\frac{\pi j}{n})
-
\sqrt{1 - 12 \sin(\frac{\pi j}{n})
+ 4 \sin(\frac{\pi j}{n})^2}}
{8 \sin(\frac{\pi j}{n})}
\Bigr]
+ O\Bigl(\frac{(\lo \lo n)^2}{(\lo n)^3}
\Bigr) .
\end{equation}
The summation 
$\sum_{j=\lceil v_n \rceil}^{J_n}$ can 
be replaced by 
$\sum_{j=\lceil \lo n \rceil}^{J_n}$. 
Indeed, from the definition of the sequence
$(v_n)$
(Appendix),
$$\sum_{j=\lceil \lo n \rceil}^{\lceil v_n \rceil}
\frac{2}{n} \, 
\lo\Bigl[\frac{1 + 
2 \sin(\frac{\pi j}{n})
-
\sqrt{1 - 12 \sin(\frac{\pi j}{n})
+ 4 \sin(\frac{\pi j}{n})^2}}
{8 \sin(\frac{\pi j}{n})}
\Bigr]
= 
O\left(
\frac{(\lo n)^{2+\epsilon}}{n}
\right).$$
Inserting the 
contributions 
\eqref{sz1} 
\eqref{sz2}
\eqref{sz3}
\eqref{sz4} in 
\eqref{logmmmrMINORANT} leads to
$$
L_{r}(\beta) =
 \lo \Lambda_r + \lo \mu_r
+
\Bigl(-
\lo \Lambda_r
-
\sum_{j=\lceil \lo n \rceil}^{J_n}
\frac{2}{n} \, \lo \bigl(2 \, \sin\bigl(\frac{\pi j}{n}\bigr)  \bigr)
\Bigr)$$
$$
+
\Bigl(-
\lo \mu_r
-
\sum_{j=\lceil \lo n \rceil}^{J_n}
\frac{2}{n} \, \lo \bigl(
\frac{1 + 
2 \sin(\frac{\pi j}{n})
-
\sqrt{1 - 12 \sin(\frac{\pi j}{n})
+ 4 \sin(\frac{\pi j}{n})^2}}
{8 \sin(\frac{\pi j}{n})}
\bigr)  \bigr)
\Bigr)
$$
\begin{equation}
\label{esti}
+O\Bigl(
\Bigl( \frac{\lo \lo n}{\lo n} \Bigr)^2
\Bigr)
\end{equation}
with the constant 
$\frac{1}{2 \pi} \arcsin(\frac{\kappa}{2})$ 
involved in the Big O.
Let us denote by
$\Delta_1$ the first
term within brackets,
resp.
$\Delta_2$ the second term
within brackets,  in \eqref{esti}
so that
\begin{equation}
\label{LrDelta1Delta2}
D(L_{r}(\beta)) = \lo ( \Lambda_r \mu_r ) + \Delta_1 + \Delta_2.
\end{equation}

\noindent
{\em Calculation of $|\Delta_1|$:}
let us estimate 
and give an upper bound of $|\Delta_1| =$
\begin{equation}
\label{diffloglambda}
\left|
\frac{-1}{\pi}\int_{0}^{2 \arcsin(\kappa/2)}
\lo \Bigl(
2 \sin (x/2)
\Bigr) dx 
-
\sum_{j=\lceil \lo n \rceil}^{J_n}
\frac{-2}{n} \, \lo \bigl(2 \, \sin\bigl(\frac{\pi j}{n}\bigr)  \bigr)
\right|.
\end{equation}
In \eqref{diffloglambda} the sums are truncated 
Riemann-Stieltjes sums of $\lo \Lambda_r$, the 
integral being 
$\lo \Lambda_r$. Referring to Stoer and Bulirsch 
(\cite{stoerbulirsch}, pp 126--128) we now replace 
$\lo \Lambda_r$ by an approximate value obtained 
by integration of an interpolation polynomial
by the methods of Newton-Cotes; we just need to know this approximate value
up to 
$O\bigl(
\Bigl( \frac{\lo \lo n}{\lo n} \Bigr)^2
\bigr)$. 
Up to 
$O\bigl(
\Bigr( \frac{\lo \lo n}{\lo n} 
\Bigr)^2
\bigr)$, we will show that: 

\noindent
(i--1) {\it an upper bound of 
\eqref{diffloglambda}  is ($\kappa$ stands for
$\kappa(1,a_{max})$ as in} Proposition
\ref{argumentlastrootJn})
$$ \frac{\arcsin(\kappa/2)}{\pi} \, \frac{1}{\lo n},$$ 
(ii--1) {\it the approximate value of $\lo \Lambda_r$ is independent of the integer $m$ (i.e. step length) used in the Newton-Cotes formulas, assuming the weights $(\alpha_q)_{q=0, 1, \ldots, m}$ associated with $m$ all positive. Indeed, if $m$ is arbitrarily large, the estimate of the integral should be very good by these methods, ideally exact at the limit ($m ``=" +\infty$)}. 

\vspace{0.2cm}

\noindent
{\it Proof of (i--1)}: 
we consider the decomposition of the interval of integration as

\noindent
$\bigl(0, 2 \arcsin(\kappa/2)\bigr] =$ 
\begin{equation}
\label{intervaldecomposition}
\bigl(0, \frac{2 \pi \lceil \lo n \rceil}{n} \bigr]
\cup \Bigl(
\bigcup_{j=\lceil \lo n \rceil}^{J_n - 1}
\bigl[
\frac{2 \pi j}{n}, \frac{2 \pi (j+1)}{n}
\bigr]
\Bigr)
\cup 
\Bigl[\frac{2 \pi J_n}{n}, 
2 \arcsin(\kappa/2)\Bigr]
\end{equation}
and proceed by calcutating the estimations of 
\begin{equation}
\label{estijj}
\left|
\frac{-1}{\pi}\int_{\frac{2 \pi j}{n}}^{\frac{2 \pi (j+1)}{n}}
\lo \Bigl(
2 \sin (x/2)
\Bigr) dx
-
\frac{-2}{n} \, \lo \bigl(2 \, \sin\bigl(\frac{\pi j}{n}\bigr)  \bigr)
\right|
\end{equation}
on the intervals $\mathcal{I}_j := \bigl[
\frac{2 \pi j}{n}, \frac{2 \pi (j+1)}{n}
\bigr]$, $j=
\lceil \lo n \rceil, 
\lceil \lo n \rceil + 1, \ldots,
J_n - 1$.
On each such $\mathcal{I}_j$, the function
$f(x)$ is approximated by its interpolation polynomial
$P_{m}(x)$, where $m \geq 1$ is the number of subintervals forming an uniform partition of
$\mathcal{I}_j$ given by
\begin{equation}
\label{yqinterpolation}
y_q = \frac{2 \pi j}{n} + q \frac{2 \pi}{n} \frac{1}{m}, \qquad q = 0, 1, \ldots, m, 
\end{equation}
of step length 
$h_{NC} := \frac{2 \pi}{n \, m}$, and
$P_m$ the interpolating polynomial 
of degree $m$ or less with
$$P_{m}(y_q) = f(y_q), \qquad \mbox{for}~~
q = 0, 1, \ldots, m.$$
The Newton-Cotes formulas
$$\int_{\frac{2 \pi j}{n}}^{\frac{2 \pi (j+1)}{n}} 
P_{m}(x) dx = h_{NC} \, \sum_{q=0}^{m} \alpha_q f(y_q)$$
provide approximate values of 
$\int_{\frac{2 \pi j}{n}}^{\frac{2 \pi (j+1)}{n}} f(x)
dx$, 
where the $\alpha_q$ are the weights obtained by 
integrating 
the Lagrange's interpolation polynomials.
Steffensen \cite{steffensen} (\cite{stoerbulirsch}, p 127) 
showed that the approximation error may be expressed as 
follows:
$$\int_{\frac{2 \pi j}{n}}^{\frac{2 \pi (j+1)}{n}} P_{m}(x) dx - \int_{\frac{2 \pi j}{n}}^{\frac{2 \pi (j+1)}{n}} f(x)dx = h_{NC}^{p+1} \cdot K \cdot f^{(p+1)}(\xi),
\qquad \xi \in \stackrel{o}{\mathcal{I}_j},$$
where $p \geq 2$ is an integer related to $m$, and $K$ a constant.

Using \cite{stoerbulirsch}, p. 128, and $m=1$, the method 
being the ``Trapezoidal rule", we have: 
``$p=2$, $K = 1/12, 
\alpha_0 = \alpha_1 = 1/2$". Then
\eqref{estijj}
is estimated by
$$\left|
\frac{1}{2} \frac{2 \pi}{n}
\left[
\frac{-1}{\pi} \, \lo \bigl(2 \, \sin\bigl(\frac{\pi j}{n}\bigr)  \bigr) +
\frac{-1}{\pi} \, \lo \bigl(2 \, \sin\bigl(\frac{\pi (j+1)}{n}\bigr)  \bigr)
\right]
-
\frac{-2}{n} \, \lo \bigl(2 \, \sin\bigl(\frac{\pi j}{n}\bigr)  \bigr)\right|$$
\begin{equation}
\label{majo}
= \frac{1}{n}
\left|
\, \lo \bigl(2 \, \sin\bigl(\frac{\pi j}{n}\bigr)  \bigr) -
 \, \lo \bigl(2 \, \sin\bigl(\frac{\pi (j+1)}{n}\bigr)  \bigr)
)\right|
=
\frac{2 \pi}{n^2}
\left|
\frac{\cos(\xi / 2)}{2 \sin (\xi / 2)}
\right|
\leq 
\frac{1}{n} \, \frac{1}{\lo n} 
\end{equation}
for some $\xi \in \stackrel{o}{\mathcal{I}_j}$, for large $n$.
The (Steffensen's) approximation error 
``$h_{NC}^3 \cdot (1/12) \cdot f^{(2)}(\xi)$"
for the trapezoidal rule, relative to
\eqref{estijj}, is
\begin{equation}
\label{steff+}
\frac{1}{\pi} \,
\left(\frac{2 \pi}{n}\right)^3 \, \frac{1}{12}
\, \left|
\frac{-1}{4 \sin^2(\xi/2)}
\right| \leq
\frac{1}{6 n} \frac{1}{(\lo n)^2} .
\end{equation}
By Proposition \ref{argumentlastrootJn}  
the integral
$$\left|
\frac{-1}{\pi}\int_{\frac{2 \pi J_n}{n}}^{2 \arcsin{\kappa/2}}
\lo \bigl(
2 \sin (x/2)
\bigr) dx
\right|\qquad \mbox{is a} \qquad 
O\bigl(\frac{1}{n}\bigr).
$$
Then, summing up the contributions of all the intervals 
$\mathcal{I}_j$, we obtain the 
following upper bound of \eqref{diffloglambda} 
\begin{equation}
\label{cekireste}
\left|
\frac{-1}{\pi}\int_{0}^{(2 \pi \lo n)/n}
\lo \bigl(
2 \sin (x/2)
\bigr) dx
\right|
+
\frac{\arcsin(\kappa/2)}{\pi} \, \frac{1}{\lo n}.
\end{equation}
with global (Steffensen's) approximation error, from \eqref{steff+}, 
$$O(\frac{1}{(\lo n)^2})$$
By integrating by parts the integral 
in \eqref{cekireste}, for large $n$, it is easy to 
show that this integral is 
$=O\left(\frac{(\lo n)^2}{n}\right)$. 
We deduce the 
following asymptotic expansion
\begin{equation}
\Delta_1 =
\frac{\rc}{\lo n}
+
O(\frac{1}{(\lo n)^2})
\qquad
\mbox{with}\quad |\rc| < \frac{\arcsin(\kappa/2)}{\pi}.
\end{equation}

\noindent
{\it Proof of (ii--1)}: Let us show that 
the upper bound 
$\frac{\arcsin(\kappa/2)}{\pi} \, \frac{1}{\lo n}$
is independent of the integer $m$ used, 
once assumed the positivity
of the weights $(\alpha_q)_{q=0, 1, \ldots, m}$. 
For $m \geq 1$ fixed, this is merely a consequence
of the relation between the weights in 
the Newton-Cotes formulas. 
Indeed, we have $\sum_{q=0}^{m} \alpha_q = m$, and therefore
$$
\left|\int_{\frac{2 \pi j}{n}}^{\frac{2 \pi (j+1)}{n}} P_{m}(x) dx - h_{NC} m f(y_0)\right|= 
h_{NC} \, 
\left|\sum_{q=0}^{m} \alpha_q (f(y_q) - f(y_0))\right|
$$
$$
\leq 
h_{NC} \, 
\bigl(\sum_{q=0}^{m} | \alpha_q | \bigr) 
\sup_{\xi \in \mathcal{L}_j}
\left| f'(\xi)
\right|.$$
Since $h_{NC} m = \frac{2 \pi}{n}$
and that
the inequality 
$\sup_{\xi \in \mathcal{L}_j}
\left| f'(\xi)
\right| \leq |f'((2 \pi \lo n)/n)|$ holds
uniformly for all $j$,
we deduce the same upper bound
as in \eqref{majo} for the Trapezoidal rule.
Summing up the contributions over all the
intervals $\mathcal{I}_j$, we obtain
the same upper bound
\eqref{cekireste} of \eqref{diffloglambda} as before.

As for the (Steffensen's) approximation errors, they make use of the successive derivatives
of the function $f(x) = \lo(2 \sin (x/2))$. We have: 
$$f'(x) = \frac{\cos(x/2)}{2 \sin(x/2)}, ~~
f''(x) = -\frac{1}{4 \sin^{2}(x/2)},
~~f'''(x) = \frac{\cos(x/2)}{4 \sin^{3}(x/2)}\ldots$$
Recursively, it is easy to show that
the $q$-th derivative of $f(x)$, $q \geq 1$, is a rational function of the two
quantities $\cos(x/2)$ 
and $\sin(x/2)$ with bounded numerator on the interval $(0, \pi/3]$, and
a denominator which is $\sin^{q}(x/2)$.
For the needs of majoration in the Newton-Cotes formulas
over each interval of the collection
$(\mathcal{I}_j)$, 
this denominator takes its smallest value
at $\xi = (2 \pi \lceil \lo n \rceil)/n$.
Therefore, for large $n$, 
the (Steffensen's) approximation error
``$h_{NC}^{p+1} \cdot K \cdot f^{(p)}(\xi)$"
on one interval $\mathcal{I}_j$ is
$$O\left(\Bigl(\frac{2 \pi}{n m}\Bigr)^{p+1}
\cdot K \cdot 
\frac{n^p}{(\pi \, \lo n)^p}\right) = 
O\left(\frac{1}{
n (\lo n)^p}
\right).$$
By summing up over the intervals
$\mathcal{I}_j$, we obtain
the global (Steffensen's) approximation error ($p \geq 2$)
$$O\left(\frac{1}{(\lo n)^p}\right)
\qquad \mbox{which is a}
\qquad 
O\left(
\left( \frac{\lo \lo n}{\lo n} \right)^2
\right).$$

\noindent
{\em Calculation of $|\Delta_2|$}:
we proceed as above for establishing an upper bound of
$$|\Delta_2|
=
\Bigl|
\frac{-1}{\pi}
\int_{0}^{2 \arcsin(\frac{\kappa(1,a_{\max})}{2})}
\! \!\lo\Bigl[\frac{1 + 2 \sin(\frac{x}{2})
-
\sqrt{1 - 12 \sin(\frac{x}{2}) 
+ 4 (\sin(\frac{x}{2}))^2}}{8 \sin(\frac{x}{2})}
\Bigr] dx\bigr.$$
\begin{equation}
\label{diffloglambda_plus}
-
\sum_{j=\lceil \lo n \rceil}^{J_n}
\frac{- 2}{n} \, \lo \bigl(
\frac{1 + 
2 \sin(\frac{\pi j}{n})
-
\sqrt{1 - 12 \sin(\frac{\pi j}{n})
+ 4 \sin(\frac{\pi j}{n})^2}}
{8 \sin(\frac{\pi j}{n})}
\bigr)  
\Bigr|
\end{equation}

In \eqref{diffloglambda_plus} the sums are truncated 
Riemann-Stieltjes sums of $\lo \mu_r$, the 
integral being 
$\lo \mu_r$. 
As above, the methods of Newton-Cotes
(Stoer and Bulirsch 
(\cite{stoerbulirsch}, pp 126--128) 
will be applied 
to compute an
approximate value of the integral up to 
$O\bigl(
\Bigl( \frac{\lo \lo n}{\lo n} \Bigr)^2
\bigr)$. 
Up to 
$O\bigl(
\Bigr( \frac{\lo \lo n}{\lo n} 
\Bigr)^2
\bigr)$, we will show that:

(i--2) {\it an upper bound of 
\eqref{diffloglambda_plus}  is ($\kappa$ stands for
$\kappa(1,a_{max})$ as in} Proposition
\ref{argumentlastrootJn})

\begin{equation}
\label{delta2upperbound}
 \frac{4 \,\arcsin(\kappa/2)}{\kappa \sqrt{2 \kappa (3-\kappa) \lo (1/\kappa)}} \, \frac{1}{\sqrt{n}}
\qquad \mbox{{\it which is a}}
~~
O\bigl(
\Bigr( \frac{\lo \lo n}{\lo n} 
\Bigr)^2
\bigr),
\end{equation}
{\it in other terms that 
\eqref{diffloglambda_plus}
is equal to zero 
up to} 
$O\bigl(
\Bigr( \frac{\lo \lo n}{\lo n} 
\Bigr)^2
\bigr)$,

(ii--2) {\it the approximate value of 
$\lo \mu_r$ is independent of the step length $m$ 
used in the Newton-Cotes formulas, 
assuming the weights $(\alpha_q)_{q=0, 1, \ldots, m}$ 
associated with $m$ all positive}. 

\vspace{0.2cm}

\noindent
{\it Proof of (i--2)}: The decomposition
of the interval 
of integration
$\bigl(0, 2 \arcsin(\kappa/2)\bigr]$
remains the same as above, given by
\eqref{intervaldecomposition}.
Let us treat the
complete interval 
of integration
$\bigl(0, 2 \arcsin(\kappa/2)\bigr]$
by subintervals.
We first
proceed by 
estimating an upper bound of
$$\Bigl|
\frac{-1}{\pi}
\int_{\frac{2 \pi j}{n}}^{\frac{2 \pi (j+1)}{n}}
\! \!\lo\Bigl[\frac{1 + 2 \sin(\frac{x}{2})
-
\sqrt{1 - 12 \sin(\frac{x}{2}) 
+ 4 (\sin(\frac{x}{2}))^2}}{8 \sin(\frac{x}{2})}
\Bigr] dx\bigr.$$
\begin{equation}
\label{estijj_plus}
-
\frac{- 2}{n} \, \lo \bigl(
\frac{1 + 
2 \sin(\frac{\pi j}{n})
-
\sqrt{1 - 12 \sin(\frac{\pi j}{n})
+ 4 \sin(\frac{\pi j}{n})^2}}
{8 \sin(\frac{\pi j}{n})}
\bigr)  
\Bigr|
\end{equation}
on the intervals $\mathcal{I}_j := \bigl[
\frac{2 \pi j}{n}, \frac{2 \pi (j+1)}{n}
\bigr]$, $j=
\lceil \lo n \rceil, 
\lceil \lo n \rceil + 1, \ldots,
J_n - 1$. 
Let
$$
F(x):= \lo\Bigl[\frac{1 + 2 \sin(\frac{x}{2})
-
\sqrt{1 - 12 \sin(\frac{x}{2}) 
+ 4 (\sin(\frac{x}{2}))^2}}{8 \sin(\frac{x}{2})}
\Bigr].
$$
On each interval
$\mathcal{I}_j$ 
the function
$F(x)$ is approximated by its interpolation polynomial
(say) $P_{F, m}(x)$, where $m \geq 1$ is the number of 
subintervals 
of
$\mathcal{I}_j$ given by their extremities $y_q$ 
by \eqref{yqinterpolation},
of step length 
$h_{NC} := \frac{2 \pi}{n \, m}$, and
$P_{F, m}$ the interpolating polynomial 
of degree $m$ or less with
$$P_{F, m}(y_q) = F(y_q), \qquad \mbox{for}~~
q = 0, 1, \ldots, m.$$
The Newton-Cotes formulas
\begin{equation}
\label{NCformula}
\int_{\frac{2 \pi j}{n}}^{\frac{2 \pi (j+1)}{n}} 
P_{F, m}(x) dx = h_{NC} \, \sum_{q=0}^{m} \alpha_q F(y_q)
\end{equation}
provide the approximate values 
$\int_{\frac{2 \pi j}{n}}^{\frac{2 \pi (j+1)}{n}} F(x)
dx$, 
where the $\alpha_q$s are the weights obtained by 
integrating 
the Lagrange's interpolation polynomials.
Using \cite{stoerbulirsch}, p. 128, and $m=1$, the method 
being the ``Trapezoidal rule", we have: $p=2$, $K = 1/12, 
\alpha_0 = \alpha_1 = 1/2$. Then
\eqref{estijj_plus}
is estimated by
$$\left|
\frac{1}{2} \frac{2 \pi}{n}
\left[
\frac{-1}{\pi} \, F\bigl(\frac{2 \pi j}{n}\bigr)  +
\frac{-1}{\pi} \, F\bigl(\frac{2\pi (j+1)}{n}\bigr) 
\right]
-
\frac{-2}{n} \,  F\bigl(\frac{2 \pi j}{n}\bigr) 
\right|$$
\begin{equation}
\label{majoF}
= \frac{1}{n}
\left|
\, F\bigl(\frac{2 \pi j}{n}\bigr) -
 \, F\bigl(\frac{2 \pi (j+1)}{n}\bigr)
)\right|
=
\frac{2 \pi}{n^2}
\left|F'(\xi)
\right|
\end{equation}
for some $\xi \in \stackrel{o}{\mathcal{I}_j}$, 
for large $n$. 
As in 
Remark 
\ref{openingangle_sin_quadratic_alginteger},
let $x = 2 \arcsin(\kappa/2)$.
The derivative
\begin{equation} 
\label{derivativeFF}
F'(y)= \frac{\cos(y/2)
(-2 \sin(y/2) + 1 -
\sqrt{4 \sin^{2}(y/2) - 12 \sin(y/2)
+1}}{4 \sin(y/2) \sqrt{4 \sin^{2}(y/2) - 12 \sin(y/2)
+1})} \, > 0
\end{equation}
is increasing on the interval 
$(0, x)$.
When $y=\frac{2 \pi J_n}{n} < x$ tends to $x^{-}$,
by Proposition \ref{argumentlastrootJn}
and Remark \ref{openingangle_sin_quadratic_alginteger},
since
$0 < \sqrt{4 \sin^{2}(y/2) - 12 \sin(y/2)
+1} \leq 1$ is close to zero for 
$y=2 \pi J_n /n$, 
the following
inequality holds
\begin{equation}
\label{zouzou1}
F'(\frac{2 \pi J_n}{n})
\leq \frac{2/\kappa}{ 
\sqrt{4 \sin^{2}(\frac{\pi J_n}{n}) - 
12 \sin(\frac{\pi J_n}{n})+1}} .
\end{equation}
The upper bound is a function 
of $n$ which comes from
the asymptotic expansion of 
$\frac{\pi J_n}{n} - \frac{x}{2}$, as
deduced from \eqref{Jnasymptotic}.
Indeed, from \eqref{Jnasymptotic}
and using Remark 
\ref{openingangle_sin_quadratic_alginteger} (ii), 
$$4 \sin^{2}(\frac{\pi J_n}{n}) 
\!-\! 
12 \sin(\frac{\pi J_n}{n})
\!+\!1 
\!=\! 
(\frac{\pi J_n}{n} - \frac{x}{2})
[8 \sin(x/2) \cos(x/2) -12 \cos(x/2)] 
\!+\! O(\frac{1}{n^2})
$$
\begin{equation}
\label{zouzou2}
= \frac{2 \kappa (3-\kappa) \lo (1/\kappa)}{n}
+
\frac{1}{n}
O\bigl(
\bigl(
\frac{\lo \lo n}{\lo n}
\bigr)^2
\bigr)
\end{equation}
From \eqref{zouzou1}
and \eqref{zouzou2} we deduce 
$|F'(\frac{2 \pi J_n}{n})|
<  \frac{(2/\kappa)}{\sqrt{2 \kappa (3-\kappa) \lo (1/\kappa)}} \sqrt{n}$.
From \eqref{majoF}, we deduce
the following upper bound of 
\eqref{estijj_plus} on each $\mathcal{I}_j := \bigl[
\frac{2 \pi j}{n}, \frac{2 \pi (j+1)}{n}
\bigr]$:
\begin{equation}
\label{pitou}
\frac{4 \pi}{\kappa \sqrt{2 \kappa (3-\kappa) \lo (1/\kappa)}} \,\frac{1}{n^{3/2}} .
\end{equation}
By summing up the contributions, 
for
$j=
\lceil \lo n \rceil, \ldots, J_n - 1$,
from \eqref{pitou}
and the asymptotics of $J_n$ given
by \eqref{Jnasymptotic},
we deduce 
the upper bound \eqref{delta2upperbound}
of $|\Delta_2|$.

Let us prove that the method of 
numerical integration we use leads to 
a (Steffensen's) approximation error 
which is a $O\bigl(
\bigl(
\frac{\lo \lo n}{\lo n}
\bigr)^2
\bigr)$.
The (Steffensen's) approximation error 
``$h_{NC}^3 \cdot (1/12) \cdot F^{(2)}(\xi)$"
for the trapezoidal rule
applied to \eqref{estijj_plus} 
(\cite{stoerbulirsch}, p. 127--128) is
\begin{equation}
\label{steff++}
\frac{1}{\pi} \,
\left(\frac{2 \pi}{n}\right)^3 \, \frac{1}{12}
\, \left|
F^{(2)}(\xi)
\right| 
\qquad \mbox{
for some~} 
\xi \in \stackrel{o}{\mathcal{I}_j}.
\end{equation}
The second derivative $F''(y)$
is positive and increasing on $(0,\frac{2 \pi J_n}{n})$.
It is easy to show that there exists a constant
$C > 0$ such that
$$F''(\frac{2 \pi J_n}{n})
\leq \frac{C}{ 
(4 \sin^{2}(\frac{\pi J_n}{n}) - 
12 \sin(\frac{\pi J_n}{n})+1)^{3/2}} .$$
Using the asymptotic expansion of $J_n$ 
(\eqref{Jnasymptotic};
Remark 
\ref{openingangle_sin_quadratic_alginteger} (ii);
\eqref{zouzou2}),
there exist 
$C_1 > 0$ such that
\begin{equation}
\label{steff++majo}
F''(\frac{2 \pi J_n}{n})
\leq  C_1 \, n^{3/2} .
\end{equation}
From \eqref{steff++} and \eqref{steff++majo}, 
summing up 
the contributions for
$j= \lceil \lo n \rceil, 
 \ldots,
J_n - 1$,  
the global (Steffensen's) approximation error
of \eqref{diffloglambda_plus} 
for $|\Delta_2|$
admits the following upper bound, 
for some constants $C'_{2} > 0, C_2 > 0$, 
$$C'_2 \,  \frac{J_n}{n^3} \, n^{3/2} =
C_2 \frac{1}{\sqrt{n}} \qquad \mbox{which is a}~~
O\bigl(
\bigl(
\frac{\lo \lo n}{\lo n}
\bigr)^2
\bigr) .$$ 
Now let us turn to the extremity intervals.
Using the Appendix, and  
\eqref{Jnasymptotic} in
Proposition \ref{argumentlastrootJn}, 
it is easy to show that
the two integrals
$$\frac{-1}{\pi} 
\int_{0}^{\frac{2 \pi \lceil \lo n \rceil}{n}}
\quad
{\rm and}
\quad
\frac{-1}{\pi} 
\int_{\frac{2 \pi J_n }{n}}^{2 \arcsin(\kappa/2)}
\qquad
\mbox{are}
\qquad
O\left(
\left(
\frac{\lo \lo n }{ \lo n} 
\right)^2
\right).$$

\noindent 
{\it Proof of (ii--2)}: 
On each interval
$\mathcal{I}_j := \bigl[
\frac{2 \pi j}{n}, \frac{2 \pi (j+1)}{n}
\bigr]$, $j=
\lceil \lo n \rceil, 
\ldots,
J_n - 1$, let us assume that
the number $m$ of 
subintervals 
of
$\mathcal{I}_j$ given by their extremities $y_q$ 
by \eqref{yqinterpolation}, is $\geq 2$.
The weights $\alpha_q$ 
in \eqref{NCformula} are assumed to be positive.

The upper bound 
 $\frac{4 \,\arcsin(\kappa/2)}{\kappa \sqrt{2 \kappa (3-\kappa) \lo (1/\kappa)}} \,
  \frac{1}{\sqrt{n}}$
of 
\eqref{diffloglambda_plus}
is independent of $m \geq 2$, 
once assumed the positivity
of the weights $(\alpha_q)_{q=0, 1, \ldots, m}$,
since, due to 
the relation between the weights in 
the Newton-Cotes formulas 
$\sum_{q=0}^{m} \alpha_q = m$, 
$$
\left|\int_{\frac{2 \pi j}{n}}^{\frac{2 \pi (j+1)}{n}} 
P_{m}(x) dx - h_{NC} m F(y_0)\right|= 
h_{NC} \, 
\left|\sum_{q=0}^{m} \alpha_q (F(y_q) - F(y_0))\right|
$$
$$
\leq 
h_{NC} \, 
\bigl(\sum_{q=0}^{m} | \alpha_q | \bigr) 
\sup_{\xi \in \mathcal{L}_j}
\left| F'(\xi)
\right|.$$
Since $h_{NC} m = \frac{2 \pi}{n}$
and that
$\sup_{\xi \in \mathcal{L}_j}
\left| F'(\xi)
\right| \leq |F'((2 \pi J_n)/n)|$ holds
uniformly for all $j=
\lceil \lo n \rceil, \ldots, J_n - 1$,
we deduce the same upper bound
\eqref{pitou} as for the Trapezoidal rule.
Summing up the contributions over all the
intervals $\mathcal{I}_j$, we obtain
the same upper bound
\eqref{delta2upperbound} 
of \eqref{diffloglambda_plus}, as before.

As for the (Steffensen's) approximation 
errors involved 
in the numerical integration
\eqref{NCformula} there are 
``$h_{NC}^{p+1} \cdot K \cdot F^{(p)}(\xi)$"
on one interval $\mathcal{I}_j$, for some $p \geq 2$.
They make use of the successive derivatives
of the function $F(x)$.
It can be shown that they contribute negligibly,
after summing up over all the intervals
$\mathcal{I}_j$, as
$
O\left(
\left( \frac{\lo \lo n}{\lo n} \right)^2
\right).$

Gathering the different terms from (i--1)(i--2),
the Steffenssen's error terms
and the error terms due to the 
numerical integration by the
Newton-Cotes method (ii--1)(ii--2),
we have proved the following theorem.

\begin{theorem}
\label{Lrasymptotictheorem}
Let $\beta > 1$ be a
reciprocal algebraic integer
such that
$n=\dyg(\beta) \geq 260$, 
with ${\rm M}(\alpha)
< 1.176280\ldots$.
The asymptotic expansion
of the minorant
$L_{r}(\beta)$ of \,$\lo {\rm M}_{r}(\beta)$
is
\begin{equation}
\label{Lrasymptotics}
L_{r}(\beta) = \lo \Lambda_r \mu_r 
+
\frac{\rc}{\lo n}
+
O\bigl(
\bigl(
\frac{\lo \lo n}{\lo n}
\bigr)^2
\bigr),
\quad
\mbox{with}\quad 0 < \rc < \frac{\arcsin(\kappa/2)}{\pi}
\end{equation}
and $\rc$ depending upon 
$n$.
\end{theorem}

\

\subsection{A Dobrowolski type minoration}
\label{S5.7}

Denote by $\rc_n$ the positive real number
$\rc$ in \eqref{Lrasymptotics}.
Let us show that it is substantially
smaller than the bound  
$\frac{\arcsin(\kappa/2)}{\pi}$.

\begin{lemma}
\label{petitrnBIGOlimit}
With the same notations as in Theorem
\ref{Lrasymptotictheorem}, there exists
an integer $\eta \geq 260$ such that 
\begin{equation}
\label{petitrn}
\left|
\rc_n
+
O\bigl(
\frac{(\lo \lo n)^2}{\lo n}
\bigr)
\right|
< \frac{\arcsin(\kappa/2)}{\pi},
\qquad n \geq \eta.
\end{equation}
\end{lemma}

\begin{proof}
Let $X_n := c \lceil \lo n \rceil$
with $c$ a positive constant such that 
$\lceil \lo n \rceil
< X_n
< J_n$.
The limit 
$\lim_{n \to \infty} X_n /n = 0$ holds.
Recall that $J_n$ is given by
\eqref{Jnasymptotic}.

The quantity $\rc_n$ comes from
the integration of the
$j$th-subdivision step
\eqref{estijj} by \eqref{majo}, 
in order to give
an estimate of the development term
of $|\Delta_{1}|$ given by
\eqref{diffloglambda}.  
This $j$th-subdivision step
of integration provides the estimated
term
(cf \eqref{majo})
\begin{equation}
\label{majorefine}
\frac{2 \pi}{n^2}
\left|
\frac{\cos(\xi / 2)}{2 \sin (\xi / 2)}
\right|
\mbox{\qquad for some} 
\, \, \xi \in \bigl(
\frac{2 \pi j}{n}, \frac{2 \pi (j+1)}{n}
\bigr).
\end{equation}
Since the cotangent function is 
positive and
strictly decreasing
on $(0, \pi/2)$,
the upper bound of \eqref{majorefine}
is naturally the one given by
the first interval of the subdivision\newline
$\Bigl[\frac{2 \pi \lceil \lo n \rceil}{n}, 
\frac{2 \pi (\lceil \lo n \rceil+1)}{n}
\Bigr]$, that is
$\frac{1}{n} \, \frac{1}{\lo n} $.
Finding a smaller upper bound of every term
\eqref{majorefine} 
for the other values 
$j \in \{\lceil \lo n \rceil + 1, \ldots,
J_n\}$
is probably
important. Our intention
is not to do it. 
We will just cut the following
summation into two parts.

\begin{equation}
\label{diffloglambdarefine}
\frac{-1}{\pi}
\int_{\frac{2 \pi \lceil \lo n \rceil}{n}}^{\frac{2 \pi J_n}{n}}
\lo \Bigl(
2 \sin (x/2)
\Bigr) dx 
-
\sum_{j=\lceil \lo n \rceil}^{J_n}
\frac{-2}{n} \, \lo \bigl(2 \, \sin\bigl(\frac{\pi j}{n}\bigr)  \bigr)
\end{equation}

\begin{equation}
\label{estijjrefine_01}
= \sum_{j=\lceil \lo n \rceil}^{X_n - 1}
\left(
\frac{-1}{\pi}\int_{\frac{2 \pi j}{n}}^{\frac{2 \pi (j+1)}{n}}
\lo \Bigl(
2 \sin (x/2)
\Bigr) dx
-
\frac{-2}{n} \, \lo \bigl(2 \, \sin\bigl(\frac{\pi j}{n}\bigr)  \bigr)
\right)
\end{equation}
 \begin{equation}
\label{estijjrefine_02}
+ \sum_{j= X_n}^{J_n - 1}
\left(
\frac{-1}{\pi}\int_{\frac{2 \pi j}{n}}^{\frac{2 \pi (j+1)}{n}}
\lo \Bigl(
2 \sin (x/2)
\Bigr) dx
-
\frac{-2}{n} \, \lo \bigl(2 \, \sin\bigl(\frac{\pi j}{n}\bigr)  \bigr)
\right).
\end{equation}
Each term of \eqref{estijjrefine_01}
is bounded by $\frac{1}{n} \, \frac{1}{\lo n} $
from above, as previously.
On the contrary, each term of
\eqref{estijjrefine_02} is such that
$$
\frac{2 \pi}{n^2}
\left|
\frac{\cos(\xi / 2)}{2 \sin (\xi / 2)}
\right|
\leq 
\frac{1}{n^2}
\left|
\frac{1}{X_n / n}\right| = 
\frac{1}{n} \, \frac{1}{c\, \lo n}.
$$
Summing up the two contributions, we
obtain the following upper bound of
\eqref{diffloglambdarefine}:
$$
(X_n -1 - \lceil \lo n \rceil) \frac{1}{n} \, \frac{1}{\lo n}
+
(J_n - X_n -1) \frac{1}{n} \, \frac{1}{c \,
\lo n}
$$
$$\leq
(c-1) \frac{1}{n} + 
\frac{1}{c}
\frac{\arcsin(\kappa/2)}{\pi \, \lo n}
$$
The first term $(c-1) \frac{1}{n}$ is
a $O\bigl(\frac{1}{n}\bigr)$ and,
multiplied by $\lo n$, 
is inserted in the
Big O of \eqref{petitrn}.
The second term
$\frac{1}{c}
\frac{\arcsin(\kappa/2)}{\pi \, \lo n}
$ is an upper bound of
$\rc_n / \lo n$. 
Let us fix the constant $c$.
Take for instance $c=3$.
The function $(\lo \lo x)^2 / \lo x$ 
tends to 0 when $x$ goes to infinity.
Therefore there exists an integer
$\eta$ such that
all the functions, depending upon $n$, 
``grouped
in the Big O" of
\eqref{petitrn} satisfy (in short form):
$$\Bigl|O\bigl(
\frac{(\lo \lo n)^2}{\lo n}
\bigr)\Bigr|
< \frac{2}{3}
\frac{\arcsin(\kappa/2)}{\pi},
\quad n \geq \eta.
$$
We deduce Lemma \ref{petitrnBIGOlimit}.
\end{proof}

The decomposition 
of $L_{r}(\beta)$
in \eqref{Lrasymptotics}
provide the following
Dobrowolski type minoration
of the Mahler measure ${\rm M}(\beta)
\geq {\rm M}_{r}(\beta)$.

\begin{theorem}
\label{dobrominorationREEL}
Let $\beta > 1$ be a reciprocal 
algebraic integer
such that $\dyg(\beta) \geq \eta$, 
with ${\rm M}(\alpha)
< 1.176280\ldots$. Then
\begin{equation}
\label{dobrominoREEL}
{\rm M}(\beta) \geq 
\Lambda_r \mu_r \, 
-
\frac{\Lambda_r \mu_r \,\arcsin(\kappa/2)}{\pi}\, 
\frac{1}{\lo (\dyg(\beta))}
\end{equation}
\end{theorem}
\begin{proof}
Taking the exponential of
\eqref{Lrasymptotics} gives
$$
{\rm M}_{r}(\beta)
\geq
\exp(L_{r}(\beta))
=
\Lambda_r \mu_r 
\Bigl(
1 
+
\frac{\rc}{\lo n}
+
O\bigl(
\bigl(
\frac{\lo \lo n}{\lo n}
\bigr)^2
\bigr)
\Bigr)
$$
and \eqref{dobrominoREEL} follows from 
Lemma \ref{petitrnBIGOlimit}.
\end{proof}

\

\subsection{The impossible convergence of ${\rm M}(\beta)$ to $1$ when $\beta > 1$ tends to $1$}
\label{S5.9}

Let $(\beta_q)_{q \geq 1}$ be an 
infinite sequence of real
reciprocal algebraic integers $> 1$ tending to 1.
Without loss of generality, we assume
that it is strictly decreasing, with
$\beta_1 \leq 1.176280\ldots$, 
$\house{\beta_1} \leq 1.176280\ldots$
${\rm M}(\beta_{1}) \leq 1.176280\ldots$,
and 
that there exists only one $\beta_q$ in an
interval $(\theta_{n}^{-1},\theta_{n-1}^{-1})$
for some $n$. 
We denote it by $\beta_{q_n}$ and have:
$$\theta_{n}^{-1}\quad < \beta_{q_n} < \quad 
\theta_{n-1}^{-1},
\qquad
\beta_{q_n} \leq 
\house{\beta_{q_n}} \leq
{\rm M}(\beta_{q_n})
\qquad (n \geq n_0).$$
We allow $n$ to run over
a strictly increasing sequence $\mathcal{I}$
of integers 
$n_0 , n_1 , n_2, \ldots$.
We assume that the sequence of Mahler measures
$({\rm M}(\beta_{q_n}))_{n \in \mathcal{I}}$ 
and the sequence of houses
$(\house{\beta_{q_n}})_{n  \in \mathcal{I}}$
are decreasing:
$$1 < \ldots < \beta_{q_{n+1}} < 
\beta_{q_n} <
\beta_{q_{n-1}} < \ldots < 
\beta_{q_{n_0}} \leq \ldots \leq
1.176280\ldots,$$
$$1 < \ldots \leq \house{\beta_{q_{n+1}}} 
\leq \house{\beta_{q_n}} \leq
\house{\beta_{q_{n-1}}} \leq \ldots \leq 
\house{\beta_{q_{n_0}}} \leq \ldots \leq
1.176280\ldots,$$
$$1 \leq \ldots
\leq {\rm M}(\beta_{q_{n+1}}) \leq {\rm M}(\beta_{q_n})
\leq {\rm M}(\beta_{q_{n-1}}) \leq \ldots 
\leq 1.176280\ldots.$$

Let us call "Main Case" when the sequence
$({\rm M}(\beta_{q_n}))_{n \in \mathcal{I}}$
is not stationary after a certain rank.
Let us call "Second Case" when the sequence
$({\rm M}(\beta_{q_n}))_{n \in \mathcal{I}}$
is stationary after a certain rank.
Whatever $\mathcal{I}$,
by
Theorem \ref{dobrominorationREEL},
we have the universal lower bound
and the asymptotic limit satisfying:
$\lim_{n \in \mathcal{I}, n \to +\infty}
{\rm M}(\beta_{q_n}) \geq \Lambda_r \mu_r$.
The universal
lower bound for
${\rm M}(\beta_{q_n})$, for any $n
\in \mathcal{I}$,
is then this limit value
$\Lambda_r \mu_r$
diminished by a
quantity calculated by means of the
Dobrowolsky type inequality, as explained
in Section \ref{S6.4}.

\

Let us give two examples of sequences of 
reciprocal algebraic integers $\beta_{q_n}> 1$ 
which converge
to $1$ which do not possess any lenticular
conjugate (except
$\beta_{q_n}^{-1}$),
to better understand the assumptions involved. 
Recall that the definition of a lenticular
zero is given in Definition \ref{lenticularzerodefinition}.

\subsubsection{``Second Case": Minimal polynomial $P_{\beta}(X) = \widetilde{P_{\beta}}(X^r)$ for $r \geq 2$ - Roots of unity}
\label{S5.9.1}

\

Let $\beta_1 = 
2+\sqrt{3} \in (1, +\infty)$. The minimal polynomial
of $\beta_1$ is
$P_{\beta_1}(X) = X^2 - 4 X + 1$. It is
reciprocal. 

\begin{lemma}
\label{beta_1_referee}
Let $r \geq 1$ be an integer and
denote by $\beta_r$ the real number in 
$(1,\infty)$ with
$\beta_{r}^{r}=\beta_1$. 
Then $\beta_r$ is a reciprocal
algebraic integer,
$\lim_{r \to \infty} \beta_r =1$, with
minimal polynomial
$X^{2r} - 4 X^r + 1$.
The Mahler measure
is constant on the family
$(\beta_r)_{r \geq 1}$, i.e.,
for all $r \geq 1$,  ${\rm M}(\beta_r)=
\beta_1$.
\end{lemma}

\begin{proof}
Say $r \geq 2$ and let $\gamma \in\cb$ satisfy
$\gamma^r = \beta_1$.
Then 
$\gamma$ is an algebraic integer.
We claim $\gamma \not\in K =\qb(\sqrt{3})$.
Indeed,
if $\gamma \in K$, then
$|\gamma| \leq \sqrt{2 + \sqrt{3}}
< 2$ and
$|\gamma'| \leq \sqrt{2 - \sqrt{3}} < 0.6$
with $\gamma'$ the conjugate of $\gamma$.
We write $\gamma = a+b \sqrt{3}$ 
with $a, b \in \zb$, so
$\gamma' = a - b \sqrt{3}$.
Then $b = (\gamma - \gamma')/(2 \sqrt{3})$
and so
$|b| \leq (2+0.6)/(2 \sqrt{3}) < 1$
which implies $b=0$.
So $|a|=|\gamma|<2$ implies
$\gamma \in \{0, \pm 1\}$, a contradiction.

So $\beta_r$  is not the $r$-th power 
of any element of $K$
if $r \geq 2$.
Moreover
$-\beta_1 /4 < 0$ is not a fourth 
power in 
the real quadratic field $K$.
From this and from a classical result from the 
theory of fields, Theorem VI.9.1
in S. Lang, {\it Algebra}, Graduate 
Texts in Mathematics, {\bf 211}, (2002),
we conclude that $X^r - \beta_1$ is irreducible
in $K[X]$ for
all $r \geq 1$.
 
So
$$X^{2 r} - 4 X^r + 1
=P_{\beta_1}(X^r)$$
is the $\zb$-minimal polynomial of $\beta_r$.
In particular $\beta_r$ is reciprocal.

Now, if $\epsilon$ runs over the
set of $r$-th roots of unity,
$\epsilon^{r} = 1$, then the set of the
conjugates of $\beta_r$
is $\{\beta_r \,\epsilon \mid {\rm all}~
\epsilon,
\epsilon^{r} = 1\}$, since
$(\beta_r \,\epsilon)^r = \beta_1$
which implies
${\rm M}(\beta_r) = {\rm M}(\beta_1)$ for all
$r \geq 2$.

The action of the $r$-th roots of unity,
for $r$ tending to infinity, does not produce reciprocal
algebraic integers of smaller Mahler measure
than ${\rm M}(\beta_1)$.
We have
$$1 \ldots < \beta_{r+1} < \beta_r < \ldots
< \beta_{2} < \beta_1 = 2+\sqrt{3}$$
with
$$\ldots =
{\rm M}(\beta_r+1)
=
{\rm M}(\beta_r)
=
\ldots
=
{\rm M}(\beta_2) = {\rm M}(\beta_1) = 2+\sqrt{3}.$$
\end{proof}
In this note we 
are concerned with the attack of the
Conjecture of Lehmer. 
Therefore,
for obtaining
a minorant of ${\rm M}(\beta_r)$,
because of the Mahler measure
remains constant on the sequence 
$(\beta_r)_{r \geq 1}$, we
replace $\dyg(\beta_r)$ by
$\dyg(\beta_{r}^{r})$ and use
the lenticular minorant
relative to $f_{\beta_{r}^{r}}(z)
= f_{\beta_1}(z)$.

The fact is that $\beta_1 = 2 + \sqrt{3}
\geq (1+\sqrt{5})/2$ and that
``$\dyg$" has not yet be defined
on $[(1+\sqrt{5})/2, \infty)$. 
By convention, say
$\dyg(\gamma)=1$ when $\gamma > (1+\sqrt{5})/2$.

So $\dyg(\beta_r)$ is replaced by
$\dyg(\beta_{r}^{r}) = 1$.
The minimal polynomial of
$\beta_1$ is relative to the Main Case.
Now the set of lenticular roots of 
$f_{\beta_1}(z)$ is 
$\{1/\beta_1\}$, and
${\rm M}(1/\beta_1)=
\beta_1$. We have proved
${\rm M}(\beta_r) \geq \beta_1$. 
The equality holds;
in this case the lenticular minorant
is exactly the Mahler measure ${\rm M}(\beta_1)
=
\beta_1$.
\vspace{0.2cm}

The general strategy is the same.
Let $\beta > 1$ be a reciprocal algebraic integer
of dynamical degree
$\dyg(\beta) \geq 260$, for which
the minimal polynomial 
$$P_{\beta}(X) =  1 + \sum_{j=1}^{d-1} a_j X^j + X^d
=
\prod_{k=1}^{d} (X-\beta^{(k)}),
\qquad \beta^{(1)}=\beta,
$$
is relative to the Main Case. 
Then,
for any $q \geq 2$, 
we define the reciprocal algebraic integer
$\beta_q > 1$, root of the reciprocal
polynomial
$$P_{\beta_q}(X) =  1 + \sum_{j=1}^{d-1} a_j X^{q j} + 
X^{q d} =
P_{\beta}(X^q).$$
For all $q \geq 2$, we have:
${\rm M}(\beta_q) = {\rm M}(\beta)$,
considering the equations
$X^q - \beta^{(k)}$,
$|\beta^{(k)}| \geq 1$,
and
if $\gamma^q =  \beta^{(k)}$,
$(\gamma \, \epsilon)^q =  \beta^{(k)}$
for any $\epsilon$,
$\epsilon^q = 1$.
For any $q \geq 2$,
there exists
$n \geq 260$ such that
$$\theta_{n}^{-1} < \beta_q <
\theta_{n-1}^{-1}.
$$
Then to find a minorant of 
${\rm M}(\beta_q)$ we
$${\rm replace~~} n= \dyg(\beta_q)
\qquad {\rm by}\qquad \dyg(\beta)$$
and 
consider the set of the lenticular zeroes of
$f_{\beta_{q}^{q}}(z)
=
f_{\beta}(z)$. Then
we obtain the lenticular minorant
(of Definition \ref{complexeALPHA})
$${\rm M}_{r}(\beta) \leq {\rm M}(\beta_q).$$ 
This amounts to
the case, called Main Case in (N),
relative to $\beta$
and $P_{\beta}(X)$.

\

\subsubsection{An example outside the Problem of Lehmer}
\label{S5.9.2}

Let $n \geq 2$ be an integer. 
Let
$\beta_{n}$ the unique root $> 1$ of the reciprocal
integer polynomial
$$P_{n}(X) = X^{2^n}-2 X^{2^n -1}
- 8 X^{2^{n-1}} - 2 X + 1 \quad \in \zb[X].$$

We show that the Conjecture of Lehmer is true for the
family $\{\beta_n : n \geq 2\}$
(in the sense that there is a common
lower bound $> 1$ for all the Mahler measures
${\rm M}(\beta_n)$). 
This is due to 
Lemma \ref{betan_mahler} 
and to the fact that the house
$\house{\beta_n}= \beta_n$ tends to 2 as $n \to \infty$
((iv) in Proposition \ref{betan_refereeFauxEx}). 
Since 2 does not belong to the interval
$(1, 1.32\ldots)$,
a dynamical degree dyg of 
$\house{\beta_n}$
cannot be defined.
Therefore a lenticular minorant
of the Mahler measure
${\rm M}(\beta_n)$ has no sense.
We cannot
expect any help from any lenticular root
of $f_{\house{\beta_n}}(z)$ in this case, since lenticular 
roots do not exist.

This example does not constitute
an attack of the problem of Lehmer, but,
interestingly, when we consider the integer
polynomial
$$P_{n}(-X) = X^{2^n}+2 X^{2^n -1}
- 8 X^{2^{n-1}} + 2 X + 1 \quad \in \zb[X],$$
we find the sequence of reciprocal
algebraic integers
$(-\gamma_n)_{n \geq 2}$
in $(1, +\infty)$ which are roots of 
$P_{n}(-z)$ and
which converges to 1, 
by (iv) in 
Proposition \ref{betan_refereeFauxEx}.
For this family it is tempting
to try to use the lenticular roots of
$f_{-\gamma_n}(z)$ for $n$ large enough,
which exist, as $n$ tends to infinity,
to establish
a minorant of 
${\rm M}(-\gamma_n) = {\rm M}(\beta_n)$.
It is hopeless.
The conditions of identification
of the lenticular roots of
$f_{-\gamma_n}(z)$
with some zeroes of $P_{n}(-z)$
in $\{z \in \cb :
-\pi/18 \leq \arg z \leq +\pi/18,~
|z| < 1\}$,
are not satisfied (cf \cite{dutykhvergergaugry2});
in this angular sector,
the only zero which is common to
$f_{-\gamma_n}(z)$ and
$P_{n}(-z)$ is 
$(-\gamma_n)^{-1}$ (cf (v) in 
Proposition \ref{betan_refereeFauxEx}).

\begin{proposition}
\label{betan_refereeFauxEx}
Let $n \geq 2$. Then
\begin{itemize}
\item[(i)] the polynomial $P_{n}(X)$ is not of the form
$Q(X^r)$ for some integer $r \geq 2$ and some 
integer polynomial $Q$,
\item[(ii)] $P_n$ is irreducible in
$\qb[X]$,
\item[(iii)] $P_n$ has no root on the unit circle,
\item[(iv)] for $n$ large enough,
$P_n$ admits four real roots
$\gamma_n , \gamma_{n}^{-1} , \beta_n , \beta_{n}^{-1}$
which have the following properties:
$$\gamma_n < -1 < \gamma_{n}^{-1} < 0
< \beta_{n}^{-1} < 1 < \beta_n,$$
$$
\beta_n = \house{\gamma'} > 2~ 
\mbox{for any conjugate}~ \gamma' ~\mbox{of}~
\beta_n, ~\mbox{in particular}~ \beta_n 
= \house{\gamma_n} > 2,$$
$$\lim_{n \to +\infty} \gamma_n = -1,
\qquad \lim_{n \to +\infty} \beta_n = 2,$$
\item[(v)] for $n$ large enough, 
$P_n$ has no root 
in the annulus $\{z \in \cb : |\gamma_{n}^{-1}| < |z| < 1\}$.
\end{itemize}
\end{proposition}

\begin{proof}
(i) This is readily due to the fact that 
$2^n$ and $2^n -1$ are coprime.

(ii) The shifted polynomial
$P_{n}(X+1)$ modulo 2
satifies:
$$P_{n}(X+1) \equiv (X+1)^{2^n} + 1 
\equiv X^{2^n} \quad (\!\!\!\!\!\!\!\mod 2 \zb[X] ).$$
Moreover $P_{n}(-1)= -2$ is not divisible by
$2^2$. By Eisenstein's criterium,
$P_{n}(X+1)$ is irreducible in
$\zb[X]$.

(iii) Assume $z$, $|z|=1$, is a zero of $P_n$. 
We have:
$$|z^{2^n}
+ 2 z^{2^n - 1}
+ 2 z + 1| \leq 1 + 2 + 2 + 1 = 6.$$
Then
$$|P_{n}(z)|
\geq \left| |-8 z^{2^{(n-1)}}|
-
|z^{2^n}+ 2 z^{2^n - 1}
+ 2 z + 1| \right|\geq 2.$$
Contradiction.

(iv) For $n$ large enough, let us prove
$$-1 - \frac{1}{2^{n-1}} < \gamma_n < -1.$$
Let us abbreviate $m= 2^{n-1}$.
Then
$$P_{n}(-1-\frac{1}{m})
=
(-1-\frac{1}{m})^{2 m}
-2 (-1-\frac{1}{m})^{2 m-1}
-8 (-1-\frac{1}{m})^{m}
-2 (-1-\frac{1}{m})
+1$$
$$
=
(1+\frac{1}{m})^{2 m}
+2 (1+\frac{1}{m})^{2 m-1}
-8 (1+\frac{1}{m})^{m}
+2 (1+\frac{1}{m})
+1,
$$
which converges to $e^2 + 2 e^2 -8 e + 3 > 0$
as $m \to \infty$. 
Since $P_{n}(-1) <0$, we obtain the existence of a root
in the interval $(-1 - \frac{1}{2^{n-1}}, -1)$.
By Descartes's rule, the number of positive 
real roots of
$P_{n}(-X)$ is equal to the number of sign changes of the polynomial, which is equal to 2.
Since $P_{n}(-X)$ is reciprocal, there is only one root
in $(-1-\frac{1}{m}, -1)$ ;
$\gamma_n$ is the only root
of $P_{n}(X)$ in this interval. 

By Descartes's rule applied to
$P_{n}(X)$, since $P_n$ is reciprocal, 
there is only one root
of $P_n$ in $(1,\infty)$. 
For $n \geq 2$, observe that
$P_{n}(-2)< 0$,
$P_{n}(-3)>0$, and then
$-3 < \beta_n < -2$.
For $n$ large enough, let us prove
that
$$
\beta_n - 2  = 2 ^{-2^{n-1}+4} (1+o(1)),
\qquad \mbox{as}~ n \to \infty.$$
Indeed, if we set $\beta_n = 2 + u$,
$$P_{n}(2+u)
= 0 =
(2+u)^{2 m}
-2 (2+u)^{2 m-1}
-8 (2+u)^{m}
-2 (2+u)
+1$$
gives, expanding at the first order,
$$u \left[
2^{2m - 1}
-8 m 2^{m-1} - 2 + \ldots
\right] = 
3 + 8 \times 2^m ,
$$
hence the result.

Now, let $\xi$ a primitive root of unity,
$\xi^m = 1$.
Let us prove that the other roots $\gamma'$
of $P_n$, $\gamma' \neq \beta_n$,
$|\gamma'| > 1$, 
are 
\begin{equation}
\label{xiroots}
\xi^{j} \gamma_n + v_j
\qquad \mbox{with}~
|v_j| \leq 2 |\gamma_n|^{-2 m+3} , \quad 
j = 1, 2, \ldots, m-1.
\end{equation}
Indeed, if we set $\gamma' = \xi^j \gamma_n + v_j$, 
we have
$$P_{n}(\gamma')
= 0 =
(\xi^j \gamma_n + v_j)^{2 m}
-2 (\xi^j \gamma_n + v_j)^{2 m-1}
-8 (\xi^j \gamma_n + v_j)^{m}
-2 (\xi^j \gamma_n + v_j)
+1.$$
Consequently, for $j \neq 0$,
expanding at the first order,
$$v_j \left[
2 \times (\xi^j \gamma_n)^{2 m -2}
-8 m (\xi^j \gamma_n)^{m-1}+ \ldots
\right] = 
2 (\xi^j -1) \gamma_n  ,
$$
hence the result.
Since $|\gamma_n| < 1+1/2^{n-1}$
we deduce
from
\eqref{xiroots}, for any $\epsilon > 0$, that
$$\lim_{n \to \infty} |\gamma_n|^{-2 \times 2^{n-1} + 3}
~\leq~ e^{-2} + \epsilon.$$
Thus, since $e^{-2} < 0.5$,
for any root $\gamma'$ of 
$P_n$, for $n$ large enough,
we have:
$|\gamma'| \leq |\gamma_n| + |v_j|
\leq 1+1/m + e^{-2} + \epsilon$, hence
$\house{\gamma'} = \beta_n$.

(v) We use Rouch\'e's Theorem and assume that
$m$ 
is large enough. The root
$1/\gamma_n$ is the unique root in
$[-1,0)$ of $P_{n}(x)$, from (iii).
For all
$r \in [-1, 1/\gamma_{n})$, we find
$0 > P_{n}(r) =
r^{2^n}-2 r^{2^n -1}
- 8 r^{2^{n-1}} - 2 r + 1,$
hence
$$|r|^{2m}+2 |r|^{2m -1}
+ 2 |r| + 1 < 8 |r|^{m}.$$
Therefore, for all 
$z \in \cb$ with $|z|=|r|$, we find
$$|P_{n}(z)+8z^m|
=|z^{2m}-2 z^{2m -1} - 2 z + 1|
\leq |z|^{2m}+2 |z|^{2m -1} + 2 |z| + 1
< |8 z^m|.$$
By Rouch\'e's Theorem
$z \mapsto P_{n}(z)$ and
$z \mapsto -8 z^m$ have the same number of roots,
counted with multiplicities,
in $\{z \in \cb : |z|< |r|\}$.
This number equals $m$
for all $r \in [-1, 1/\gamma_{n})$.
Therefore $P_{n}(z)$ has no roots 
with absolute value in
$(|1/\gamma_{n}|, 1)$. 
\end{proof}

\begin{lemma}
\label{betan_mahler}
For $n$ large enough, ${\rm M}(\beta_n) =
{\rm M}(\gamma_n) \geq 4.6$.
\end{lemma}

\begin{proof}

From  Proposition
\ref{betan_refereeFauxEx}, 
the number of conjugates of $\beta_n$
of modulus $> -\gamma_n$
is exactly $2^{n-1}$.
Therefore,  the Mahler measure of $\beta_n$
satisfies
$${\rm M}(\beta_n) > \beta_n (-\gamma_{n})^{2^{n-1}-1}.$$ 
A lower bound 
of $-\gamma_n$, as a function of $n$, 
is obtained as follows.
We have: $P_{n}(-1) = -2$, 
$P_{n}(-1-\frac{1}{m})$ which tends to
$e^2 + 2 e^2 - 8 e +3 > 0$ as $m \to \infty$. 
Observe that
$P'_{n}(-1) = - 2^n$ which tends to $-\infty$ as
$n \to \infty$,
and 
$$P'_{n}(-1-\frac{1}{m}) = 
2m (-1-\frac{1}{m})^{2m-1}
- 2 (2m -1)(-1-\frac{1}{m})^{2m-2}
- 8 m (-1-\frac{1}{m})^{m-1} -2
$$ 
which tends to $-\infty$ as
$m \to \infty$, with
$\frac{1}{m} P'_{n}(-1-\frac{1}{m})$ converging to
$-6 e^2 + 8 e < 0$ .
Since
$\displaystyle P_{n}(-1-1/m) -P_{n}(\gamma_n) = 
P_{n}(-1-1/m) =
(-1-\frac{1}{m} - \gamma_n) P'_{n}(-1-\frac{1}{m})$ 
at the first order, we deduce that
$m (1+\frac{1}{m} + \gamma_n)$ 
converges to 
$\displaystyle \frac{3 e^2 -8 e+3}{6 e^2 - 8 e}$ 
as $m$ tends to $+\infty$.
Therefore, when $m$ is large,
we have approximately
$$-\gamma_n =  +1 + \frac{1}{m}\frac{3 (e^2 -1)}{6 e^2 - 8 e}.$$
We deduce, for any $\epsilon > 0$,
$${\rm M}(P_n) = {\rm M}(\beta_n) 
~>~ 2 \times e^{\bigl(
\frac{3 (e^2 -1)}{6 e^2 - 8 e}\bigr)} - \epsilon,
\quad \approx 4.66\ldots - \epsilon.
$$
as soon as $n$ is large enough.
\end{proof}

\begin{remark}

Lemma \ref{betan_mahler} shows that
the family $(\beta_n)$ 
is not concerned with the Problem of Lehmer.
It is concerned with
the problem of the topology and the search for
limit points of the set
of Mahler measures 
of algebraic numbers in the half-line
$(1,\infty)$.
Indeed, the Mahler measure ${\rm M}(\beta_n)$
(calculated by Graeffe's method)
tends to a limit, as 
$n \to \infty$.

\

\begin{center}
\begin{tabular}{lc}
\hline

$n$ & \qquad ${\rm M}(\beta_n)=
{\rm M}(-\gamma_n)$\\

\hline
2 & \qquad 7.095126\ldots\\
3 & \qquad 7.273581\ldots\\
4 & \qquad 7.275408\ldots\\
5 & \qquad 7.275409\ldots\\
6 & \qquad 7.275409\ldots\\
\ldots & \qquad \ldots\\
\hline
\end{tabular}
\end{center}

\

Whether this limit $7.275409\ldots$
is algebraic or transcendental
is not known.
The characterization of limit points
has been tackled by Boyd
and Mossinghoff \cite{boydmossinghoff}, Deninger
\cite{deninger} and many subsequent
contributions \cite{vergergaugrySurvey}, 
but never by means of
dynamical zeta functions of numeration systems.
It asks the question of how it could be
investigated by such means.
\end{remark}

\section{Minoration of the Mahler measure M$(\alpha)$ for $\alpha$ a reciprocal complex algebraic integer of house $\house{\alpha} > 1$ close to one. Proofs of the Conjectures}
\label{S6}

Let $\alpha$ be a nonreal complex 
reciprocal algebraic integer, for which
$\house{\alpha} > 1$ (case {\bf (ii)} 
in Section $\S$\ref{S1})
is close to one.  
The minimal polynomial of $\alpha$ is
denoted by
$P_{\alpha}(X) \in \zb[X]$. 
If $\alpha^{(i)}$  
is a conjugate of maximal
modulus of $\alpha$, $\alpha^{(i)}$ 
is conjugated with
$\overline{\alpha^{(i)}}$,
$(\alpha^{(i)}){}^{-1}$,
$\overline{(\alpha^{(i)}){}^{-1}}$; 
the house
$\house{\alpha}$ of $\alpha$, resp.
its inverse $\house{\alpha}^{-1}$,
is root of the quadratic equation
$$X^2 - \alpha^{(i)} \overline{\alpha^{(i)}} =0,
\qquad \mbox{resp. of} \quad
X^2 - (\alpha^{(i)})^{-1} (\overline{\alpha^{(i)}})^{-1} =0 .$$
The house $\house{\alpha}$ 
and its inverse $\house{\alpha}^{-1}$
are real algebraic integers
of degree $\leq \deg(\alpha)+2$ for which 
we assume
$1 <\house{\alpha} < \Theta = \theta_{5}^{-1}$.
The mapping $\alpha \to \house{\alpha}, 
\mathcal{O}_{\overline{\qb}} \to \mathcal{O}_{\overline{\qb}}
\cap (1, \infty)$ is 
not continuous.

Writing $\beta = \house{\alpha}$, 
the preceding analytic functions
$\zeta_{\beta}(z)$,
$f_{\beta}(z)$
of the R\'enyi-Parry dynamical system 
of the $\beta$-shift,
defined in Section $\S$\ref{S4},
can be applied once $\house{\alpha}> 1$ 
is close enough to $1^{+}$; 
the quantities
$\dyg(\beta)$, 
$\lc_{\beta}$, ${\rm M}(\beta)$,
${\rm M}_{r}(\beta)$ are also
well-defined.
The minoration
of the Mahler measure
${\rm M}(\house{\alpha})$
is of Dobrowolski type as in 
Theorem \ref{dobrominorationREEL},
as a function of the dynamical degree
$\dyg(\beta)$.
The domain $\Omega_n$ on which there is
fracturability of
the polynomial $P_{\beta}(z)$
is defined in Proposition
\ref{splitBETAdivisibility+++}.

\subsection{Fracturability of the minimal polynomial of $\alpha$ by the Parry Upper function at $\house{\alpha}$}
\label{S6.1}

The following Theorem is
an extension of 
Theorem \ref{splitBETAdivisibility}
and Proposition \ref{splitBETAdivisibility+++}.

\begin{theorem}
\label{divisibilityALPHA}
Let $\alpha$ be a nonreal complex 
reciprocal algebraic integer, 
for which
$1 < \house{\alpha} < \Theta =
\theta_{5}^{-1}$, 
with ${\rm M}(\alpha)
< 1.176280\ldots$.
Denote
$\beta = \house{\alpha}$.
Then the following formal decomposition of 
the minimal polynomial of $\alpha$
\begin{equation}
\label{decompoAlphaBeta}
P_{\alpha}(X) = P_{\alpha}^{*}(X)
=
U_{\alpha}(X) \times
f_{\beta}(X)
\end{equation}
holds, as the product
of the Parry Upper function 
at $\beta$
\begin{equation}
\label{formlacualphaformel}
f_{\beta}(X)= G_{\dyg(\beta)}(X)
+ X^{m_1} + X^{m_2} + X^{m_3}+ \ldots.
\end{equation}
with
$m_0:= \dyg(\beta)$,
$m_{q+1} - m_q \geq \dyg(\beta)-1$ for $q \geq 0$,
and the invertible formal series 
$U_{\alpha}(X) \in \zb[[X]]$, 
quotient of $P_{\alpha}$ by
$f_{\beta}$.

The specialization $X \to z$ of the formal variable
to the complex variable leads to
the identity between analytic functions,
obeying the Carlson-Polya dichotomy  
as:
\begin{equation}
\label{decompozzzALPHA}
P_{\alpha}(z) = U_{\alpha}(z) \times
f_{\beta}(z)
\quad
\left\{
\begin{array}{cc}
\mbox{on}~ \cb & \mbox{if $\house{\alpha}$ is a Parry number},\\
&\\
\mbox{on}~ |z| < 1 & 
\mbox{~if~} \house{\alpha} \mbox{is 
a nonParry number, with $|z|=1$}\\
& \mbox{as natural boundary
for both $U_{\alpha}$ and $f_{\beta}$.}
\end{array}
\right.
\end{equation}
Assume $\dyg(\beta) \geq 260$. 
Then

\parindent=0cm
(i) 
the minimal polynomials
of $\alpha$ and $\beta$
are equal: $P_{\alpha} = P_{\beta}$
and $\beta$ is reciprocal,

(ii) the identity 
$U_{\alpha}(z) = -\zeta_{\beta}(z) \times
P_{\beta}(z)$ holds
as meromorphic functions
on the domain of definition
of  $f_{\beta}$,

(iii)
the integer power series 
$U_{\alpha}(z)$
is a nonconstant holomorphic function
on the domain $\Omega_n$, and has no zero
in $\Omega_n$,

(iv) the lenticulus of lenticular zeroes of 
$\alpha$ is that of $\beta$, namely
$$\lc_{\alpha} =
\lc_{\beta} = \{
\beta^{-1}\}
\cup 
\bigcup_{j=1}^{J_n}
(\{\omega_{j,n}\} 
\cup 
\{\overline{\omega_{j,n}}\}) 
\subset \Omega_n,$$ 
where the
$\omega_{j,n}$ and 
$\overline{\omega_{j,n}}$ are the conjugates
of $\beta^{-1}$,

(v) 
for any zero
$\omega_{j,n} \in \lc_{\beta}$, 
\begin{equation}
\label{nevervanishesA}
U_{\alpha}(\omega_{j,n}) = 
\frac{P'_{\alpha}(\omega_{j,n})}{f'_{\beta}(\omega_{j,n})}
\neq 0 \, ,~~
U_{\alpha}(\overline{\omega_{j,n}}) = 
\frac{P'_{\alpha}(\overline{\omega_{j,n}})}{f'_{\beta}(\overline{\omega_{j,n}})}
\neq 0 \,
~~ \mbox{and}
~~
U_{\alpha}(\beta^{-1}) = 
\frac{P'_{\alpha}(\beta^{-1})}
{f'_{\beta}(\beta^{-1})}
\neq 0.
\end{equation}

\end{theorem}

\begin{proof}
There exists
an integer $n \geq 6$ such that
$\house{\alpha}$
lies between two successive Perron
numbers of the family
$(\theta_{n}^{-1})_{n \geq 5}$, as
$
\theta_{n}^{-1} \leq \house{\alpha}
< \theta_{n-1}^{-1}$. 
Then the Parry Upper function
$f_{\house{\alpha}}(z)$ at 
$\house{\alpha}$
has the form:
\begin{equation}
\label{formlacuALPHA}
f_{\house{\alpha}}(z) = 
-1 + z + z^n + z^{m_1} + z^{m_2}
+ z^{m_{3}}+ \ldots
\end{equation}
with $m_0 = n$ and
$m_{q+1}-m_q \geq n-1$ for $q \geq 0$.
Whether $\house{\alpha}$ is a Parry number or 
a nonParry number is unkown. 
In any case,
$f_{\house{\alpha}}(\house{\alpha}^{-1}) = 0$
and
the zero
$\house{\alpha}^{-1}$ of $f_{\house{\alpha}}(z)$
is simple.
Let us write the Parry Upper function 
in the generic form
$f_{\house{\alpha}}(z) = 
- 1 + \sum_{j \geq 1} t_j z^j$.

Let us show that the formal decomposition
\eqref{decompoAlphaBeta} is always possible. 
We proceed as in the proof of 
Theorem \ref{splitBETAdivisibility}.
Indeed, if
we put
$U_{\alpha}(X) = 
-1 + \sum_{j \geq 1} b_j X^j$, and
$P_{\alpha}(X) = 1 + a_1 X + a_2 X^2 + \ldots
a_{d-1} X^{d-1}
+ X^d $, (with $a_j = a_{d-j}$),
the formal identity 
$P_{\alpha}(X) = U_{\alpha}(X) \times 
f_{\house{\alpha}}(X)$
leads to the existence of the coefficient 
vector $(b_j)_{j \geq 1}$ of
$U_{\alpha}(X)$, as a function
of $(t_j)_{j \geq 1}$ and
$(a_i)_{i = 1,\ldots, d-1}$, as:
$b_1 = -(a_1 + t_1)$,
and, for $r = 2, \ldots, d-1$,
\begin{equation}
\label{bedeRecurrencedebutALPHA}
b_r = -(t_r + a_r - 
\sum_{j=1}^{r-1}b_j t_{r-j}) 
\quad
\mbox{with} \quad
b_d = -(t_d + 1 - 
\sum_{j=1}^{d-1} b_j t_{r-j}),
\end{equation}
\begin{equation}
\label{bedeRecurrenceALPHA}
b_r = -t_r + 
\sum_{j=1}^{r-1} b_j t_{r-j} \quad \mbox{for}~~  r > d.
\end{equation}
For $j \geq 1$,
$b_j \in \zb$,
and the integers
$b_r , r > d$,
are determined recursively
by \eqref{bedeRecurrenceALPHA}
from 
$\{b_0 = -1, b_1 , b_2 , \ldots , b_d \}$.
Every $b_j$ in 
$\{b_1 , b_2 , \ldots , b_d \}$
is computed from the coefficient vector
of $P_{\alpha}(X)$
using \eqref{bedeRecurrencedebutALPHA}, 
starting by $b_1 = -1 - a_1$. The disk 
of convergence of $U_{\alpha}(z)$
has a radius  
$\geq \theta_{n-1}$ by Theorem 
\ref{splitBETAdivisibility}.

Let us show (i). Assume
the contrary, i.e.
$P_{\alpha} \neq P_{\beta}$.
We will proceed
as in \S \ref{S5.4.2} by constructing
another 
rewriting trail from
"$P_{\alpha}$" 
to
"$f_{\beta}$", the one
from 
"$P_{\beta}^{*}$" 
to
"$f_{\beta}$" being already studied
(a priori the two polynomials
$P_{\beta}$ and $P_{\beta}^{*}$
may be different).

The starting point is 
the two identities
$P_{\alpha}(\alpha^{-1})=0$
and
$f_{\beta}(\beta^{-1})=0$.
They provide
a $\alpha$-representation of 1
and a
$\beta$-representation of 1, the second 
one being the R\'enyi $\beta$-expansion of 1:
\begin{equation}
\label{equa1Pcomplex}
1= -a_1 \alpha^{-1} - a_2 \alpha^{-2}
- a_3 \alpha^{-3} + \ldots - a_{d-1} 
\alpha^{-(d-1)} - \alpha^{-d}
=
1 - P_{\alpha}(\alpha^{-1})
,
\end{equation}
\begin{equation}
\label{equa1fcomplex}
1= t_1 \beta^{-1} + t_2 \beta^{-2}
+
t_3 \beta^{-3}+ \ldots = 1 + 
f_{\beta}(\beta^{-1}).
\end{equation}
The goal consists in 
constructing an infinite chain
of intermediate
$(\alpha, \beta)$-representa-tions of 1 
between them, by
``restoring" 
the digits
$t_i$ of $f_{\beta}$ in
\eqref{equa1fcomplex}
one after the other
from \eqref{equa1Pcomplex}.
The first step is the addition
of 
$(\beta^{-1}+
a_1 \alpha^{-1}) P_{\alpha}(\alpha^{-1})
= 0$
to \eqref{equa1Pcomplex}.
Then, denoting by
$S_{q}(z) = -1 +\sum_{j=1}^{q} t_j z^j, 
q \geq 1$,
the $q$-th polynomial section of $f_{\beta}$,
 we obtain
$$
1= \beta^{-1}
+ R_{1}(\alpha^{-1}, \beta^{-1})
=
(1+S_{1}(\beta^{-1}))
+
R_{1}(\alpha^{-1}, \beta^{-1})
$$
with $R_{1}(X,Y) \in \zb[X,Y]$,
$$
R_{1}(\alpha^{-1}, \beta^{-1})
=
(a_{1}^{2} - a_2) \alpha^{-2}
+
(a_{1} a_{2} - a_3) \alpha^{-3}
+
(a_{1} a_{3} - a_4) \alpha^{-4}+ \ldots
$$
$$+
a_{1} \alpha^{-1} \beta^{-1}
+
a_{2} \alpha^{-2} \beta^{-1}
+
a_{3} \alpha^{-3} \beta^{-1}
+\ldots
$$
Let $A_{1}(\alpha^{-1}, \beta^{-1})
= -1 + (a_{1} \alpha^{-1} + \beta^{-1})
$. 
The bivariate polynomial
$A_{1}(X,Y) =
-1 + (a_1 X + Y)$ belongs to
$\zb[X,Y]$.
We deduce, at the first step,
$$0 =
A_{1}(\alpha^{-1}, \beta^{-1})
P_{\alpha}(\alpha^{-1})
= S_{1}(\beta^{-1})
+
R_{1}(\alpha^{-1}, \beta^{-1})
$$
Iterating this process we deduce,
for every
$q \geq 1$,
the existence
of two polynomials
$A_{q}, R_{q} \in \zb[X,Y]$, with
$\deg_{X}(A_q) \leq q$,
$\deg_{Y}(A_q) \leq q$,
$A_{q}(0,0) = -1$,
$R_{q}(0,0)=0$,
such that
\begin{equation}
\label{AqPbetadecompocomplex}
0=
A_{q}(\alpha^{-1}, \beta^{-1})
P_{\alpha}(\alpha^{-1}) 
= 
S_{q}(\beta^{-1}) +
R_{q}(\alpha^{-1}, \beta^{-1}).
\end{equation}
But, for $q \geq 1$,
$0 = S_{q}(\beta^{-1}) +
(f_{\beta}(\beta^{-1})
-
S_{q}(\beta^{-1}) )$.
Hence, the
quantities
$$R_{q}(\alpha^{-1}, \beta^{-1})
=f_{\beta}(\beta^{-1})
-
S_{q}(\beta^{-1}),
\qquad q \geq 1,$$
do not depend upon
$\alpha^{-1}$, but only upon
$\beta^{-1}$.
Let $\sigma$ the $\qb$-endomorphism
of the number field
$\qb(\alpha, \beta, \beta^{-1})
\neq \qb(\beta, \beta^{-1})$ defined by
$\sigma(\alpha^{-1})=
\beta^{-1}$ leaving invariant
the subfield
$\qb(\beta, \beta^{-1})$.
Applying
$\sigma$ to
\eqref{AqPbetadecompocomplex}
gives
\begin{equation}
\label{AqPbetadecompocomplexEGAL}
0=
A_{q}(\beta^{-1}, \beta^{-1})
\,P_{\alpha}(\beta^{-1}), \qquad q \geq 1.
\end{equation}
If we assume that
$P_{\alpha}(\beta^{-1}) \neq 0$
then we should
have all  the (nonzero) polynomials
$A_{q}(X,X)$, $q \geq 1,$ 
in the ideal generated by
$P_{\beta}^{*}(X)$ in $\zb[X]$.
As multiples of
$P_{\beta}^{*}(X)$
we should have:
$\deg(A_{q}(X,X)) \geq \deg( P_{\beta}^{*})$.
But $\dyg(\beta) \geq 260$ implies that
$\deg(\beta) = \deg(P_{\beta}^{*})$
is large since
the number of roots
of $P_{\beta}^{*}$
is at least the number 
$1 + 2 J_n$ of lenticular roots
(Theorem \ref{thm2lenticuli}).
Therefore it suffices
to take a value of $q$ small enough
to obtain a contradiction.
We deduce
$P_{\alpha}(\beta^{-1}) = 0$.
Therefore
$P_{\alpha}=P_{\beta}^{*}$ 
and $\beta=\house{\alpha}$ is reciprocal
since $P_{\alpha}$ is reciprocal.
We deduce
$P_{\alpha} = P_{\beta}$.

Let us prove (ii). 
Since
$\beta > 1$ is a reciprocal algebraic integer,
$\beta$ is not a simple Parry number.
By Theorem \ref{parryupperdynamicalzeta},
$\zeta_{\beta}(z)=
- 1/f_{\beta}(z)$.
The Parry Upper function $f_{\beta}(z)$
has coefficients in the finite set
$\{-1, 0, 1\}$, and 
therefore obeys the Carlson-Polya dichotomy.
The domain of definition
of $\zeta_{\beta}(z)$, as a 
meromorphic function, is that of
$f_{\beta}(z)$, that is: $\cb$ if
$\beta$ is a Parry number,
the open unit disk 
if $\beta$ is not a Parry number.
On this domain of definition
the fracturability of the minimal
polynomial $P_{\alpha}$ comes
from that of
$P_{\beta}$ by Proposition \ref{splitBETAdivisibility+++}, as
$$P_{\alpha}(z) =
P_{\beta}(z)
=
(-\zeta_{\beta}(z) P_{\alpha}(z))
\times 
f_{\beta}(X).
$$

Let us prove (iii), (iv) and (v). 
The holomorphy
of $-\zeta_{\beta}(z) P_{\alpha}(z)$
on $\Omega_n$ is a consequence
of Proposition \ref{splitBETAdivisibility+++}.
The relations \eqref{nevervanishesA}
come from the derivation of
\eqref{decompozzzALPHA} at the conjugates
$\omega_{j,n}$ and
$\overline{\omega_{j,n}}$
of $\beta^{-1}$ in the domain
$\Omega_n$.

\begin{definition}
\label{complexeALPHA}
Let $\alpha$ be a nonreal complex reciprocal 
algebraic integer
such that
$1 < \house{\alpha} < \Theta = \theta_{5}^{-1}$.
For $n = \dyg(\alpha) \geq 260$,
with ${\rm M}(\alpha)
< 1.176280\ldots$,
the {\em dynamical degree of $\alpha$} 
is defined by
$\dyg(\alpha) := \dyg(\house{\alpha})$;
the reduced Mahler measure of $\alpha$ 
is
$${\rm M}_{r}(\alpha) := 
{\rm M}_{r}(\house{\alpha})
= \house{\alpha}
\prod_{j=1}^{J_n} |\omega_{j,n}|^{-2}
.$$
\end{definition}

The {\em domain of fracturability}
of the minimal polynomial
$P_{\alpha}(X)$ is 
the largest domain
in the open unit disk
on which the analytic function
$-\zeta_{\beta}(z) P_{\alpha}(z)$ is a
nonconstant holomorphic
function which does not vanish in this domain.
It contains $\Omega_n$.
\end{proof}

\subsection{A Dobrowolski type minoration with the dynamical degree of the house - Proof of Theorem~\ref{mainDOBROWOLSLItypetheorem}} 
\label{S6.2}

Let $\alpha$ be a reciprocal algebraic 
integer such that 
$\beta = \house{\alpha}$ has
dynamical degree $\dyg(\beta) \geq 260$, 
with ${\rm M}(\alpha)
< 1.176280\ldots$.
The first nonreal complex root
$\omega_{1,n}$ of the lenticulus
$\lc_{\alpha}$
is a continuous function
of $\beta$ by 
\cite{flattolagariaspoonen}. 
By Corollary 
\ref{zeroesParryUpperfunctionContinuity} and
Theorem \ref{divisibilityALPHA} the other
lenticular roots of 
the Parry Upper function
$f_{\house{\alpha}}(z)$
are continuous functions of
$\beta = \house{\alpha}$. 
These facts suggest the 
Conjecture that the (true) Mahler measure
${\rm M}(\alpha)$ is a continuous function
of the house $\house{\alpha}$ of
$\alpha$.
On the contrary, the nonderivability
of the function $\beta 
= \house{\alpha} \to \omega_{1,\dyg(\beta)}$ 
conjectured in
 \cite{flattolagariaspoonen}
suggests that
 the (true) Mahler measure
${\rm M}(\alpha)$ is nowhere 
derivable as a function
of $\beta 
= \house{\alpha}$.

By Theorem \ref{divisibilityALPHA}, 
since
$P_{\alpha} = P_{\house{\alpha}}$, the 
minoration of the Mahler measure
M$(\alpha)$ is deduced from
the minoration
of the Mahler measure
M$(\beta)$. The following Theorems
are readily deduced from
Theorem \ref{Lrasymptotictheorem},
Lemma \ref{petitrnBIGOlimit} and
Theorem \ref{dobrominorationREEL}.

 \begin{theorem}
\label{Lrasymptotictheoremcomplexe}
Let $\alpha$,
$\house{\alpha} > 1$, 
be a reciprocal algebraic integer
such that
$n=\dyg(\house{\alpha}) \geq 260$, with ${\rm M}(\alpha)
< 1.176280\ldots$.
The asymptotic expansion
of the minorant
$L_{r}(\alpha)$ of \newline
\,$\lo {\rm M}_{r}(\alpha)$
is
\begin{equation}
\label{Lrasymptoticscomplexe}
L_{r}(\alpha) := \lo \Lambda_r \mu_r 
+
\frac{\rc}{\lo n}
+
O\bigl(
\bigl(
\frac{\lo \lo n}{\lo n}
\bigr)^2
\bigr),
\quad
\mbox{with}\quad |\rc| < \frac{\arcsin(\kappa/2)}{\pi}
\end{equation}
and $\rc$ depending upon $\house{\alpha}$ and $n$.
\end{theorem}

\begin{theorem}
\label{dobrominorationCOMPLEX}
Let $\alpha$,
$\house{\alpha} > 1$, 
be a reciprocal 
algebraic integer
such that $\dyg(\alpha) \geq \eta$, with ${\rm M}(\alpha)
< 1.176280\ldots$. Then
\begin{equation}
\label{dobrominoCOMPLEX}
{\rm M}(\alpha) \geq 
\Lambda_r \mu_r \, 
-
\frac{\Lambda_r \mu_r \,\arcsin(\kappa/2)}{\pi}\, 
\frac{1}{\lo (\dyg(\alpha))}
\end{equation}
\end{theorem}

In the case where $\alpha$ is the conjugate of
a Perron number
$\theta_{n}^{-1}$ , for some 
$n \geq 260$,
the minorant in \eqref{dobrominoCOMPLEX}
has to be replaced
by that
of Theorem \ref{maincoro5}
for the trinomials $G_n$,
taking higher values.
Comparatively, if $\alpha$ is a nonzero 
nonreciprocal
algebraic integer, which is not a root of unity,
the Mahler measure M$(\alpha)$
is uniformly bounded from below 
by Smyth's lower bound $\Theta$ \cite{smyth}.

\subsection{Proof of the Conjecture of Lehmer (Theorem \ref{mainLEHMERtheorem})}
\label{S6.4}

Let $\alpha \neq 0$ be a
reciprocal algebraic integer
which is not a root of unity,
such that $\dyg(\alpha) \geq \eta$
with $\eta \geq 259$.
Since ${\rm M}(\alpha) =
 {\rm M}(\alpha^{-1})$
there are three cases to be considered:
\begin{enumerate}
\item[(i)]
the house of $\alpha$ satisfies $\house{\alpha} \geq \theta_{5}^{-1}$,
\item[(ii)]
the dynamical degree of $\alpha$ satisfies:
$6 \leq \dyg(\alpha) < \eta$, with ${\rm M}(\alpha)
< 1.176280\ldots$,
\item[(iii)]
the dynamical degree of $\alpha$ 
satisfies: 
$\dyg(\alpha) \geq \eta$, with ${\rm M}(\alpha)
< 1.176280\ldots$.
\end{enumerate}
In the first case, ${\rm M}(\alpha) \geq 
\theta_{5}^{-1} \geq \theta_{259}^{-1}
\geq \theta_{\eta}^{-1} > 1$. 
In the second case,
${\rm M}(\alpha) \geq 
\theta_{\eta}^{-1}$.
In the third case, 
the Dobrowolski type inequality
\eqref{dodobrobro}
gives the following
lower bound of the Mahler measure
$${\rm M}(\alpha) \geq 
\Lambda_r \mu_r \, 
-
\frac{\Lambda_r \mu_r \,\arcsin(\kappa/2)}{\pi ~\lo (\dyg(\alpha))}\, 
$$
$$
\geq
\Lambda_r \mu_r \, 
-
\frac{\Lambda_r \mu_r \,\arcsin(\kappa/2)}{\pi ~\lo (\eta)}\, 
\geq
\Lambda_r \mu_r \, 
-
\frac{\Lambda_r \mu_r \,\arcsin(\kappa/2)}{\pi ~\lo (259)}\, 
, = 1.14843\ldots
$$
by 
Theorem \ref{MahlerMINORANTreal}
and Theorem \ref{dobrominorationCOMPLEX}.
This lower bound is numerically greater than
$\theta_{259}^{-1} = 1.016126\ldots$,
itself greater than
$\theta_{\eta}^{-1}$. 
Therefore, in any case,
the lower bound 
$\theta_{\eta}^{-1}$
of ${\rm M}(\alpha)$ holds true.
We deduce the claim.

\subsection{Proof of the Conjecture of Schinzel-Zassenhaus (Theorem \ref{mainSCHINZELZASSENHAUStheorem})}
\label{S6.5}

\begin{proposition}
\label{degreeBETApolymini}
Let $\alpha$,
$\house{\alpha} > 1$, 
 be a reciprocal algebraic
integer such that
$\dyg(\alpha) \geq 260$,
with ${\rm M}(\alpha)
< 1.176280\ldots$. 
The degree 
$\deg(\alpha)$ 
of $\alpha$ 
is related to its dynamical degree 
$\dyg(\alpha)$ by
\begin{equation}
\label{dygdeg}
\dyg(\alpha) \Bigl(
\frac{2 \arcsin\bigl( \frac{\kappa}{2} \bigr)}{\pi}
\Bigr)
+
\Bigl(
\frac{2 \kappa \, \lo \kappa}
{\pi \,\sqrt{4 - \kappa^2}} 
\Bigr)
 ~\leq~ \deg(\alpha).
 \end{equation}
\end{proposition}

\begin{proof}
By Theorem \ref{omegajnexistence}
and Proposition 
\ref{splitBETAdivisibility+++}
the number of zeroes in the lenticulus
$\lc_{\alpha}$ is
$1 + 2 J_n $,
with $n:= \dyg(\alpha)$
; these zeroes
are all conjugates of
$\alpha$. The total number of conjugates 
of $\alpha$
is the degree $\deg(\alpha)$ of
the minimal polynomial $P_{\alpha}$.
By Proposition \ref{argumentlastrootJn},
$$1 + 2 J_n = 
\frac{2 n}{\pi}
\bigl(
\arcsin\bigl( \frac{\kappa}{2} \bigr) 
\bigr)
+
\Bigl(
\frac{2 \kappa \, \lo \kappa}
{\pi \,\sqrt{4 - \kappa^2}}
\Bigr)
+ \bigl(1 
+ 
\frac{1}{n}
O\bigl(
\bigl(
\frac{\lo \lo n}{\lo n}
\bigr)^2
\bigr)\bigr).
$$
The inequality 
\eqref{dygdeg} follows.

\end{proof}

\begin{theorem}
\label{schinzelZ}
Let $\alpha$,
$\house{\alpha} > 1$, 
be a
reciprocal algebraic integer
which is not a root of unity,
with ${\rm M}(\alpha)
< 1.176280\ldots$.
Then
\begin{equation}
\label{kappo}
\house{\alpha} \geq 1 + \frac{c}{\deg(\alpha)},
\qquad
\mbox{with}~~ 
c = \theta_{\eta}^{-1} -1.
\end{equation}
\end{theorem}

\begin{proof}
There are two cases: either (i) $\house{\alpha} 
\geq \theta_{\eta}^{-1}$, or (ii) $n \geq \eta + 1$.
(i) If $\house{\alpha} \geq 
\theta_{\eta}^{-1}$,
then,
whatever the degree $\deg(\alpha) \geq 1$,
$$\house{\alpha} \geq 1 +\frac{(\theta_{\eta}^{-1} -1)}
{\deg(\alpha)}.$$

(ii) The minoration of the house
$\beta = \house{\alpha}$ can easily be obtained
as a function of the dynamical degree of $\alpha$.
Let $n = \dyg(\beta)$ and assume  
$n \geq \eta + 1$.
By definition $\theta_{n}^{-1}
\leq \beta < \theta_{n-1}^{-1}$.
Theorem 1.8 in \cite{vergergaugry6} 
(cf also \cite{vergergaugry6} \S 5.3)
implies 
\begin{equation}
\label{tintin}
\beta = \house{\alpha}
~\geq~ \theta_{n}^{-1}
~\geq~ 1 + 
\frac{(\lo n) (1 - \frac{\lo \lo n}{\lo n})}{n}.
\end{equation}
From Proposition
\ref{degreeBETApolymini},
\begin{equation}
\label{milou}
\frac{1}{n} =
\frac{1}{\dyg(\beta)}
\geq
\frac{2 \arcsin(\kappa/2)}{\pi~\deg(\alpha)}
\left(
1 + \frac{\kappa \lo \kappa}{n \, \arcsin(\kappa/2) \, \sqrt{4 - \kappa^2}}
\right) .
\end{equation}
The function
$\frac{\lo x - \lo \lo x}{\lo x}\left(
1 + \frac{\kappa \lo \kappa}{x \, \arcsin(\kappa/2) \, \sqrt{4-\kappa^2}}
\right)$ 
is increasing for $x \geq 260$.
From \eqref{tintin} and \eqref{milou} 
we deduce
$$\house{\alpha} \geq 1 + \frac{\tilde{c}}{\deg(\alpha)}$$
with
$\tilde{c}= \frac{2}{\pi}
\frac{\lo 260 - \lo \lo 260}{\lo 260}
\Bigl(
\arcsin(\kappa/2) + \frac{ \kappa \lo \kappa}{260 \, \sqrt{4-\kappa^2}}
\Bigr) = 0.0375522\ldots$.

From (i) and (ii), we deduce
that \eqref{kappo} holds with
$c = \min\{\tilde{c}, (\theta_{\eta}^{-1} -1)\}
= (\theta_{\eta}^{-1} -1)$
since
$\theta_{259}^{-1} -1 = 0.016126\ldots
< \tilde{c}$, 
for every
nonzero reciprocal algebraic integer
$\alpha$ which is not a root of unity.
\end{proof}

\section{Proof of the Conjecture of Lehmer for Salem numbers}
\label{S7}

The set of Pisot numbers admits the minorant
$\Theta$ by a result of Siegel \cite{siegel}. 
Theorem \ref{secondseriesUU} 
implies boundedness 
from below to 
the set of Salem numbers.

\subsection{Existence and localization of the first nonreal root of the Parry Upper function $f_{\beta}(z)$ of modulus $<1$ in the cusp of the fractal of Solomyak}
\label{S7.1}

The lenticular roots of the 
Parry Upper function
$f_{\beta}(z)$ were
studied in Section \S \ref{S5}
and identified as conjugates of 
the reciprocal algebraic integer $\beta > 1$,
for $\beta$ close enough to one,
${\rm M}(\beta) < 1.176280\ldots$.
Their number is $1+ 2 J_n$. This 
was done under the assumption
that the quantity $J_n$
has a sense, that is for 
a regime of asymptotic expansions 
of the roots $z_{j,n}$
of $G_n$ the closest to $|z|=1$
valid outside the "bump sector" 
(cf Appendix). This is the reason why
$n$ has been taken above $260$.

In the present Section the ``emergence"
of such lenticuli of roots is used, 
at small values of $n$. 
By emergence is meant that 
the number of lenticular roots
of $f_{\beta}(z)$ takes the odd values
$1, 3, 5, \ldots$ when
$n$ increases from 6 to higher values, 
the lenticuli being successively
of the type
$$
\{\beta^{-1}\},
\{\beta^{-1}, \omega_{1,n},
 \overline{\omega_{1,n}}\},
\{\beta^{-1}, \omega_{1,n},
 \overline{\omega_{1,n}},
\omega_{2,n},
 \overline{\omega_{2,n}} \}, \ldots
$$
for 
$\theta_{n}^{-1} \leq \beta
<
\theta_{n-1}^{-1}$, the value 3 corresponding to the ``emergence".
This study of the emergence of 
3-tuples of lenticular roots
does not call for the regime of
asymptotic
expansions of the roots of
$G_n$ outside the ``bump sector".
On the contrary it
calls for a regime
relative to the lenticular roots
which lie
inside the ``bump sector",
in particular
for the roots of $G_n$
the closest to the real axis, to deal with
the control of the existence of
the lenticular
root $\omega_{1,n}$.
The different regimes of asymptotic
expansions of the roots
$z_{j,n}$ of $G_n$
are recalled in Section \S \ref{S4}
and the limits of validity
of the 3 regimes summarized 
in the Appendix.

The regime of 
asymptotic expansions
dedicated to the first nonreal
lenticular root $\omega_{1,n}$
has 
Proposition \ref{zedeUNmodule}
for consequence. 
Proposition \ref{zedeUNmodule}
is used in
the proof of
Theorem \ref{_cercleoptiSALEM}.
Corollary
\ref{smallestSALEM}
asserts the existence
of a lenticulus of at least 3 roots
of $f_{\beta}(z)$ 
as soon as 
$n = \dyg(\beta)$ is
$\geq 32$.
The identification of
these lenticular roots
as conjugates
of $\beta^{-1}$
is done in \S \ref{S7.2}.

\begin{theorem}
\label{_cercleoptiSALEM}
Let $n \geq 32$.
Denote by $C_{1,n}
:= \{z \mid |z-z_{1,n}| = 
\frac{\pi |z_{1,n}|}{n \, a_{\max}} \}$
the circle centered at the first root
$z_{1,n}$ of $G_{n}(z) = -1 + z + z^n$.
Then the 
condition of Rouch\'e 
\begin{equation}
\label{rouchecercleSALEM}
\frac{
\left|z\right|^{2 n -1}}{1 - |z|^{n-1}}
~<~
\left|-1 + z + z^n \right| , 
\qquad \mbox{for all}~ z \in C_{1,n},
\end{equation}
holds true. 
\end{theorem}

\begin{proof}
Let $a \geq 1$ and $n \geq 18$.
Denote by $\varphi := \arg (z_{1,n})$
the argument of the first root
$z_{1,n}$ (in ${\rm Im}(z) > 0$).  
Since $-1 + z_{1,n} + z_{1,n}^n = 0$, 
we have $|z_{1,n}|^n = |-1 + z_{1,n}|$.
Let us write $z= z_{1,n}+ 
\frac{\pi |z_{1,n}|}{n \, a} e^{i \psi}
=
z_{1,n}(1 + \frac{\pi}{a \, n} e^{i (\psi - \varphi)})$
the generic element belonging to 
$C_{1,n}$, with
$\psi \in [0, 2 \pi]$.
Let $X := \cos(\psi - \varphi)$.
Let us show that if the inequality
\eqref{rouchecercleSALEM} of Rouch\'e 
holds true for $X =+1$,
then it holds true
for all $X \in [-1,+1]$, that is for 
every argument $\psi \in [0, 2 \pi]$,
i.e. for every 
$z \in C_{1,n}$.
As in the proof
of Theorem \ref{cercleoptiMM},
$$
\left|1 + \frac{\pi}{a \,n} e^{i (\psi - \varphi)}
\right|^{n}
=
\exp\Bigl(
\frac{\pi \, X}{a}\Bigr)
\times 
\left(
1 - \frac{\pi^2}{2 a^2 \, n} (2 X^2 -1) 
+O(\frac{1}{n^2})
\right) 
$$
and
$$
\arg\left(
\Bigl(1 + \frac{\pi}{a \, n} e^{i (\psi - \varphi)}
\Bigr)^{n}\right)
=
sgn(\sin(\psi - \varphi))
\times
\left( ~\frac{\pi \, \sqrt{1-X^2}}{a}
[1 -
\frac{\pi \, X}{a \, n}
]
+O(\frac{1}{n^2})
\right)
.$$
Moreover,
$$
\left|1 + \frac{\pi}{a \, n} 
e^{i (\psi - \varphi)}
\right|
=
\left|1 + \frac{\pi}{a \, n} 
(X \pm i \sqrt{1-X^2})
\right|
=1 + \frac{\pi \, X}{a \, n} + O(\frac{1}{n^2}).
$$
with
$$\arg(1 + \frac{\pi}{a \, n} e^{i (\psi - \varphi)})
= 
sgn(\sin(\psi - \varphi)) \times
\frac{\pi \sqrt{1 - X^2}}{a \, n} 
+ O(\frac{1}{n^2}).
$$
For all $n \geq 18$, from
Proposition \ref{zedeUNmodule}, 
we have
\begin{equation}
\label{devopomainSALEM}
|z_{1,n}|
= 1 - \frac{\lo n - \lo \lo n}{n}
+ \frac{1}{n} O \left(
\frac{\lo \lo n}{\lo n}\right).
\end{equation}
from which we deduce the following equality, 
up to $O(\frac{1}{n})$ - terms,
$$
|z_{1,n}|
\,
\left|1 + \frac{\pi}{a \, n} e^{i (\psi - \varphi)}\right|
=
|z_{1,n}| .
$$
Then the left-hand side term of \eqref{rouchecercleSALEM}
is
$$\frac{
\left|z\right|^{2 n -1}}{1 - |z|^{n-1}}
=
\frac{|-1 + z_{1,n}|^2 
\left|1 + \frac{\pi}{a \, n} e^{i (\psi - \varphi)}
\right|^{2 n}}
{|z_{1,n}| \, 
\left|1 + \frac{\pi}{a \, n} e^{i (\psi - \varphi)}\right|
-
|-1 + z_{1,n}| \,
\left|1 + \frac{\pi}{a \, n} e^{i (\psi - \varphi)}\right|^{n}}$$

\begin{equation}
\label{rouchegaucheSALEM}
=
\frac{|-1 + z_{1,n}|^2 
\left(
1 - \frac{\pi^2}{a \, n} (2 X^2 -1) 
\right)
\exp\bigl(
\frac{2 \pi \, X}{a}\bigr)
}
{|z_{1,n}|
\left|1 + \frac{\pi}{a \, n} e^{i (\psi - \varphi)}\right|
-
|-1 + z_{1,n}| \,
\left(
1 - \frac{\pi^2}{2 a \, n} (2 X^2 -1) 
\right)
\exp(
\frac{\pi \, X}{a})
}
\end{equation}
up to
$\frac{1}{n} 
O \left(
\frac{\lo \lo n}{\lo n}
\right)$
-terms (in the terminant).
The right-hand side term of 
\eqref{rouchecercleSALEM} is 
$$\left|-1 + z + z^n \right|
=
\left|
-1 + z_{1,n}\Bigl(1 + \frac{\pi}{n \, a} e^{i (\psi - \varphi)}\Bigr) +
z_{1,n}^{n}
\Bigl(1 + \frac{\pi}{n \, a} e^{i (\psi - \varphi)}
\Bigr)^n
\right|
$$
$$=
\left|
-1 + z_{1,n}
(1 \pm i \frac{\pi \sqrt{1 - X^2}}{a \, n})
(1 + \frac{\pi \, X}{a \, n})
+
(1 - z_{1,n})
\left(
1 - \frac{\pi^2}{2 a^2 \, n} (2 X^2 -1) 
\right)
\right. \hspace{3cm} \mbox{}
$$
\begin{equation}
\label{rouchedroiteSALEM}
\left.
\times
\exp\bigl(
\frac{\pi \, X}{a}
\bigr) \,
\exp\Bigl(
\pm \,
i \,
\Bigl( ~\frac{\pi \, \sqrt{1-X^2}}{a}
[1 - \frac{\pi \, X}{a \, n}] 
\Bigr)
\Bigr)
+ O(\frac{1}{n^2})
\right|
\end{equation}
Let us consider
\eqref{rouchegaucheSALEM}
and
\eqref{rouchedroiteSALEM}
at the first order for the asymptotic expansions, 
i.e. up to $O(1/n)$ - terms instead of
up to 
$O(\frac{1}{n}(\lo \lo n/ \lo n))$ - terms or
$O(1/n^2)$ - terms.
\eqref{rouchegaucheSALEM} becomes:
$$\frac{|-1+z_{1,n}|^2 \exp(\frac{2 \pi X}{a})}
{|z_{1,n}| - |-1+z_{1,n}| \exp(\frac{\pi X}{a})}$$
and \eqref{rouchedroiteSALEM} is equal to:
$$|-1 + z_{1,n}|
\left|
1 -
\exp\bigl(
\frac{\pi \, X}{a}
\bigr) \,
\exp\Bigl(
\pm \,
i \,
\frac{\pi \, \sqrt{1-X^2}}{a} 
\Bigr)
\right|
$$
and is independent of the sign of 
$\sin(\psi - \varphi)$.
Then
the inequality \eqref{rouchecercleSALEM} is 
equivalent to
\begin{equation}
\label{roucheequiv1SALEM}
\frac{|-1+z_{1,n}|^2 \exp(\frac{2 \pi X}{a})}
{|z_{1,n}| - |-1+z_{1,n}| \exp(\frac{\pi X}{a})}
<
|-1+z_{1,n}| \, 
\left|
1 -
\exp\bigl(
\frac{\pi \, X}{a}
\bigr) \,
\exp\Bigl(
\pm \,
i \,
\frac{\pi \, \sqrt{1-X^2}}{a} 
\Bigr)
\right|
,
\end{equation}
and \eqref{roucheequiv1SALEM} to
\begin{equation}
\label{amaximalfunctionXSALEM}
\frac{|-1 + z_{1,n}|}{|z_{1,n}|}
~  < ~ \, 
\frac{\left|
1 -
\exp\bigl(
\frac{\pi \, X}{a}
\bigr) \,
\exp\Bigl(
 \,
i \,
\frac{\pi \, \sqrt{1-X^2}}{a} 
\Bigr)
\right|
\exp\bigl(
\frac{-\pi \, X}{a}
\bigr)}{\exp\bigl(
\frac{\pi \, X}{a}
\bigr) +\left|
1 -
\exp\bigl(
\frac{\pi \, X}{a}
\bigr) \,
\exp\Bigl(
 \,
i \,
\frac{\pi \, \sqrt{1-X^2}}{a} 
\Bigr)
\right|} = \kappa(X,a).
\end{equation}
The right-hand side function
$\kappa(X,a)$ is 
a function of $(X, a)$, 
on $[-1, +1] \times [1, +\infty)$.
which is strictly decreasing for any
fixed $a$, 
and reaches its minimum
at $X=1$; this minimum is always 
strictly positive. 
Consequently 
the inequality of Rouch\'e
\eqref{rouchecercleSALEM} will be satisfied
on $C_{1,n}$ once it is 
satisfied at $X = 1$, as claimed.

Hence, up to
$O(1/n)$-terms, the Rouch\'e condition
\eqref{amaximalfunctionXSALEM}, 
for any fixed $a$,
will be satisfied (i.e. for any
$X \in [-1,+1]$)
by the set of integers
$n = n(a)$ for which $z_{1,n}$
satisfies:
\begin{equation}
\label{amaximalfunctionSALEM}
\frac{|-1 + z_{1,n}|}{|z_{j,n}|} 
< \kappa(1,a) 
=
\frac{\left|
1 -
\exp\bigl(
\frac{\pi}{a}
\bigr)
\right|
\exp\bigl(
\frac{-\pi}{a}
\bigr)}{\exp\bigl(
\frac{\pi}{a}
\bigr) +\left|
1 -
\exp\bigl(
\frac{\pi}{a}
\bigr)
\right|} ,
\end{equation}
equivalently, from Proposition 
\ref{zedeUN},
\begin{equation}
\label{amaximalfunctionnnSALEM}
\frac{\lo n - \lo \lo n}{n} 
< \frac{\kappa(1,a)}
{1 + \kappa(1,a)} .
\end{equation}
In order to 
obtain 
the largest possible range 
of values of $n$,
the value of $a \geq 1$ has to be chosen
such that $a \to \kappa(1,a)$ is maximal
in \eqref{amaximalfunctionnnSALEM}
(Figure \ref{h1a}).
In the proof of Theorem 
\ref{cercleoptiMM} 
we have seen that
the function $a \to \kappa(1,a)$ 
reaches its maximum
$\kappa(1, a_{\max}) := 0.171573\ldots$
at $a_{\max} = 5.8743\ldots$.
We take
$a = a_{\max}$.

The slow decrease
of the functions of the variable $n$
involved in 
the terminants when $n$ tends to infinity,
as a factor of uncertainty on
\eqref{amaximalfunctionnnSALEM},
has to be taken into account
in \eqref{amaximalfunctionnnSALEM}.
It amounts to check 
numerically 
whether \eqref{rouchecercleSALEM} 
is satisfied 
for the small values
$18 \leq n \leq 100$
for $a = a_{\max}$, or not. 
Indeed, for the
large enough values of $n$,
the inequality
\eqref{amaximalfunctionnnSALEM}
is satisfied since
$\lim_{n \to \infty}
\frac{\lo n - \lo \lo n}{n} = 0$.
On the computer, 
the critical threshold
of $n = 32$  
is easily calculated, with
$(\lo 32 - \lo \lo 32)/32 = 
0.0694628\ldots$.
Then
$$\frac{\lo n - \lo \lo n}{n}
<
\frac{\kappa(1,a_{\max})}
{1 + \kappa(1,a_{\max})} = 0.146447\ldots\qquad
\mbox{for all
$n \geq 32$} .$$
Let us note that
the last inequality
also holds for some values of
$n$ less
than $ 32$.
\end{proof}

\begin{corollary}
\label{smallestSALEM}
Let $\beta > 1$ be a reciprocal algebraic 
number
such that
$\dyg(\beta) \geq 32$. Then
the Parry Upper function
$f_{\beta}(z)$
admits a simple zero 
$\omega_{1,n}$ (of modulus
$< 1$) in the 
open disk $D(z_{1,n}, 
\frac{\pi |z_{1,n}|}{n \, a_{\max}})$.
\end{corollary}

\begin{proof}
The polynomial $G_{n}(z)$ has
simple roots.
Since
\eqref{rouchecercleSALEM} is satisfied, 
the Theorem of Rouch\'e states that
$f_{\beta}(z)$
and
$G_{n}(z) = -1 + z + z^n$
have the same number of roots, 
counted with
multiplicities, in the open disk
$D(z_{1,n}, 
\frac{\pi |z_{1,n}|}{n \, a_{\max}})$, giving the existence of an unique zero
$\omega_{1,n}$. 
\end{proof}

\subsection{Identification of the first lenticular root as a conjugate}
\label{S7.2}

Let us prove that the first zero 
$\omega_{1,n}$ of
$f_{\beta}(z)$ is a zero of the minimal
polynomial of $\beta$, using rewriting 
polynomials as in \S \ref{S5.4}. 

\begin{theorem}
\label{secondseriesUU}
Let $\beta > 1$ be a reciprocal algebraic 
number
such that 
$\dyg(\beta) \geq 32$
with ${\rm M}(\beta) < 1.176280\ldots$. 
Then
(i) the first zero 
$\omega_{1,n}$ of
$f_{\beta}(z)$ is
a conjugate of $\beta$,
(ii) the minimal polynomial 
$P_{\beta}(X)$
of $\beta$ 
is fracturable
on the 
union of
the disks
$D_{1,n} =
\{ z \mid |z-z_{1,n}| < 
\frac{\pi |z_{1,n}|}{n \, a_{\max}}\}$
and $\overline{D_{1,n}}$
in the sense that
the analytic function
$U_{\beta}(z) =
-\zeta_{\beta}(z) P_{\beta}(z)$
is a nonconstant holomorphic
function which does not vanish
on this domain,
and that 
\begin{equation}
\label{factoUUSALEM}
P_{\beta}(z) ~=~ (-\zeta_{\beta}(z)
P_{\beta}(z))
\times f_{\beta}(z), \qquad
z \in D_{1,n} \cup 
\overline{D_{1,n}},
\end{equation}
(iii) the function $U_{\beta}(z)
\in \zb[[z]]$  
takes 
the following values at
$\omega_{1,n} \in D_{1,n}$,
resp.
$\overline{\omega_{1,n}} 
\in \overline{D_{1,n}}$,
\begin{equation}
\label{UbetaNonnulenOmegajn}
U_{\beta}(\omega_{1,n})=
\frac{P'_{\beta}(\omega_{1,n})}
{f'_{\beta}(\omega_{1,n})} \neq 0,
\quad
U_{\beta}(\overline{\omega_{1,n}})=
\frac{P'_{\beta}(\overline{\omega_{1,n}})}
{f'_{\beta}(\overline{\omega_{1,n}})} \neq 0,
\end{equation}
(iv) the first zero $\omega_{1,n}
= \omega_{1,n}(\beta)$ of
$f_{\beta}(z)$ is a nonreal complex zero
of modulus $< 1$ of the minimal polynomial
$P_{\beta}(z)$ which is 
a continuous function of
$\beta$.
\end{theorem}

\begin{proof}
The domain of definition
of the meromorphic function
$f_{\beta}(z)$ is $\cb$ 
or the open unit disk
with $|z|=1$ as natural boundary,
according to the Carlson-Polya
dichotomy (Bell and Chen \cite{bellchen},
Carlson \cite{carlson}
\cite{carlson2},
Polya \cite{polya}, 
Szeg\H{o} \cite{szego}). 
In the first case $\beta$ is a Parry number 
and it is a nonParry number in the second case.
In both cases
$f_{\beta}(z)$ is holomorphic in
$D_{1,n} \cup \overline{D_{1,n}}$,
these disk being included in
$|z| < 1$. 
The function
$f_{\beta}(z)$ admits only one simple zero
in each disk
by Corollary \ref{smallestSALEM}.
Let us prove (i) and (ii).
Since $\beta$ is assumed reciprocal,
the power series $f_{\beta}(z)$ is never 
a polynomial and $\beta$ is not a simple Parry
number. Therefore 
$f_{\beta}(z) = -1/\zeta_{\beta}(z)$
so that the quotient
$P_{\beta}(z)/f_{\beta}(z) \in \zb[[z]]$
is equal to
$-\zeta_{\beta}(z) \times P_{\beta}(z)$
on the domain of definition of 
$f_{\beta}(z)$.
Hence the identity
$P_{\beta}(z) ~=~ (-\zeta_{\beta}(z)
P_{\beta}(z))
\times f_{\beta}(z)$ is satisfied
for $|z| < 1$.

As in 
\S \ref{S5.4.2} we can construct
the infinite sequence
of $\beta$-representations of 1
from ``$P_{\beta}$" to ``$f_{\beta}$"
by restoring the digits
of $f_{\beta}$ one after the other.
The corresponding 
rewriting trail starts at
\eqref{equa1P} and ends at the 
R\'enyi $\beta$-expansion of 1
which is given by \eqref{equa1f}.
Then we obtain the analogue 
statement of
Proposition \ref{APversfbeta}:
{\em if $x$ is a zero of 
$P_{\beta}$
in $D_{1,n} \cup
\overline{D_{1,n}}$, then $x$ is a zero
of $f_{\beta}$, that is
$x = \omega_{1,n}$
or
$x = \overline{\omega_{1,n}}$}.
Conversely, 
in order to show that
$x= \omega_{1,n}$ is a zero of $P_{\beta}$,
we can construct 
the rewriting trail
from the $s$-th polynomial
section ``$S_s$" of
$f_{\beta}$ to
``$P_{\beta}$", i.e.
the sequence of 
$\gamma_{s}$-representations of 1
from the R\'enyi
$\gamma_{s}$-expansion of 1 given by
\eqref{gammasexpansionSs}
to
\eqref{equa1Pgammas},
at the unique zero
$\gamma_{s}^{-1} \in D_{0,n}$ of $S_s$,
by restoring the digits
of $1-P_{\beta}(X)$ one after the other.
The important point is that
$s$ should be taken
large enough
by the analogue of
Proposition \ref{Bcyclolenticules}:
{\em 
let $\beta > 1$ be a 
reciprocal algebraic integer having
$\dyg(\beta) \geq 32$.
Let 
$f_{\beta}(x) =
-1 + x + x^n +
x^{m_1} + x^{m_2} + \ldots + x^{m_s} +\ldots$
be the Parry Upper function at $\beta$
and, for $s \geq 0$, denote its $s$-th 
polynomial section by
$$f(x) =
-1 + x + x^n +
x^{m_1} + x^{m_2} + \ldots + x^{m_s}
\quad \mbox{factorized as} ~= A(x) B(x) C(x),$$
where $s \geq 1$, $m_1 - n \geq n-1$, 
$m_{j+1}-m_j \geq n-1$ for $1 \leq j < s$,
where $A$ is the cyclotomic part, 
$B$ the reciprocal noncyclotomic part,
$C$ the nonreciprocal part of $f$.

There exists $s_0$ (depending upon $n$)
such that
the reciprocal
noncyclotomic part $B$
of $f(x)$, if any,
does not vanish on the lenticular root
$\omega_{1,n}$
of $f$, as soon as 
$s \geq s_0$.
}
Then, taking the limit
$s \to \infty$, with
$\lim_{s \to \infty} \gamma_{s}^{-1}
=
\beta^{-1}$, we obtain the analogue
of Proposition \ref{ASversPbeta}:
{\em the lenticular zero
$\omega_{1,n}$ of $f_{\beta}$ 
in $D_{1,n}$ is a zero of $P_{\beta}$}.

Let us prove (iii). It suffices to
take the
derivatives
$$P'_{\beta}(z) = U'_{\beta}(z) \times 
f_{\beta}(z)
+
U_{\beta}(z) \times f'_{\beta}(z)$$
of \eqref{factoUUSALEM} 
at $\omega_{1,n}$
and $\overline{\omega_{1,n}}$.
Let us prove (iv). It is a consequence 
of Corollary
\ref{zeroesParryUpperfunctionContinuity}.
\end{proof}

\subsection{A lower bound for the set of Salem numbers. Proof of Theorem \ref{mainpetitSALEM}}
\label{S7.3}

Assume that $\beta$ is a Salem number
of dynamical degree
$n = \dyg(\beta) \geq 32$. 
Its minimal polynomial $P_{\beta}(X)$
would admit $\beta$, $1/\beta$ as real roots, 
the remaining roots being
on the unit circle, as 
(nonreal) complex-conjugated pairs.
By Theorem \ref{secondseriesUU} 
it would 
admit the pair of nonreal roots
$(\omega_{1,n} , \overline{\omega_{1,n}})$ 
as well, but
$|\omega_{1,n}| < 1$.
This fact is impossible.
We deduce that
$\beta > \theta_{31}^{-1}=
1.08545\ldots$.

\section{Sequences of reciprocal algebraic integers converging to $1^+$ in house and limit equidistribution of conjugates on the unit circle. Proof of Theorem \ref{main_EquidistributionLimitethm}}
\label{S8}

In this Section, given a sequence
$(\alpha_q)$ of reciprocal algebraic integers,
such that $|\alpha_q| > 1$,
$\dyg(\alpha_q) \geq 260$ tending to infinity,
we consider
the limit geometry of all the conjugates
of the $\alpha_q$s. 
The limit equidistribution, 
restricted to an arc of the unit circle, of
the lenticular
conjugates was used in the proof of Theorem
\ref{MahlerMINORANTreal}. 
We generalize this fact to all the conjugates
with limit arc the complete unit circle.
We will make use 
of Belotserkovski's Theorem
\cite{belotserkovski},
recalled below as Theorem
\ref{belotserkovskitheorem},
which prefigurates
Bilu's theorem on 
the $n$-dimensional torus 
\cite{bilu};
the discrepancy function
of equidistribution
given by this theorem is well adapted to 
become a function 
of only the dynamical degree.

\begin{theorem}[Belotserkovski]
\label{belotserkovskitheorem}
Let $F(x) = a \prod_{i=1}^{m} (x- \alpha^{(i)})
\in \cb[x], m \geq 1$,
be a polynomial with roots
$\alpha^{(k)} = r_k e^{i \varphi_k}$,
$0 \leq \varphi_k \leq 2 \pi$. For
$0 \leq \varphi \leq \psi \leq 2 \pi$,
denote
$N_{F}(\varphi , \psi)\! =
Card\{k \mid \varphi \leq \varphi_k \leq \psi\}$.
Let $0 \leq \epsilon, \delta \leq 1/2$
and
$$\sigma_{dis} = 
\max\left(
m^{-1/2} \, \lo (m+1),
\sqrt{- \epsilon \, \lo (\epsilon)},
\sqrt{- \delta \,\lo (\delta)}
\right).$$ 
If
$|r_k - 1| \leq \epsilon$
for
$1 \leq k \leq m$,
and
$|\lo a| \leq \delta m$ are satisfied,
then, for some (universal, in the sense 
that it does not depend upon $F$) constant $C > 0$, 
\begin{equation}
\label{discrep_prorotata}
\left|
\frac{1}{m} N_{F}(\varphi, \psi)
-
\frac{\psi - \varphi}{2 \pi}\right|
\leq 
C\, \sigma_{dis} \qquad
\mbox{for all}~~
0 \leq \varphi \leq \psi \leq 2 \pi .
\end{equation}
\end{theorem}

Let us recall some useful notations.
The multiplicative group of nonzero elements
of $\cb$, resp. $\qb$, is denoted by 
$\cb^{\times}$, resp.
$\qb^{\times}$. The unit Dirac measure
supported at $\omega \in \cb$
is denoted
by $\delta_{\omega}$.
We denote by
$\mu_{\tb}$ the (normalized)
Haar measure
(unit Borel measure), invariant by rotation,
that is supported on the unit circle
$\tb
=
\{z \in \cb \mid |z|=1\}$,
compact subgroup of $\cb^{\times}$,
i.e. with
$\mu_{\tb}(\tb) = 1$.
Given $\alpha \in \overline{\qb}^{\times}$,
of degree
$m = \deg(\alpha)$,
we define
the unit Borel measure
(probability)
$$\mu_{\alpha} =
\frac{1}{\deg(\alpha)}
\sum_{j=1}^{\deg(\alpha)}
\delta_{\sigma(\alpha)}$$
on $\cb^{\times}$, 
the sum being taken over
all $m$ embeddings 
$\sigma : \qb(\alpha) \to \cb$.
A sequence $\{\gamma_s\}$
of points of
$\overline{\qb}^{\times}$
is said to be {\em strict} 
if any proper subgroup
of $\overline{\qb}^{\times}$ contains 
$\gamma_s$ for only finitely many values of
$s$.

Theorem 6.2 in \cite{vergergaugry6}
shows that limit equidistribution of conjugates
occurs on the unit circle
for the sequence of Perron numbers
$\{\theta_{n}^{-1} \mid n = 2, 3, 4, \ldots\}$,
as $\mu_{\theta_{n}^{-1}} \to
\mu_{\tb}, n \to \infty$. 
All these Perron numbers have
a Mahler measure $> \Theta$.
We now give a generalization of this limit
result
to convergent sequences of algebraic integers of 
small Mahler measure,
$< \Theta$, where ``{\em convergence 
to $1$}" 
has to be taken in the sense 
of the ``house".
The Theorem is Theorem 
\ref{main_EquidistributionLimitethm}.

{\it Proof of Theorem \ref{main_EquidistributionLimitethm}}:
(i) Denote generically by
$\alpha \in \mathcal{O}_{\overline{\qb}}$
any
element of 
$(\alpha_{q})_{q \geq 1}$.
Let $m = \deg(\alpha)$
and
$\beta = \house{\alpha} 
\in (\theta_{n}^{-1},
\theta_{n-1}^{-1}), n \geq 260$. 
Using the inequality
\eqref{dygdeg}
 between 
$m$
and the dynamical degree
$n = \dyg(\alpha)
=\dyg(\beta)$,
there exists a constant $c_1 > 0$ such that
$$
\frac{\lo (m+1)}{\sqrt{m}}
\leq c_1 \frac{\lo n}{\sqrt{n}}
.$$
On the other hand, 
the minimal polynomial
$P_{\alpha} = P_{\beta}$
is reciprocal and all its roots
$\alpha^{(k)}$, including $\beta$
and $1/\beta$
by Theorem \ref{divisibilityALPHA},
lie in the annulus
$\{z \mid
\frac{1}{\beta} \leq |z| \leq \beta
\}$.
As a consequence, using 
Theorem \ref{betaAsymptoticExpression},
there exists a constant $c_2 > 0$ such that
$$||\alpha^{(k)}| -1| \leq \epsilon ,
\quad 1 \leq k \leq m,
\quad \mbox{with}~~
\epsilon = c_2 \frac{\lo n}{n} .
$$
We take $\delta = 0$
in the definition of $\sigma$
in Theorem \ref{belotserkovskitheorem}
since $P_{\alpha}$ is monic.
We deduce that the discrepancy function,
i.e. the upper bound
in the rhs of 
\eqref{discrep_prorotata},
is equal to
$C \sigma_{dis} = c_3 \frac{\lo n}{\sqrt{n}}$
for some constant $c_3 > 0$.
Hence,
\begin{equation}
\label{discrep_prorotataALPHA}
\left|
\frac{1}{m} N_{P_{\alpha}}(\varphi, \psi)
-
\frac{\psi - \varphi}{2 \pi}\right|
\leq 
c_3 \frac{\lo n}{\sqrt{n}}
 \qquad
\mbox{for all}~~
0 \leq \varphi \leq \psi \leq 2 \pi .
\end{equation}
The discrepancy function of 
\eqref{discrep_prorotataALPHA} tends to
$0$ if $n$ tends to infinity.
By Theorem \ref{betaAsymptoticExpression}
and Theorem \ref{nfonctionBETA}, for
$1 < \beta < \theta_{260}^{-1}$,
$$\beta~ \to ~1^+ \Longleftrightarrow~
n = \dyg(\beta) \to \infty,$$
so that the sequence of Galois orbit measures
in \eqref{haarmeasurelimite} 
converge for the weak topology
as a function of the dynamical degree.

(ii) The sequence $(\alpha_q)$ is strict
since the sequence
$(\house{\alpha_q})$ only admits 1 as
limit point:
$\limsup_{q \to \infty} \house{\alpha_q}
=
\lim_{q \to \infty}\house{\alpha_q}=1$
and the number 
Card$\{
\alpha_q \in 
(\theta_{n}^{-1} , \theta_{n-1}^{-1})\}$
between two successive Perron numbers of 
$(\theta_{n}^{-1})$, for every $n \geq 3$, 
is finite.
In the space of probality measures
equipped with the weak topology,
the reformulation of 
\eqref{discrep_prorotataALPHA}
means \eqref{haarmeasurelimite}, 
equivalently
\eqref{haarlimitefunvtions}.

\section{Some consequences: Salem numbers, and a Conjecture of Margulis}
\label{S9} 

A first consequence concerns
the difference between two
successive Salem numbers
generating the same number field.
We are in the context of
root separation theorems 
\cite{beresnevichbernikgotze}
\cite{bugeaudmignotte}
\cite{evertse}
\cite{guting}
\cite{mahler2}
and the representability of real
algebraic numbers 
as a difference of two Mahler measures
\cite{drungilasdubickas}.

\begin{theorem}
\label{main_coro_salemnumbersSAMEdegree} 
Suppose that $\tau$ and $\tau'$ are
Salem numbers with $\tau' \in \qb(\tau)$.
Then
$$\tau' > \tau~~
\Longrightarrow ~~\tau' - \tau \geq 
\theta_{31}^{-1} (\theta_{31}^{-1} - 1)
= 0.0927512\ldots$$
\end{theorem}

\begin{proof}
Since $\tau' - \tau = (\tau'/\tau - 1) \tau$
we deduce the lower bound from
\cite{salem}, p. 169,
Proposition 3
in \cite{smyth6}, since
$\tau'/\tau$ is a Salem number, and
Theorem \ref{mainpetitSALEM}.
\end{proof}

\begin{remark}
For each Pisot number $\theta$ the 
``Construction of Salem"
(\cite{guichardvergergaugry}, 
Theorem IV in Salem \cite{salem})
gives two convergent sequences $(\tau'_n)_n$
and $(\tau'_n)_n$ of Salem numbers,
such that $\tau_n < \theta < \tau'_n$
for $n$ large enough, 
and
$$\lim_{n \to \infty} \tau_n =
\theta =\lim_{n \to \infty} \tau'_n.$$
From Theorem \ref{main_coro_salemnumbersSAMEdegree},
since $\lim_{n \to \infty} (\tau'_n - \tau_n)=0$,
we have:
$\qb(\tau_n) \neq \qb(\tau'_n)$ for $n$ large enough.
\end{remark}

Ghate and Hironaka
(in \cite{ghatehironaka} p. 304)
mention that if
Lehmer's Conjecture is true, then
the following
Conjecture of Margulis is also true.

\begin{theorem}[ex-Margulis Conjecture]
\label{CJ13}
Let $G$ be a connected semi-simple group over 
$\rb$,
with 
rank$_{\rb}(G) \geq 2$. Then there is a 
neighbourhood
$U \subset G(\rb)$ of the identity such that for 
any
irreducible cocompact lattice
$\Gamma \subset G(\rb)$, the intersection
$\Gamma \cap U$ consists only
of elements of finite order.
\end{theorem}
Let us give a proof of Theorem 
\ref{CJ13}
from the two 
arguments of Margulis
(\cite{margulis}, Theorem (B) p. 322): 
(i) first,
the arithmeticity Theorem 1.16 in
\cite{margulis}, p. 299, 
and (ii) the following statement
(Margulis \cite{margulis}, p. 322):

{\it Let $P(x) = x^n + a_{n-1} x^{n-1}
+ \ldots + a_0$ be an irreducible monic 
polynomial with integral coefficients.
Denote by
$\beta_{1}(P), \ldots, \beta_{n}(P)$
the roots of $P$ and by $m(P)$ the number 
of those $i$ with $1 \leq i \leq n$
and $|\beta_{i}(P)| \neq 1$.
Then
\begin{equation}
\label{margulisminoration}
{\rm M}(P) = \prod_{1 \leq i \leq n}
\max\{1, |\beta_{i}(P)|\} > d
\end{equation}
where the constant $d > 1$ depends only
upon $m(P)$ (and does not depend upon $n$). 
}
The minorant $d$ is universal and
is given by
Theorem \ref{mainLEHMERtheorem}
(ex-Lehmer's Conjecture), hence the result.

But
the dependency
of the minorant $d$ of
${\rm M}(P)$
in \eqref{margulisminoration},
expected by Margulis,
with the
number of roots $m(P)/2$
lying outside 
the closed unit disk, 
or equivalently inside the open unit disk
(the polynomial $P$ can be assumed
of small Mahler measure $< \Theta$, hence
reciprocal by Smyth's Theorem), 
and
not with the degree $n$
of $P$, is not clear in view of
Theorem \ref{omegajnexistence}, 
Theorem \ref{divisibilityALPHA}
and
Proposition \ref{argumentlastrootJn}. 
Indeed, the following minorant
$$m(P) \geq 2 (1 + J_{\dyg(P)})$$
can only be deduced from the present study,
this minorant being two times
the cardinal of the lenticulus of roots
associated with the dynamical
degree of the house of the polynomial 
$P$; what can be said is that
the degree $n=\deg(P)$
of $P$ is not involved
in this minorant of $m(P)$, and therefore
that the constant $d$ in 
\eqref{margulisminoration}
is likely to depend upon the
dynamical degree $\dyg(P)$.

\section{Appendix}
\label{S10}
 
\subsection{Notations}
\label{S10.1}

Let $P(X) \in \zb[X]$,
$m=\deg(P) \geq 1$.
The {\it reciprocal polynomial} of $P(X)$ 
is $P^*(X)=X^m P(\frac{1}{X})$.
The polynomial $P$ is reciprocal if $P^*(X)=P(X)$.
If $P(X) = a_0 \prod_{j=1}^{m} (X- \alpha_j) 
= a_0 X^m + a_1 X^{m-1}+ \ldots +a_m$, with
$a_i \in \cb$, $a_0 a_m \neq 0$,
and roots $\alpha_j$, the {\it Mahler measure} of $P$ is
\begin{equation}
\label{mahlermeasuredefinition}
{\rm M}(P) := |a_0| \prod_{j=1}^{m} \max\{1, |\alpha_j|\}. 
\end{equation}
The absolute Mahler measure
of $P$ is 
${\rm M}(P)^{1/\deg(P)}$,
denoted by
$\mathcal{M}(P)$.
The Mahler measure of an algebraic number $\alpha$ 
is the Mahler of its minimal polynomial
$P_{\alpha}$: ${\rm M}(\alpha) := {\rm M}(P_{\alpha})$.
For any algebraic number
$\alpha$ the house $\house{\alpha}$ of
$\alpha$ is the maximum modulus of its conjugates, 
including $\alpha$ itself; 
by Jensen's formula the Weil height $h(\alpha)$ of $\alpha$
is $\lo {\rm M}(\alpha)/\deg (\alpha)$.
By its very definition, ${\rm M}(P Q) =
{\rm M}(P) {\rm M}(Q)$ (multiplicativity).
The Mahler measure of a nonzero polynomial
$P(x_1, \ldots, x_n) \in \cb[x_1, \ldots, x_n]$
is defined by
\begin{equation}
\label{mahlermeasuredefinitions}
{\rm M}(P)
:=
\exp\left(
\frac{1}{(2 i \pi)^n}
\int_{\tb^n}
\lo |P(x_1, \ldots, x_n)|
\frac{d x_1}{x_1}\ldots \frac{d x_n}{x_n} 
\right),
\end{equation}
where $\tb^n = \{(z_1,\ldots, z_n)
\in \cb^n \mid |z_1| = \ldots = |z_n|=1\}$ is the unit torus 
in dimension $n$. If $n=1$, 
by Jensen's formula, it is given by
\eqref{mahlermeasuredefinition}.

A {\it Perron number} is either $1$ or 
a real algebraic integer $\theta > 1$
such that the Galois conjugates
$\theta^{(i)}, i \neq 0$, 
of $\theta^{(0)} := \theta$ satisfy: 
$|\theta^{(i)}| < \theta$. 
Denote by $\mathbb{P}$ the set of Perron numbers. 
A {\it Pisot number} is a Perron number $> 1$ for which 
$|\theta^{(i)}| < 1$ for all $i \neq 0$.
The smallest Pisot number
is denoted by
$\Theta = 1.3247\ldots$, dominant root of
$X^3 -X-1$.
A Salem number is
an algebraic integer $\beta > 1$ such that its Galois conjugates
$\beta^{(i)}$ satisfy:
$|\beta^{(i)}| \leq 1$ for all $i=1, 2, \ldots, m-1$,
with $m = {\rm deg}(\beta) \geq 1$,
$\beta^{(0)} = \beta$ and at least one conjugate 
$\beta^{(i)}, i \neq 0$,
on the unit circle.  
All the 
Galois conjugates 
of a Salem number $\beta$ lie on the unit circle, 
by pairs of complex conjugates, except
$1/\beta$ which lies in the open interval $(0,1)$.
Salem numbers are of even degree $m \geq 4$.
The set of Pisot numbers, resp. Salem numbers, 
is denoted by S, resp. by T. 
If $\tau \in$ S or T, then
${\rm M}(\tau) = \tau$.

The set of algebraic numbers, resp. 
algebraic integers, in $\cb$, 
is denoted by
$\overline{\qb}$, resp.
$\mathcal{O}_{\overline{\qb}}$.
The $n$th cyclotomic polynomial is denoted by
$\Phi_{n}(z)$.
The (na\"ive) height
of a polynomial $P$
is the maximum of the absolute value
of the coefficients of $P$.

For $x > 0$, $\lfloor x\rfloor$, $\{x\}$ and
$\lceil x \rceil$ denotes respectively the integer
part, resp. the fractional part,
resp. the smallest integer greater than or equal to $x$.
For $\beta > 1$ any real number, 
the map
$T_{\beta}: [0,1] \to [0,1],
x \to \{\beta x\}$ denotes the $\beta$-transformation.
With
$T_{\beta}^{0} := T_{\beta}$, its
iterates are denoted by 
$T_{\beta}^{(j)} := 
T_{\beta}(T_{\beta}^{j-1})$
for $j \geq 1$. 
A real number $\beta > 1$ 
is a Parry number if 
the sequence
$(T_{\beta}^{(j)}(1))_{j \geq 1}$ 
is eventually periodic;
a Parry number is called simple if in particular 
$T_{\beta}^{(q)}(1) = 0$ for some integer
$q \geq 1$. 
The set of Parry numbers is denoted by
$\pb_{P}$.
The terminology chosen by Parry in \cite{parry} 
has changed: $\beta$-numbers are now 
called Parry numbers, in honor of W. Parry.

For $x > 0$,
$\lo^{+} x$ denotes $\max\{0, \lo x\}$.
Let $\mathcal{F}$ be an infinite subset
of the set of nonzero
algebraic numbers which are not a root of unity;
we say that the {\em Conjecture of Lehmer is true
for $\mathcal{F}$} 
if there exists
a constant $c_{\mathcal{F}} > 0$ such that
${\rm M}(\alpha) \geq 1 + c_{\mathcal{F}}$
for all $\alpha \in \mathcal{F}$.

\subsection{Angular asymptotic sectorization of the roots $z_{j,n}, \omega_{j,n}$, of the Parry Upper functions, in lenticular sets of zeroes -- notations for transition regions}
\label{S10.2}

The Poincar\'e
asymptotic expansions of the roots $z_{j,n}$ 
of $G_{n}(z)=-1+z+z^n$,
lying in the first quadrant of $\cb$, 
are
divergent formal series of functions of the 
couple of {\it two variables} which is:
\begin{itemize}
\item[$\bullet$] $\displaystyle \bigl( n , \frac{j}{n} \bigr),$ 
\quad \quad in the angular sector: \quad 
$\displaystyle \frac{\pi}{2} >~ \arg z ~>~ 2 \pi \frac{\lo n}{n}$,
\item[$\bullet$] $\displaystyle \bigl( n , \frac{j}{\lo n} \bigr),$ \hspace{0.1cm} in the angular sector 
(``bump" sector):
\, $\displaystyle 2 \pi \frac{\lo n}{n} ~>~ \arg z ~\geq~ 0$. 
\end{itemize}
 
\noindent
In the bump sector (cusp sector of Solomyak's fractal
$\mathcal{G}$,
$\S$ \ref{S3.2}), the roots $z_{j,n}$
are dispatched into the two
subsectors:

\begin{itemize}
\item[$\bullet$] $ 2 \pi \frac{\sqrt{(\lo n) (\lo \lo n)}}{n} ~>~ \arg z ~>~ 0$,
\item[$\bullet$] $ 2 \pi \frac{\lo n}{n} ~>~ \arg z ~>~ 2 \pi \frac{\sqrt{(\lo n) (\lo \lo n)}}{n}$.
\end{itemize}

The relative angular size of the bump sector, as
$(2 \pi \frac{\lo n}{n})/(\frac{\pi}{2})$,
tends to zero, 
as soon as $n$ is large enough.
By transition region, we mean
a small neighbourhood
of the argument :
$$\arg z  = 2 \pi \frac{\lo n}{n} \quad {\rm or~ of} \quad 
2 \pi \frac{\sqrt{(\lo n) (\lo \lo n)}}{n}.$$
Outside these two transition regions, a
dominant
asymptotic expansion of $z_{j,n}$ exists.
In a transition region an asymptotic expansion
contains more $n$-th order terms 
of the same order of magnitude ($n=2, 3, 4$).
These two neighbourhoods are defined as follows.
Let $\epsilon \in (0, 1)$ small enough. 
Two strictly increasing  sequences
of real numbers $(u_n), (v_n)$
are introduced, which satisfy: 
$$ 
\lfloor n/6 \rfloor ~>~ v_n ~>~ \lo n, \quad
\lo n ~>~ u_n ~>~ \sqrt{(\lo n) (\lo \lo n)},
\quad {\rm for}~ n \geq n_0 = 18,$$
such that
$$\lim_{n \to \infty}
\frac{v_n}{n} ~=~
\lim_{n \to \infty}
\frac{\sqrt{(\lo n) (\lo \lo n)}}{u_n}
~=~
\lim_{n \to \infty} \frac{u_n}{\lo n}
~=~ \lim_{n \to \infty} \frac{\lo n}{v_n} ~=~ 0$$
and
\begin{equation}
\label{unvndifference}
v_n - u_n = O( (\lo n)^{1+\epsilon} )
\end{equation}
with the constant 1 involved in the big O.
The roots $z_{j,n}$ lying in the first transition region about $2 \pi (\lo n)/n$ are such that:
$$2 \pi \frac{v_n}{n} ~>~ \arg z_{j,n} ~>~ 2 \pi \frac{(2 \lo n - v_n)}{n},$$
and the
roots $z_{j,n}$ lying in the second transition region
about $\frac{2 \pi \sqrt{(\lo n) (\lo \lo n)}}{n}$
are such that:
$$2 \pi \, \frac{u_n}{n} ~>~ \arg z_{j,n} ~>~ 2 \pi \, \frac{2 \sqrt{(\lo n) (\lo \lo n)} - u_n}{n}.$$

In Proposition \ref{zjjnnExpression}, for simplicity's sake, these two transition regions are schematically denoted by 
$$\arg z \asymp 2 \pi \, \frac{(\lo n)}{n}
\quad
\mbox{resp.}\quad
\arg z \asymp 2 \pi \, \frac{\sqrt{(\lo n)(\lo \lo n)}}{n}.$$
By complementarity, the other sectors are schematically written:
$$
2 \pi \, \frac{\sqrt{(\lo n) (\lo \lo n)}}{n} ~>~ \arg z ~>~ 0$$
instead of 
$$2 \pi \, \frac{2 \sqrt{(\lo n) (\lo \lo n)} - u_n}{n} ~>~ \arg z ~>~ 0;$$
resp.
$$
2 \pi \, \frac{\lo n}{n} ~>~ \arg z ~>~
2 \pi \, \frac{\sqrt{(\lo n) (\lo \lo n)}}{n}
$$
instead of 
$$2 \pi \, \frac{2 \lo n  - v_n}{n} ~>~ \arg z ~>~ 2 \pi \, \frac{u_n}{n};$$
resp. 
$$
\frac{\pi}{2} ~>~ \arg z ~>~ 2 \pi \, \frac{\lo n}{n}   
\quad \mbox{instead of} \quad 
\frac{\pi}{2} ~>~ \arg z ~>~ 2 \pi \, \frac{v_n}{n}.$$

\section*{Acknowledgements}

The author wishes to express special
thanks to the referees for 
valuable comments, critics, and suggestions.
It is also a pleasure to
thank J.-W. M. van Ittersum for
asking pertinent questions to the 
author.

\end{document}